\documentclass[11pt,leq,english]{smfart}
\usepackage{smfthm}
\usepackage{amssymb}
\usepackage{amsbsy}
\usepackage{amscd}
\usepackage{arydshln}
\usepackage[mathscr]{eucal}
\usepackage{nicefrac}

\usepackage[dvipdf]{graphicx}
\usepackage[all]{xy}

\makeatletter
\renewcommand{\thesubsubsection}{\arabic{subsubsection})}
\makeatother
\usepackage{verbatim}
\usepackage{version}
\usepackage{color}
\newenvironment{NB}{
\color{red}{\bf NB}. \footnotesize
}{}
\newenvironment{NB2}{
\color{blue}{\bf NB}. \footnotesize
}{}
\excludeversion{NB}
\excludeversion{NB2}
\newtheorem{Theorem}[equation]{Theorem}
\newtheorem{Corollary}[equation]{Corollary}
\newtheorem{Lemma}[equation]{Lemma}
\newtheorem{Proposition}[equation]{Proposition}

\theoremstyle{definition}
\newtheorem{Definition}[equation]{Definition}
\newtheorem{Example}[equation]{Example}

\newtheorem{Conjecture}[equation]{Conjecture}

\theoremstyle{remark}
\newtheorem{Remark}[equation]{Remark}
\newtheorem{Remarks}[equation]{Remarks}
\newtheorem*{Claim}{Claim}
\numberwithin{equation}{subsection}
\newcommand{\thmref}[1]{Theorem~\ref{#1}}
\newcommand{\secref}[1]{\S\ref{#1}}
\newcommand{\lemref}[1]{Lemma~\ref{#1}}
\newcommand{\propref}[1]{Proposition~\ref{#1}}
\newcommand{\corref}[1]{Corollary~\ref{#1}}
\newcommand{\subsecref}[1]{\S\ref{#1}}

\newcommand{\remref}[1]{Remark~\ref{#1}}

\newcommand{\defeq}{\overset{\operatorname{\scriptstyle def.}}{=}}
\newcommand{\CC}{{\mathbb C}}
\newcommand{\Z}{{\mathbb Z}}
\newcommand{\Q}{{\mathbb Q}}
\newcommand{\RR}{{\mathbb R}}
\newcommand{\proj}{{\mathbb P}}
\newcommand{\SL}{\operatorname{\rm SL}}

\newcommand{\algsl}{\operatorname{\mathfrak{sl}}} % because \sl="slant"

\newcommand{\gl}{\operatorname{\mathfrak{gl}}}

\newcommand{\g}{{\mathfrak g}}
\newcommand{\h}{{\mathfrak h}}
\newcommand{\Spec}{\operatorname{Spec}\nolimits}

\newcommand{\End}{\operatorname{End}}
\newcommand{\Hom}{\operatorname{Hom}}
\newcommand{\Ext}{\operatorname{Ext}}
\newcommand{\Ker}{\operatorname{Ker}}
\newcommand{\Coker}{\operatorname{Coker}}
\newcommand{\Ima}{\operatorname{Im}}

\newcommand{\codim}{\mathop{\text{\rm codim}}\nolimits}
\newcommand{\rank}{\operatorname{rank}}

\newcommand{\id}{\operatorname{id}}
\newcommand{\ve}{\varepsilon}
\newcommand{\topdeg}{{\operatorname{top}}} % top degree
\newcommand{\Wedge}{{\textstyle \bigwedge}}
\newcommand{\Mreg}{M^{\operatorname{reg}}}
\newcommand{\shfO}{\mathcal O}
\newcommand{\bA}{\mathbf A}
\newcommand{\MR}[1]{}
\newcommand{\linf}{\ell_\infty}

\newcommand{\fT}{\mathfrak T}
\renewcommand{\AA}{{\mathbb A}}

\newcommand{\cL}{\mathcal L}
\newcommand{\Uh}[2][G]{{\mathcal U}_{#1}^{#2}}
\newcommand{\Bun}[2][G]{\operatorname{Bun}_{#1}^{#2}}
\newcommand{\UhL}[1]{\Uh[L]{#1}}
\newcommand{\BunL}[1]{\operatorname{Bun}_L^{#1}}
\newcommand{\UhGl}[1]{\Uh[{G,\lambda}]{#1}}
\newcommand{\BunGl}[1]{\operatorname{Bun}_{G,\lambda}^{#1}}

\newcommand{\BunLl}[1]{\operatorname{Bun}_{L,\lambda}^{#1}}
\newcommand{\UhP}[1]{\Uh[P]{#1}}
\newcommand{\UhPm}[1]{\Uh[{P_-}]{#1}}

\newcommand{\Gi}[1]{\widetilde{\mathcal U}_{#1}}

\newcommand{\cUh}[2][G]{{}^c\Uh[#1]{#2}}
\newcommand{\cBun}[2][G]{{}^c\Bun[#1]{#2}}
\newcommand{\cGi}[1]{{}^c\Gi{#1}}

\newcommand\ZZ{\mathbb Z}
\newcommand\grh{\mathfrak h}
\newcommand\grg{\mathfrak g}

\newcommand{\scr}{\mathscr}
\newcommand{\cC}{\mathcal C}

\newcommand{\Perv}{\operatorname{Perv}}
\newcommand{\W}[2]{W^{(#1)}_{#2}}
\newcommand{\bW}[2]{\widehat W^{(#1)}_{#2}}
\newcommand{\mW}[2]{\widetilde{W}^{(#1)}_{#2}}
\newcommand{\sW}[2]{w^{(#1)}_{#2}}
\newcommand{\GG}{\mathbb G}
\newcommand{\TT}{\mathbb T}
\newcommand{\blam}{{\boldsymbol\lambda}}
\newcommand{\bmu}{{\boldsymbol\mu}}
\newcommand{\ba}{{\boldsymbol a}}
\newcommand{\mL}{\widetilde{L}}

\newcommand{\bigzerol}{\smash{\hbox{\Huge 0}}}
\newcommand{\bigzerou}{\lower1.7ex\hbox{\Huge 0}}
\newcommand{\normal}[1]{\mbox{\large\bf:}#1\mbox{\large\bf:}}
\newcommand{\bara}{{\bar a}}
\newcommand{\barb}{{\bar b}}
\newcommand{\dst}{d_{\mathrm{st}}}
\newcommand{\bF}{\mathbf F}
\newcommand{\cohdeg}{{}^{\text{c}}\!\operatorname{deg}}
\newcommand{\bB}{\mathbf B}
\newcommand{\bR}{\mathbf R}

\newcommand{\bn}{{\boldsymbol n}}
\newcommand{\bm}{{\boldsymbol m}}

\newcommand\alp{\alpha}
\newcommand\aff{\operatorname{aff}}
\newcommand\PP{\mathbb{P}}
\newcommand\calZ{\mathcal{Z}}

\newcommand\x{\times}

\newcommand{\pol}{\delta}
\newcommand{\Heis}{\mathfrak{Heis}}
\newcommand{\Vir}{\mathfrak{Vir}}
\newcommand{\IC}{\operatorname{IC}}
\newcommand{\br}{\mathbf r}
\newcommand{\fa}{\mathfrak a}
\newcommand{\ft}{\mathfrak t}
\newcommand{\Stab}{\operatorname{Stab}}

\newcommand{\Aut}{{\operatorname{Aut}}}
\newcommand{\scF}{\mathscr F}
\newcommand{\scG}{\mathscr G}
\newcommand{\scH}{\mathscr H}

\newcommand{\cA}{\mathcal A}
\newcommand{\hR}{\mathcal I}

\newcommand\Sym{\operatorname{Sym}}
\newcommand{\gr}{\operatorname{gr}}
\newcommand{\la}{\langle}
\newcommand{\ra}{\rangle}
\newcommand{\fU}{\mathfrak U}

\newcommand{\scW}{\mathscr W}

\usepackage{url}

\makeindex

\newcommand{\texorpdfstring}[2]{#1}

\begin{document}

\title[Instanton moduli spaces and $\scW$-algebras]
{Instanton moduli spaces and $\scW$-algebras
%{\rm Preliminary Version (\today)}
}
\author[A.~Braverman]{Alexander Braverman}
\address{Department of Mathematics, Brown University, 151 Thayer st.,
Providence, Rhode Island 02912, USA}
\curraddr{Department of Mathematics, University of Toronto and Perimeter
Institute of Theoretical Physics, Waterloo, Ontario, Canada, N2L 2Y5}
\email{braval@math.toronto.edu}
\author[M.~Finkelberg]{Michael Finkelberg}
\address{National Research University Higher School of Economics,
Russian Federation,
Department of Mathematics, 6 Usacheva st., Moscow 119048;
Skolkovo Institute of Science and Technology}
\email{fnklberg@gmail.com}
\author[H.~Nakajima]{Hiraku Nakajima}
\address{Research Institute for Mathematical Sciences,
Kyoto University, Kyoto 606-8502,
Japan}
\email{nakajima@kurims.kyoto-u.ac.jp}

\dedicatory{In memory of Kentaro Nagao}

%\subjclass[2010]{Primary 14C05; Secondary 14D21, 14J60}


\altkeywords{AGT, cohomologie d'intersection \'equivariante, espace de
  Gieseker, alg\`ebre de Heisenberg, restriction hyperbolique, espace
  de modules d'instantons, forme de Kac-Shapovalov, accouplement de
  Poincar\'e, espace d'Uhlenbeck, enveloppe stable, W-alg\`ebre,
  vecteur de Whittaker}

\begin{abstract}
We describe the (equivariant) intersection cohomology of certain moduli spaces (``framed Uhlenbeck spaces") together with some structures on them (such as e.g.\ the Poincar\'e pairing) in terms of representation theory of some vertex operator algebras (``$\scW$-algebras").
\end{abstract}

\begin{altabstract}
    Nous d\'ecrivons la cohomologie d'intersection (\'equivariante) de
    certains espaces de modules (les ``espaces d'Uhlenbeck
    encadr\'es'') ainsi que certaines structures sur ces espaces de
    cohomologie (comme par exemple l'accouplement de dualit\'e de
    Poincar\'e) en termes de la th\'eorie des repr\'esentations de certaines
    alg\`ebres vertex (les ``$\scW$-alg\`ebres'').
\end{altabstract}

\maketitle

\begin{NB}
\begin{verbatim}
\newcommand{\Uh}[2][G]{{\mathcal U}_{#1}^{#2}}
\end{verbatim}

\verb+\Uh{d}+ yields $\Uh{d}$.
\end{NB}

\setcounter{tocdepth}{2}
\tableofcontents

\begin{NB}
\subsection*{To do list}

\begin{itemize}
\item Complete \secref{sec:whittaker-state}. Give the proof of
  Conjectures~\ref{c1} : {\color{red}$\checkmark$}
\item Write introduction : SB {\color{red}$\checkmark$}
\item Check the proof of \lemref{lem:Fock} : SB,MF
\item Give the reference to Laumon's theorem at line \ref{Laumon} : SB {\color{red}$\checkmark$}
\item Give the reference to the dimension estimate at lines
  \ref{dim_est1}, \ref{dim_est2} : HN {\color{red}$\checkmark$}
\item Generalize the description of irreducible components of
  $\Uh[P,0]d$ in \propref{prop:Un} to arbitrary $G$. This is not
  necessary now, but may be useful for type $B_2$, $G_2$ : ALL {\color{red}$\checkmark$}
\item Give the proof of \lemref{lem:inter=pair} : ALL {\color{red}$\checkmark$}
\item Study the additional condition (2) in \lemref{lem:cochar} : SB,MF
\item Check the computation \eqref{eq:Vira} : SB,MF
\item Check the computation \eqref{eq:92} : SB,MF
\item Check the assertion \eqref{eq:32} : SB,MF
\item Check \subsecref{sec:highest-weight} : SB,MF
\item Check \subsecref{sec:kac-shapovalov-form} : SB,MF
\item Use character formula to complete the proof of
  \propref{prop:lambda}, \propref{prop:span} : Either SB or MF, as it
  depends on \cite{BFG}. {\color{red}$\checkmark$}
\item Fill the gap in the proof of \propref{prop:diff_proof} : ALL {\color{red}$\checkmark$}
\item Give the proof of \propref{prop:SpanP} : ALL {\color{red}$\checkmark$}
\item Give the proof of \lemref{lem:geomind} : ALL {\color{red}$\checkmark$}
\item Check HN's comments in \secref{sec:append-exactn} : SB {\color{red}$\checkmark$}
\item Write the proof of the embedding of $\scW_\bA(\g)$ into the
  $\scW$-algebra of Levi : SB {\color{red}$\checkmark$}
\end{itemize}
\end{NB}

\newcommand\calW{\mathscr W}
\newcommand\grn{\mathfrak n}
\newcommand\fin{\operatorname{fin}}
\newcommand\calA{\mathcal A}
\newcommand\calF{\mathcal F}
\newcommand\calG{\mathcal G}
\newcommand\IH{\operatorname{IH}}
\section{Introduction}

The main purpose of this paper is to describe the (equivariant) intersection cohomology of certain moduli spaces (``framed Uhlenbeck spaces")
together with some structures on them (such as e.g.\ the Poincar\'e pairing) in terms of representation theory of some vertex
operator algebras (``$\calW$-algebras"). In this introduction we first briefly introduce the relevant geometric and algebraic objects
(cf.\ Subsections \ref{int-geom} and \ref{int-alg}) and then state our main result (in a somewhat weak form) in Subsection \ref{int-main-local} (a more precise version is discussed in \ref{int-main-global}). In Subsection \ref{int-motiv} we discuss the motivation for our results and relate them to some previous works.
In \subsecref{int-tech} we mention earlier works from which we obtain
strategy and techniques of the proof.

\subsection{Uhlenbeck spaces}\label{int-geom}

Let $G$\index{G@$G$ (group)} be an almost simple simply-connected algebraic
group over $\CC$ with Lie algebra $\mathfrak g$. Let also $\mathfrak h$ be a Cartan subalgebra of $\mathfrak g$.

Let $\Bun{d}$\index{Bun@$\Bun{d}$} be the moduli space of algebraic
$G$-bundles over the projective plane $\proj^2$ (over $\mathbb C$)
with the instanton number $d$ and with trivialization at the line at
infinity $\linf$.\index{linfinity@$\linf$ (line at infinity)} It is a
non-empty smooth quasi-affine algebraic variety of dimension
$2dh^\vee$ for $d\in\Z_{\ge 0}$, where $h^\vee$ is the dual Coxeter
number of $G$.

By results of Donaldson \cite{MR763753} (when $G$ is classical) and
Bando \cite{Bando} (when $G$ is arbitrary) $\Bun{d}$ is homeomorphic
to the moduli space of anti-self-dual connections (instantons) on
$S^4$ modulo gauge transformations $\gamma$ with $\gamma(\infty) = 1$
where the structure group is the maximal compact subgroup of $G$.  We
will use an algebro-geometric framework, as we can use various tools.

It is well-known that $\Bun{d}$ has a natural partial compactification $\Uh{d}$\index{Uh@$\Uh{d}$}, called the Uhlenbeck space.
Set-theoretically, $\Uh{d}$ can be described as follows:
$$
\Uh{d}=\bigsqcup\limits_{0\leq d'\leq d}\Bun{d'}\times S^{d-d'}(\AA^2),
$$
where $S^{d-d'}(\AA^2)$ denotes the corresponding symmetric power of the affine plane $\AA^2$.

The variety $\Uh{d}$ is affine and it is always singular unless $d=0$. It has a natural action of the group $G\times GL(2)$, where
$G$ acts by changing the trivialization at $\linf$ and $GL(2)$ just acts on $\proj^2$ (preserving $\linf$). In what follows,
it will be convenient for us to restrict ourselves to the action of $\GG=G\times \CC^*\times \CC^*$ where $\CC^*\times \CC^*$ is the diagonal subgroup of $GL(2)$.

\begin{Remark}
    The compactification of the moduli space of instantons on a
    compact $C^\infty$ $4$-manifolds, as a topological space, was
    introduced by Donaldson, based on the earlier fundamental work by
    Uhlenbeck. See \cite[Notes to Section 4.4.1]{MR1079726} for
    further historical comments. This construction works for any
    compact Lie group, i.e., any reductive group $G$, and also the
    case when we take the quotient only by gauge transformations
    $\gamma$ with $\gamma(\infty) = 1$ as above.

    A construction as an affine variety was given in \cite{BFG}, which
    is one of our main references. See \remref{rem:Gieseker} for
    comments in type $A$.
\end{Remark}

\subsection{Main geometric object}\label{int-ih}
The main object of our study on the geometric side is the
$\GG$-equivariant intersection cohomology
$\IH^*_{\GG}(\Uh{d})$.\index{IHG@$\IH^*_{\GG}(\Uh{d})$} By the
definition, it is endowed with the following structures:

\medskip
1) It is a module over $H^*_{\GG}(\mathrm{pt})$. The latter algebra can be canonically identified with the algebra of polynomial functions
on $\mathfrak h\times \CC^2$ which are invariant under $W$, where $W$ is the Weyl group of $G$. In what follows we shall denote this ring by $\bA_G$; let also $\bF_G$\index{FFG@$\bF_G = \CC(\ve_1,\ve_2,\ba)^W$} denote its field of fractions.
We shall typically denote an element of $\grh\times \CC^2$ by
$(\ba,\varepsilon_1,\varepsilon_2)$.

2) There exists a natural symmetric (Poincar\'e) pairing $\IH^*_{\GG}(\Uh{d})\underset {\bA_G}\otimes \IH^*_{\GG}(\Uh{d})\to \bF_G$
(this follows from the fact that $(\Uh{d})^{T\times \CC^2}$ consists of one point).

3) For every $d\geq 0$ we have a canonical unit cohomology class $|1^d\rangle\in \IH^*_{\GG}(\Uh{d})$.\index{1d@$\vert 1^d\rangle$}

\medskip
\noindent
The main purpose of this paper is to describe the above structures in terms of representation theory.
To formulate our results, we need to introduce the main algebraic player -- the $\calW$-algebra.

\subsection{Main algebraic object: $\calW$-algebras}\label{int-alg}
In this subsection we recall some basic facts and constructions from the theory of $\calW$-algebras (cf.\ \cite{F-BZ} and references therein). First, we need to recall the notion of Kostant-Whittaker reduction for finite-dimensional Lie algebras.

Let $\grg$ be as before a simple Lie algebra over $\CC$ with the universal enveloping algebra $U(\grg)$. Let us choose a triangular decomposition
$\grg=\grn_+\oplus\grh\oplus \grn_-$ for $\grg$. Let $\chi\colon\grn_+\to \CC$ be a non-degenerate character of $\grn_+$, i.e.\ a Lie algebra
homomorphism such that $\chi|_{\grn_{+,i}}\neq 0$ for every vertex $i$ of the Dynkin diagram of $\grg$ (here $\grn_{+,i}$ denotes the corresponding simple root subspace). Then we can define the finite $\calW$-algebra of $\grg$ (to be denoted by $\calW_{\fin}(\grg)$) as
the quantum Hamiltonian reduction of $U(\grg)$ with respect to $(\grn_+,\chi)$. In other words,
we have
$$
\calW_{\fin}(\grg)=\Hom_{U(\grg)}(U(\grg)\underset{U(\grn_+)}\otimes\CC_{\chi},U(\grg)\underset{U(\grn_+)}\otimes\CC_{\chi}).
$$

\noindent
A well-known result of Kostant~\cite[Theorem~2.4.2]{MR0507800} asserts that

(1f) $\calW_{\fin}(\grg)$ is naturally isomorphic to the center ${\mathcal Z}(\grg)$ of $U(\grg)$.

In particular, we have

(2f) The algebra $\calW_{\fin}(\grg)$ has a natural embedding into $S(\grh)$, whose image coincides with the algebra
$S(\grh)^W$.

(3f) The algebra $\calW_{\fin}(\grg)$ is a polynomial algebra in some variables $F^{(1)},...,F^{(\ell)}$, where $\ell=\operatorname{rank}(\grg)$.
Each $F^{(\kappa)}$ is homogeneous as an element of $S(\grh)^W$ of some degree $d_{\kappa}+1\geq 2$.

(4f) The algebra $\calW_{\fin}(\grg)$ is isomorphic to the algebra $\calW_{\fin}(\grg^{\vee})$.

\noindent
Feigin and Frenkel (cf.\ \cite{F-BZ} and references therein) have generalized the above results to the case of affine Lie algebras.
Namely, let $\grg((t))$ denote the Lie algebra of $\grg$-valued formal loops.
It has a natural central extension
$$
0\to \CC\to \hat{\grg}\to\grg((t))\to 0
$$
(this extension depends on a choice of an invariant form on $\grg$ which we choose so that the the squared length of every short
coroot is equal to 2).
\begin{NB}
    Long roots have square length 2.
\end{NB}%
The group $\CC^*$ acts naturally on $\hat{\grg}$ by ``loop rotation" and the same is true for its Lie algebra $\CC$.
We let $\grg_{\aff}$ be the semi-direct product of $\hat{\grg}$ and $\CC$ (for the above action).

For every $k\in \CC$ one can consider the algebra $U_k(\hat{\grg})$ --- this is the quotient of $U(\hat{\grg})$ by the ideal generated
by $\mathbf{1}-k$ where $\mathbf{1}$ denotes the generator of the central $\CC\subset \grg_{\aff}$.
Let us also extend $\chi$ to $\grn_+((t))$ by taking the composition of the residue map $\grn_+((t))\to \grn_+$ with $\chi:\grn_+\to \CC$. Abusing slightly the notation, we shall denote this map again by $\chi$.

The $\calW$-algebra $\calW_k(\grg)$\index{Wkg@$\scW_k(\g)$} is roughly
speaking the Hamiltonian reduction of $U_k(\hat{\grg})$ with respect
to $(\grn_+((t)),\chi)$.  However, the reader must be warned that
rigorously this reduction must be performed in the language of vertex
operator algebras; in particular, $\calW_k(\grg)$ is a vertex operator
algebra (cf.\ again \cite{F-BZ} for the relevant definitions).

Unlike in the finite case, the algebra $\calW_k(\grg)$ is usually non-commutative (unless $k=-h^{\vee}$).
The main results of Feigin and Frenkel about $\calW_k(\grg)$ can be summarized as follows (notice the similarities between
(1f)-(4f) and (1w)-(4w)):
%----------------------------------------------------------------------------
%\begin{Theorem}\label{FF-intro}\leavevmode
%\begin{enumerate}
%  \item

(1w)
  The algebra $\calW_{-h^{\vee}}(\grg)$ can be naturally identified with the center of the \textup(vertex operator algebra version of\textup)
  $U_{-h^{\vee}}(\hat{\grg})$.

%  \item
(2w)
  Let $\Heis(\h)$ denote the central extension of $\grh((t))$ corresponding to the bilinear form on $\grh$ chosen above. Abusing the notation we shall use the same symbol for the corresponding vertex operator algebra. Also for any $k\in \CC$ we can consider the corresponding algebra $\Heis_k(\grh)$ \textup(``Heisenberg algebra of level $k$''\textup).\footnote{Note that for all $k\neq 0$ these algebras are isomorphic.}
  Then for generic $k$ there exists a canonical embedding $\calW_k(\grg)\hookrightarrow \Heis_{k+h^{\vee}}(\grh)$.
%  \item

(3w)
  The algebra $\calW_k(\grg)$ is generated \textup(in the sense of \cite[15.1.9]{F-BZ}\textup)
  by certain elements $W^{(\kappa)}$, $\kappa=1,\cdots,\ell$ of conformal dimension $d_\kappa+1$. This \textup(among other things\textup)
  means that for every module $M$ over $\calW_k(\grg)$ and every $\kappa=1,\ldots,\ell$ there is a well defined
  field $Y(W^{(\kappa)},z)=\sum\limits_{n\in \ZZ} W^{(\kappa)}_n z^{-n-d_{\kappa}-1}$ where $W^{(\kappa)}_n$ can be regarded as a linear endomorphism
  of $M$.

%    \item 
(4w)  
  Suppose $k$ is generic.  There is a natural isomorphism
  $\scW_k(\grg)\simeq \scW_{k^{\vee}}(\grg^{\vee})$ where
  $(k+h^{\vee}_{\grg})(k^{\vee}+h^{\vee}_{\grg^{\vee}})=r^{\vee}$
  where $r^{\vee}$ is the lacing number of $\grg$ \textup(i.e.\ the
  maximal number of edges between two vertices of the Dynkin diagram
  of $\grg$\textup). We shall call this isomorphism \emph{the
  Feigin-Frenkel duality}.

  % \end{enumerate}
  % \end{Theorem}

The representation theory of $\scW_k(\grg)$ has been extensively studied (cf.\ for example \cite{Arakawa2007}).
In particular, to any $\lambda\in\mathfrak h^*$ one can attach a Verma module $M(\lambda)$\index{Mlambda@$M(\lambda)$} over $\calW_k(\grg)$ and $M(\lambda_1)$ is isomorphic to $M(\lambda_2)$ if $\lambda_1+\rho$ and $\lambda_2+\rho$ are on the same orbit of the Weyl group. This module carries
a natural (Kac-Shapovalov) bilinear form, with respect to which the  operator $W^{(\kappa)}_n$ is conjugate to $W^{(\kappa)}_{-n}$ (up to sign).
This module can be obtained as the Hamiltonian reduction of the corresponding Verma module for $\grg$.
%--------------------------------------------------------------------------------------------------
\subsection{The main result: localized form}\label{int-main-local}
Let us set
$$
M^d_{\bF_G}(\ba)=\IH^*_{\GG}(\Uh{d})\underset {\bA_G}\otimes \bF_G;\quad M_{\bF_G}(\ba)=\bigoplus\limits_{d=0}^{\infty} M^d_{\bF_G}(\ba).
\index{MFG@$M_{\bF_G}(\ba)$}
$$
It is easy to see that $M^d_{\bF_G}(\ba)$ is also naturally isomorphic to  $\IH^*_{\GG,c}(\Uh{d})\underset{\bA_G}\otimes\bF_G$ where the subscript
$_c$ stands for cohomology with compact support.

Let us also set
$$
k=-h^{\vee}-\frac{\ve_2}{\ve_1}.
$$

Then (a somewhat weakened) form of our main result is the following:
%--------------------------------------------------------------------------------
\begin{Theorem}\label{main-intro-loc}
Assume that $G$ is simply laced and let us identify $\grh$ with $\grh^*$ by means of the invariant form such that $(\alp,\alp)=2$ for
every root of $\grg$. Then
there exists an action
of the algebra $\calW_k(\grg)$ on $M_{\bF_G}(\ba)$ such that
%\begin{enumerate}

  \textup{(1)} %\item
  The resulting module is isomorphic to the Verma module $M(\lambda)$ over $\calW_k(\grg)$ where
  $$
  \lambda=\frac{\ba}{\ve_1}-\rho
  $$
  \textup(here we take $\bF_T =
  \operatorname{Frac}(H^*_T(\mathrm{pt}))$\index{FFT@$\bF_T = \CC(\ve_1,\ve_2,\ba)$|textit} as our field of
  scalars\textup).

  \textup{(2)} %\item 
  Under the above identification a twisted Poincar\'e pairing
  on $M_{\bF_G}(\ba)$ goes over to the Kac-Shapovalov form on
  $M(\lambda)$.  \textup(The twisting will be explained in
  \subsecref{sec:kac-shapovalov-form}.\textup)
  \begin{NB}
      Corrected on Sep.~14, 2014
  \end{NB}%

  \textup{(3)} %\item
  Under the above identification the grading by $d$ corresponds to the grading by eigenvalues of $L_0$.
  \item
  Let $d\ge 1$, $n > 0$. We have
\begin{equation}\label{eq:Whittaker-intro}
  W^{(\kappa)}_n |1^d\rangle =
  \begin{cases}
     \pm \ve_1^{-1}\ve_2^{-h^{\vee}+1}|1^{d-1}\rangle &
    \text{if $\kappa=\ell$ and $n=1$},
    \\
    0 & \text{otherwise}.
  \end{cases}
\end{equation}
%\end{enumerate}
\end{Theorem}

\begin{Remarks}
1) We believe that the sign in (\ref{eq:Whittaker-intro}) is actually always $``+"$, however, currently we don't know how to eliminate the sign issue. Note, however, that (\ref{eq:Whittaker-intro}) still defines the scalar product $\langle 1^d|1^d\rangle$
unambiguously. Also (assuming that the above sign issue can be settled) it follows from (\ref{eq:Whittaker-intro}) that if we formally
set $w=\sum_d |1^d\rangle$ then we have
$$
W^{(\kappa)}_n (w) =
  \begin{cases}
   \ve_1^{-1}\ve_2^{-h^{\vee}+1} w &
    \text{if $\kappa=\ell$ and $n=1$},
    \\
    0 & \text{otherwise}.
\end{cases}
$$

Sometimes we shall write $w_{\ba,\ve_1,\ve_2}$ to emphasize the dependence on the corresponding parameters.

2) The assumption that $G$ is simply laced is essential for Theorem \ref{main-intro-loc} to hold as stated. However, we believe that
a certain modified version of Theorem \ref{main-intro-loc} holds in the non-simply laced case as well, although at the moment we don't
 have a proof of this modified statement (cf.\ subsection \ref{non-simply-laced} for a brief discussion of the non-simply laced case).

3) Since $\Uh{d}$ is acted on by the full $GL(2)$ and not just by $\CC^*\times \CC^*$, it follows that the vector space $M_{\bF_G}(\ba)$ has a natural automorphism which induces the involution $\ve_1\leftrightarrow \ve_2$ on $\bF$ (and leaves $\ba$ untouched).
Note that changing $\ve_1$ to $\ve_2$ amounts to changing $k=-h^{\vee}-\frac{\ve_2}{\ve_1}$ to $k^{\vee}=-h^{\vee}-\frac{\ve_1}{\ve_2}$
and we have $(k+h^{\vee})(k^{\vee}+h^{\vee})=1$. Note also that we are assuming that $\grg$ is simply laced, so
$\grg$ is isomorphic to $\grg^{\vee}$ and the above geometrically defined automorphism is in fact a corollary of the Feigin-Frenkel duality (cf.\ (1w)--(4w)).
%Theorem \ref{FF-intro}).
\end{Remarks}

\begin{NB}
%----------------------------------------------------------------
\subsection{Relation to previous works}\label{int-motivNB}
There are several motivations for the above result. One comes from the so called Alday-Gaiotto-Tachikawa conjecture (cf.\ \cite{AGT} for $G=SL(2)$ and \cite{ABCDEFG} for general $G$) in mathematical physics, which claims certain duality between $N=2$ super-symmetric gauge theory on ${\mathbb R}^4$
and the 2-dimensional Toda conformal field theory. Theorem \ref{main-intro-loc} might probably be viewed as a mathematical corollary of that physical conjecture, although we are not aware of a reference with the physical derivation of something like \ref{main-intro-loc}. However, it is worthwhile to note that the pairing $\langle 1^d|1^d\rangle$ is sometimes called ``the equivariant integral of 1 over $\Uh{d}$"
and their generating function
\begin{equation}
Z(Q,\ba,\ve_1,\ve_2)=\sum\limits_{d=0}^{\infty} Q^d \langle 1^d|1^d\rangle
\end{equation}
is called ``the instanton part of the Nekrasov partition function for pure $N=2$ gauge theory".
It follows from \ref{main-intro-loc} that
$$
Z(Q,\ba,\ve_1,\ve_2)=\langle Q^{-L_0} w_{\ba,\ve_1,\ve_2}|w_{\ba,\ve_1,\ve_2}\rangle.
$$
This statement appears as a conjecture in \cite{Gaiotto} for $G=SL(2)$ and in \cite{ABCDEFG} for general $G$.

In addition, there are now many purely mathematical results in the spirit of Theorem \ref{main-intro-loc}. First, the fact that the character of  $M_{\bF_G}(\ba)$ is equal to the character of a Verma module over $\calW_k(\grg)$ follows from the main result of
\cite{BFG}. In addition, the paper \cite{BraInstantonCountingI} by the first-named author contains a result, which is very similar to Theorem \ref{main-intro-loc} (but technically much simpler). Namely, in the situation of {\em loc. cit.} on the representation theory side one deals with the affine
Lie algebra $\grg_{\aff}$ instead of the corresponding $W$-algebra, and on the geometric side one needs to replace the Uhlenbeck
spaces $\Uh{d}$ by certain {\em flag Uhlenbeck spaces}. In fact, it is important to note that when the original group $G$ is not simply laced, the main result of \cite{BraInstantonCountingI} relates the equivariant intersection cohomology of the flag Uhlenbeck spaces for
the group $G$ with the representation theory of the affine Lie algebra $\grg_{\aff}^{\vee}$, whose root system is dual to that
of $\grg_{\aff}$. A somewhat simpler construction exists also for the finite-dimensional Lie algebra $\grg^{\vee}$ -- in that case on the geometric side one has to work with the so called space of {\em based quasi-maps into the flag variety of $\grg$}, also known as Zastava
spaces (cf.\ \cite{Bra-icm} for a survey on these spaces). In \cite{BFFR} a similar result is conjectured (and proved in type A) for
{\em finite $W$-algebras} for $\grg^{\vee}$ (in that case on the geometric side one works with the so called parabolic Zastava spaces - cf.\ again \cite{Bra-icm} for the relevant definitions).

Last but not least we must mention the fundamental work \cite{MO}, which contains (among many other things) a result which is essentially
equivalent to Theorem \ref{main-intro-loc} in the case when $G=SL(r)$. More precisely, when $G=SL(r)$, the Uhlenbeck spaces $\Uh{d}$ have a natural semi-small symplectic resolution $\Gi{r}^d\to \Uh{d}$; here $\Gi{r}^d$ is the moduli space of torsion-free sheaves
on $\proj^2$ together with a trivialization at $\proj^1_{\infty}$ of generic rank $r$ and of second Chern class $r$.
The authors of \cite{MO} work with the equivariant cohomology of $\Gi{r}^d$ rather than with equivariant intersection cohomology
of $\Uh{d}$, which is slightly bigger. As a result on the representation theory side they get a Verma module over
$\calW(\gl_r)$ (this algebra is isomorphic to the tensor product of $\calW(\mathfrak{s}\mathfrak{l}_r)$
with a (rank 1) Heisenberg algebra). We should also mention that we use the construction of \cite{MO} for $r=2$ in a crucial way for the proof of Theorem \ref{main-intro-loc}.

Let us also mention the work \cite{SV}
where the authors prove a result similar to the above result of \cite{MO} by a different method.
\end{NB}

\subsection{Relation to previous works}\label{int-motiv}

We discuss previous works related to the above result here and later
in \subsecref{int-tech}. This subsection is devoted for those works
related to statements themselves, and \subsecref{int-tech} is for
those which give us a strategy and techniques of the proof.

\begin{NB}
    I do not refer any physical papers earlier than AGT, nor even
    mention there are.
\end{NB}

First we discuss the statements (1),(2),(3).
There are many previous works in almost the same pattern: We consider
moduli spaces of instantons or variants on complex surfaces, and their
homology groups or similar theory.
Then some algebras similar to affine Lie algebras act on direct sums
of homology groups, where we sum over various Chern classes.

The first example of such a result was given by the third-named author
\cite{Na-quiver,Na-alg}. The $4$-manifold is $\CC^2/\Gamma$ for a
nontrivial finite subgroup $\Gamma\subset SU(2)$, and the gauge group
is $U(r)$. The direct sum of homology groups of symplectic resolutions
of Uhlenbeck spaces, called {\it quiver varieties\/} in more general
context, is an integrable representation of the affine Lie algebra
$\mathfrak g_{\Gamma,\aff}$ of level $r$.
Here $\mathfrak g_\Gamma$ is a simple Lie algebra of type $ADE$
corresponding to $\Gamma$ via the McKay correspondence, and
$\mathfrak g_{\Gamma,\aff}$ is its affine Lie algebra.

This result nicely fitted with the $S$-duality conjecture on the
modular invariance of the partition function of $4d$ $N=4$
supersymmetric gauge theory by Vafa-Witten \cite{Vafa-Witten}, as
characters of integrable representations are modular forms.
It was understood that the correspondence \cite{Na-quiver,Na-alg}
should be understood in the framework of a duality in string theories
\cite{Vafa2}. There are lots of subsequent developments in physics
literature since then.

In mathematics, the case $\Gamma = \{e\}$ was subsequently treated by
\cite{MR1441880} and Grojnowski \cite{MR1386846} for $r=1$, and by
Baranovsky \cite{Baranovsky} for general $r$. The corresponding
$\mathfrak g_{\Gamma,\aff}$ is the Heisenberg algebra, i.e., the
affine Lie algebra associated with the trivial Lie algebra $\gl_1$, in
this case.

For $\Gamma=\{e\}$, the symplectic resolution $\Gi{r}^d\to
\Uh{d}$\index{UhTilder@$\Gi{r}^d$} of the Uhlenbeck space $\Uh{d}$ is
given by the moduli space of torsion-free sheaves on $\proj^2$
together with a trivialization at $\linf$ of generic rank
$r$ and of second Chern class $d$. We call it {\it the Gieseker
  space\/} in this paper.
For general $\Gamma$, we have its variant. All have description
in terms of representations of quivers by variants of the ADHM
description, and hence are examples of quiver varieties.
(See \remref{rem:Gieseker} for historical comments.)

\begin{NB}
    The integrable representation becomes irreducible if we only take
    middle degree homology groups (for $\Gamma\neq\{e\}$).
\end{NB}

This result was extended to an action of the quantum toroidal algebra
$\mathbf U_q(\mathbf L\mathfrak g_{\Gamma,\aff})$ on the equivariant
K-theory of the moduli spaces when $\Gamma\neq\{e\}$
\cite{Na-qaff,MR1989196}. A variant for equivariant homology groups
was given by Varagnolo \cite{Varagnolo}.

In all these works, the action was given by introducing
correspondences in products of moduli spaces, which give generators of
the algebra. In particular, the constructions depend on {\it good\/}
presentations of algebras.
The case $\Gamma=\{e\}$ was studied much later, as we explain below,
as the corresponding algebra, which would be $\mathbf U_q(\mathbf
L({\gl}_1)_{\aff})$, was considerably more difficult.

\begin{NB}
    Misha, I add your result.
\end{NB}

Let us also mention that the second-named author with Kuznetsov
\cite{MR1799936} constructed an action of the affine Lie algebra
$\widehat{\gl}_r$ on the homology group of moduli spaces of parabolic
sheaves on a surface, called {\it flag Gieseker spaces\/} or {\it
  affine Laumon spaces\/} when the surface is $\proj^2$, the parabolic
structure is put on a line and the framing is added. (Strictly
speaking, the action was constructed on the homology group of the
fibers of morphisms from flag Gieseker spaces to flag Uhlenbeck
spaces. The action for the whole variety is constructed much later by
Negut \cite{Negut} in the equivariant K-theory framework.)

Let us turn to works on the inner product $\la 1^d|1^d\ra$, which
motivate the statement (4). It is given by the equivariant integration
of $1$ over $\Uh{d}$, and their generating function
\begin{equation}\label{Nekrasov function}
   Z(Q,\ba,\ve_1,\ve_2)=\sum\limits_{d=0}^{\infty} Q^d \langle 1^d|1^d\rangle
\end{equation}
is called ``the instanton part of the Nekrasov partition function for
pure $N=2$ supersymmetric gauge theory" \cite{Nekrasov}.
This partition function has been studied intensively in both
mathematical and physical literature. In particular, a result, which
is very similar to Theorem~\ref{main-intro-loc}(1)$\sim$(4) (but
technically much simpler) was proved by the first-named author
\cite{BraInstantonCountingI}. Namely, in the situation of
\cite{BraInstantonCountingI} on the representation theory side one
deals with the affine Lie algebra $\grg_{\aff}$ instead of the
corresponding $\scW$-algebra, and on the geometric side one needs to
replace the Uhlenbeck spaces $\Uh{d}$ by
{\em flag Uhlenbeck spaces} $\calZ^{\alp}_G$\index{ZalphaG@$\calZ^\alpha_G$}.
In fact, it is important to note that when the original
group $G$ is not simply laced, the main result of
\cite{BraInstantonCountingI} relates the equivariant intersection
cohomology of the flag Uhlenbeck spaces for the group $G$ with the
representation theory of the affine Lie algebra $\grg_{\aff}^{\vee}$,
whose root system is dual to that of $\grg_{\aff}$.
A somewhat simpler construction exists also for the finite-dimensional Lie algebra $\grg^{\vee}$ -- in that case on the geometric side one has to work with the so called space of {\em based quasi-maps into the flag variety of $\grg$}, also known as Zastava
spaces (cf.\ \cite{Bra-icm} for a survey on these spaces).
\begin{NB}
    \cite{BFFR} = [BFFR] is later than \cite{AGT} = [AGT]. Therefore
    this part should be moved later, I think.
\end{NB}

The Nekrasov partition functions are equal for $\Uh{d}$ and for flag
Uhlenbeck spaces at $\ve_2 = 0$, and it is enough for some purposes,
say to determine Seiberg-Witten curves, but they are different in
general. Therefore it was clear that we must replace
$\grg_{\aff}^\vee$ by something else, but we did not know what it is.

A breakthrough was given in a physics context by
Alday-Gaiotto-Tachikawa \cite{AGT} (AGT for short). They conjectured
that the partition functions for $G=SL(2)$ with four fundamental
matters and adjoint matters are conformal blocks of the Virasoro
algebra.
They provided enough mathematically rigorous evidence, say numerical
checks for small instanton numbers. They also give physical intuition
that this correspondence is coming from an observation that $4d$ $N=2$
supersymmetric gauge theories are obtained by compactifying the $6d$
theory on a Riemann surface: the Virasoro algebra naturally lives on
the Riemann surface, which cannot be directly seen from the $4d$ side.
They also guessed that the Virasoro algebra is replaced by the
$\scW$-algebra for a group $G$ of type $ADE$.

There is a large literature in physics after AGT, especially for type
$A$. We do not give the list, though those works are implicitly
related to ours. We mention only one which was most relevant for us,
it is \cite{ABCDEFG} by Keller et al, where the statement (4) was
written down for the first time for general $G$. (There is an earlier
work by Gaiotto for $G=SL(2)$ \cite{Gaiotto}, and various others for
classical groups.)

Around the same time when \cite{AGT} appeared in a physics context,
there was an independent advance on the understanding of the algebra
$\mathbf U_q(\mathbf L({\gl}_1)_{\aff})$ acting on the K-theory of
resolutions of Uhlenbeck spaces of type $A$ by Feigin-Tsymbaliuk
\cite{MR2854154} and Schiffmann-Vasserot \cite{MR3018956}. They
noticed that that $\mathbf U_q(\mathbf L({\gl}_1)_{\aff})$ is
isomorphic to various algebras, which had been studied in different
contexts: a Ding-Iohara algebra, a shuffle algebra with the wheel
conditions, the Hall algebra for elliptic curves, and an algebra
studied by Miki \cite{MR2377852}.
Combined with the AGT picture, we understand that $\mathbf U_q(\mathbf
L({\gl}_1)_{\aff})$ is the limit of the deformed $\scW(\algsl_r)$, or
$\scW(\gl_r)$ by the reason explained below, when $r\to\infty$.

In \cite{BFFR} a similar result is conjectured (and proved in type A) for
{\em finite $W$-algebras} associated with a nilpotent element $e\in \grg^{\vee}$, which is principal in some Levi subalgebra (in that case on the geometric side one works with the so called parabolic Zastava spaces - cf. \cite{Bra-icm} for the relevant definitions).

Finally Maulik-Okounkov \cite{MO} and Schiffmann-Vasserot \cite{SV} proved
Theorem~\ref{main-intro-loc} in the case when $G=SL(r)$.
More precisely, they work with the equivariant cohomology of
$\Gi{r}^d$ rather than with equivariant intersection cohomology of
$\Uh{d}$, which is slightly bigger. As a result on the representation
theory side they get a Verma module over $\calW(\gl_r)$ (this algebra
is isomorphic to the tensor product of
$\calW(\mathfrak{s}\mathfrak{l}_r)$ with a (rank 1) Heisenberg
algebra). We should also mention that we use the construction of
\cite{MO} for $r=2$ in a crucial way for the proof of Theorem
\ref{main-intro-loc}.

\begin{Remark}\label{rem:Gieseker}
    Gieseker constructed a moduli space of semistable sheaves on a
    projective surface \cite{MR466475}. A morphism from Gieseker's
    moduli space to Uhlenbeck compactification was constructed by Li
    and Morgan \cite{MR1205451,MR1231956}. See \cite[Ch.~8]{MR2665168}
    as a modern reference.

    There is an alternative approach for the case of bundles with
    trivialization over $\proj^2$: The ADHM description \cite{ADHM} of
    instantons on $S^4$ describes the moduli space as a space of
    certain linear maps modulo the action of the unitary group. The
    Uhlenbeck space naturally arises by dropping an open condition,
    and considering a larger space (see
    \cite[Ch.~3]{MR1079726}). Furthermore this description is an
    affine algebro-geometric quotient \cite{MR763753}, and one can
    introduce a GIT quotient by perturbing the stability condition
    \cite[Ch.~3]{Lecture}. It gives the moduli space of torsion free
    sheaves with trivialization. The morphism from Gieseker space to
    Uhlenbeck space is also naturally defined.
\end{Remark}

\subsection{Hyperbolic restriction}
One of the main technical tools used in the proof of Theorem~\ref{main-intro-loc} is the notion of {\em hyperbolic restriction}.
Let us recall the general definition of this notion.

Let $X$ be an algebraic variety endowed with an action of $\CC^*$. Then $X^{\CC^*}$ is a closed subvariety of $X$.
Let $\calA_X$ denote the corresponding attracting set. Let $i\colon X^{\CC^*}\to\calA_X$ and $j\colon \calA_X\to X$ be the natural embeddings.
Then we have the functor $\Phi=i^*j^!$\index{UZphi@$\Phi$ (hyperbolic restriction functor)} from the derived category of constructible sheaves on $X$ to the derived
category of constructible sheaves on $X^{\CC^*}$. This functor has been extensively studied by
Braden in \cite{Braden}. In particular, the main result of \cite{Braden} says that $\Phi$ preserves the semi-simplicities of complexes.

Assume that we have a symplectic resolution $\pi\colon Y\to X$ in the sense of \cite{MO} and assume in addition that the above $\CC^*$-action
lifts to $Y$ preserving the symplectic structure. Let $\calF=\pi_*\CC_Y[\dim X]$ (where $\CC_Y$  denotes the constant sheaf
on $Y$). Then we have
\begin{Theorem}\label{nakajima-symplectic}
%    \begin{enumerate}
      \textup{(1)} %\item 
      \cite{VV2}
        $\Phi(\calF)$ is isomorphic to $\pi_*\CC_{Y^{\CC^*}}[\dim X^{\CC^*}]$.
        
       \textup{(2)} %\item 
       Maulik-Okounkov's stable envelope \cite{MO} gives us a
        choice of an isomorphism in \textup{(1)}.
    %\end{enumerate}
%$\Phi(\calF)$ is naturally isomorphic to $\pi_*\CC_{Y^{\CC^*}}[\dim X^{\CC^*}]$.
\end{Theorem}

See \cite{tensor2} for the proof. Though both $\Phi(\calF)$ and
$\pi_*\CC_{Y^{\CC^*}}[\dim X^{\CC^*}]$ are isomorphic semi-simple
perverse sheaves, the proof of \cite{VV2} only gives us a canonical
filtration on the former whose associated graded is canonically
isomorphic to the latter. Then the stable envelope \cite{MO} gives us a
choice of a splitting.

Now we specialize the above discussion to the following situation. Let $P\subset G$ be a parabolic subgroup of $G$ with Levi subgroup $L$. Let us choose a subgroup $\CC^*\subset Z(L)$ (here $Z(L)$ stands for the center of $L$) such that the fixed point set of its adjoint action on $P$ is $L$ and the attracting set is equal to all of $P$. Let now $X=\Uh{d}$. We denote by $\UhL{d}$ the fixed point set of
the above $\CC^*$ on $\Uh{d}$ and by $\UhP{d}$ the corresponding attracting set.
It is easy to see that if $L$ is not a torus, then $\UhL{d}$ is just homeomorphic to ${\mathcal U}_{[L,L]}^d$ (and if $L$ is a torus, then
$\UhL{d}$ is just $S^d(\CC^2)$). See \subsecref{subsec:fixed}.
 Often we are going to drop the instanton number $d$ from
the notation, when there is no fear of confusion. We let
$i$\index{i@$i$ (inclusion $\UhL{d}\to\UhP{d}$)} and $p$\index{p@$p$
  (projection $\UhP{d}\to\UhL{d}$)} denote the corresponding maps from
$\UhL{}$ to $\UhP{}$ and from $\UhP{}$ to $\UhL{}$. Also we denote by
$j$\index{j@$j$ (inclusion $\UhP{d}\to \Uh{d}$)} the embedding of
$\UhP{}$ to $\Uh{}$. We have the
diagram
\begin{equation}
    \label{eq:1intro}
    \UhL{} \overset{p}{\underset{i}{\leftrightarrows}}
    \UhP{} \overset{j}{\rightarrow} \Uh{},
\end{equation}

Thus we can consider the corresponding hyperbolic restriction functor
$\Phi_{L,G}=i^* j^!$\index{UZphiLG@$\Phi_{L,G}$} (note that the
functor actually depends on $P$ and not just on $L$, but
it %is easy to see that it doesn't
depend on the choice of $\CC^*\subset Z(L)$ made above, as we will
explain in \subsecref{sec:hyperb-restr}).

The following is one of the main technical results used in the proof of Theorem~\ref{main-intro-loc}:
\begin{Theorem}\label{perverse-intro}
%\begin{enumerate}
 \textup{(1)} %\item
  Let $P_1\subset P_2$ be two parabolic subgroups and let $L_1\subset L_2$ be the corresponding Levi subgroups. Then
  we have a natural isomorphism of functors $\Phi_{L_1,G}\simeq \Phi_{L_1,L_2}\circ \Phi_{L_2,G}$.

    \textup{(2)} %\item
  For $P$ and $L$ as above the complex $\Phi_{L,G}(\IC(\Uh{d}))$ is perverse and semi-simple.
  Moreover, the same is true for any semi-simple perverse sheaf on $\Uh{d}$ which is constructible with respect to the natural stratification.
%\end{enumerate}
\end{Theorem}
Note that when $G=SL(r)$, it is easy to deduce Theorem~\ref{perverse-intro} from \thmref{nakajima-symplectic}(1), since in this case
the scheme $\Uh{d}$ has a symplectic resolution $\Gi{r}^d$.

%-------------------------------------------------------------------------------------------------------------------------------

\subsection{Sketch of the proof}\label{sketch}
The proof of Theorem \ref{main-intro-loc} will follow the following plan:

1) Replace $\GG=G\times \CC^*\times \CC^*$-equivariant cohomology with $\TT=T\times \CC^*\times \CC^*$-equivariant cohomology.
Note that the former is just equal to the space of $W$-invariants in the latter, so if we define an action
of $\calW_k(\grg)$ on $\oplus \IH^*_{\TT}(\Uh{d})\underset {\bA_T}\otimes \bF_T$
\index{IHT@$\IH^*_{\TT}(\Uh{d})$}
(where $\bA_T=H^*_{T\times \CC^*\times \CC^*}(\mathrm{pt})$
and $\bF_T$ is its field of fractions) and check that if commutes with the action of $W$, we get an action of $\calW_k(\grg)$ on
$\oplus_d \IH^*_{\GG}(\Uh{d})\underset {\bA_G}\otimes \bF_G$.

2) We are going to construct an action of $\Heis_{k+h^{\vee}}(\grh)$ on $\oplus_d \IH^*_{\TT}(\Uh{d})\underset {\bA_T}\otimes \bF_T$
and then get the action of $\calW_k(\grg)$ by using the embedding $\calW_k(\grg)\hookrightarrow \Heis_{k+h^{\vee}}(\grh)$.
It should be noted that the above $\Heis_{k+h^{\vee}}(\grh)$-action will have several ``disadvantages" that will disappear
when we restrict ourselves to $\calW_k(\grg)$. For example, this action will depend on a certain auxiliary choice (a choice of a Weyl
chamber).
\begin{NB}
I do not agree with this.

Also it will not have any nice compatibility properties with the Poincar\'e pairing on $\oplus_d \IH^*_{\TT}(\Uh{d})\underset {\bA_T}\otimes \bF_T$.
\end{NB}

3) The action of the Heisenberg algebra on $\oplus_d \IH^*_{\TT}(\Uh{d})\underset {\bA_T}\otimes \bF_T$ will be constructed
in the following way. Let us choose a Borel subgroup $B$ containing the chosen maximal torus $T$.
We can identify $\oplus_d \IH^*_{\TT}(\Uh{d})\underset {\bA_T}\otimes \bF_T$ with $\oplus_d H^*_{\TT}\Phi_{T,G}(\IC(\Uh{d}))\underset {\bA_T}\otimes \bF_T$, so it is enough to define an action of the Heisenberg algebra on the latter. For this it is enough to define
the action of $\Heis(\CC\alp_i^{\vee})$ for every simple coroot $\alp_i^{\vee}$ of $G$ (and then check the corresponding relations).
Let $P_i$ denote the corresponding sub-minimal parabolic subgroup containing $B$.
Let also $L_i$ be its Levi subgroup (it is canonical after the choice of $T$). Note that $[L_i,L_i]\simeq SL(2)$.
Using the isomorphism
$\Phi_{T,G}(\IC(\Uh{d}))\simeq \Phi_{T,L_i}\circ\Phi_{L_i,G}(\IC(\Uh{d}))$ and Theorem~\ref{nakajima-symplectic}, we define
the action of $\Heis(\CC\alp_i^{\vee})$ on $\oplus_d H^*_{\TT}\Phi_{T,G}(\IC(\Uh{d}))\underset {\bA_T}\otimes \bF_T$ using the results
of \cite{MO} for $G=SL(2)$.

Here it is important for us to write down $\Phi_{L_i,G}(\IC(\Uh{d}))$
in terms of $\IC(\Uh[L_i]{d'})$ ($d'\le d$) and local systems on
symmetric products in a `canonical' way.
In particular, we need to construct a base in the multiplicity space of
$\IC(\Uh[L_i]{d'})$ in $\Phi_{L_i,G}(\IC(\Uh{d}))$.
For $G=SL(r)$, this follows from the stable envelope, thanks to
Theorem~\ref{nakajima-symplectic}(2).
For general $G$, this argument does not work, and we use the
factorization property of Uhlenbeck spaces together with the special
case $G=SL(2)$. A further detail is too complicated to be explained in
Introduction, so we ask an interested reader to proceed to the main text.

4) We now need to check the relations between various $\Heis(\CC\alp_i^{\vee})$. For this we have two proofs. One reduces it again to
the results of
\cite{MO} for $G=SL(3)$ (note that since we assume that $G$ is simply laced, any connected rank 2 subdiagram of the Dynkin diagram of $G$ is of type $A_2$). The other goes through the theory of certain ``geometric" $R$-matrices (cf.\  Chapter~\ref{sec:R-matrix}).
The proof of assertions (2) and (3) of Theorem~\ref{main-intro-loc} is more or less straightforward. The proof of assertion (4) is more technical and we are not going to discuss it in the Introduction. Let us just mention that for that proof we need a stronger form of the first 3 statements of Theorem~\ref{main-intro-loc} which is briefly discussed below.
%---------------------------------------------------------------------------------------------------------------------\ref{main-intro-loc}

\subsection{Relation to previous works -- technical parts}\label{int-tech}

Let us mention previous works which give us a strategy and techniques of the proof.

First of all, we should mention that the overall framework of the
proof is the same as those in \cite{MO,SV}. We realize the
Feigin-Frenkel embedding of $\scW_k(\mathfrak g)$ into
$\Heis_{k+h^\vee}(\mathfrak h)$ in a geometric way via the fixed point
$(\Uh{d})^{\CC^*} = \Uh[L]{d}$,
as is explained the geometric realization in 3),4) in
\subsecref{sketch}.
This was first used in \cite{MO,SV} for type $A$.

What we do here is to replace the equivariant homology of Gieseker
spaces $\Gi{r}^d$ by intersection cohomology of $\Uh{d}$ as the former
exists only in type $A$.
Various foundational issues were discussed in the joint work of the
first and second-named authors with Gaitsgory \cite{BFG}.
In particular, the fact that the character of $M_{\bF_G}(\ba)$ is
equal to the character of a Verma module over $\calW_k(\grg)$ follows
from the main result of \cite{BFG}.
(For type A, it was done earlier in the joint work of the third-named
author with Yoshioka. See \cite[Exercise~5.15]{Lecture} and its
solution in \cite{MR2095899}.)

A search of a replacement of Maulik-Okounkov's stable envelope
\cite{MO} was initiated by the third-named author \cite{tensor2}.
In particular, the relevance of the hyperbolic restriction functor
$\Phi$ and the statement \thmref{nakajima-symplectic}(2) were found.
Therefore our technical aim is to find a `canonical' isomorphism
between $\Phi_{L,G}(\IC(\Uh{d}))$ and a certain perverse sheaf on $\Uh[L]d$.

Let us also mention that \thmref{nakajima-symplectic}(1) was proved
much earlier by Varagnolo-Vasserot \cite{VV2} in their study of quiver
varieties. The functor $\Phi$ realized tensor products of
representations of $\mathfrak g_{\Gamma,\aff}$.  (Strictly speaking,
only quiver varieties of finite types were considered in \cite{VV2}. A
slight complication occurs for quiver varieties of affine types which
give $\mathfrak g_{\Gamma,\aff}$. See \cite[Remark~1]{tensor2} for
detail.)

When we do not have a symplectic resolution like $\Gi{d}^r$, we need another tool to analyze $\Phi$.
Fortunately the hyperbolic restriction functor was studied by
Mirkovi\'c-Vilonen \cite{MV,MV2} in the context of the geometric
Satake isomorphism, which asserts the category of
$G(\CC[[t]])$-equivariant perverse sheaves on the affine Grassmannian
$\operatorname{Gr}_G = G(\CC((t)))/G(\CC[[t]])$ is equivalent to the
category of finite dimensional representations of the Langlands dual
$G^\vee$ of $G$ as tensor categories.
The hyperbolic restriction functor realizes the restriction from
$G^\vee$ to its Levi subgroup.

In particular, it was proved that $\Phi$ sends perverse sheaves to
perverse sheaves. This was proved by estimating dimension of certain
subvarieties of $\operatorname{Gr}_G$, now called Mirkovi\'c-Vilonen
cycles. The proof of \thmref{perverse-intro} is given in the same
manner, replacing Mirkovi\'c-Vilonen cycles by attracting sets of the
$\CC^*$-action.

It is clear that we should mimic the geometric Satake isomorphism from
the conjecture of the first and second-named authors
\cite{braverman-2007} which roughly says the following: it is
difficult to make sense of perverse sheaves on the double affine
Grassmannian, i.e., the affine Grassmannian
$\operatorname{Gr}_{G_{\aff}}$ for the affine Kac-Moody group
$G_{\aff}$. But perverse sheaves on $\Uh{d}$ (and more generally
instanton moduli spaces on $\CC^2/\Gamma$ with $\Gamma = \ZZ/k\ZZ$)
serve as their substitute. Then they control the representation theory
of $G_{\aff}^\vee$ at level $k$.

This conjecture nicely fits with the third-named author's works
\cite{Na-quiver,Na-alg} on quiver varieties via I.~Frenkel's
level-rank duality for the affine Lie algebra of type $A$
\cite{MR675108}. Namely in the correspondence between moduli spaces
and representation theory, the gauge group determines the rank, and
$\Gamma$ the level respectively in the double affine Grassmannian. And
the role is reversed in quiver varieties.

In \cite{BF2}, the first and second-named authors proposed a functor,
acting on perverse sheaves, which conjecturally gives tensor products
of $G_{\aff}^\vee$. This proposal was checked in \cite{Na-branching}
for type $A$, by observing that the same functor gives the branching
from $\mathfrak g_{\Gamma,\aff}$ to the affine Lie algebra of a Levi
subalgebra. The interchange of tensor products and branching is
again compatible with the level-rank duality.

Here in this paper, tensor products and branching appear in the
opposite side: The hyperbolic restriction functor $\Phi$ realizes the
tensor product in the quiver variety side, as we mentioned
above. Therefore it should correspond to branching in the dual affine
Grassmannian side. This is a philosophical explanation why the study of
analog of Mirkovi\'c-Vilonen cycles is relevant here.

\subsection{The main result: integral form}\label{int-main-global}
The formulation of Theorem~\ref{main-intro-loc} has an obvious drawback: it is only formulated in terms
of {\em localized} equivariant cohomology.
\begin{NB}
I do not agree:

     In fact, localization with respect to $\ba$ is not very essential,
but
\end{NB}%
First of all,
it is clear that as stated Theorem~\ref{main-intro-loc} only has a chance to work over
the the localized field $\bF=\CC(\ve_1,\ve_2)$
\index{FF@$\bF = \CC(\ve_1,\ve_2)$|textit}
rather than over $\bA=\CC[\ve_1,\ve_2]$.
\index{AA@$\bA = \CC[\ve_1,\ve_2]$|textit}
The reason is that our formula for the level
$k=-h^{\vee}-\frac{\ve_2}{\ve_1}$ and the highest weight $\lambda =
\frac{\ba}{\ve_1} - \rho$ are not elements of $\bA$.
\begin{NB}
However, for many purposes,
\end{NB}%
For many purposes, it is convenient to have an $\bA$-version of
Theorem~\ref{main-intro-loc}. In fact, technically in order to prove
the last assertion of Theorem~\ref{main-intro-loc} we need such a
refinement of the first 3 assertions (the reason is that we need to
use the cohomological grading which is lost after localization).
In earlier works \cite{MO,SV} for type $A$, the $\bA$-version appears
only implicitly, as operators $\W{\kappa}{n}$ are given by cup
products on Gieseker spaces. But in our case, Uhlenbeck spaces are
singular, and we need to work with intersection cohomology groups. Hence
$\W{\kappa}n$ do not have such descriptions.

So, in order to formulate a non-localized version of Theorem~\ref{main-intro-loc} one needs to define
an $\bA$-version $\calW_{\bA}(\grg)$
\index{WAg@$\scW_\bA(\g)$|textit}
of the $\calW$-algebra (such that after tensoring with $\bF$ we get the algebra
$\calW_k(\grg)$ with $k=-h^{\vee}-\frac{\ve_2}{\ve_1}$). We also want this algebra to be graded (such that
the degrees of $\ve_1$ and $\ve_2$ are equal to 2); in addition we need analogs of statements~(2w) and~(3w). 
%of Theorem~\ref{FF-intro}. 
This is performed in the Appendix B. Let us note, that although this $\bA$-form is motivated by geometry, it can be defined
purely in an algebraic way, following the work of Feigin and Frenkel.
As far as we know, this $\bA$-form does not appear in the literature before.
As a purely algebraic application, we can remove the genericity
assumption in~(4w). 
%\thmref{FF-intro}(4). 
The third named author learns from
Arakawa that this was known to him before, but the proof is not written.
% we generally follow
% the work of Fegin and Frenkel there, some technical details become considerably more complicated.
After this we prove an  $\bA$-version of Theorem~\ref{main-intro-loc} in Chapter~\ref{sec:whittaker-state}.

The non-localized equivariant cohomology groups also give us a refined
structure in our construction. We construct $\scW_\bA(\g)$-module
structures on four modules
\begin{multline*}
    \bigoplus_d \IH^*_{\GG,c}(\Uh{d}), \quad
    \bigoplus_d H^*_{\TT,c}(\Phi_{T,G}(\IC(\Uh{d})))
\\
    \bigoplus_d H^*_{\TT}(\Phi_{T,G}(\IC(\Uh{d}))), \quad
    \bigoplus_d \IH^*_{\GG}(\Uh{d}),
\end{multline*}
where the subscript $c$ stands for cohomology with compact support.
They become isomorphic if we take tensor products with $\bF_T$, i.e.,
in the localized equivariant cohomology. But they are different over
$\bA_G$ and $\bA_T$.
We show that they are {\it universal\/} Verma, Wakimoto modules
$M_\bA(\ba)$\index{MA@$M_\bA(\ba)$|textit},
$N_\bA(\ba)$\index{NA@$N_\bA(\ba)$|textit},
and
their duals respectively.
Here by a Wakimoto module, we mean the pull-back of a Fock space via the embedding of $\scW(\g)$ in $\Heis(\mathfrak h)$.
They are universal in the sense that we can specialize to Verma/Wakimoto and their duals at any evaluation $\bA_G\to\CC$, $\bA_T\to\CC$.
This will be important for us to derive character formulas for simple
modules, which will be discussed in a separate publication.

The importance of the integral form and the application to character
formulas were first noticed in the context of the equivariant K-theory
of the Steinberg variety and the affine Hecke algebra (see \cite{CG}),
and then in quiver varieties \cite{Na-qaff} and parabolic Laumon spaces (=
handsaw quiver varieties) \cite{handsaw}.

\subsection{Remarks about non-simply laced case}\label{non-simply-laced}
We have already mentioned above that verbatim Theorem~\ref{main-intro-loc} doesn't hold for non-simply laced $G$.
However, we expect that the following modification of Theorem \ref{main-intro-loc} should hold.

First, let $\calG$ be any affine Lie algebra in the sense of \cite{Kac} with connected Dynkin diagram.
For example, $\calG$ can be untwisted, and in this case it is isomorphic to a Lie algebra of the form $\grg_{\aff}$ for some
simple finite-dimensional Lie algebra. But in addition there exist twisted affine Lie algebras. We refer the reader to
\cite{Kac} for the relevant definitions; let us just mention that every twisted $\calG$ comes from a pair $(\calG',\sigma)$
where $\calG'=\grg_{\aff}$ for some simply laced simple finite-dimensional Lie algebra $\grg$ and $\sigma$ is a certain automorphism of
$\grg$ of finite order.

The Dynkin diagram of $\calG$ comes equipped with a special ``affine" vertex. We let $G_{\calG}$ denote the semi-simple and simply connected group whose Dynkin diagram is obtained from that of $\calG$ by removing that vertex.

To such an algebra one can attach another affine Lie algebra $\calG^{\vee}$ ---
``the Langlands dual Lie algebra". By definition, this is just the Lie algebra whose generalized Cartan matrix is transposed to that of $\calG$. It is worthwhile to note that:

1) If $\grg$ is a simply laced finite-dimensional simple Lie algebra, then $\grg_{\aff}^{\vee}$ is isomorphic to $\grg_{\aff}$ (which
is also the same as $(\grg^{\vee})_{\aff}$ in this case).

2) In general, if $\grg$ is not simply laced, then $\grg_{\aff}^{\vee}$ is not isomorphic to $(\grg^{\vee})_{\aff}$. In fact, if $\grg$
is not simply laced, then $\grg_{\aff}^{\vee}$ is always a twisted Lie algebra.

It turns our that one can define the Uhlenbeck spaces $\mathcal U^d_{\calG}$ for any affine Lie algebra $\calG$ in such a way that
that $\mathcal U^d_{\calG}=\Uh{d}$ when $\calG=\grg_{\aff}$ and $\grg=\operatorname{Lie}(G)$ (the definition uses the corresponding simply laced
algebra $\grg$ and its automorphism $\sigma$ mentioned above). We are not going to explain the definition here
(we shall postpone it for a later publication). This scheme is endowed with an action of the group $G_{\calG}\times \CC^*\times \CC^*$.

In addition to $\calG$ as above one can also attach a $W$-algebra $\calW(\calG)$. Then we expect the following to be true:
\begin{Conjecture}\label{non-simp}
There exists an action of $\calW(\calG)$ on $\oplus \IH^*_{G(\calG^{\vee})\times \CC^*\times \CC^*}(\mathcal U_{\calG^{\vee}}^d)$ satisfying properties similar to those of Theorem~\ref{main-intro-loc}.
\end{Conjecture}

Let us discuss one curious corollary of the above conjecture. Let $\grg$ be a finite-dimensional simple Lie algebra.
Set $\calG_1=\grg_{\aff}^{\vee}, \calG_2=(\grg^{\vee})_{\aff}^{\vee}$. Then Conjecture \ref{non-simp} together with Feigin-Frenkel
duality imply that there should be an isomorphism between $\IH^*_{G(\calG_1)\times \CC^*\times \CC^*}(\mathcal U_{\calG_1}^d)$
and
$\IH^*_{G(\calG_2)\times \CC^*\times \CC^*}(\mathcal U_{\calG_2}^d)$ which sends $\frac{\ve_2}{\ve_1}$ and to $r^{\vee}\frac
{\ve_1}{\ve_2}$.
It would be interesting to see whether this isomorphism can be constructed geometrically (let us note that the naive guess
that
there exists an isomorphism between $\mathcal U_{\calG_1}^d$ and $\mathcal U_{\calG_2}^d$ giving rise to the above isomorphism between
$\IH^*_{G(\calG_1)\times \CC^*\times \CC^*}(\mathcal U_{\calG_1}^d)$
and
$\IH^*_{G(\calG_2)\times \CC^*\times \CC^*}(\mathcal U_{\calG_2}^d)$ is probably wrong).
This question might be related to the work \cite{Vafa} where the author explains how to derive
the 4-dimensional Montonen-Olive duality for non-simply laced groups from 6-dimensional (2,0) theory.

\subsection{Further questions and open problems}
In this subsection we indicate some possible directions for future research on the subject (apart from generalizing everything
to the non-simply laced case, which was discussed before).

\subsubsection{VOA structure and CFT}
Our results imply that the space $M_{\bF_G}(\ba)$ has a natural vertex operator algebra structure. It would be extremely interesting
to construct this structure geometrically.

The AGT conjecture predicts a duality between $N=2$ $4d$ gauge
theories and $2d$ conformal field theories (CFT). The equivariant
intersection cohomology group $M_{\bF_G}(\ba)$ is just the quantum
Hilbert space associated with $S^1$, appeared as a boundary of a
Riemann surface. We should further explore the $4d$ gauge theory from
CFT perspective, as almost nothing is known so far.

\subsubsection{Gauge theories with matter}
Our results give a representation-theoretic interpretation of the Nekrasov partition function of the {\em pure} $N=2$ super-symmetric
gauge theory on $\RR^4$. For physical reasons it is also interesting to study gauge theories with matter. Mathematically it usually means that in the definition of the partition function (\ref{Nekrasov function}) one should replace the equivariant integral of 1 by the equivariant integral of some other (intersection) cohomology class. However, when $G$ is not of type $A$ even the definition of
the partition function is not clear to us. Namely, for $G=SL(r)$ one usually works with the Gieseker space $\Gi{r}^d$ instead
of $\Uh{d}$. In this case the cohomology classes in question are usually defined as Chern classes of certain natural sheaves
$\Gi{r}^d$ (such as, for example, the tangent sheaf). Since $\Uh{d}$ is singular and we work with {\em intersection cohomology}
such constructions don't literally make sense for $\Uh{d}$.

\subsubsection{The case of $\mathbb C^2/\Gamma$}It would be interesting to try and generalize our results to Uhlenbeck space
of $\mathbb C^2/\Gamma$. Here we expect the case when $\Gamma$ is a cyclic group to be more accessible than the general case;
in fact, in this case one should be able to see connections with  \cite{braverman-2007},\cite{BF2} and on the other hand with
\cite{2011JHEP...07..079B,2013CMaPh.319..269B}.
On the other hand the theory of quiver varieties deals with general
$\Gamma$, but the group $G$ is of type $A$, as we mentioned in \subsecref{int-motiv}.
The case when both $\Gamma$ and $G$ are not of type $A$ seems more
difficult. Note that we must impose $\ve_1=\ve_2$, therefore the level
$k = -h^\vee - \ve_2/\ve_1$ cannot be deformed. In particular, the
would-be $\scr W$-algebra does not have a classical limit.

\subsubsection{Surface operators}

As we have already mentioned in \subsecref{int-motiv}, there are {\it
  flag Uhlenbeck spaces\/} parametrizing (generalized) $G$-bundles on
$\proj^2$ with parabolic structure on the line $\proj^1$. A type of
parabolic structure corresponds to a parabolic subgroup $P$ of $G$.
Generalizing results in two extreme cases, $P=B$ in
\cite{BraInstantonCountingI} and $P=G$ in this paper, it is expected
that the equivariant intersection cohomology group admits a
representation of the $\scr W$-algebra associated with the principal
nilpotent element in the Lie algebra $\mathfrak l$ of the Levi part of
$P$. (We assume $G$ is of type $ADE$, and the issue of Langlands
duality does not occur, for brevity.) This is an affine version of the conjecture in \cite{BFFR} mentioned before.
Moduli of $G$-bundles with parabolic structure of type $P$ is called a
surface operator of Levi type $\mathfrak l$ in the context of $N=4$
supersymmetric gauge theory \cite{MR2459305}.

However there is a surface operator corresponding to
{\it arbitrary\/} nilpotent element $e$ in $\operatorname{Lie}G$
proposed in \cite{MR3027737}, which is supposed to have the symmetry
of $\scr W(\g,e)$, the $\scr W$-algebra associated with $e$. We do not understand what kind of parabolic structures nor equivariant intersection cohomology groups we should consider if $e$ is not regular in Levi.

%-------------------------------------------------------------------------------------------

\subsection{Organization of the paper}
In Chapter~\ref{sec:preliminaries}  we discuss some generalities about Uhlenbeck spaces. Chapter~\ref{sec:local}
is devoted to the general discussion of hyperbolic restriction
and Chapter~\ref{sec:hyper-uhlenbeck} ---
to hyperbolic restriction on Uhlenbeck spaces. In Chapter~\ref{sec:typeA} we relate the constructions and results of Chapter~\ref{sec:hyper-uhlenbeck} to certain
constructions of \cite{MO} in the case when $G$ is of type $A$. Chapter~\ref{sec:w-algebra-repr} is devoted to the construction of the action of the algebra $\calW_k(\grg)$ on $M_{\bF_G}(\ba)$
along the lines presented above. Chapter~\ref{sec:R-matrix} is devoted to the discussion of ``geometric $R$-matrices".
\subsection{Some notational conventions}\label{sec:convent}

\begin{enumerate}
  \renewcommand{\theenumi}{\roman{enumi}}%
  \renewcommand{\labelenumi}{\textup{(\theenumi)}}%

\item A partition $\lambda$ is a nonincreasing sequence
  $\lambda_1\ge\lambda_2\ge \cdots$ of nonnegative integers with
  $\lambda_N = 0$ for sufficiently large $N$. We set $|\lambda| = \sum
  \lambda_i$, $l(\lambda) = \# \{ i \mid \lambda_i \neq 0\}$. We also
  write $\lambda = (1^{n_1} 2^{n_2}\cdots)$ with $n_k = \# \{ i\mid
  \lambda_i = k\}$.

    \item The equivariant cohomology group $H^*_\GG(\mathrm{pt})$ of a
  point is canonically identified with the ring of invariant
  polynomials on the Lie algebra $\operatorname{Lie}\GG$ of $\GG$. The
  coordinate functions for the two factors $\CC^*$ are denoted by
  $\ve_1$, $\ve_2$ respectively. We identify the ring of invariant
  polynomials on $\g = \operatorname{Lie}G$ with the ring of the Weyl
  group invariant polynomials on the Cartan subalgebra $\h$ of $\g$.
%   Thus
%   \begin{equation}
% %      \label{eq:56}
%       H^*_\GG(\mathrm{pt}) = \CC[\ve_1,\ve_2]\otimes S(\h^*)^W
%   \end{equation}
  When we consider the simple root $\alpha_i$ as
  a polynomial on $\h$, we denote it by $a^i$.

    \item For a variety $X$, let $D^b(X)$ denote the bounded derived
  category of complexes of constructible $\CC$-sheaves on $X$. 
  Let $\IC(X_0,\mathcal L)$ denote the intersection cohomology complex
  associated with a local system $\mathcal L$ over a Zariski open
  subvariety $X_0$ in the smooth locus of $X$. We denote it also by
  $\IC(X)$ if $\mathcal L$ is trivial.
  When $X$ is smooth and irreducible, $\cC_X$\index{CCX@$\cC_X$}
  denotes the constant sheaf on $X$ shifted by $\dim X$. If $X$ is a
  disjoint union of irreducible smooth varieties $X_\alpha$, we
  understand $\cC_X$ as the direct sum of $\cC_{X_\alpha}$.

\item\label{conv2} We make a preferred degree shift for the
  Borel-Moore homology group (with complex coefficients), and denote
  it by $H_{[*]}(X)$\index{HstarX@$H_{[*]}(X)$}. The shift is coming
  from a related perverse sheaf, which is clear from the context.
  For example, if $X$ is smooth, $\cC_X$ is a perverse sheaf. Hence
  $H_{[*]}(X) = H_{*+\dim X}(X)$ is a natural degree shift, as it is
  isomorphic to $H^{-*}(X, \cC_X)$.
  More generally, if $L$ is a closed subvariety in a smooth variety
  $X$, we consider $H_{[*]}(L) = H_{*+\dim X}(L) = H^{-*}(L,j^!
  \cC_X)$, where $j\colon L\to X$ is the inclusion.

\item We use the ADHM description of framed torsion free sheaves on
  $\proj^2$ at several places. We change the notation $(B_1,B_2,i,j)$
  in \cite[Ch.~2]{Lecture} to $(B_1,B_2,I,J)$ as $i$, $j$ are used for
  different things.
\end{enumerate}

\subsection{Acknowledgments}

We dedicate this paper to the memory of Kentaro Nagao, who passed away
on Oct. 22, 2013. His enthusiasm to mathematics did not stop until the
end, even under a very difficult health condition. H.N.\ enjoyed
discussions on results in this paper, and other mathematics with him
last summer.

It is clear that this paper owes a lot to earlier works by D.~Maulik,
A.~Okounkov, O.~Schiffmann and E.~Vasserot. The authors thank them for
their patient explanation on the works. They also thank the referee for his/her careful reading of the manuscript and useful comments, and a suggestion of a different proof of \corref{cor:dimest}.

A.B. is grateful to D.~Gaitsgory for very helpful discussions related to Appendix A as well as with the proof of
Theorem \ref{thm:nonlocal}.

H.N.\ thanks to Y.~Tachikawa who explained physical intuition to him,
and to T.~Arakawa, R.~Kodera and S.~Yanagida who taught him various
things on $\scW$-algebras.

This work was started at 2012 fall, while H.N.\ was visiting the
Simons Center for Geometry and Physics, and he wishes to thank the
Center for warm hospitality. We thank the Abdus Salam International
Centre for Theoretical Physics, the Euler International Mathematical
Institute, the Higher School of Economics, Independent University of
Moscow, where we continued the discussion.

A.B.\ was partially supported by the NSF grant DMS-1200807 and by the Simons foundation.
Also this work was completed when A.B. was visiting the New Economic School of Russia and the Mathematics
Department of the Higher School of Economics and he would like to thank those institutions for their hospitality.

M.F.\ was supported by The National Research
University Higher School of Economics Academic Fund Program in 2014/2015
(research grant No. 14-01-0047).

The research of H.N.\ was supported in part by JSPS Kakenhi Grant
Numbers 00201666, 25220701, 30165261, 70191062.

%%% Local Variables:
%%% mode: latex
%%% TeX-master: "BFN"
%%% End:

%\input{intro}

\section{Preliminaries}\label{sec:preliminaries}

A basic reference to results in this chapter is \cite{BFG}, where
\cite{AHS,Lecture} are quoted occasionally.

\subsection{Instanton number}\label{subsec:instnumber}

We define an instanton number of a $G$-bundle $\mathcal F$ over
$\proj^2$. It is explained in, for example, \cite{AHS}.
Since it is related to our assumption that $G$ is simply-laced, we
briefly recall the definition.

The instanton number is the characteristic class associated with an
invariant bilinear form $(\ ,\ )$ on the Lie algebra $\g$ of
$G$. Since we assume $G$ is simple, the bilinear form is unique up to
scalar. We normalize it so that the square length of the highest root
$\theta$ is $2$.
\index{( , )@$(\ ,\ )$}

When $G = SL(r)$, it is nothing but the second Chern class of the
associated complex vector bundle.

For an embedding $SL(2)\to G$ corresponding to a root $\alpha$, we can
induce a $G$-bundle $\mathcal F$ from an $SL(2)$-bundle $\mathcal
F_{SL(2)}$. Then the corresponding instanton numbers are related by
\begin{equation}
  d(\mathcal F) = d(\mathcal F_{SL(2)})\times \frac2{(\alpha,\alpha)}.
\end{equation}
Since we assume $G$ is simply-laced, we have $(\alpha,\alpha) = 2$ for
any root $\alpha$. Thus the instanton number is preserved under the
induction.

\subsection{Moduli of framed \texorpdfstring{$G$}{G}-bundles}

Let $\Bun{d}$ be the moduli space of $G$-bundles with trivialization
at $\linf$ of instanton number $d$ as before.
\begin{NB}
  If the following is explained in Introduction, it will be deleted.
\end{NB}%
We often call them {\it framed $G$-bundles}.

The tangent space of $\Bun{d}$\index{Bun@$\Bun{d}$} at $\mathcal F$ is
equal to the cohomology group $H^1(\proj^2,\g_{\mathcal F}(-\linf))$,
where $\g_{\mathcal F}$ is the vector bundle associated with $\mathcal
F$ by the adjoint representation $G\to GL(\g)$ (\cite[3.5]{BFG}).
Other degree cohomology groups vanish, and hence the dimension of
$H^1$ is given by the Riemann-Roch formula.
It is equal to $2dh^\vee$ (\cite{AHS}). Here
$h^\vee$\index{hcheck@$h^\vee$ (dual Coxeter number)} is the dual
Coxeter number of $G$, appearing as the ratio of the Killing form and
our normalized inner product $(\ ,\ )$.

It is known that $\Bun{d}$ is connected, and hence irreducible
(\cite[Prop.~2.25]{BFG}).

It is also known that $\Bun{d}$ is a holomorphic symplectic
manifold. Here the symplectic form is given by the isomorphism
\begin{equation}
  H^1(\proj^2,\g_{\mathcal F}(-2\linf))
  \xrightarrow{\cong}
  H^1(\proj^2,\g^*_{\mathcal F}(-\linf)),
\end{equation}
where $\g\cong\g^*$ is induced by the invariant bilinear form, and
$\shfO_{\proj^2}(-\linf)\to\shfO_{\proj^2}(-2\linf)$ is given by the
multiplication by the coordinate $z_0$ corresponding to $\linf$.
The tangent space $T_{\mathcal F}\Bun{d}\cong H^1(\proj^2,\g_{\mathcal
  F}(-\linf))$ is isomorphic also to $H^1(\proj^2,\g_{\mathcal
  F}(-2\linf))$ and the above isomorphism can be regarded as
$T_{\mathcal F}\Bun{d}\to T_{\mathcal F}^*\Bun{d}$. It is
nondegenerate and closed. (See \cite[Ch.~2, 3]{Lecture} for
$G=SL(r)$. General cases can be deduced from the $SL(r)$-case by a
faithful embedding $G\to SL(r)$.)

\subsection{Stratification}

Let $\Uh{d}$ be the Uhlenbeck space for $G$.
It has a stratification
\begin{equation}
  \label{eq:strat}
  \Uh{d} = \bigsqcup \BunGl{d_1},\quad
  \BunGl{d_1} = \Bun{d_1}\times S_\lambda\AA^2,
  \index{BunGlambda@$\BunGl{d_1}$}
\end{equation}
where the sum runs over pairs of integers $d_1$ and partitions
$\lambda$ with $d_1+|\lambda| = d$. Here
$S_\lambda\AA^2$\index{SlambdaA@$S_\lambda\AA^2$} is a stratum of the
symmetric product $S^{|\lambda|}\AA^2$, consisting of configurations
of points whose multiplicities are given by $\lambda$, that is
\begin{equation}
  \label{eq:Sl}
  S_\lambda\AA^2 = \left\{
    \sum \lambda_i x_i\in S^{|\lambda|}\AA^2 \middle|\,
      \text{$x_i\neq x_j$ for $i\neq j$}\right\}
\end{equation}
for $\lambda = (\lambda_1\ge\lambda_2\ge\cdots)$. We have
\begin{equation}
  \dim \BunGl{d_1} = 2(d_1 h^\vee + l(\lambda)).
\end{equation}

Let $\UhGl{d_1}$ be the closure of $\BunGl{d_1}$. We
have a finite morphism
\begin{equation}\label{eq:finite}
  \Uh{d_1}\times\overline{S_\lambda\AA^2}\to \UhGl{d_1},
\end{equation}
extending the identification $\Bun{d_1}\times S_\lambda\AA^2 =
\BunGl{d_1}$, where $\overline{S_\lambda\AA^2}$ is the closure of
$S_\lambda\AA^2$ in $S^{|\lambda|}\AA^2$.

\subsection{Factorization}\label{sec:factorization}

For any projection $a\colon\AA^2 \to \AA^1$\index{a@$a$ (projection)}
we have a natural map $\pi^d_{a,G}\colon \Uh{d} \to
S^d\AA^1$\index{paidpg@$\pi^d_{p,G}$ (factorization morphism)}. See
\cite[\S6.4]{BFG}. It is equivariant under $\GG =
G\times\CC^*\times\CC^*$: it is purely invariant under $G$. We also
change the projection $a$ according to the $\CC^*\times\CC^*$-action.

Let us explain a few properties. Let $\mathcal F\in\Bun{d}$. It is a
principal $G$-bundle over $\proj^2$ trivialized at $\linf$, but can be
also considered as a $G$-bundle over $\proj^1\times\proj^1$
trivialized at the union of two lines $\{\infty\}\times\proj^1$ and
$\proj^1\times\{\infty\}$. We extend $a$ to
$\proj^1\times\proj^1\to\proj^1$.
Then $\pi^d_{a,G}(\mathcal F)$ measures how the restriction of
$\mathcal F$ to a projective line $a^{-1}(x)$ differs from the trivial
$G$-bundle for $x\in\proj^1$. If $x$ is disjoint from
$\pi^d_{a,G}(\mathcal F)$, then $\mathcal F|_{a^{-1}(x)}$ is a trivial
$G$-bundle. If not, the coefficient of $x$ in $\pi^d_{a,G}(\mathcal
F)$ counts non-triviality with an appropriate multiplicity. (See
\cite[\S4]{BFG}.)

On the stratum $\Bun{d_1}\times S_{\lambda}\AA^2$, $\pi^d_{a,G}$ is
given as the sum of $\pi^{d_1}_{a,G}$ and the natural morphism
$S_{\lambda}\AA^2\to S^{|\lambda|}\AA^1$ induced from $a$. This
property comes from the definition of the Uhlenbeck as a space of
quasi-maps. (See \cite[\S\S1,2]{BFG}.)

For type $A$, it is given as follows in terms of the ADHM description
$(B_1,B_2,I,J)$ (see \cite[Ch.~2]{Lecture}): let $B_a$ be the linear
combination of $B_1$, $B_2$ corresponding to the projection
$a\colon\AA^2\to \AA^1$. Then $\pi^d_{a,G}$ is the characteristic
polynomial of $B_a$. (See \cite[Lem.~5.9]{BFG}.)

Moreover, most importantly, this map enjoys the factorization
property, which says the following. Let us write $d = d_1+d_2$ with
$d_1, d_2 > 0$. Let $(S^{d_1}\AA^1\times S^{d_2}\AA^1)_0$ be the open
subset of $S^{d_1}\AA^1\times S^{d_2}\AA^1$ where the first divisor is
disjoint from the second divisor. Then we have a natural isomorphism
\begin{equation}
    \label{eq:4}
    \Uh{d}\times_{S^d\AA^1} (S^{d_1}\AA^1\times S^{d_2}\AA^1)_0
    \cong (\pi^{d_1}_{a,G} \times \pi^{d_2}_{a,G})^{-1}(
    (S^{d_1}\AA^1\times S^{d_2}\AA^1)_0).
\end{equation}
See \cite[Prop.~6.5]{BFG}. We call $\pi^d_{a,G}$ the {\it
  factorization morphism}. Often we are going to make statements
about $\Uh{d}$ and we are going to prove them by induction on $d$;
\eqref{eq:4} will usually allow us to say that the inductive step is
trivial away from the preimage under $\pi^d_{a,G}$ of the main
diagonal in $S^d\AA^1$. In this case we are going to say that (the
generic part of) the induction step ``follows by the factorization
argument''.

\begin{NB}
\subsection{Parabolic subgroups}

Let us fix a Borel subgroup $B$ and a maximal torus $T$. Then we have
simple roots $\alpha_i\in X^*(T) = \Hom(T,\GG_m)$. Let $I$ denote the
set of simple roots. Let us choose a one parameter subgroup
$\lambda\in X_*(T) = \Hom(\GG_m,T)$ from the fundamental chamber,
i.e., $\langle \lambda, \alpha_i\rangle > 0$ for any $i\in I$. Then we
recover $B$ and $T$ from $\lambda$ as
\begin{equation}
  B = \left\{ g\in G\, \middle|\,\text{$\lim_{t\to 0}
  \lambda(t) g\lambda(t)^{-1}$ exists}\right\},
\quad
  T = G^{\lambda(\GG_m)}.
\end{equation}

We identify $G/T$ as the space of one parameter subgroups
\begin{equation}
  G/T \ni h \mapsto \{ h\lambda h^{-1}\colon \GG_m\to G \}.
\end{equation}
Let
\begin{equation}
  \begin{split}
    & B^h \defeq \left\{ g\in G\, \middle|\, \text{$\lim_{t\to 0}
  (h\lambda(t) h^{-1}) g (h\lambda(t) h^{-1})^{-1}$
  exists}\right\} = h B h^{-1},
\\
    & T^h \defeq G^{h\lambda(\GG_m) h^{-1}} = h T h^{-1}.
  \end{split}
\end{equation}
All pairs $(B^h, T^h)$ of Borel subgroups and tori appear in this
way. Therefore $G/T$ can be considered as the space of pairs.

Choose a subset $J\subset I$. Then we choose $\lambda_J\in X_*(T)$ in
so that
\begin{equation}
  \begin{split}
  & \langle \lambda_J, \alpha_i\rangle > 0 \quad\text{for $i\notin J$},
\\
  & \langle \lambda_J, \alpha_i\rangle = 0 \quad\text{for $i\in J$}.
  \end{split}
\end{equation}
Then we define
\begin{equation}
  \begin{split}
    & P_J \defeq \left\{ g\in G\, \middle|\, \text{$\lim_{t\to 0}
  \lambda_J(t) g\lambda_J(t)^{-1}$ exists}\right\},
\\
    & L_J \defeq G^{\lambda_J(\GG_m)}.
  \end{split}
\end{equation}
It is a parabolic subgroup and its Levi subgroup corresponding to
$J$. (Note $P_\emptyset = B$, $L_\emptyset = T$, $P_I = L_I = G$.)

We can define $P_J^h$, $L_J^h$ in the same way as above. The space of
pairs $(P, L)$ of parabolic subgroups and its Levi subgroups is equal
to $G/L_J$.
\begin{NB2}
  This is not true. Misha's message on Oct. 30: Consider $G=SL(3)$ and
  $L$ is of type $(2,1)$. The standard parabolic $P$ and opposite
  parabolic $P_-$ is not conjugate. Let us take $w_0$, the longest
  element of the Weyl group. We have a parabolic subgroup $P^{w_0}$,
  but its Levi subgroup is of type $(1,2)$.
\end{NB2}

Let $B^\text{op}$ be the opposite Borel subgroup. It is given by the
longest element $w_0$ in the Weyl group $W = N(T)/T$ as
$B^{w_0}$. Note that $w_0$ induces an involutive diagram automorphism
$\varpi$ of $I$ by
\begin{equation}
  w_0(\alpha_i) = - \alpha_{\varpi(i)}.
\end{equation}
For example, for type $A_\ell$, it is the reflection with respect to
the center in the Dynkin diagram. In the McKay correspondence, $I$ is
the set of nontrivial irreducible representations of a finite subgroup
$\Gamma$ of $SL_2(\CC)$. Then $\varpi$ is given by the dual representation.

We have a projection $G/T\to G/L_J$. If we regard $w_0$ as an element
in $G/L_J$, then $L_J^{w_0} = L_{\varpi(J)}$. This probably means that
what is written after \eqref{eq:decomp2} is {\it wrong}. We have
decomposition
\begin{equation}
  U^B_{T,G} = U^{P_i}_{T,L_i} \oplus (U^B_{T,L_i})^\perp.
\end{equation}
If we move this by $w_0$, we get
\begin{equation}
  U^{B^w_0}_{T,G} = U^{P_i^{w_0}}_{T,L_{\varpi(i)}} \oplus
  (U^{P_i^{w_0}}_{T,L_{\varpi(i)}})^\perp.
\end{equation}

Therefore the inner product using the trivialization of the local
system should not be a standard one. It should be given by
\begin{equation}
  d(x, w_0 y) \qquad\text{for $x,y\in\h$}
\end{equation}
under $U^d_{T,G}\cong\h$.
\end{NB}%

% Time-stamp: <2016-05-06 14:08:38 nakajima>
\newcommand{\HHom}{\mathcal Hom}

\section{Localization}\label{sec:local}

\subsection{General Statement}\label{sec:localization}

Let $T$ be a torus acting on $X$ and $Y$ be a closed invariant subset
containing $X^T$. Let $\varphi\colon Y\to X$ be the inclusion. Let
$U\defeq X\setminus Y$ and $\psi\colon U\to X$ be the inclusion.
Let $\scF\in D^b_T(X)$. We consider distinguished triangles
\begin{equation}\label{eq:15}
\begin{gathered}
  \varphi_!\varphi^! \scF \to \scF \to \psi_*\psi^* \scF \xrightarrow{+1},
\\
  \psi_!\psi^! \scF \to \scF \to \varphi_*\varphi^* \scF \xrightarrow{+1}.
\end{gathered}
\end{equation}

Denote the Lie algebra of $T$ by $\ft$.
Natural homomorphisms
\begin{gather}
  \label{eq:tri1}
  H^*_T(X,\scF) \to H^*_T(X, \varphi_*\varphi^* \scF)\cong
  H^*_T(Y, \varphi^* \scF),
  \\
  \label{eq:tri2}
  H^*_T(Y,\varphi^! \scF) \cong H^*_T(X, \varphi_*\varphi^! \scF) \cong
  H^*_T(X, \varphi_!\varphi^! \scF) \to H^*_T(X, \scF)
\end{gather}
become isomorphisms after inverting an element $f\in
\CC[\ft]$ such that
\begin{equation}\label{eq:cond}
  \{ x\in \ft\mid f(x) = 0\} \supset \bigcup_{x\in X\setminus Y}
  \operatorname{Lie}(\operatorname{Stab}_x).
\end{equation}
See \cite[(6.2)]{GKM}. These assertions follow by observing
\(
    H^*_T(X; \psi_!\psi^! \scF) = H^*_T(X, Y ; \scF)
\)
and
\(
    H^*_T(X; \psi_*\psi^*\scF) = H^*_T(U ; \scF)
\)
are torsion in $\CC[\ft]$.
\begin{NB}
  For \eqref{eq:tri1} we note that
  \begin{equation}
    H^*_T(X; \psi_!\psi^! \scF) = H^*_T(X; \psi_! \psi^* \scF) = H^*_T(X, Y ; \scF)
  \end{equation}
  is torsion. For \eqref{eq:tri2} we note that
  \begin{equation}
    H^*_T(X; \psi_*\psi^*\scF) = H^*_T(U ; \scF)
  \end{equation}
  is torsion.
\end{NB}%
The same is true also for cohomology groups with compact supports. We
call these statements the {\it localization theorem}.

\begin{NB}
  One remark which Sasha pointed out me was the following: If $T$ acts
  trivially on $Y$ (say $Y = X^T$), and $\scF = \CC_X$, we have
  \begin{equation}
    H^*_T(Y, \varphi^* \scF) = H^*_T(Y,\CC_Y)
    \cong H^*(Y)\otimes H^*_T(\mathrm{pt}).
  \end{equation}
However for a general equivariant complex $\scF$, we may not have
\begin{equation}
  H^*_T(Y, \varphi^* \scF) \cong H^*(Y, \varphi^* \underline{\scF})\otimes
  H^*_T(\mathrm{pt}),
\end{equation}
where $\underline{\scF}$ is the underlying (non-equivariant) complex of
$\scF$. We have the spectral sequence with the right hand side as
$E_2$-term converging to the left hand side \cite[(5.5.4)]{GKM}. But
it may fail to degenerate. See \cite[(5.6)]{GKM}.

I think that this problem does not arise for a hyperbolic restriction:
It sends a semisimple complex to a semisimple complex, and a
semisimple complex satisfies a degeneration of the spectral sequence
above. It should be written somewhere in \cite{GKM}.
\end{NB}

We now suppose that we have an action of $\CC^*\times\CC^*$ commuting
with the $T$-action such that
\begin{equation}\label{eq:assump}
  \begin{minipage}[htbp]{.85\linewidth}
\begin{itemize}
\item $X^{\CC^*\times\CC^*}$ is a single point, denoted by $0$.
\item If $n_1$, $n_2 > 0$, $(t^{n_1}, t^{n_2})\cdot x$ goes to $0$
  when $t\to 0$.
\end{itemize}
  \end{minipage}
\end{equation}
In fact, it is enough to have a $\CC^*$-action for the result below,
but we consider a $\CC^*\times\CC^*$-action, as the Uhlenbeck space
has natural $\CC^*\times\CC^*$-action.

Let $\TT = T\times\CC^*\times\CC^*$.

\begin{NB}
  The following was stated several times by Sasha. I give the proof
  now.
\end{NB}

\begin{Lemma}\label{lem:centralfiber}
  The natural homomorphisms $H^*_\TT(X, \scF) \to H^*_\TT(Y, \varphi^* \scF)$,
  $H^*_{\TT,c}(Y, \varphi^! \scF)\to H^*_{\TT,c}(X,\scF)$ are isomorphisms for
  $\scF\in D^b_\TT(X)$.
\end{Lemma}

\begin{proof}
    Let $b_0^X\colon \{0\}\to X$, $b_0^Y\colon \{0\}\to Y$ be
    inclusions, and $a_X\colon X\to \{ 0\}$, $a_Y\colon Y\to \{0\}$ be
    the obvious morphisms. Since $0$ is the unique fixed point of an
    attracting action of $\CC^*\times\CC^*$ by our assumption,
    adjunction gives us isomorphisms
    $(a_X)_*\xrightarrow{\cong}(b_0^X)^*$,
    $(a_Y)_*\xrightarrow{\cong}(b_0^Y)^*$ on equivariant objects by
    \cite[Lemma~6]{Braden}. Therefore we have a diagram
  \begin{equation}
    \begin{CD}
    H^*_\TT(X,\scF) @>>> H^*_\TT(Y, \varphi^* \scF)
\\
   @V{\cong}VV @VV{\cong}V
\\
    H^*_\TT((b_0^X)^* \scF) @= H^*_\TT((b_0^Y)^*\varphi^* \scF),
    \end{CD}
  \end{equation}
  where the lower horizontal equality follows from $\varphi b_0^Y =
  b_0^X$. If $\scF$ is a sheaf, other three homomorphisms are given by
  restrictions, therefore the diagram is commutative. Hence it is also
  so for $\scF\in D^b_\TT(X)$ by a standard argument. Taking the dual
  spaces, we obtain the second assertion.
\end{proof}

\subsection{The case of Ext algebras}

Let $\scF$, $\scG\in D^b_T(X)$. We claim that
\begin{gather}\label{eq:31}
  \Ext_{D^b_T(X)}(\scF,\scG)\to \Ext_{D^b_T(Y)}(\varphi^! \scF, \varphi^! \scG),
  \\
  \Ext_{D^b_T(X)}(\scF,\scG)\to \Ext_{D^b_T(Y)}(\varphi^* \scF, \varphi^* \scG)\label{eq:38}
\end{gather}
are isomorphisms after inverting an appropriate element $f$. Taking
adjoint and considering \eqref{eq:15}, we see that it is enough to
show that
\begin{equation}
    \Ext_{D^b_T(X)}(\psi_*\psi^* \scF, \scG), \quad
    \Ext_{D^b_T(X)}(\scF,\psi_!\psi^! \scG)
\end{equation}
are torsion. Let us observe that
\begin{equation}
  \Ext_{D^b_T(X)}(\psi_*\psi^* \scF, \psi_*\psi^* \scF) \cong
  \Ext_{D^b_T(U)}(\psi^* \psi_*\psi^* \scF, \psi^* \scF)
\end{equation}
is torsion, as it is an equivariant cohomology group over $U$. Then
multiplying the identity endomorphism of $\psi_*\psi^* \scF$ to
$\Ext_{D^b_T(X)}(\psi_*\psi^* \scF, \scG)$, we conclude that
$\Ext_{D^b_T(X)}(\psi_*\psi^* \scF, \scG)$ is torsion.
The same argument applies also to $\Ext_{D^b_T(X)}(\scF,\psi_!\psi^! \scG)$.

\begin{NB}
  Here is my earlier NB. It is recorded for my memory.

  We want to have the same statement for the Ext algebras. Namely we
  want to show that
  \begin{gather}
    \Ext_{D^b_T(X)}(\scF,\scG)\to \Ext_{D^b_T(Y)}(\varphi^! \scF, \varphi^! \scG),
\\
   \Ext_{D^b_T(X)}(\scF,\scG)\to \Ext_{D^b_T(Y)}(\varphi^* \scF, \varphi^* \scG)
  \end{gather}
  are isomorphism after inverting an element $f$. Sasha says that these
  follow once we consider the `localized equivariant derived category
  localizing the space of homomorphisms'. It may be OK, but I am not
  comfortable with it, so let me consider a down-to-earth approach.

  We first give a sheaf theoretic interpretation of homomorphisms. We
  have
\begin{equation}
  \varphi_*\HHom(\varphi^! \scF, \varphi^! \scG)
  \cong \HHom(\varphi_!\varphi^! \scF, \scG)
\end{equation}
by \cite[3.1.10]{KaSha}. Hence we have a natural homomorphism
\begin{equation}
  \HHom(\scF,\scG) \to \varphi_* \HHom(\varphi^! \scF, \varphi^! \scG),
\end{equation}
and hence
\begin{equation}\label{eq:ext}
  \Ext_{D^b_T(X)}(\scF,\scG)\to \Ext_{D^b_T(Y)}(\varphi^!\scF, \varphi^! \scG).
\end{equation}
This is compatible with \eqref{eq:tri2}:
\begin{equation}
  \begin{CD}
  H^*_T(Y,\varphi^! \scF)\otimes  \Ext_{D^b_T(X)}(\scF, \scG)
 @>\id\times\eqref{eq:ext}>>
 H^*_T(Y,\varphi^! \scF)\otimes  \Ext_{D^b_T(Y)}(\varphi^!\scF, \varphi^!\scG)
 @>>> H^*_T(Y,\varphi^! \scG)
\\
   @V{\eqref{eq:tri2}\times\id}VV @. @VV\eqref{eq:tri2}V
\\
  H^*_T(X,\scF)\otimes \Ext_{D^b_T(X)}(\scF,\scG)
  @= H^*_T(X,\scF)\otimes \Ext_{D^b_T(X)}(\scF,\scG)
  @>>> H^*_T(X,\scG)
  \end{CD}
\end{equation}
is commutative. If we set $\scF=\scG$, the upper row gives an
$\Ext_{D^b_T(X)}(\scF,\scF)$-module structure on $H^*_T(Y,\varphi^!\scF)$.  The
commutativity means that \eqref{eq:tri2} is a homomorphism of
$\Ext_{D^b_T(X)}(\scF,\scF)$-modules.

To show that \eqref{eq:ext} is an isomorphism after the localization,
we need that
\begin{equation}
  \Ext_{D^b_T(X)}(\psi_*\psi^* \scF, \scG)
\end{equation}
is torsion. But I do not see why it is so.

Similarly we have
\begin{equation}
  \varphi_*\HHom(\varphi^* \scF, \varphi^* \scG) \cong \HHom(\scF, \varphi_*\varphi^*\scG)
\end{equation}
\cite[2.6.15]{KaSha}, and hence
\begin{equation}\label{eq:ext2}
  \Ext_{D^b_T(X)}(\scF,\scG) \to \Ext_{D^b_T(Y)}(\varphi^*\scF, \varphi^*\scG).
\end{equation}
This is compatible with \eqref{eq:tri1}:
\begin{equation}
  \begin{CD}
    H^*_T(Y,\varphi^*\scF) \otimes \Ext_{D^b_T(X)}(\scF,\scG) @>>>
    H^*_T(Y,\varphi^*\scF) \otimes \Ext_{D^b_T(Y)}(\varphi^*\scF,\varphi^* \scG) @>>>
    H^*_T(Y,\varphi^*\scG)
\\
   @AAA @. @AAA
\\
    H^*_T(X,\scF) \otimes \Ext_{D^b_T(X)}(\scF, \scG) @=
    H^*_T(X,\scF) \otimes \Ext_{D^b_T(X)}(\scF, \scG)
    @>>> H^*_T(X,\scG)
  \end{CD}
\end{equation}
is commutative. This diagram has the same meaning as above.
\eqref{eq:ext2} is an isomorphism over localization, if
\begin{equation}
  \Ext_{D^b_T(X)}(\scF, \psi_!\psi^! \scG)
\end{equation}
is torsion. Again I do not see why it is so.
\end{NB}

\subsection{Attractors and repellents}\label{sec:attr-repell}

Let $X$ be a $T$-invariant closed subvariety in an affine space with a
linear $T$-action. Let $A\subset T$ be a subtorus and $X^A$ denote the
fixed point set.

Let $X_*(A)$ be the space of cocharacters of $A$. It is a free
$\Z$-module. Let
\begin{equation}
  \fa_\RR = X_*(A)\otimes_\Z\RR.
\end{equation}

Let $\Stab_x$ be the stabilizer subgroup of a point $x\in X$. A {\it
  chamber\/} $\mathfrak C$ is a connected component of
\begin{equation}
    \label{eq:29}
    \fa_\RR \setminus \bigcup_{x\in X\setminus X^A} X_*(\Stab_x)\otimes_\Z\RR.
\end{equation}

We fix a chamber $\mathfrak C$. Choose a cocharacter $\lambda$ in
$\mathfrak C$. Let $x\in X^A$. We introduce {\it attracting\/} and
{\it repelling\/} sets:
\begin{multline}
    \label{eq:3}
    \mathcal A_x = \left\{y \in X \,\middle|\,
      \begin{minipage}{.6\linewidth}
      the map $t\mapsto \lambda(t)(y)$
      extends to a map $\AA^1 \to X$ sending $0$ to $x$
      \end{minipage}
    \right\},
    \\
    \mathcal R_x = \left\{y \in X\,\middle|\,
      \begin{minipage}{.6\linewidth}
      the map $t\mapsto \lambda(t^{-1})(y)$ extends
      to a map $\AA^1 \to X$ sending $0$ to $x$
      \end{minipage}\right\}.
\end{multline}
These are closed subvarieties of $X$, and independent of the choice
of $\lambda\in\mathfrak C$.
\begin{NB}
    To see the independence, we may assume $X$ is affine. Then the
    assertion is clear.
\end{NB}
 Similarly we can define $\mathcal
A_X$\index{AX@$\mathcal A_X$ (attracting set)}, $\mathcal
R_X$\index{RX@$\mathcal R_X$ (repelling set)} if we do not fix the
point $x$ as above. Note that $X^{A}$ is a closed subvariety of both
$\mathcal A_X$ and $\mathcal R_X$; in addition we have the natural
morphisms $\mathcal A_X \to X^{A}$ and $\mathcal R_X \to X^{A}$.

\begin{NB}
\subsection{Attractors and repellents}

Let us denote by $\Lambda^\vee_{Z(L)}$ its lattice of characters and
by $\Lambda^{\vee +}_{Z(L)}$ the sub-semi-group generated by the
weights of $Z(L)$ on $\mathfrak n_P$.
Let $Z^+(L) = \Spec \CC[\Lambda^{\vee+}Z(L)]$. This is a semi-group,
containing $Z(L)$ as an open dense subgroup. It also has a canonical
point $0$. Let now $X$ be a scheme endowed with a $Z(L)$-action and
let $x$ be a $Z(L)$-fixed point. Set
\begin{multline}
%    \label{eq:3}
    \mathcal A_x = \{y \in X \mid \text{the map $z\mapsto z(y)$
      extends to a map $Z^+(L) \to X$ sending $0$ to $x$}\},
    \\
    \mathcal R_x = \{y \in X\mid \text{the map $z\mapsto z^{-1}(y)$ extends
      to a map $Z^+(L) \to X$ sending $0$ to $x$}\}.
\end{multline}
These are locally closed subsets of $X$. Similarly we can define
$\mathcal A_X$\index{AX@$\mathcal A_X$ (attracting set)}, $\mathcal
R_X$\index{RX@$\mathcal R_X$ (repelling set)} if we do not fix the
point $x$ as above. Note that $X^{Z(L)}$ is a closed subset of both
$\mathcal A_X$ and $\mathcal R_X$; in addition we have the natural
maps $\mathcal A_X \to X^{Z(L)}$ and $\mathcal R_X \to X^{Z(L)}$.
\end{NB}

\subsection{Hyperbolic restriction}\label{sec:hyperb-restr-1}

We continue the setting in the previous subsection. We choose a
chamber in $\fa_\RR$, and consider the diagram
\begin{equation}\label{eq:113}
  X^A\overset{p}{\underset{i}{\leftrightarrows}}
  \cA_X \xrightarrow{j} X,
\end{equation}
where $i$, $j$ are embeddings, and $p$ is defined by $p(y) =
\lim_{t\to 0}\lambda(t)y$.

\begin{NB}
  Oct. 10: I change the definition of the hyperbolic restriction from
  $i^!j^*$ to $i^* j^!$. Or equivalently exchange $\cA_X$ and
  $\mathcal R_X$.
\end{NB}

We consider Braden's {\it hyperbolic restriction functor\/}
\cite{Braden} defined by $\Phi = i^*j^!$.\index{UZphi@$\Phi$ (hyperbolic restriction functor)} (See also a recent paper
\cite{2013arXiv1308.3786D}.)
Braden's theorem says that we have a canonical isomorphism
\begin{equation}\label{eq:Braden}
  i^* j^!
  \cong  i_-^! j_-^*
\end{equation}
on weakly $A$-equivariant objects, where $i_-$, $j_-$ are defined
as in \eqref{eq:113} for $\mathcal R_X$ instead of $\mathcal A_X$.

Braden proved his theorem for a normal algebraic variety. It is not
known that $\Uh{d}$ is normal or not. Therefore we use a more general
result~\cite[Theorem~3.1.6]{2013arXiv1308.3786D}.

Note also that $i^*$ and $p_*$ are isomorphic on weakly equivariant
objects, we have $\Phi = p_*  j^!$. (See \cite[(1)]{Braden}.)

Let $\scF\in D^b_T(X)$. A homomorphism
\begin{equation}\label{eq:48}
  H^*_T(X^A, i^* j^! \scF) \cong H^*_T(X^A,p_* j^!\scF)
  = H^*_T(\cA_X, j^! \scF) \to
 H^*_T(X, \scF)
\end{equation}
becomes an isomorphism after inverting a certain element by the
localization theorem in the previous subsection, applied to the pair
$\cA_X\subset X$.

We also have two naive restrictions
\begin{equation}\label{eq:naive}
  H^*_T(X^A, (j\circ i)^! \scF), \quad
  H^*_T(X^A, (j\circ i)^* \scF).
\end{equation}
For the first one, we have a homomorphism to the hyperbolic restriction
\begin{equation}\label{eq:26}
  H^*_T(X^A, (j\circ i)^! \scF)\to H^*_T(X^A, i^* j^! \scF),
\end{equation}
which factors through $H^*(\cA_X, j^! \scF)$. Then it also becomes
an isomorphism after inverting an element.

The second one in \eqref{eq:naive} fits into a commutative diagram
\begin{equation}
  \begin{CD}
    H^*_T(X^A, i^* j^! \scF)
    @>>>
    H^*_T(X^A, (j\circ i)^* \scF)
\\
@AAA @AAA
\\
    H^*_T(\cA_X, j^! \scF) @>>>
    H^*_T(\cA_X, j^* \scF).
  \end{CD}
\end{equation}
Two vertical arrows are isomorphisms after inverting an element
$f$. The lower horizontal homomorphism factors through $H^*_T(X,
\scF)$ and the resulting two homomorphisms are isomorphisms after
inverting an element, which we may assume equal to $f$. Therefore the
upper arrow is also an isomorphism after inverting an element.

\subsection{Hyperbolic semi-smallness}\label{subsec:hypsemismall}

Braden's isomorphism $p_* j^! \cong (p_-)_! j_-^*$ implies that $p_*
j^!$ preserves the purity of weakly equivariant mixed
sheaves. (\cite[Theorem~8]{Braden}). In particular, $p_* j^! \IC(X)$ is
isomorphic to a direct sum of shifts of intersection cohomology
complexes (\cite[Theorem~2]{Braden}).

Braden's result could be viewed as a {\it formal\/} analog of the
decomposition theorem (see \cite[Theorem~8.4.8]{CG} for example).
We give a sufficient condition so that $p_* j^! \IC(X)$ remains
perverse (and semi-simple by the above discussion) in this subsection.
This result is a formal analog of the decomposition theorem for {\it
  semi-small\/} morphisms (see
\cite[Proposition~8.9.3]{CG}). Therefore we call the condition the
{\it hyperbolic semi-smallness}.
This condition, without its naming, appeared in \cite{MV,MV2}
mentioned in the introduction.
We give the statement in a general setting, as it might be useful also in other situations.

Let $X$, $X^A$ as before.
Let $X = \bigsqcup X_\alpha$ be a stratification of $X$ such that
$i_\alpha^! \IC(X)$, $i_\alpha^* \IC(X)$ are locally constant sheaves
up to shifts.
Here $i_\alpha$ denotes the inclusion $X_\alpha\to X$.
We suppose that $X_0$ is the smooth locus of $X$ as a convention.
\begin{NB}
    Therefore $i_{0}^! \IC(X) = i_0^* \IC(X) = \cC_{X_0}
    = \CC_{X_0}[\dim X]$.
\end{NB}%

We also suppose that the fixed point set $X^A$ has a stratification
$X^A = \bigsqcup Y_\beta$ such that the restriction of $p$ to
$p^{-1}(Y_\beta)\cap X_\alpha$ is a topologically locally trivial
fibration over $Y_\beta$ for any $\alpha$, $\beta$ (if it is
nonempty). We assume the same is true for $p_-$.
We take a point $y_\beta\in Y_\beta$.

\begin{Definition}\label{def:hypsemismall}
    We say $\Phi$ is {\it hyperbolic semi-small\/} if the following
    two estimates hold
    \begin{equation}\label{eq:est}
        \begin{split}
            & \dim p^{-1}(y_\beta)\cap X_\alpha\le \frac12 
            (\dim X_\alpha - \dim Y_\beta),
\\
            & \dim p_-^{-1}(y_\beta)\cap X_\alpha\le \frac12 
            (\dim X_\alpha - \dim Y_\beta).
        \end{split}
    \end{equation}
\end{Definition}

In order to state the result, we need a little more notation.
We have two local systems over $Y_\beta$, whose fibers at a point
$y_\beta$ are $H_{\dim X - \dim Y_\beta}(p^{-1}(y_\beta)\cap X_0)$ and
$H^{\dim X - \dim Y_\beta}_c(p_-^{-1}(y_\beta)\cap X_0)$ respectively.
\begin{NB}
    Let $i_{y_\beta}$ denote the inclusion of $y_\beta$ to
    $Y_\beta$. We consider components of $\Phi(\IC(X))$ involving
    local systems on $Y_\beta$. To compute the former, we study
    $H^{\dim Y_\beta}(i_{y_\beta}^!)$. We use the base change that the
    fiber is
    \begin{equation*}
        \begin{split}
            & H_{\dim X - \dim Y_\beta}(p^{-1}(y_\beta)\cap X_0) 
    = H^{-\dim X + \dim Y_\beta}(p^{-1}(y_\beta)\cap X_0, \mathbb D_{p^{-1}(y_\beta)\cap X_0})
\\
    =\; & H^{-\dim X + \dim Y_\beta}(
    p^{-1}(y_\beta)\cap X_0,
    \tilde j^! \mathbb D_{X_0})
\\
    =\; & H^{\dim Y_\beta}(
    p^{-1}(y_\beta)\cap X_0, \tilde j^! \cC_{X_0})
    = H^{\dim Y_\beta}(i_{y_\beta}^!, p_* j^! \cC_{X_0}),
        \end{split}
    \end{equation*}
where $\tilde j$ is the inclusion of $p^{-1}(y_\beta)\cap X_0$.
For the latter, we study $H^{-\dim Y_\beta}(i_{y_\beta}^*)$. We have
\begin{equation*}
    \begin{split}
    & H^{\dim X - \dim Y_\beta}_c(p_-^{-1}(y_\beta)\cap X_0)
    = H^{- \dim Y_\beta}_c(p_-^{-1}(y_\beta)\cap X_0, \tilde j_-^* \cC_{X_0})
\\
   =\; & H^{- \dim Y_\beta}(i^*_{y_\beta} (p_-)_! j_-^* \cC_{X_0}).
    \end{split}
\end{equation*}
\end{NB}%
Note that $p^{-1}(y_\beta)\cap X_0$ and $p_-^{-1}(y_\beta)\cap X_0$
are at most $(\dim X - \dim Y_\beta)/2$-dimensional if $\Phi$ is
hyperbolic semi-small. In this case, cohomology groups have bases
given by $(\dim X - \dim Y_\beta)/2$-dimensional irreducible components
of $p^{-1}(y_\beta)\cap X_0$ and $p_-^{-1}(y_\beta)\cap X_0$
respectively.
Let $H_{\dim X - \dim Y_\beta}(p^{-1}(y_\beta)\cap X_0)_\chi$ and
$H^{\dim X - \dim Y_\beta}_c(p_-^{-1}(y_\beta)\cap X_0)_\chi$ denote
the components corresponding to a simple local system $\chi$ on
$Y_\beta$.

\begin{Theorem}\label{thm:hypsemismall}
    Suppose $\Phi$ is hyperbolic semi-small. Then $\Phi(\IC(X))$ is
    perverse and it is isomorphic to
    \begin{equation*}
        \bigoplus_{\beta,\chi} \IC(Y_\beta,\chi)\otimes
        H_{\dim X-\dim Y_\beta}(p^{-1}(y_\beta)\cap X_0)_\chi.
    \end{equation*}
Moreover, we have an isomorphism
\begin{equation*}
    H_{\dim X-\dim Y_\beta}(p^{-1}(y_\beta)\cap X_0)_\chi
    \cong
    H^{\dim X-\dim Y_\beta}_c(p_-^{-1}(y_\beta)\cap X_0)_\chi.
\end{equation*}
\end{Theorem}

The proof is similar to one in \cite[Theorem~3.5]{MV2}, hence the
detail is left as an exercise for the reader. In fact, we only use the
case when $X^T$ is a point, and we explain the argument in detail
for that case in \thmref{thm:twistedmain}.

The same assertion holds for $\IC(X_0,\mathcal L)$ the intersection
cohomology complex with coefficients in a simple local system
$\mathcal L$ over $X_0$, if we put $\mathcal L$ also to cohomology
groups of fibers.

Note that $\Phi(\IC(X_\beta,\mathcal L_\beta))$ is also perverse
for a local system $L_\beta$ on $X_\beta$, and isomorphic to
\begin{equation*}
    \bigoplus_{\beta,\chi} \IC(Y_\beta,\chi)\otimes
    H_{\dim X_\beta-\dim Y_\beta}(p^{-1}(y_\beta)\cap X_\beta)_\chi.
\end{equation*}
Conversely, if $\Phi(\IC(X_\beta,\mathcal L_\beta))$ is perverse, we
have the dimension estimates \eqref{eq:est}. It is because the top
degree cohomology groups are nonvanishing, and contribute to nonzero
perverse degrees.  See the argument in Corollary~\ref{cor:dimest} for
detail.

\begin{NB}
  Writing the following subsection was the main purpose of this note,
  but I will probably move it to the second preprint, discussing
  Kazhdan-Lusztig type conjecture for the $\scW$-algebra. The remaining
  part should be absorbed into BFN.tex

\subsection{Ext algebra of hyperbolic restriction}

Let us note that we have also
\begin{equation}
    X^A\overset{p_-}{\underset{i_-}{\leftrightarrows}}
    \mathcal R_X \xrightarrow{j_-} X
\end{equation}
such that $j_-\circ  i_-= j\circ i$. Braden \cite{Braden} proved that
$i_-^! j_-^* = i^* j^!$. Therefore we have
\begin{equation}
  \Ext_{D^b_T(X^A)}(i^* j^! \scF, i^* j^! \scG)
  \cong
  \Ext_{D^b_T(X^A)}(i_-^! j_-^* \scF, i^* j^! \scG)
\end{equation}
for $\scF,\scG\in D^b_T(X)$.
We have natural morphisms $i_-^! \to i_-^*$ and $i^!\to i^*$
\cite[(8.5.1)]{CG}, hence we have a homomorphism
\begin{equation}\label{eq:hom1}
    \Ext_{D^b_T(X^A)}(i_-^* j_-^* \scF, i^! j^! \scG)
    \to
    \Ext_{D^b_T(X^A)}(i_-^! j_-^* \scF, i^* j^! \scG).
\end{equation}
The left hand side is
\begin{equation}
  \begin{split}
     \Ext_{D^b_T(X)}(\scF, (j\circ i)_* (j\circ i)^! \scG)
     \cong
     \Ext_{D^b_T(X)}(\scF, (j\circ i)_! (j\circ i)^! \scG).
  \end{split}
\end{equation}
Therefore we also have a natural homomorphism
\begin{equation}\label{eq:hom2}
    \Ext_{D^b_T(X^A)}(i_-^* j_-^* \scF, i^! j^! \scG)
    \to
    \Ext_{D^b_T(X)}(\scF, \scG).
\end{equation}

\begin{Proposition}
  \textup{(1)} Two homomorphisms \eqref{eq:hom1}, \eqref{eq:hom2}
  commute with Yoneda products, explained as below.

  \textup{(2)} The homomorphism \eqref{eq:hom1} \textup(resp.\
  \eqref{eq:hom2}\textup) becomes an isomorphism after inverting an
  element $f\in\CC[\ft]$ with \eqref{eq:cond} for $x\in
  \cA_X\times_{X^A}\mathcal R_X\setminus X^A$ \textup(resp.\
  $X\setminus X^A$\textup).
\end{Proposition}

In particular, we have an algebra isomorphism
\begin{equation}
  \Ext_{D^b_T(X)}(\scF, \scF)_f \cong
  \Ext_{D^b_T(X^A)}(i^* j^! \scF, i^* j^! \scF)_f
\end{equation}
for an appropriate $f\in\CC[\ft]$.

\begin{proof}
  (1) Let us first consider how the homomorphism \eqref{eq:hom2}
  behaves under Yoneda product. Let us omit the equivariant derived
  category $D^b_X(\ )$ from the notation of $\Ext$ groups for brevity.

  Let $\scF$, $\scG$, $\scH\in D^b_T(X)$. Let $\alpha = j\circ i = j_-\circ
  i_-$. We consider
\begin{equation}
  \begin{CD}
    \Ext(\scF, \alpha_! \alpha^! \scG) \otimes \Ext(\scG, \alpha_! \alpha^! \scH)
    @>>> \Ext(\scF, \alpha_!\alpha^! \scH)
\\
   @VVV
   @VVV
\\
    \Ext(\scF, \scG) \otimes \Ext(\scG, \scH)
    @>>> \Ext(\scF, \scH),
  \end{CD}
\end{equation}
where the upper horizontal arrow is given by composing
$\alpha_!\alpha^!\scG\to \scG$. The vertical arrows are given by
compositions of $\alpha_!\alpha^!\scG\to \scG$ and $\alpha_!\alpha^!\scH\to
\scH$. The commutativity of the diagram is obvious.

Next we turn to the compatibility of \eqref{eq:hom1} with the Yoneda
product. We consider the diagram
\begin{equation}
  \begin{CD}
  \Ext(i_-^* j_-^* \scF, i^! j^! \scG)\otimes \Ext(i_-^* j_-^* \scG, i^! j^! \scH)
  @>>>
  \Ext(i_-^* j_-^* \scF, i^! j^! \scH)
\\
   @VVV @VVV
\\
  \Ext(i_-^! j_-^* \scF, i^* j^! \scG)\otimes \Ext(i_-^! j_-^* \scG, i^* j^! \scH)
@>>>
  \Ext(i_-^! j_-^* \scF, i^* j^! \scH),
  \end{CD}
\end{equation}
where the upper horizontal arrow is given by composing a natural
homomorphism $i^! j^! \scG \to i^* j^* \scG = i_-^* j_-^* \scG$ in the
middle. The lower horizontal arrow is given by Braden's isomorphism
$i^* j^! \scG\xrightarrow{\cong}i_-^! j_-^*\scG$. Vertical arrows are
\eqref{eq:hom1}. The commutativity of the diagram means that
\begin{equation}
  i^! j^! \scG \to i^* j^! \scG \xrightarrow{\cong}i_-^! j_-^* \scG
  \to i_-^* j_-^* \scG
\end{equation}
is the natural homomorphism above.
This follows from the definition of $i^* j^!
\scG\xrightarrow{\cong}i_-^! j_-^*\scG$ in \cite{Braden}.
\begin{NB2}
  I need to check whether this is true. I need to read \cite{Braden}
  carefully. In this \verb+NB+, we swap $i$, $i_-$ and $j$, $j_-$ as
  this was written before we change the definition of the hyperbolic
  restriction.

  Let us recall that $i_-^* j_-^! \scG\to i^! j^* \scG$ was constructed as
  \begin{equation}
    i_-^* j_-^! \scG \to i_-^* j_-^! j_* j^* \scG \cong i_-^* (i_-)_* i^! j^* \scG
    \cong i^! j^* \scG,
  \end{equation}
  where the first homomorphism is the adjunction $\id\to j_* j^*$, the
  second is the base change for the cartesian square
  \begin{equation}
    \begin{CD}
      X^A @>i_->> \mathcal R_X
\\
     @V{i}VV @VV{j_-}V
\\
     \cA_X @>>j> X,
    \end{CD}
  \end{equation}
  and the third one is also the adjunction $i_-^* (i_-)_*\to \id$,
  which is an isomorphism on equivariant objects
  (\cite[Lemma~6]{Braden}).

  Suppose $X^A = \mathrm{pt}$. Then Braden's isomorphism is written by
  the global cohomology group as
  \begin{equation}
    H^*(X,X\setminus\mathcal R_X,\scG) \xrightarrow{j^*}
    H^*(\cA_X,\cA_X\setminus X^A,\scG) \xrightarrow{\cong}
    H^*_c(\cA_X,\scG).
  \end{equation}
  Now $i_-^!j_-^! \scG\to i_-^* j_-^! \scG$ and $i^! j^* \scG\to i^*j^* \scG$ are
  restriction maps and given respectively by
  \begin{equation}
    H^*(X,X\setminus X^A,\scG) \to H^*(X,X\setminus\mathcal R_X,\scG), \qquad
    H^*_c(\cA_X, \scG) \to H^*(X^A,\scG).
  \end{equation}
  The composition is just given by the restriction map
  \begin{equation}
    H^*(X,X\setminus X^A,\scG) \to H^*(X^A,\scG).
  \end{equation}
  This is nothing but $(j_- i_-)^! \scG \to (j_- i_-)^* \scG$.
\end{NB2}

(2)
We replace \eqref{eq:hom1} as
\begin{equation}\label{eq:hom3}
  \Ext(i_-^* j_-^* T, i^! j^! \scG)
  \to
  \Ext((p_-)_! j_-^* T, p_* j^! \scG),
\end{equation}
as natural homomorphisms $p_*\to i^*$, $i_-^! \to (p_-)_!$ given by
adjunction are isomorphisms on objects above \cite[Lemma~6]{Braden}.
\begin{NB2}
  We have two ways to use adjunction. The first one is $i^!\to i^! p^!
  p_! = p_!$. The second one is $i^! = p_! i_! i^! \to p_!$. They are
  the same. But I do not see how this follows from a formal
  manipulation. I consider their adjoint $p_* \to i^*$ on a sheaf
  $D$. Then $(p_* D)(U) = D(p^{-1}(U))$, $(i^* D)(U) = \varinjlim_{V}
  D(V)$ where $V$ ranges through the family of open neighborhoods of
  $i(U)$. Now $p^{-1}(U)\supset i(U)$, we have the canonical
  restriction homomorphism $(p_* D)(U)\to (i^* D)(U)$. Both $p_* = i^*
  p^* p_* \to i^*$ and $p_* \to p_* i_* i^* = i^*$ are equal to this
  homomorphism. For the first one, we have $(p^*p_* D)(V) =
  \varinjlim_W D(p^{-1}(W))$ where $W$ ranges throughout the family of
  open neighborhoods of $p(V)$. Then we have the restriction
  homomorphism $(p^* p_* D)(V) \to D(V)$ as $p^{-1}(W)\supset
  V$. Then its pullback by $i^*$ is the canonical restriction
  homomorphism. For the second one, we have $(i_*i^* D)(V) =
  (i^*D)(i^{-1}(V)) = \varinjlim_W D(W)$ where $W$ ranges thorough the
  family of open neighborhoods of $i(i^{-1}(V))$. Since $V\supset
  i(i^{-1}(V))$, we have a restriction $D(V) \to (i_*i^* D)(V)$. Then
  we pushforward it by $p_*$ to get the canonical restriction
  homomorphism.
\end{NB2}%
We have two ways to go from $(p_-)_! j_-^* \scF$ to $i_-^* j_-^*
\scF$. First, as above, we use the inverse of the adjunction $(p_-)_!
j_-^* \scF\to i_-^! j_-^* \scF$ and then compose it with $i_-^! j_-^* \scF\to
i_-^* j_-^* \scF$. Second we compose $(p_-)_!  j_-^*\scF \to (p_-)_* j_-^*\scF$
and the adjunction $(p_-)_* j_-^*\scF \to i_-^* j_-^* \scF$.
\begin{NB2}
  Similarly two ways for $i^! j^! \scG \to p_* j^! \scG$.
\end{NB2}%
They are the same, i.e., $i_-^!\to i_-^*$ is equal to the composite of
$i_-^!\to (p_-)_!\to (p_-)_* \to i_-^*$.
\begin{NB2}
In fact, the first arrow is given by the
adjunction as $i^! = p_! i_! i^! \to p_!$. Therefore the composite of
the first and the second is $i^! = p_* i_* i^! = p_* i_! i^! \to
p_*$. The arrow $p_*\to i^*$ is given by $p_* = i^* p^* p_* \to
i^*$. Therefore the composite of all is
\(
   i^! = i^* p^* p_* i_! i^! \to i^*
\)
applied first by $i_!i^!\to\id$ and then by $p^* p_*\to\id$.
If we first apply $p^* p_*\to\id$, we get $i^* i_! i^!\to i^*$. But
$i^* i_! i^! = i^!$ (see \cite[3.1.12]{KaSha}) and $i^* i_! i^!\to
i^*$ is nothing but $i^!\to i^*$.
\end{NB2}%

Thus we further replace \eqref{eq:hom3} by
\begin{equation}\label{eq:l2}
  H^*_T(X^A, \delta_L^! (i_-\times i)^! \scH) \to
  H^*_T(X^A, \delta_L^! (p_-\times p)_* \scH),
\end{equation}
where $\delta_L\colon X^A\to X^A\times X^A$ is the diagonal
embedding and $\scH = (j_-\times j)^! (\scF^\vee\boxtimes \scG)$.
See \cite[(8.3.16)]{CG}.

\begin{NB2}
Now the homomorphism is given by composing
\(
   \delta_L^!(i_-\times i)^! \to \delta_L^!(p_-\times p)_!
   \to \delta_L^!(p_-\times p)_*.
\)
\end{NB2}%
We have a cartesian square
\begin{equation}\label{eq:l4}
  \begin{CD}
  \mathcal R_X\times_{X^A}\cA_X @>{\tilde p}>> X^A
\\
  @V{\tilde\delta}VV @VV\delta_LV
\\
  \mathcal R_X\times \cA_X @>>p_-\times p> X^A\times X^A.
  \end{CD}
\end{equation}
Hence the right hand side of \eqref{eq:l2} is written as
\begin{equation}\label{eq:rhs}
  H^*_T(X^A, \tilde p_* \tilde\delta^! \scH)
  \cong H^*_T(\mathcal R_X\times_{X^A}\cA_X,
  \tilde\delta^! \scH).
\end{equation}

On the other hand, the left hand side of \eqref{eq:l2} is
\begin{equation}\label{eq:lhs}
  H^*_T(X^A, \delta_L^! (i_-\times i)^! \scH)
  \cong
    H^*_T(X^A, \tilde i^! \tilde \delta^! \scH),
\end{equation}
where $\tilde i$ is defined by
\begin{equation}
  \begin{CD}
      \mathcal R_X\times_{X^A}\cA_X @<{\tilde i}<< X^A
\\
  @V{\tilde\delta}VV @VV\delta_LV
\\
  \mathcal R_X\times \cA_X @<<i_-\times i< X^A\times X^A.
  \end{CD}
\end{equation}
Now let us check that the homomorphism from \eqref{eq:lhs} to
\eqref{eq:rhs} is the special case of \eqref{eq:tri2}, as $X^A$ is the
fixed point set in $\mathcal R_X\times_{X^A}\cA_X$ with respect
to $T$.

Let us consider the diagram
\begin{equation}\label{eq:l3}
  \begin{CD}
  \tilde \delta^! \scH @>>> \tilde \delta^! (p_-\times p)^! (p_-\times p)_! \scH
  @= \tilde p^! \delta_L^! (p_-\times p)_! \scH
\\
  @VVV @. @VV{!\to *}V
\\
  \tilde p^! \tilde p_! \tilde\delta^! \scH
  @>>{!\to *}>
  \tilde p^! \tilde p_* \tilde\delta^! \scH
  @= \tilde p^! \delta_L^! (p_-\times p)_* \scH.
  \end{CD}
\end{equation}
Our homomorphism from \eqref{eq:lhs} to \eqref{eq:rhs} is $\tilde i^!$
applied to the composite of arrows starting from $\delta^! \scH$ to
$\tilde p^! \tilde p_* \tilde\delta^! \scH$ passing through the right
upper corner.
We prove that this diagram is commutative later. Then our homomorphism
is the composite of
\begin{equation}\label{eq:l1}
  H^*_T(X^A, \tilde i^!\tilde\delta^! \scH) \xrightarrow{\cong}
  H^*_T(X^A, \tilde p_!\tilde\delta^! \scH) \to
  H^*_T(X^A, \tilde p_*\tilde\delta^! \scH).
\end{equation}
The left arrow is induced from the adjunction as
\(
   \tilde i^! \to \tilde i^!\tilde p^! \tilde p_! = \tilde p_!,
\)
as before, but it is the same as
\(
  \tilde i^! = \tilde p_! \tilde i_! \tilde i^!
  \to \tilde p_!,
\)
as can be checked by a standard argument.
\begin{NB2}
  See the above \verb+NB2+.
\end{NB2}%
Now the composition is induced by $\tilde i^! = \tilde p_* \tilde i_*
\tilde i^! = \tilde p_* \tilde i_! \tilde i^! \to \tilde p_*$.
Hence \eqref{eq:l1} is identified with
\begin{equation}
  H^*_T(\mathcal R_X\times_{X^A}\cA_X, \tilde i_!\tilde i^!
  \tilde\delta^!\scH)
  \to
  H^*_T(\mathcal R_X\times_{X^A}\cA_X, \tilde\delta^!\scH)
\end{equation}
given by the adjunction $\tilde i_!\tilde i^!\to\id$. This is nothing
but \eqref{eq:tri2}. Thus the assertion follows.

Let us prove the commutativity of \eqref{eq:l3}. Since all involved
maps appear in the cartesian square \eqref{eq:l4}, let us change the
notation as
\begin{equation}
  \begin{CD}
    X\times_Z Y @>\tilde f>> Y
\\
   @V{\tilde g}VV @VV{g}V
\\
   X @>>f> Z
  \end{CD}
\end{equation}
and show that
\begin{equation}
  \begin{CD}
    \tilde g^! @>>> \tilde g^! f^! f_! @= \tilde f^! g^! f_!
\\
   @VVV @. @VVV
\\
   \tilde f^! \tilde f_! \tilde g^! @>>> \tilde f^! \tilde f_* \tilde g^!
   @= \tilde f^! g^! f_*
  \end{CD}
\end{equation}
is commutative. We apply the Verdier duality to the diagram to get
\begin{equation}
  \begin{CD}
    \tilde g^* @<<< \tilde g^* f^* f_* @= \tilde f^* g^* f_*
\\
   @AAA @. @AAA
\\
   \tilde f^* \tilde f_* \tilde g^* @<<< \tilde f^* \tilde f_! \tilde g^*
   @= \tilde f^* g^* f_!.
  \end{CD}
\end{equation}
Now the commutativity can be checked by considering the case of
functors applied to a sheaf.

The proof for \eqref{eq:hom2} is easier and is omitted.
\end{proof}
\end{NB}

\subsection{Recovering the integral form}

We assume \eqref{eq:assump} and also that $X$ is affine. We consider
the hyperbolic restriction with respect to $T$.
\begin{NB}
  (instead of $A$.)
\end{NB}%

Let $\bA_T = \CC[\operatorname{Lie}(\TT)] = \CC[\ve_1,\ve_2,\ba]$
and $\bF_T$ be its quotient field.
\index{AAT@$\bA_T = \CC[\ve_1,\ve_2,\ba]$|textit}
\index{FFT@$\bF_T = \CC(\ve_1,\ve_2,\ba)$|textit}

We further assume that
\(
  H^*_{\TT,c}(X,\scF)
\)
is torsion free over $H^*_\TT(\mathrm{pt}) = \bA_T$, i.e,
\(
  H^*_{\TT,c}(X,\scF)\to H^*_{\TT,c}(X,\scF)\otimes_{\bA_T}\bF_T
\)
is injective.
This property for the Uhlenbeck space will be proved in
\lemref{lem:free}.
\begin{NB2}
  The above assumption is added on Oct. 13.
\end{NB2}

We consider a homomorphism
\begin{equation}\label{eq:tensorhom}
  H^*_{\TT,c}(X, \scF) \cong H^*_{\TT,c}(X^T,i^! j^! \scF)
  \to H^*_{\TT,c}(X^T, i^* j^! \scF)
\end{equation}
for $\scF\in D^b_\TT(X)$. The first isomorphism is given in
\lemref{lem:centralfiber}. By the localization theorem, the second
homomorphism becomes an isomorphism after inverting an element $f\in
\CC[\operatorname{Lie}\TT]$ which vanishes on the union of the Lie
algebras of the stabilizers of the points $x\in\cA_X\setminus
X^T$.

\begin{Theorem}\label{thm:nonlocal}
  Consider the intersection
  \(
    H^*_{\TT,c}(X^T, i^* j^! \scF)\cap H^*_{\TT,c}(X^T,
    i_-^* j_-^! \scF)
  \)
  in $H^*_{\TT,c}(X, \scF)\otimes_{\bA_T} \bF_T$. It coincides with
  $H^*_{\TT,c}(X,\scF)$.
\end{Theorem}

The proof occupies the rest of this subsection. We first give a key
lemma studying stabilizers of points in $\cA_X\setminus X^T$.

\begin{Lemma}\label{lem:cochar}
  Suppose that $(\lambda^\vee,n_1, n_2)$ is a cocharacter of $\TT$
  such that either of the followings holds
  \begin{enumerate}
  \item $\lambda^\vee$ is dominant and $n_1$, $n_2 > 0$.
  \item $\lambda^\vee$ is regular dominant and $n_1$, $n_2 \ge 0$.
  \end{enumerate}
  Then there is no point in $\cA_X\setminus X^T$ whose stabilizer
  contains $(\lambda^\vee,n_1,n_2)(\CC^*)$.
\end{Lemma}

\begin{NB}
  The above is taken from Sasha's message on June 7, 2013. Originally
  it was assumed only (1) that $\lambda^\vee$ is dominant, i.e.,
  $\langle \lambda^\vee, \alpha_i\rangle \ge 0$ for any simple root
  $\alpha_i$, and $n_1$, $n_2 > 0$. The argument gave a stronger
  consequence $x\in X^\TT$ (rather than $x\in X^T$). It seems that
  this assumption is not enough for our purpose as we cannot exclude,
  say the case $\mu = 0$, $m_1 = 0$, $m_2 > 0$.
\end{NB}%

\begin{proof}
  Assume $\lambda$ is dominant and $n_1$, $n_2\ge 0$.

  Suppose that $x\in \cA_X$ is fixed by
  $(\lambda^\vee,n_1,n_2)(\CC^*)$. Then we have
  \begin{equation}\label{eq:tx}
    \lambda^\vee(t^{-1})\cdot x = (t^{n_1}, t^{n_2})\cdot x.
  \end{equation}
  Since $\lambda^\vee$ is dominant, its attracting set contains
  $\cA_X$. Therefore the left hand side has a limit when
  $t\to\infty$. On the other hand, the right hand side has a limit
  when $t\to 0$. Therefore $\CC^*\ni t\mapsto
  \lambda^\vee(t^{-1})\cdot x\in X$ extends to a morphism $\proj^1\to
  X$. As $X$ is affine, such a morphism must be constant, i.e.,
  \eqref{eq:tx} must be equal to $x$.

  If $n_1$, $n_2 > 0$, $x$ must be the unique $\CC^*\times\CC^*$ fixed
  point. It is contained in $X^T$.

  If $\lambda^\vee$ is regular, $x$ is fixed by $T$, that is $x\in
  X^T$.
\end{proof}

\begin{proof}[Proof of Theorem~{\rm\ref{thm:nonlocal}}]
  Let $\alpha$ be an element in $H^*_{\TT,c}(X, \scF)$ which is not
  divisible by any non-constant element of $\bA_T$. Let $J_\alpha^\pm$
  be two fractional ideals of $\bA_T$ consisting of those rational
  functions $f$ such that $f\alpha\in H^*_{\TT,c}(X^T, i^* j^!\scF)$ and
  $f\alpha\in H^*_{\TT,c}(X^T, i_-^* j_-^!\scF)$ respectively.  We need
  to show that $J_\alpha^+ \cap J_\alpha^- = \bA_T$. Note that a
  priori the right hand side is embedded in the left hand side.

  Let $f\in J_\alpha^+$. Then $f = g/h$ where $g$, $h\in \bA_T$ and
  $h$ is a product of linear factors of the form $(\mu, m_1, m_2)$
  such that
\begin{itemize}
\item $\langle \lambda^\vee, \mu\rangle > 0$ for a regular dominant
  coweight $\lambda^\vee$, and
  \begin{NB}
    Sasha wrote a sum of positive roots. But we can multiply a
    positive root by a positive scalar, right ?
  \end{NB}
\item  $m_1$, $m_2\ge 0$ with at least one of them nonzero.
\end{itemize}
In fact, we have $\langle (\lambda^\vee,n_1,n_2),(\mu, m_1,
m_2)\rangle \neq 0$ for any $(\lambda^\vee,n_1,n_2)$ as in
\lemref{lem:cochar}. Taking a regular dominant coweight $\lambda^\vee$
and $n_1$, $n_2 = 0$, we get the first condition. Next we take
$\lambda = 0$ and $n_1$, $n_2 > 0$ and get the second condition.

Similarly for $f = g/h\in J_\alpha^-$, $h$ is a product of
$(\mu,m_1,m_2)$ with $\langle \lambda^\vee,\mu\rangle < 0$ for a
regular dominant coweight $\lambda^\vee$, and the same conditions for
$(m_1,m_2)$ as above.
\begin{NB}
  Sasha used an involution $\tau$ sending $(\mu,m_1,m_2)$ to
  $(-\mu,m_1,m_2)$. But I do not understand why it is necessary.
\end{NB}%
Then there are no linear factors satisfying both conditions, hence we have
$J_\alpha^+\cap J_\alpha^- = \bA_T$.
\end{proof}

\begin{NB2}
  Question : Do we have a similar characterization for the Ext algebra
  ? If we would have it, we can say probably that generators of the
  integral form of the $\scW$-algebra are in the Ext algebra: generators
  of integral Virasoro are in the Ext algebra for $\Uh[SL(2)]d$.
\end{NB2}

% \bibliographystyle{myamsplain}
% \bibliography{nakajima,mybib,tensor2,MO}

% \end{document}

%%% Local Variables:
%%% mode: latex
%%% TeX-master: "BFN"
%%% End:

\section{Hyperbolic restriction on Uhlenbeck spaces}\label{sec:hyper-uhlenbeck}

This chapter is of technical nature, but will play a quite important
role later. Feigin-Frenkel realized the $\scW$-algebra $\scW_k(\g)$
\index{Wkg@$\scW_k(\g)$|textit} in the Heisenberg algebra $\Heis(\h)$
associated with the Cartan subalgebra $\h$ of $\g$.  (See
\cite[Ch.~15]{F-BZ}.)

We will realize this picture in a geometric way. In \cite{MO}
Maulik-Okounkov achieved it by {\it stable envelopes\/} which relate the
cohomology group of Gieseker space to that of the fixed point set with
respect to a torus.
The former is a module over $\scW_k(\g)$ and the latter is a Heisenberg
module.
In \cite{SV} Schiffmann-Vasserot also related two cohomology groups by
a different method.

We will take a similar approach, but we need to use a sheaf theoretic
language, as Uhlenbeck space is singular. We use the {\it hyperbolic
  restriction functor} in \subsecref{sec:hyperb-restr-1}, and combine
it with the theory of stable envelopes. This study was initiated by the
third author \cite{tensor2}. A new and main result here is
\thmref{thm:perverse}, which says that perversity is preserved under
the hyperbolic restriction in our situation.

\begin{NB}
    Sasha, please put an explanation on the result of
    Mirkovi\'c-Vilonen \cite{MV2}.

    Sasha says that there is no formal relation. I agree, but at least
    that they proved that a hyperbolic restriction preserves the
    perversity in their situation.
\end{NB}

We fix a pair $T\subset B$ of a maximal torus $T$ and a Borel subgroup
$B$, and consider only parabolic subgroups $P$ containing $B$, except
we occasionally use opposite parabolic subgroups $P_-$ until
\subsecref{sec:autg-invariance}. In \subsecref{sec:autg-invariance},
we consider other parabolic subgroups also.

\subsection{A category of semisimple perverse sheaves}

Let $\IC(\BunGl{d},\rho)$ denote the intersection cohomology (IC)
complexes, where $\rho$ is a simple local system on $\BunGl{d} =
\Bun{d}\times S_\lambda\AA^2$ corresponding to an irreducible
representation of $S_{n_1}\times S_{n_2}\times\cdots$ via the covering
\begin{equation}
  (\AA^2)^{n_1}\times (\AA^2)^{n_2}\times
  \cdots \setminus\text{diagonal}
  \to
  S_{\lambda}\AA^2,
\end{equation}
where $\lambda = (1^{n_1} 2^{n_2}\cdots)$. (Recall $S_\lambda\AA^2$ is
a stratum of $S^{|\lambda|}\AA^2$, see \eqref{eq:Sl}.)

\begin{Definition}
Let $\Perv(\Uh{d})$\index{perveU@$\Perv(\Uh{d})$} be the additive
subcategory of the abelian category of semisimple perverse sheaves on
$\Uh{d}$, consisting of finite direct sums of $\IC(\BunGl{d},\rho)$.
\end{Definition}

By abuse of notation, we use the same notation $\IC(\BunGl{d},\rho)$
even if $\rho$ is a reducible representation of $S_{n_1}\times
S_{n_2}\times\cdots$. It is the direct sum of the corresponding
simple IC sheaves.

If $\rho$ is the trivial rank $1$ local system, we omit $\rho$ from
the notation and denote the corresponding IC complex by
$\IC(\BunGl{d})$, or $\IC(\UhGl{d})$.

Furthermore, we omit $\lambda$ from the notation when it is the empty
partition $\emptyset$. Therefore $\IC(\Uh{d})$ means
$\IC(\Bun[G,\emptyset]{d})$.

Objects in $\Perv(\Uh{d})$ naturally have structures of equivariant
perverse sheaves in the sense of \cite{BL} with respect to the group
action $\GG = G\times\CC^*\times\CC^*$ on $\Uh{d}$. We often view
$\Perv(\Uh{d})$ as the subcategory of equivariant perverse sheaves.

\subsection{Fixed points}\label{subsec:fixed}

Let $P$ be a parabolic subgroup of $G$ with a Levi subgroup $L$. Let
$A = Z(L)^0$ denote the connected center of $L$.
\begin{NB}
  $Z(L)$ is not connected in general. Misha's message on Oct. 30:
  consider $G = SL(4)$ and the Levi subgroup of type $(2,2)$. Then the
  center consists of diagonal matrices of the form $(x,x,y,y)$ with
  $x^2 y^2 = 1$. Therefore we have components correspond to $xy = \pm
  1$.
\end{NB}%
Let $\BunL{d}$ denote the moduli space of $L$-bundles on $\proj^2$
with trivialization at $\linf$ of `instanton number $d$'. The latter
expression makes sense, since the notion of instanton number, defined
as in \subsecref{subsec:instnumber}, corresponds to a choice of a
bilinear form on the coweight lattice, which is the same for $G$ and
for $L$.

Suppose that $\mathcal F\in \Bun{d}$ is fixed by the $A$-action. It
means that bundle automorphisms at $\linf$ parametrized by $A$
extend to the whole space $\proj^2$. The extensions are
unique. Therefore the structure group $G$ of $\mathcal F$ reduces to
the centralizer of $A$, which is $L$. Hence $(\Bun{d})^{A} =
\BunL{d}$.

Let us consider the fixed point subvariety
\begin{equation}
    \UhL{d} = (\Uh{d})^{A} \index{UhL@$\UhL{d}$}
\end{equation}
in the Uhlenbeck space.
Then we have an induced stratification
\begin{equation}
    \UhL{d} = \bigsqcup_{d_1+d_2=d, \lambda \vdash d_2} \BunLl{d_1},
    \quad \BunLl{d_1} = \BunL{d_1} \times S_\lambda\AA^2.
\end{equation}

Strictly speaking, our $\UhL{d}$ depends on the choice of the
embedding $L\to G$, therefore should be denoted, say by
$\Uh[L,G]{d}$. We think that there is no fear of confusion.

Note that $[L,L]$ is again semi-simple and simply-connected.
\begin{NB}
  Do you have a reference ? I googled and find

  \url{http://mathoverflow.net/questions/126516/simply-connected-algebraic-groups-and-reductive-subgroups-of-maximal-rank}.

It is referred to Borel-Tits. Which paper ?

And I give some detail on the relation between $\Uh[{[L,L]}]{d}$ and
$\UhL{d}$. Please check it.
\end{NB}%
(See \cite[Cor.~4.4]{Borel-Tits}.)
Suppose that we have only one simple factor. Since we assume $G$ is
simply-laced, $[L,L]$ is also. The instanton number is the same for
$G$ and $[L,L]$. Otherwise we define the instanton number for $[L,L]$
by the invariant form on $\operatorname{Lie}([L,L])$ induced from one
on $\g$.

We only have trivial framed $L/[L,L]$-bundles as $H^2(\proj^2)$ is
$1$-dimensional hence the first Chern class of a framed bundle
vanishes. Thus we have
\begin{equation}\label{eq:BunLL}
  \BunL{d_1} = \Bun[{[L, L]}]{d_1}.
\end{equation}

Since $[L,L]$ is a subgroup of $G$, we have the induced closed
embedding $\Uh[{[L,L]}]{d}\to \Uh{d}$ (see \cite[Lem.~6.2]{BFG}),
which clearly factors as
\begin{equation}
  \Uh[{[L,L]}]{d}\to \UhL{d}.
\end{equation}
By \eqref{eq:BunLL}, this map is bijective.
\begin{NB}
  We do not know $\UhL{d}$ is normal.
\end{NB}%
Since both spaces are closed subschemes of $\Uh{d}$, we have

\begin{Proposition}\label{prop:LL}
  The morphism $\Uh[{[L,L]}]{d}\to \UhL{d} = (\Uh{d})^A$ is a
  homeomorphism between the underlying topological spaces.
\end{Proposition}

We are interested in perverse sheaves on $\UhL{d}$, hence we only need
underlying topological spaces. Hence we may identify $\UhL{d}$ and
$\Uh[{[L,L]}]{d}$.
We define the category $\Perv(\UhL{d})$ in the same way as
$\Perv(\Uh{d})$.

\begin{Example}\label{ex:torus_fixed}
  The case when $L$ is a maximal torus $T$ is most important. We have
\begin{equation}
  \Uh[T]{d} = S^d\AA^2 = \bigsqcup_{\lambda\vdash d} S_\lambda\AA^2,
\end{equation}
as we do not have nontrivial framed $T$-bundles.
\end{Example}

\subsection{Polarization}\label{sec:polarization}

Following \cite[\S3.3.2]{MO}, we introduce the notion of a {\it
  polarization\/} of a normal bundle of the smooth part of a fixed
point component.

Let us give a definition in a general situation. Suppose a torus $A$
acts on a holomorphic symplectic manifold $X$, preserving the
symplectic structure. Let $Z$ be a connected component of $X^A$ and
$N_Z$ be its normal bundle in $X$. Consider $A$-weights of a fiber of
$N_Z$.
Let $e(N_Z)|_{H^*_A(\mathrm{pt})}$ be the $H^*_A(\mathrm{pt})$-part of
the Euler class of the normal bundle, namely the product of all
$A$-weights of a fiber of $N_Z$.
Since $A$ preserves the symplectic form, $Z$ is a symplectic
submanifold, and weights of $N_Z$ appear in the pairs
$(\alpha_i,-\alpha_i)$. Hence
\begin{equation}\label{eq:e^2}
  (-1)^{(\codim Z)/2} e(N_Z)|_{H^*_A(\mathrm{pt})} = \prod \alpha_i^2
\end{equation}
is a perfect square. A choice of a square root
$\pol$\index{delta@$\pol$ (polarization)} of \eqref{eq:e^2} is called a
{\it polarization\/} of $Z$ in $X$.

In the next subsection we consider attractors and repellents. We have
a polarization $\pol_{\text{rep}}$ given by product of weights in
repellent directions. However this will not be a right choice to save
signs. Our choice of the polarization $\pol$, which follows
\cite[Ex.~3.3.3]{MO}, will be explained in \subsecref{sec:Gifixed} for
Gieseker spaces, and in \subsecref{sec:anotherbase2} for Uhlenbeck
spaces. Then we understand $\pol = \pm 1$, depending on whether it is
the same as or the opposite to $\pol_{\text{rep}}$, in other words we
identify $\pol$ with $\pol/\pol_{\text{rep}}$, as
$\delta_{\text{rep}}$ is clear from the context.

Note that a polarization does not make sense unless the variety $X$ is
smooth. Therefore we restrict the normal bundle to $Z\cap\Bun{d} =
Z\cap\Bun[L]d$ and consider a polarization there for Uhlenbeck spaces.

However a fixed point component $Z$, in general, does not intersect
with $\Bun{d}$. Say $Z\cap\Bun{d} = \emptyset$ if $L = T$. We do not
consider a polarization of $Z$ in this case, and smooth cases are
enough for our purpose.

\subsection{Definition of hyperbolic restriction functor}\label{sec:hyperb-restr}

We now return to the situation when $X = \Uh{d}$. We choose a
parabolic subgroup $P$ with a Levi subgroup $L$ as before.

We consider the setting in
\S\S\ref{sec:attr-repell},\ref{sec:hyperb-restr-1} with $A =
Z(L)^0$. Then \eqref{eq:29} is the hyperplane arrangement induced by
roots:
\begin{equation}
  \fa_\RR \setminus \bigcup_{\alpha} \{ \left.\alpha\right|_{\fa_\RR} = 0 \},
\end{equation}
where the union runs over all positive roots $\alpha$ which do not vanish
on $\fa_\RR$. The chambers are in one to one correspondence to the parabolic
subgroups containing $L$ as their Levi (associated parabolics).
Therefore the fixed $P$ determines a `positive' chamber.

We denote the corresponding attracting and repelling sets $\mathcal
A_X$, $\mathcal R_X$ by $\UhP{d}$\index{UhP@$\UhP{d}$} and
$\UhPm{d}$. Often we are going to drop the instanton number $d$ from
the notation, when there is no fear of confusion. We let
$i$\index{i@$i$ (inclusion $\UhL{d}\to\UhP{d}$)} and $p$\index{p@$p$
  (projection $\UhP{d}\to\UhL{d}$)} denote the corresponding maps from
$\UhL{}$ to $\UhP{}$ and from $\UhP{}$ to $\UhL{}$. Also we denote by
$j$\index{j@$j$ (inclusion $\UhP{d}\to \Uh{d}$)} the embedding of
$\UhP{}$ to $\Uh{}$. We shall sometimes also use similar maps $i_-$,
$j_-$ and $p_-$ where $\UhP{}$ is replaced with $\UhPm{}$. We have
diagrams
\begin{equation}
    \label{eq:1}
    \UhL{} \overset{p}{\underset{i}{\leftrightarrows}}
    \UhP{} \overset{j}{\rightarrow} \Uh{},
\qquad
    \UhL{} \overset{p_-}{\underset{i_-}{\leftrightarrows}}
    \UhPm{} \overset{j_-}{\rightarrow} \Uh{}.
\end{equation}

\begin{Definition}
    We define the functor $\Phi_{L,G}$\index{UZphiLG@$\Phi_{L,G}$} by
    $i^* j^! = p_* j^!$.
\end{Definition}
\begin{NB}
  Sometimes, we are going to write $\Phi_{L,G}^d$ (when we want to fix
  $d$).

  Do we really need the notation $\Phi_{L,G}^d$ ? It seems to me that
  $P$-dependence is more important. So why do not we write
  $\Phi_{L,G}^P$ when we want to emphasize $P$ ?

  \begin{NB2}
    Oct. 24: Now it is comment out.
  \end{NB2}
\end{NB}

We apply it to weakly $A$-equivariant objects, in particular on
$\Perv(\Uh{d})$.

\noindent
{\bf Warning.} Of course, the functor $\Phi_{L,G}$ depends on $P$ and
not just on $L$. When we want to emphasize $P$, we write
$\Phi^P_{L,G}$. Otherwise $P$ is always chosen so that $P\supset B$
for the fixed Borel subgroup $B$.

Let us justify our notation $\UhP{}$ for the attracting set. We have a
one parameter subgroup $\lambda\colon\GG_m\to G$ such that
\begin{equation}\label{eq:PLoneparam}
  \begin{split}
  P & = \left\{ g\in G\, \middle|\, \text{$\lim_{t\to 0}
  \lambda(t) g\lambda(t)^{-1}$ exists}\right\},
\\
  L &= G^{\lambda(\mathbb G_m)}
  = \left\{ g\in G\, \middle|\, \text{$
  \lambda(t) g = g\lambda(t)$ for any $t\in\mathbb G_m$}\right\}.
  \end{split}
\end{equation}
Then we have
\begin{equation}
  \begin{split}
    \UhP{} & \defeq \left\{ x\in \Uh{}\, \middle|\, \text{$\lim_{t\to 0}
  \lambda(t)\cdot x$ exists}\right\},
\\
  \UhL{} &\defeq (\Uh{})^{\lambda(\mathbb G_m)}
  = \left\{ x\in \Uh{}\, \middle|\, \text{$
  \lambda(t)\cdot x = x$ for any $t\in\mathbb G_m$}\right\}.
  \end{split}
\end{equation}
We embed $G$ into $SL(N)$ and consider the corresponding space for $G
= SL(N)$. We use the ADHM description for $\Uh[SL(N)]{}$ to identify
it with the affine GIT quotient as in \cite[Ch.~3]{Lecture}. Then
$SL(N) = SL(W)$, and $\UhP{}$ coincides with the variety
$\pi(\mathfrak Z)$ studied in \cite[\S3]{Na-Tensor}. Here $\pi$ is
Gieseker-Uhlenbeck morphism, and $\mathfrak Z$ is the attracting set
in the Gieseker space, which will be denoted by $\Gi{P}$ later.

In \cite[Rem.~3.16]{Na-Tensor} it was remarked that $\mathfrak Z$
parametrizes framed torsion free sheaves having a filtration $E = E^0
\supset E^1 \supset \cdots \supset E^k \supset E^{k+1} = 0$.
If all $F^i = E^i/E^{i+1}$ are locally free, $E$ is a $P$-bundle. Thus
$\UhP{}$ contains a possibly empty open subset $p^{-1}(\BunL{})$
consisting of $P$-bundles.

Let us, however, note that $\UhP{}\cap \Bun{}$ is not entirely
consisting of $P$-bundles, hence larger than
$p^{-1}(\BunL{})$: Consider a short exact sequence
\begin{equation*}
   0 \to F^2 \to E \to F^1 = \mathcal I_x \to 0,
\end{equation*}
arising from the Koszul resolution of the skyscraper sheaf at a point
$x\in\AA^2$. Here $\mathcal I_x$ is the ideal sheaf for $x$. Then
$E\in\UhP{}\cap\Bun{}$, but $E$ is not a $P$-bundle as $F^1$ is not
locally free.
More detailed analysis will be given in the proof of \propref{prop:Un}.

\subsection{Associativity}\label{sec:ass}

\begin{Proposition}\label{prop:trans}
  Let $Q$ be another parabolic subgroup of $G$, contained in $P$ and
  let $M$ denote its Levi subgroup. Let $Q_L$ be the image of $Q$ in
  $L$ and we identify $M$ with the corresponding Levi group.
  Then we have a natural isomorphism of functors
    \begin{equation}
        \label{eq:5}
       \Phi_{M, L} \circ \Phi_{L,G} \cong \Phi_{M,G}.
    \end{equation}
\end{Proposition}

\begin{proof}
    It is enough to show that
    \begin{equation}
        \label{eq:2}
            \UhP{} \times_{\UhL{}} \Uh[{Q_L}]{} = \Uh[Q]{},
    \end{equation}
    as
    \begin{equation*}
      p'_* j^{\prime !} p_* j^!
      = p'_* p''_* j^{\prime\prime !} j^!
      = (p'\circ p'')_* (j\circ j'')^!
    \end{equation*}
    in the diagram
    \begin{equation}\label{eq:CDpj}
      \begin{CD}
        \Uh[Q]{} @>{j''}>> \UhP{} @>{j}>> \Uh{}
\\
        @V{p''}VV @VV{p}V @.
\\
        \Uh[Q_L]{} @>{j'}>> \UhL{} @.
\\
        @V{p'}VV @. @.
\\
        \Uh[M]{} @. @.
      \end{CD}
    \end{equation}

    The left hand side of \eqref{eq:2} is just equal to
    $p^{-1}(\Uh[{Q_L}]{})$. By embedding $G$ into $SL(N)$ we may
    assume that $G = SL(N)$.
    \begin{NB}
        $\UhP{}$ is defined via a one parameter subgroup, therefore it
        is respected under the embedding $G\to SL(N)$.
    \end{NB}%
    In this case, we use the ADHM description
    to describe $\UhP{}$, $\Uh[Q]{}$, $\Uh[Q_L]{}$. By \cite[Proof of
    Lemma~3.6]{Na-Tensor}, they are consisting of data $(B_1,B_2,I,J)$
    such that $J F(B_1,B_2) I$ are in $P$, $Q$, $Q_L$ respectively,
    i.e., upper triangular in appropriate sense, for any products
    $F(B_1,B_2)$ of $B_1$, $B_2$ of arbitrary order. Now the assertion
    is clear.
    \begin{NB}
      Sep.15 : I give a proof via the ADHM.
    \end{NB}
\end{proof}

\subsection{Preservation of perversity}

The following is our first main result:

\begin{Theorem}\label{thm:perverse}
    $\Phi_{L,G}(\IC(\Uh{d}))$ is perverse \textup(and semi-simple,
    according to \cite[Theorem 2]{Braden}\textup). Moreover, the same is
    true for any perverse sheaf in $\Perv(\Uh{d})$.
\end{Theorem}

The proof will be given in \secref{sec:append-exactn}.

Let us remark that the result is easy to prove for type $A$, see
\cite[\S4.4, Lemma~3]{tensor2}. The argument goes back to an earlier
work by Varagnolo-Vasserot \cite{VV2}.

\begin{NB}
  Sep. 16 : The following subsection is moved from
  \subsecref{sec:anotherbase}, as the argument is already used.

  \begin{NB2}
    Oct. 23 : Don't we need the following subsection in the argument
    in Appendix ? It seems to me that you need it at the last
    moment. If so, why do not we move the following subsection to
    before this subsection ?
  \end{NB2}

  \begin{NB2}
    Answer is NO. It uses only the hyperbolic restriction of the
    trivial bundle on a smooth variety is trivial.
  \end{NB2}
\end{NB}

\subsection{Hyperbolic restriction on \texorpdfstring
{$\Bun[L]{d}$}{BunLd}}\label{sec:open}

Let us consider the restriction of $\Phi_{L,G}(\IC(\Uh{d}))$ to the
open subset $\Bun[L]{d}$ in this subsection.

For simplicity, suppose that $[L,L]$ has one simple factor so that the
instanton numbers of $L$-bundles are the same as those of
$[L,L]$-bundles. In particular, $\Bun[L]{d}$ is irreducible. Then
$\IC(\Uh[L]{d})$ is a simple perverse sheaf, and we study
\begin{equation}\label{eq:LG}
  \Hom_{\Perv(\UhL{d})}(\IC(\Uh[L]{d}),\Phi_{L,G}(\IC(\Uh{d}))).
\end{equation}

We restrict \eqref{eq:1} to the open subsets consisting of genuine
bundles:
\begin{equation}
    \label{eq:10}
    \BunL{d} \overset{p}{\underset{i}{\leftrightarrows}}
    p^{-1}(\BunL{d})
    \overset{j}{\rightarrow} \Bun{d}.
\end{equation}
Let us take $\mathcal F\in \BunL{d}$. Then the tangent space of
$\BunL{d}$ at $\mathcal F$ is $H^1(\proj^2, \mathfrak l_{\mathcal
  F}(-\linf))$, where $\mathfrak l$ is the Lie algebra of $L$. This is
the subspace of $H^1(\proj^2, \g_{\mathcal F}(-\linf)) = T_{\mathcal
  F}\Bun{d}$, consisting of $Z(L)^0$-fixed elements.
The normal bundle of $\BunL{d}$ in $\Bun{d}$ splits into the sum of
$H^1(\proj^2, \mathfrak n_{\mathcal F}(-\linf))$ and $H^1(\proj^2,
\mathfrak n^-_{\mathcal F}(-\linf))$, where $\mathfrak n$ is the nil
radical of $\mathfrak p = \operatorname{Lie}P$, and $\mathfrak n^-$ is
its opposite.
They correspond to attracting and repellent directions respectively.
Then $p^{-1}(\BunL{d})$ is a vector bundle over $\BunL{d}$, whose
fiber at $\mathcal F$ is $H^1(\proj^2, \mathfrak n_{\mathcal F}(-\linf))$. It
parametrizes framed $P$-bundles.
\begin{NB}
    For type $A$, this space parametrizes the extension of $E_2$ by
    $E_1$, hence is equal to $\Ext^1(E_1, E_2(-\linf))$. In general,
    it is a generalization of the Ext group, as above.
\end{NB}%
The morphism $p$ is the projection and $i$ is the inclusion of the
zero section. Therefore we have the Thom isomorphism between
$i^*j^!(\cC_{\Bun{d}})$ and $\cC_{\BunL{d}}$ up to shift.

Note further that $\dim p^{-1}(\BunL{d})$ is the half of the sum of
dimensions of $\BunL{d}$ and $\Bun{d}$, as $H^1(\proj^2, \mathfrak
n_{\mathcal F}(-\linf))$ and $H^1(\proj^2, \mathfrak n^-_{\mathcal
  F}(-\linf))$ are dual to each other with respect to the symplectic
form.
Hence a shift is unnecessary, and the Thom isomorphism gives the
canonical identification $i^* j^! (\cC_{\Bun{d}})\cong
\cC_{\BunL{d}}$.
Therefore we normalize the canonical homomorphism
\begin{equation}\label{eq:1LG}
  1^d_{L,G}\in \Hom_{\Perv(\UhL{d})}(\IC(\UhL{d}),
  \Phi_{L,G}(\IC(\Uh{d})))
\end{equation}
so that it is equal to the Thom isomorphism on the open subset.
\index{1LG@$1^d_{L,G}$}

\begin{NB}
  Sep. 16 :
  This homomorphism is not a right one, as it must be corrected by a
  polarization.

  For $L = L_i$, the correction will be given by a bipartite coloring
  of the vertexes of Dynkin diagram, though I still need to check the
  detail.
  \begin{NB2}
    I checked the detail. It is explained in
    \subsecref{sec:anotherbase2}.
  \end{NB2}

  If this problem will be fixed, I may change my definition of
  $1^d_{L,G}$. But it applies only to the case $L = L_i$, so I
  probably will not.
\end{NB}

Note also that a homomorphism in \eqref{eq:LG} is determined by its
restriction to $\BunL{d}$, hence \eqref{eq:LG} is $1$-dimensional from
the above observation. And $1^d_{L,G}$ is its base.

If $[L,L]$ has more than one simple factors $G_1$, $G_2$, \dots,
$\Bun[L]d$ is not irreducible as it is isomorphic to
$\bigsqcup_{d_1+d_2+\dots=d} \Bun[G_1]{d_1}\times
\Bun[G_2]{d_2}\times\cdots$.
Then $\IC(\Uh[L]d)$ must be understood as the direct sum
\begin{equation}
  \bigoplus_{d_1+d_2+\dots=d} \IC(\Bun[G_1]{d_1}\times
  \Bun[G_2]{d_2}\times\cdots).
\end{equation}
In particular, \eqref{eq:LG} is not $1$-dimensional. But it does not
cause us any trouble. We have the canonical isomorphism for each
summand, and $1^d_{L,G}$ is understood as their sum.

\subsection{Space \texorpdfstring{$U^d$}{Ud} and its base}\label{sec:Ud}

We shall introduce the space $U^d$ of homomorphisms from
$\cC_{S_{(d)}\AA^2}$ to $\Phi_{L,G}(\IC(\Uh{d}))$ and study its
properties in this subsection. A part of computation is a byproduct of
the proof of \thmref{thm:perverse} (see \lemref{lem:Ud}).
The study of $U^d$ will be continued in the remainder of this chapter,
and also in the next chapter.

\begin{Definition}
For $d > 0$, we define a vector space
\begin{equation}\label{eq:Ud}
  \begin{split}
  U^d_{L,G} \equiv\index{ULG@$U^d_{L,G}$}
  U^d &\defeq \Hom_{\Perv(\UhL{d})}(\cC_{S_{(d)}\AA^2},\Phi_{L,G}(\IC(\Uh{d})))\\
  &= H^{-2}(S_{(d)}\AA^2,\xi^! \Phi_{L,G}(\IC(\Uh{d}))),
  \end{split}
\end{equation}
where $(d)$ is the partition of $d$ consisting of a single entry $d$,
and $\xi\colon S_{(d)}\AA^2\to \UhL{d}$ is the inclusion.
\index{SdA2@$S_{(d)}\AA^2$}
\end{Definition}

We use the notation $U^d$, when $L$, $G$ are clear from the context.
\begin{NB}
  Sep. 16 : I move the definition of $U^d_{L,G}$ here.
\end{NB}

Since the hyperbolic restriction $\Phi_{L,G}$ depends on $P$, the
space $U^d_{L,G}$ depends also on $P$. When we want to emphasize $P$,
we denote it by $U^{d,P}_{L,G}$\index{ULGdP@$U^{d,P}_{L,G}$} or simply
by $U^{d,P}$.

We have a natural evaluation homomorphism
\begin{equation}
    \label{eq:6}
    U^d\otimes \cC_{S_{(d)}\AA^2}\to \Phi_{L,G}(\IC(\Uh{d})),
\end{equation}
which gives the isotypical component of $\Phi_{L,G}(\IC(\Uh{d}))$
corresponding to the simple perverse sheaf $\cC_{S_{(d)}\AA^2}$.

By the factorization \secref{sec:factorization} together with the
Thom isomorphism $i^*j^!(\cC_{\Bun{d_1}})\cong \cC_{\BunL{d_1}}$,
\begin{NB}
  Sep. 16 : I add `Thom isomorphism'. This was missing in the previous
  version, but we certainly need it.

  We may introduce
  \begin{equation}
    W^d = \Hom_{\Perv(\UhL{d})}(\IC(\UhL{d}), \Phi_{L,G}(\IC(\Uh{d}))),
  \end{equation}
  and the canonical isomorphism is
  \begin{equation}
        \Phi_{L,G}(\IC(\Uh{d})) \cong \bigoplus
    \IC(\BunLl{d_1},\rho)\otimes W^{d_1}.
  \end{equation}
  I do not know which is better, at this moment, but I take an
  economical approach, fixing $W^{d_1}\cong \CC$.
\end{NB}%
we get

\begin{Proposition}\label{prop:hyperbolic}
  We have the canonical isomorphism in $\Perv(\Uh[L]{d})$:
\begin{equation}
    \label{eq:7}
    \Phi_{L,G}(\IC(\Uh{d})) \cong \bigoplus
    \IC(\BunLl{d_1},\rho).
\end{equation}
Here $\rho$ is the \textup(semisimple\textup) local system on
$\BunLl{d_1} = \BunL{d_1}\times S_\lambda\AA^2$ with $\lambda =
(1^{n_1} 2^{n_2}\cdots)$ corresponding to the representation of
$S_{n_1}\times S_{n_2}\times\cdots$ on $(U^1)^{\otimes n_1}\otimes
(U^2)^{\otimes n_2}\otimes\cdots$ given by permutation of factors.

Moreover the isomorphism is also in the equivariant category with
respect to $L\times\CC^*\times\CC^*$.
\end{Proposition}

For example, the isotypical component for the intersection cohomology
complex $\IC(\Bun[L,\lambda]{d_1})$ for the trivial simple local system
is
\begin{equation}
    \label{eq:9}
    \Sym^{n_1} U^1\otimes \Sym^{n_2} U^2\otimes \cdots,
\end{equation}
where $\Sym$ denotes the symmetric power.

The second statement is the consequence of the first as the spaces of
homomorphisms between objects in $\Perv(\Uh[L]{d})$ are canonically
isomorphic for equivariant category with respect to
$L\times\CC^*\times\CC^*$ and non-equivariant one.  (See
\cite[1.16(a)]{Lu-cus2}.)
\begin{NB}
  I cannot find a statement in \cite{BL}. Is this a standard result ?
  \begin{NB2}
    If we consider the corresponding statement for equivariant
    $D$-modules, it is obvious: Vector fields generated by Lie
    algebras determined the equivariant structure on the $D$-module
    for a {\it connected\/} group.
  \end{NB2}%
\end{NB}%
Therefore \eqref{eq:7} is an isomorphism in the equivariant derived
category, though we use the factorization, which is {\it not\/}
equivariant with respect to $\CC^*\times\CC^*$.

\begin{NB}
  The following will be used later, but the proof was not given
  before. Dec. 12.
\end{NB}

\begin{Lemma}\label{lem:Fock}
  Suppose $L=T$. We have
  \begin{equation}
    H^*(S^d\AA^2, \Phi_{T,G}(\IC(\Uh{d})))
    \cong \bigoplus_{|\lambda|=d} \Sym^{n_1}U^1\otimes\Sym^{n_2}U^2\otimes
    \cdots% \otimes H^*(\overline{S_\lambda\AA^2}),
  \end{equation}
where $\lambda = (1^{n_1}2^{n_2}\dots)$.
\end{Lemma}

\begin{proof}
  Since $L=T$, we have $\Uh[T]{d} = S^d\AA^2$. See
  Example~\ref{ex:torus_fixed}. Then the assertion means that only
  trivial representation of $S_{n_1}\times S_{n_2}\times\cdots$
  contribute to the global cohomology group.

  Let $U$ be an open subset of $(\AA^2)^{n_1}\times (\AA^2)^{n_2}
  \times\cdots$ consisting of pairwise disjoint $n_1$ ordered points,
  $n_2$ ordered points, and so on in $\AA^2$. Forgetting orderings, we
  get an $(S_{n_1}\times S_{n_2}\times\cdots)$-covering $p \colon U\to
  S_\lambda\AA^2$. The pushforward of the trivial rank $1$ system with
  respect to $p$ is the regular representation $\rho_{\mathrm{reg}}$
  of $(S_{n_1}\times S_{n_2}\times\cdots)$.

  Since $p$ extends to a finite morphism $(\AA^2)^{n_1}\times
  (\AA^2)^{n_2} \times\cdots\to \overline{S_\lambda\AA^2}$, we have
  $\IC(S_\lambda\AA^2,\rho_{\mathrm{reg}}) =
  p_*(\cC_{(\AA^2)^{n_1}\times (\AA^2)^{n_2} \times\cdots})$. By the
  K\"unneth theorem, the global cohomology group $H^*(\bullet)$ of the
  right hand side is $H^*((\AA^2)^{n_1})\otimes
  H^*((\AA^2)^{n_2})\otimes\cdots$. This is $1$-dimensional, and
  corresponds to the trivial isotypical component of
  $\rho_{\mathrm{reg}}$. Now the assertion follows.
\end{proof}

Let us continue the study of $U^d$.
Let us note that all of our spaces $\Uh{d}$, $\Uh[L]{d}$, $\Uh[P]{d}$
have trivial factors $\AA^2$ given by the center of instantons, or the
translation on the base space $\AA^2$ except $d=0$ where $\Uh{0} =
\Uh[L]{0} = \Uh[P]{0} = \mathrm{pt}$. We assume $d\neq 0$
hereafter. Let $\cUh{d}$ denote the centered Uhlenbeck space at the
origin, thus we have $\Uh{d} = \cUh{d}\times \AA^2$.
Let us compose factorization morphisms $\pi^d_{h,G}$, $\pi^d_{v,G}$
for the horizontal and vertical projections $h\colon\AA^2\to\AA^1$,
$v\colon\AA^2\to\AA^1$ with the sum map $\sigma\colon S^d\AA^1\to
\AA^1$. Then $\cUh{d} =
(\sigma\pi^d_{h,G}\times\sigma\pi^d_{v,G})^{-1}(0,0)$.
\index{UhcG@$\cUh{d}$ (centered Uhlenbeck space)}
We use the notation $\cUh[L]{d}$, $\cUh[P]{d}$ for $\Uh[L]{d}$,
$\Uh[P]{d}$ cases. The diagrams \eqref{eq:1} factor and induce the
diagrams for the centered spaces, and the factorization is compatible
with the hyperbolic restriction. Let us use the same notation for $i$,
$j$, $p$ for the centered spaces. Then we have
\begin{equation}
  U^{d} = H^0(\xi_0^! p_* j^!\IC(\cUh{d})),
\end{equation}
where $\xi_0$ is the inclusion of the single point $d\cdot 0$ in
$\cUh[L]{d}$.
Here $d\cdot 0$ is the point in $S_{(d)}\AA^2$, the origin with
multiplicity $d$.

By base change we get
\begin{equation}\label{eq:17}
%  \begin{split}
  U^d %& \cong H^{0}(S_{(d)}\AA^2, \xi^! p_* j^{!} \IC(\cUh{d}))
%\\
%   &
%   \cong H^{-2}(p^{-1}(S_{(d)}\AA^2), \tilde j^{!} \IC(\Uh{d})),
   \cong H^{0}(p^{-1}(d\cdot 0), \tilde j^{!} \IC(\cUh{d})),
%  \end{split}
\end{equation}
where $\tilde j\colon p^{-1}(d\cdot 0)\to \cUh{d}$ is the
inclusion.
\begin{NB}
Consider the fiber square
\begin{equation}
    \begin{CD}
        p^{-1}(d\cdot 0) @>{\tilde \xi_0}>> \cUh[P]{d}
\\
       @V{p}VV @VV{p}V
\\
       \{ d\cdot 0\} @>{\xi_0}>> \cUh[L]{d}.
    \end{CD}
\end{equation}
We have $\xi_0^! p_* = p_* \tilde\xi_0^!$. Hence
\begin{equation}
    \begin{split}
    H^{0}(\xi_0^! p_* j^{!} \IC(\cUh{d}))
    &= H^{0}(p^{-1}(d\cdot 0),
    (j\circ \tilde\xi_0)^{!} \IC(\cUh{d}))
\\
   &= H^{0}(p^{-1}(d\cdot 0),
    \tilde j^{!} \IC(\cUh{d})).
    \end{split}
\end{equation}
\end{NB}

We have
\begin{Lemma}\label{lem:dimU}
    \begin{equation}
        \label{eq:8}
        \dim U^d = \rank G - \rank [L,L].
    \end{equation}
\end{Lemma}

\begin{proof}
%\subsection{Proof of Lemma \ref{lem:dimU}}
  \begin{NB}
    The first sentence is comment out, as I changed the order of
    \propref{prop:hyperbolic} and \lemref{lem:dimU}.

  First, let us note that Proposition \ref{prop:hyperbolic} is
  independent of Lemma \ref{lem:dimU}. In addition,
  \end{NB}
  According to a theorem of Laumon~\cite{Laumon-chi},
  \begin{NB}
      \linelabel{Laumon}
  \end{NB}%
  given a constructible complex $F$ on a complex algebraic variety
  $X$, and a morphism $f\colon X\to Y$, the classes 
  $[Rf_*F]$ and $[Rf_!F]$ in the Grothendieck
  group of constructible complexes on $Y$ coincide.
  In particular, $\chi(X,F)=\chi_c(X,F)$.
  It follows that the Euler characteristic of the
  stalk of $\IC(\Uh{d})$ at a point of $S_{(d)}\AA^2$ is equal to the
  Euler characteristic of the stalk of $\Phi_{L,G}(\IC(\Uh{d}))$ at
  the same point; the former was computed in Theorem 7.10 in \cite{BFG}.
  \begin{NB}
    Put the reference of Laumon.
  \end{NB}%

  \begin{NB}
    What do you mean ?

  Let us not prove Lemma \ref{lem:dimU}.
  \end{NB}%

  Now let us give a proof in the case $L=T$. Then it is easy to see
  that Proposition \ref{prop:hyperbolic} implies that the stalk of
  $\Phi_{T,G}(\IC(\Uh{d}))$ at a point of $S_{(d)}\AA^2$ is isomorphic
  to $\Sym^d(\oplus_i U^i_{T,G})$, where we regard $\oplus_i
  U^i_{T,G}$ as a graded vector space (with the natural grading coming
  from $i$) and the super-script $d$ means degree $d$ with respect to
  that grading. On the other hand,~\cite[Theorem~7.10]{BFG} implies
  that a similar description fits the stalk of $\IC(\Uh{d})$ at a point
  of $S_{(d)}\AA^2$ if we disregard the cohomological grading
  (the ``first'' grading in the language of~\cite{BFG}) and
  take a $\rank(G)$-dimensional space $V^i$ in place of $U^i_{T,G}$ above.
  We get $\dim U^d_{T,G}=\rank(G)$ for every $d$ by induction in $d$.

  Let us now consider the case of arbitrary $L$. Again, it is easy to
  deduce from Proposition \ref{prop:hyperbolic} that the stalk of
  $\Phi_{T,L}(\Phi_{L,G}(\IC(\Uh{d})))\simeq \Phi_{T,G}(\IC(\Uh{d}))$
  at a point of $S_{(d)}\AA^2$ is isomorphic to
$$
\bigoplus\limits_{d_1+d_2=d} \Sym^{d_1}(\oplus_i U^i_{T,L})\otimes \Sym^{d_2}(\oplus_j U^j_{L,G}),
$$
where the meaning of the super-scripts $d_1$ and $d_2$ is as above. In view of the preceding paragraph, we get
$\dim U^d_{L,G}=\rank(G)-\rank([L,L])$.
\end{proof}

The dimension estimate \corref{cor:dimest}
\begin{NB}
    \linelabel{dim_est1}
\end{NB}
\begin{NB}
  Reference to somewhere in the appendix. I must be careful, as the
  argument is not straight.
\end{NB}%
and the argument in \cite[Prop.~3.10]{MV2}
implies that
\begin{equation}\label{eq:107}
  \begin{split}
    & H^{0}(p^{-1}(d\cdot 0), \tilde j^{!} \IC(\cUh{d}))
    \\
    \cong \; & H^{0}(p^{-1}(d\cdot 0)\cap\Bun{d}, \tilde
    j^{!} \IC(\cUh{d}))
    \\
    = \;& H_{[0]}(p^{-1}(d\cdot 0)\cap\Bun{d},\CC).
  \end{split}
\end{equation}
Here we use the degree shift convention of the Borel-Moore homology
group (see Convention~(\ref{conv2})), which is shift by $\dim\cUh{d} =
2dh^\vee-2$ in this case.

Let us set
\begin{equation}\label{247}
  \Uh[P,0]d \defeq p^{-1}(d\cdot 0). \index{UhP0@$\Uh[P,0]d$}
\end{equation}
The subscript $0$ stands for $d\cdot 0$, and this convention will be
also used later. More generally, we denote $p^{-1}(x)$ by $\Uh[P,x]d$ for
$x\in \UhL{d}$\index{UhPx@$\Uh[P,x]d$}.

Then $H_{[0]}(\Uh[P,0]d\cap\Bun{d},\CC)$ has a base given by
$(dh^\vee-1)$-dimensional irreducible components of
$\Uh[P,0]d\cap\Bun{d}$. The dimension estimate \corref{cor:dimest}
\begin{NB}
\linelabel{dim_est2}
\end{NB}
\begin{NB}
  Reference to somewhere in the appendix. I must be careful, as the
  argument is not straight.
\end{NB}%
implies that $\Uh[P,0]d\cap\Bun{d'}$ ($d' < d$) is
lower-dimensional. Therefore

\begin{Lemma}\label{lem:Ud}
We have
\begin{equation}\label{eq:tu}
  U^d \cong H_{[0]}(\Uh[P,0]d).
\end{equation}
This space has a base given by $(dh^\vee-1)$-dimensional irreducible
components of $\Uh[P,0]d$.
\end{Lemma}

\subsection{Irreducible components}\label{sec:irred-comp}

Let us describe $(dh^\vee - 1)$-dimensional irreducible components of
$\Uh[P,0]d$ for $P=B$ explicitly.
We believe that there is no irreducible component of smaller dimension (see Remark~\ref{rem:dimofirr}), but we do not have a proof.

First consider the case $G = SL(2)$. By \lemref{lem:dimU} we have
$\dim U^d = 1$, and hence $\Uh[B,0]d$ has only one $(2d-1)$-dimensional
irreducible component. As we have
observed in the previous subsection, it is the closure of
$\Uh[B,0]d\cap \Bun{d}$. In \subsecref{sec:spaceVd}, it will be shown
that $\Uh[B,0]d\cap \Bun{d}$ consists of rank 2 vector bundles $E$
arising from a short exact sequence
\begin{equation}\label{eq:114}
    0 \to \shfO\to E \to \mathcal I\to 0
\end{equation}
compatible with framing, where $\mathcal I$ is an ideal sheaf of
colength $d$.

For a general $G$, consider the diagram \eqref{eq:CDpj} with $M = T$,
$L = L_i$ the Levi subgroup corresponding to a simple root $\alpha_i$.
Note that $[L_i,L_i] \cong SL(2)$, and hence $\Uh[{L_i}]d$ is
homeomorphic to $\Uh[SL(2)]d$. Therefore
$\Uh[B_{L_i},0]d\cap\Bun[L_i]d$ is irreducible of dimension $2d-1$ by
the above consideration.

\begin{Proposition}\label{prop:irred-comp}
    The irreducible components of $\Uh[B,0]d$ of dimension $dh^\vee-1$
    are the closures of $p^{-1}(\Uh[B_{L_i},0]d\cap \Bun[L_i]d)$ for
    $i\in I$.
\end{Proposition}

\begin{Definition}
    Let us denote the closure of $p^{-1}(\Uh[B_{L_i},0]d\cap
    \Bun[L_i]d)$ by $Y_i$.\index{Yi@$Y_i$}
\end{Definition}

\begin{proof}
    Consider the upper right part of \eqref{eq:CDpj}, which is
    \eqref{eq:1}.
    Its restriction to the open subset $\Bun[L_i]{d}$ has been
    described in \subsecref{sec:open}. As $p$ is a vector bundle whose
    rank is equal to the half of the codimension of $\Bun[L_i]{d}$ in
    $\Bun[G]{d}$, it follows that the inverse image
    $p^{-1}(\Uh[B_{L_i},0]d\cap \Bun[L_i]d)$ is irreducible and has
    dimension $dh^\vee - 1$.
    Therefore the closure of $p^{-1}(\Uh[B_{L_i},0]d\cap \Bun[L_i]d)$
    is an irreducible component of $\Uh[B,0]d$.

    Since $\dim U^d = \rank G$ by \lemref{lem:dimU}, it is enough to
    check that $p^{-1}(\Uh[B_{L_i},0]d\cap \Bun[L_i]d)\neq
    p^{-1}(\Uh[B_{L_j},0]d\cap \Bun[L_j]d)$ if $i\neq j$. When $G =
    SL(r)$, $\Uh[B,0]d\cap \Bun[G]{d}$ consists of vector bundles $E$
    having a filtration $0 = E_0\subset E_1\subset\cdots\subset E_r =
    E$ compatible with framing.  Moreover $p^{-1}(\Uh[B_{L_i},0]d\cap
    \Bun[L_i]d)$ consists of those with $c_2(E_i/E_{i-1}) = d$ and
    $c_2(E_j/E_{j-1}) = 0$ for $j\neq i$.
    Therefore $p^{-1}(\Uh[B_{L_i},0]d\cap \Bun[L_i]d)\neq
    p^{-1}(\Uh[B_{L_j},0]d\cap \Bun[L_j]d)$ for $i\neq j$.  (See
    \subsecref{sec:spaceVd} for detail.)
    For a general $G$, we embed $G$ into $SL(N)$. Then we need to
    replace $B$ by a parabolic $P$, but $p^{-1}(\Uh[B_{L_i},0]d\cap
    \Bun[L_i]d)$ is embedded into a corresponding space, and the same
    argument still works.
\end{proof}

\subsection{A pairing on \texorpdfstring{$U^d$}{Ud}}\label{sec:pairing}

Let us introduce a pairing between $U^{d,P}$ and $U^{d,P_-}$ in this
subsection.

We combine Braden's isomorphism \eqref{eq:Braden} with the natural
homomorphism $\xi_0^!\to \xi_0^*$ to get
\begin{equation}
  H^0(\xi_0^! i^* j^!\IC(\cUh{d})) \to
  H^0(\xi_0^* i_-^! j_-^*\IC(\cUh{d})).
\end{equation}
The right hand side is dual to
\begin{equation}
  U^{d,P_-} = H^0(\xi_0^! i_-^* j_-^!\IC(\cUh{d})).
\end{equation}
Thus we have a pairing between $U^{d,P}$ and $U^{d,P_-}$.
Following the convention in \cite[3.1.3]{MO}, we multiply the pairing
by the sign $(-1)^{\dim \cUh{d}/2} = (-1)^{dh^\vee-1}$.
\begin{NB}
  Let us explain why we take $\dim\cUh{d}/2 = dh^\vee-1$ instead of
  $\dim\Uh{d}/2 = dh^\vee$. Consider the Gieseker space with
  $h^\vee=1$, $d=2$. We have $\Gi{1}^1 = \AA^2\times T^*\proj^1$,
  $\cGi{1}^1 = T^*\proj^1$. The intersection pairing of $\proj^1$ in
  $T^*\proj^1$ is $-2$. But we want to identify it with $d (e_i,e_i) =
  d = 2$, where $e_i$ is the coordinate vector of $\Z^{h^\vee}$ with
  the standard inner product.
\end{NB}%
Let us denote it by $\la\ ,\ \ra$.\index{< , >@$\la\ ,\ \ra$} When we want to
emphasize that it depends on the choice of the parabolic subgroup $P$,
we denote it by $\la\ ,\ \ra_P$.

Since $\xi_0^! \CC_{d\cdot 0} \to \xi_0^* \CC_{d\cdot 0}$ is obviously
an isomorphism, this pairing is nondegenerate.

The transpose of the homomorphism $U^{d,P}\to
(U^{d,P_-})^\vee$ is a linear map $U^{d,P_-}\to
(U^{d,P})^\vee$. It is
\begin{equation}
  H^0(\xi_0^! i_-^* j_-^!\IC(\cUh{d})) \to
  H^0(\xi_0^* i^! j^*\IC(\cUh{d})),
\end{equation}
given by the transpose of the composite of $\xi_0^!\to\xi_0^*$ and
Braden's isomorphism $i^*j^!\to i_-^! j_-^*$. They are the same as
original homomorphisms $\xi_0^!\to\xi_0^*$ and $i_-^* j_-^!\to i^!
j^*$ respectively.
\begin{NB}
  The transpose of $\xi_0^! \to \xi_0^*$ is $\xi_0^! = (\xi_0^*)^\vee
  \to (\xi_0^!)^\vee = \xi_0^*$. This is the same as the original
  $\xi_0^!\to \xi_0^*$.
\end{NB}%
It means that
\begin{equation}\label{eq:41}
  \la u, v\ra_{P} = \la v, u\ra_{P_-} \quad
  \text{for $u\in U^{d,P}$, $v\in U^{d,P_-}$},
\end{equation}
where $\la\ ,\ \ra_{P_-}$ is the pairing defined with respect to the
opposite parabolic, i.e., one given after exchanging $i$, $j$ and $i_-$,
$j_-$ respectively.

\begin{NB}
  The following is the previous formulation, defining $\la\ ,\ \ra$ as a
  pairing on $U^d_{L,G}$. But it does not work as $P$ and $P_-$ are not
  conjugate under $G$ in general:

Recall that our hyperbolic restriction $\Phi_{L,G}$ depends on the
choice of a parabolic subgroup $P$, and also on $L$. However as in the
argument in \cite[Proof of Th.~3.6]{MV2}, we observe that $U^d_{L,G} =
H_{[0]}(\Uh[P,0]{d})$ forms a local system when we move $L\subset P$
in $G/L_0$ with a fixed $L_0$. It is a trivial local system as $G/L_0$
is simply connected.
\begin{NB2}
  I wonder that we need a stronger dimension estimate is necessary for
  this argument.
\end{NB2}%
\begin{NB2}
  Let us also remark that the local system is embedded in the local
  system of cohomology over the affine Grassmannian, which is trivial
  in \cite{MV2}. But the latter does not make sense, so we need to use
  that $G/L_0$ is simply connected.
\end{NB2}%
Therefore we have a canonical identification among different choices
of $P$. In particular, we have an isomorphism
\begin{equation}
  U^d_{L,G} =
  H^0(\xi_0^! i^* j^!\IC(\cUh{d}))
  \cong
  H^0(\xi_0^! i_-^* j_-^! \IC(\cUh{d})).
\end{equation}
Note that the constant sheaf $\CC_{d\cdot 0}$ at the point $d\cdot 0$
appears instead of $\cC_{S_{(d)}\AA^2}$ in \eqref{eq:Ud} in the
centered situation. Therefore $\xi_0^!$ does nothing.
The right hand side is isomorphic to
\begin{equation}\label{eq:opp}
  H^0(\xi_0^! i^! j^* \IC(\cUh{d}))
\end{equation}
by \cite[Th.~1]{Braden}. We compose the natural
homomorphism $\xi_0^!\to \xi_0^*$ and see that $H^0(\xi_0^* i^! j^*
(\IC(\cUh{d})))$ is the dual space to $H^0(\xi_0^! i^* j^!\IC(\cUh{d}))$.
Thus we have a pairing on $U^d_{L,G} = H^0(\xi_0^! i^* j^!\IC(\cUh{d}))$.
Following the convention in \cite[3.1.3]{MO}, we multiply the pairing
by the sign $(-1)^{\dim \cUh{d}/2} = (-1)^{dh^\vee-1}$.
\begin{NB2}
  Let us explain why we take $\dim\cUh{d}/2 = dh^\vee-1$ instead of
  $\dim\Uh{d}/2 = dh^\vee$. Consider the Gieseker space with
  $h^\vee=1$, $d=2$. We have $\Gi{1}^1 = \AA^2\times T^*\proj^1$,
  $\cGi{1}^1 = T^*\proj^1$. The intersection pairing of $\proj^1$ in
  $T^*\proj^1$ is $-2$. But we want to identify it with $d (e_i,e_i) =
  d = 2$, where $e_i$ is the coordinate vector of $\Z^{h^\vee}$ with
  the standard inner product.
\end{NB2}%
Let us denote it by $\la\ ,\ \ra$.

Since $\xi_0^! \CC_{d\cdot 0} \to \xi_0^* \CC_{d\cdot 0}$ is obviously
an isomorphism, this pairing is nondegenerate.

\begin{Lemma}
  The pairing $\la\ ,\ \ra$ is symmetric.
\end{Lemma}

\begin{proof}
  If we denote by $\varphi$ the isomorphism
  $U^d_{L,G}\xrightarrow{\cong} (U^d_{L,G})^\vee$, the assertion means
  $\varphi^\vee = \varphi$. Going back to the definition of $\varphi$,
  we see that it is equivalent to the assertion that two homomorphisms
  from $U^d_{L,G}$ to \eqref{eq:opp}, one as above and another again
  by the composition of Braden's isomorphism and the identification by
  the triviality of the local system, but in the opposite order. This
  assertion is clear from the definition of Braden's isomorphism.
\end{proof}

\begin{NB2}
  If we do not like the centered Uhlenbeck space, we can work on
  $\Uh{d}$: We consider $\xi_0\colon \{d\cdot 0\}\to \Uh[L]{d}$, then
  $\xi_0^* \cC_{S_{(d)}\AA^2} = \xi_0^* \CC_{S_{(d)}\AA^2}[2] =
  \CC_{d\cdot 0}[2]$, $\xi_0^! \cC_{S_{(d)}\AA^2} =
  (\xi_0^*\cC_{S_{(d)}\AA^2})^\vee = (\CC_{d\cdot 0}[2])^\vee =
  \CC_{d\cdot 0}[-2]$. Therefore the pairing becomes just $0$. But we
  work on the equivariant setting and consider $H^{-2}(\xi_0^! i^!j^*
  \IC(\Uh{d}))\to H^2(\xi_0^*i^!j^* \IC(\Uh{d}))$, the natural homomorphism
  $\xi_0^!\to\xi_0^*$ multiplied by $-\ve_1\ve_2$.
\end{NB2}
\end{NB}

\subsection{Another base of \texorpdfstring{$U^d$}{Ud}}\label{sec:anotherbase}

We next construct another base of $U^d = U^d_{T,G}$ for $L=T$, which
is ($\rank G$)-dimensional by \lemref{lem:dimU}.
This new base is better behaved under hyperbolic restrictions than the
previous one given by irreducible components.

This subsection is preliminary, and the construction will be completed
in \subsecref{sec:anotherbase2}.

We study $U^d_{T,G}$, using the associativity of the hyperbolic
localization (\propref{prop:trans}) for $M = T$, $L = L_i$ the Levi
subgroup corresponding to a simple root $\alpha_i$.
Since various Levi subgroups appear, we use the notation $U^d_{T,G}$
indicating groups we are considering.

%Let $L_i$ the Levi subgroup corresponding to a simple root $\alpha_i$.
%
Note that $[L_i,L_i] \cong SL(2)$, and hence $\Uh[{L_i}]d$ is
homeomorphic to $\Uh[SL(2)]d$. We understand $\IC(\Uh[{L_i}]d)$ as
$\IC(\Uh[SL(2)]d)$ and apply \lemref{lem:dimU} to see that
\begin{equation}\label{eq:1_Li}
  U_{T,L_i}^d = \Hom_{\Perv(\Uh[T]d)}
  (\cC_{S_{(d)}\AA^2}, \Phi_{T,L_i}(\IC(\Uh[{L_i}]d)))
\end{equation}
is $1$-dimensional.
\begin{NB}
    I hope that there is a canonical element $1_{L_i}$ in
    $U_{T,L_i}^d$.

    Dec. 5, 2012: This could be given by the
    fundamental class of a certain variety, which turns out to be
    irreducible in this situation.

    Apr. 30, 2013: Unfortunately this is not true. We need to choose a
    basis by MO's stable envelop.
\end{NB}%
In the next chapter, we shall introduce an element $1^d_{L_i}$ in
$U_{T,L_i}^d$ using the theory of the stable envelope in \cite{MO}.
\index{1Li@$1^d_{L_i}$|textit}

\begin{NB}
  Sep. 16 :
  A paragraph explaining $\Hom_{\Perv(\UhL{d})}(\IC(\UhL{d}),
  \Phi_{L,G}(\IC(\Uh{d})))\cong\CC$ is moved to \subsecref{sec:open}.
\end{NB}

Taking $L = L_i$ in the construction in \subsecref{sec:open}, we apply
the functor $\Phi_{T,L_i}$. By \propref{prop:trans} we get an element
\begin{equation}
    \label{eq:11}
\Phi_{T,L_i}(1^d_{L_i,G})\in
    \Hom_{\Perv(\Uh[T]d)}(\Phi_{T,L_i}(\IC(\Uh[{L_i}]d)),
    \Phi_{T,G}(\IC(\Uh{d}))).
\end{equation}
\begin{NB}
We have used $\Phi_{T,L_i}\Phi_{L_i,G}(\IC(\Uh{d}))
=\Phi_{T,G}(\IC(\Uh{d}))$.
\end{NB}%
Composing with the element $1^d_{L_i}$ in $U_{T,L_i}^d$ mentioned just
above, we get
\begin{equation}
    \label{eq:12}
 \Phi_{T,L_i}(1^d_{L_i,G})\circ 1^d_{L_i}\in U^d_{T,G}.
\end{equation}

We have ($\rank G$)-choices of $i$. Then we will show that
\begin{equation}\label{eq:anotherbase}
  \{ \Tilde\alpha^d_i \defeq
  \Phi_{T,L_i}(\pol 1^d_{L_i,G})\circ 1^d_{L_i}\}_i
\index{alphatildai@$\Tilde\alpha^d_i$|textit}
\end{equation}
gives a basis of $U^d_{T,G}$ in the next subsection.
Here we will introduce an appropriate polarization $\pol = \pm 1$,
using a consideration of rank $2$ case.  See \eqref{eq:polarization}.
Moreover, this will give us an identification $U^d_{T,G}$ with the
Cartan subalgebra $\h$ of $\g$ such that $\Tilde\alpha^d_i$ is sent to
the $i^{\mathrm{th}}$ simple coroot $\alpha_i^\vee$.
See a remark after \propref{prop:HeisRel}.
\begin{NB}
  Here is a copy of my message on Dec.\ 19, 2012:

  We still need to check that various Heisenberg algebras coming from
  have correct commutation relations, i.e., . But this can be checked
  by the reduction to the rank 2 case. Since we only consider ADE, we
  only need to consider the cases and , and these cases are already OK
  if we know that our construction is the same as [MO]. This is what I
  wrote `We may need a rank 2 consideration for the proof...' in my
  note, p.5. And this I learn what Sasha told us. Once we know the
  correct commutation relation, the linear independence of the base of
  (3.12) is clear, I believe.
\end{NB}%

We normalize the inclusion $\IC(\Uh[L_i]d)\to\Phi_{L_i,G}(\IC(\Uh{d}))$
by $\pol 1^d_{L_i,G}$ as above. Then the projection
$\Phi_{L_i,G}(\IC(\Uh{d}))\to \IC(\Uh[L_i]d)$ is also determined, as
$\IC(\Uh[L_i]d)$ has multiplicity $1$ in $\Phi_{L_i,G}(\IC(\Uh{d}))$
(see \subsecref{sec:open}).
Therefore we have the canonical isomorphism
\begin{equation}\label{eq:decomp}
   \Phi_{L_i,G}(\IC(\Uh{d})) \cong
   \IC(\Uh[L_i]d) \oplus \IC(\Uh[L_i]d)^\perp,
\end{equation}
where $\IC(\Uh[L_i]d)^\perp$ is the sum of isotypical components for
simple factors not isomorphic to $\IC(\Uh[L_i]d)$.
Applying $\Phi_{T,L_i}$ and using $\Phi_{T,L_i}\Phi_{L_i,G} =
\Phi_{T,G}$, we get an induced decomposition
\begin{equation}\label{eq:decomp2}
  U^d_{T,G} = U^d_{T,L_i}\oplus (U^d_{T,L_i})^\perp.
\end{equation}

This decomposition is orthogonal with respect to the pairing in
\subsecref{sec:pairing} in the following sense.
We have the decomposition
\(
  U^{d,B_-}_{T,G} = U^{d,B_-\cap L_i}_{T,L_i}\oplus
  (U^{d,B_-\cap L_i}_{T,L_i})^\perp
\)
for the opposite Borel $B_-$, and
\begin{equation}\label{eq:110}
  \la U^d_{T,L_i}, (U^{d,B_-\cap L_i}_{T,L_i})^\perp\ra = 0
  = \la (U^{d}_{T,L_i})^\perp, U^{d,B_-\cap L_i}_{T,L_i}\ra.
\end{equation}
Moreover the restriction of the pairing to $U^{d,B\cap L_i}_{T,L_i}$,
$U^{d,B_-\cap L_i}_{T,L_i}$ coincides with one defined via $\Uh[L_i]d$.

\begin{NB}
    These assertions follow from as the diagram
    \begin{equation}
        \begin{CD}
        (p'\circ p'')_* (j \circ j'')^! @>>>
        (p_-'\circ p_-'')_! (j_-\circ j_-'')^!
\\
   @| @|
\\
       p'_* j^{\prime !} p_* j^! @>>>
       p'_{-!} j_-^{\prime *} p_{-!} j_-^*
        \end{CD}
    \end{equation}
    is commutative, where the vertical arrows are as in the proof of
    \propref{prop:trans}, and the horizontal arrows are Braden's
    isomorphisms. See 2013-11-20 Note.
\end{NB}

\subsection{Dual base}

Let $\Tilde\alpha^{d,-}_i$ denote the element defined as
$\Tilde\alpha^{d}_i$ for the opposite Borel. We shall prove
\begin{equation}
    \label{eq:111}
    \la [Y_j],\Tilde\alpha^{d,-}_i\ra = \pm \delta_{ij}(-1)^{d-1} d
\end{equation}
modulo the computation for $G=SL(2)$, corresponding to the case $i=j$
in this subsection. The computation for $G=SL(2)$ will be given in
\remref{rem:SL(2)}.
This formula means that $\Tilde\alpha^{d,-}_i$ is the dual base to the
base given by irreducible components $Y_j$ with respect to the pairing
$\frac{(-1)^{d-1}}d\la\ ,\ \ra$ up to sign.

Consider the diagram \eqref{eq:CDpj} for the centered version, where
we take $M=T$, $L=L_i$ as in \subsecref{sec:anotherbase}. Let us
consider the open embedding of $\cBun[L_i]d$ to $\cUh[L_i]d$. We
have the corresponding restriction homomorphism
\begin{equation}
    \label{eq:104}
    \begin{split}
    U^d_{T,G} &= H^0(\xi_0^! (p'\circ p'')_* (j\circ j'')^! \IC(\cUh[G]d))
\\
   &\cong H^0(\xi_0^! p'_* j^{\prime!} \Phi_{L_i,G}(\IC(\cUh[G]d)))
\\
   &\cong H^0(p^{\prime-1}(d\cdot 0),
   \Tilde j^{\prime !}\Phi_{L_i,G}(\IC(\cUh[G]d)))
\\
   & \quad \to
   H^0(p^{\prime-1}(d\cdot 0)\cap\cBun[L_i]d,
   \Tilde j^{\prime !}\Phi_{L_i,G}(\IC(\cUh[G]d))),
    \end{split}
\end{equation}
where $\Tilde{j'}$ is the restriction of $j'$ to $p^{\prime-1}(d\cdot 0)$.
When we restrict $\Phi_{L_i,G}(\IC(\cUh[G]d))$ to the open set
$\cBun[L_i]d$, the first summand $\IC(\Uh[L_i]d)$ in the
decomposition \eqref{eq:decomp} is replaced by the constant sheaf
$\cC_{\cBun[L_i]d}$, and the second summand is killed.
Therefore we have an isomorphism
\begin{multline*}
   H^0(p^{\prime-1}(d\cdot 0)\cap\cBun[L_i]d,
   \Tilde{j'}^!\Phi_{L_i,G}(\IC(\cUh[G]d)))
\\
   \cong H_{[0]}(p^{\prime-1}(d\cdot 0)\cap\cBun[L_i]d, \CC)
   \cong U^d_{T,L_i},
\end{multline*}
where the second isomorphism is nothing but \eqref{eq:107} for $G$
replaced by $L_i$.

Thus the projection $U^d_{T,G}\to U^d_{T,L_i}$ to the first summand in
\eqref{eq:decomp2} is nothing but the restriction homomorphism
we have just constructed.

Let us further consider the restriction of the upper right corner of
the diagram \eqref{eq:CDpj} to the open subset $\cBun[L_i]d$.
Then
\begin{equation*}
    p^{-1}(p^{\prime-1}(d\cdot 0)\cap\cBun[L_i]d) =
    p^{-1}(\Uh[B_{L_i},0]d\cap\cBun[L_i]d)
\end{equation*}
has been studied in \subsecref{sec:irred-comp}: Its closure is an
irreducible component of $\Uh[B,0]d$.
By the base change the restriction to $\cBun[L_i]d$ is replaced by one
to $p^{-1}(\Uh[B_{L_i},0]d\cap\cBun[L_i]d)$, and we can replace
relevant $\IC$ sheaves by constant sheaves. The Thom isomorphism gives
us $p_*j^! \cC_{\Bun[G]d} \cong \cC_{\Bun[L_i]d}$ as in \subsecref{sec:open}.
Note that the intersection of an irreducible component $Y_j$ of
\propref{prop:irred-comp} with the open subset
$p^{-1}(\Uh[B_{L_i},0]d\cap\cBun[L_i]d)$ is lower-dimensional if
$i\neq j$, as $p^{-1}(\Uh[B_{L_i},0]d\cap\cBun[L_i]d)$ is irreducible.
Therefore the fundamental class of $Y_j$ goes to $0$ under the
restriction.
Hence we have \eqref{eq:111} for $i\neq j$ by \eqref{eq:110}.
In fact, we will see that $Y_j\cap
p^{-1}(\Uh[B_{L_i},0]d\cap\cBun[L_i]d) = \emptyset$ for type A in
\subsecref{sec:spaceVd}, and the same is true for any $G$ thanks to an
embedding $G\to SL(N)$.

The Thom isomorphism sends $[Y_i]$ to $[\Uh[B_{L_i},0]d]$ from the
definition of $Y_i$.
The sign in \eqref{eq:111} appears as we multiply the Thom isomorphism
by a polarization $\delta$ (see \eqref{eq:polarization} below).
Therefore the computation of \eqref{eq:111} for
$i=j$ is reduced to the case $G=SL(2)$.
The relevant computation will be given in \remref{rem:SL(2)} as we
mentioned above.
\begin{NB}
    Since $\pm = o(i)^d$, we have $\pm (-1)^{d-1} =
    -(-o(i))^d$. Compare this with \eqref{eq:112} for $d=1$.
\end{NB}

\subsection{\texorpdfstring{$\Aut(G)$}{Aut(G)} invariance}\label{sec:autg-invariance}

Let $\Aut(G)$ be the group of automorphisms of $G$. Its natural action
on $\Bun{d}$ extends to $\Uh{d}$ (\cite[\S6.1]{BFG}).
\begin{NB}
    When a principal bundle is given by a transition function
    $\varphi_{\alpha\beta}\colon U_\alpha\cap U_\beta\to G$, we define
    a new principal bundle by the transition function
    $g(\varphi_{\alpha\beta})$. The cocycle condition
    $\varphi_{\alpha\beta}\varphi_{\beta\gamma}\varphi_{\gamma\alpha}
    = 1$ is preserved, as $g$ is a group automorphism. If $g(\bullet)
    = g_0\bullet g_0^{-1}$ for $g_0\in G$, the principal bundle is
    isomorphic to the original one, and the action is just given by a
    change of the framing.
\end{NB}%

\begin{NB}
  We must be careful on the $\Aut(G)$-action in the following
  section. For type $A$, $\Aut(G)/G = \{ \pm 1\}$ is the Dynkin
  diagram automorphism given by the reflection at the center. In terms
  of $\Bun{d}$, it corresponds to associate a {\it dual\/} vector
  bundle. In particular, it does not extend to an action on the Gieseker
  space $\Gi{r}^d$.

  In the ADHM description, the diagram automorphism is given by
  \begin{equation}
      \label{eq:45}
      [(B_1,B_2,I,J)]\mapsto [(B_1^t, B_2^t, -J^t, I^t)].
  \end{equation}
  This does not preserve the stability condition.
\end{NB}%

Let us fix a cocharacter $\lambda\colon \GG_m\to G$, and consider our
construction with respect to $\sigma\circ\lambda$ for $\sigma\in \Aut(G)$.
Here $L = G^{\lambda(\GG_m)}$ is considered as a fixed Levi subgroup.
Substituting $\sigma\circ\lambda$ into $\lambda$ in the formula
\eqref{eq:PLoneparam}, we define a pair $(P^\sigma,L^\sigma)$ of a parabolic
subgroup and its Levi part. The action $\varphi_\sigma\colon \Uh{d}\to
\Uh{d}$ induces $\varphi_\sigma\colon \Uh[P]d\to \Uh[P^{\sigma}]d$,
$\varphi_\sigma\colon \Uh[L]d\to \Uh[L^{\sigma}]d$, and we have a commutative
diagram
\begin{equation}    \label{eq:39}
  \begin{CD}
    \Uh[L]{d} @>i>> \Uh[P]d @>j>> \Uh{d}
\\
     @V{\varphi_\sigma}VV @V{\varphi_\sigma}VV @V{\varphi_\sigma}VV
\\
    \Uh[L^{\sigma}]{d} @>>{i_{\sigma}}> \Uh[P^{\sigma}]{d} @>>{j_{\sigma}}> \Uh{d},
  \end{CD}
\end{equation}
% \begin{equation}    \label{eq:39}
% %    \Aut(G)/T \times S^d\AA^2 =
% %    \bigsqcup_{g\in \Aut(G)/T} %\{ g\} \times
%     \Uh[L^g]{d}
% %    \overset{\Tilde p = \sqcup p_g}{\underset{
% %        \Tilde i = \sqcup i_g}{\leftrightarrows}}
%     \overset{p_g}{\underset{i_g}{\leftrightarrows}}
% %    \bigsqcup_{g\in \Aut(G)/T} %\{ g\}\times
%     \Uh[P^g]{d}
% %    \overset{\Tilde j = \sqcup j_g}{\rightarrow} \Aut(G)/T\times \Uh{d},
%     \overset{j_g}{\rightarrow}
% %    \Aut(G)/T\times
%     \Uh{d},
% \end{equation}
where the subscript $\sigma$ indicates morphisms between spaces for
$\sigma\in \Aut(G)$.

Since $\IC(\Uh{d})$ is an $\Aut(G)$-equivariant perverse sheaf, we have
an isomorphism $\varphi_\sigma^*\IC(\Uh{d})\cong \IC(\Uh{d})$. Therefore we
have an isomorphism
\begin{equation}\label{eq:44}
  i^* j^! \IC(\Uh{d}) \cong \varphi_\sigma^* i_{\sigma}^* j_{\sigma}^! \IC(\Uh{d}).
\end{equation}

The isomorphism \eqref{eq:44} is equivariant in the following sense:
The right hand side is a $\TT^\sigma=T^\sigma\times\CC^*\times\CC^*$-equivariant
perverse sheaf, while the left hand side is $\TT$-equivariant. The
isomorphism \eqref{eq:44} respects equivariant structures under the
group isomorphism $\sigma\colon \TT\xrightarrow{\cong} \TT^{\sigma}$. In
particular, we have an isomorphism
\begin{equation}
    \label{eq:43}
    \varphi_\sigma\colon
    H^*_{\TT}(\Uh[L]d, i^*j^! \IC(\Uh{d}))
    \xrightarrow{\cong} H^*_{\TT^{\sigma}}(\Uh[L^{\sigma}]d, i_{\sigma}^*j_{\sigma}^! \IC(\Uh{d})),
\end{equation}
which respects the $H^*_{\TT}(\mathrm{pt})$ and
$H^*_{\TT^{\sigma}}(\mathrm{pt})$ structures via $\TT\cong \TT^{\sigma}$.

In the same way, we obtain a canonical isomorphism
\begin{equation}
    \label{eq:40}
    U_{L,G}^{d,P} \xrightarrow{\cong} U_{L^{\sigma},G}^{d,P^{\sigma}},
\end{equation}
which is denoted also by $\varphi_\sigma$ for brevity.

The pairing $\la\ ,\ \ra$ in \subsecref{sec:pairing} is compatible
with $\varphi_\sigma$: Let us denote by $\la\ ,\ \ra_{P^\sigma}$ the
pairing between $U_{L^\sigma,G}^{d,P^\sigma}$ and $U_{L^\sigma, G}^{d,
  P^\sigma_-}$.
We have $\varphi_\sigma \colon U_{L,G}^{d,P_-}\xrightarrow{\cong}
U_{L^\sigma, G}^{d, P^\sigma_-}$ as above, and the following holds
\begin{equation}
  \la\varphi_\sigma(u), \varphi_\sigma(v)\ra_{P^\sigma} = \la u,v\ra_P,
  \qquad u\in U_{L,G}^{d,P}, v\in U_{L,G}^{d,P_-}.
\end{equation}

The decomposition \eqref{eq:decomp} is transferred under $\varphi_\sigma$ to
\begin{equation}
    \label{eq:46}
    i_\sigma^* j_\sigma^! \IC(\Uh{d}) \cong
    \pm \IC(\Uh[L_i^\sigma]{d})\oplus \IC(\Uh[L_i^\sigma]{d})^\perp.
\end{equation}
Here the sign $\pm$ means that we multiply the projection to
$\IC(\Uh[L_i^\sigma]{d})$ by $\pm$, according to whether $\sigma$
respects the polarization $\pol$ for $\Uh[L_i]{d}$ and
$\Uh[L_i^\sigma]{d}$ or not.
Our polarization will be invariant under inner automorphisms, so the
sign depends on diagram automorphisms $\Aut(G)/\operatorname{Inn}(G)$.
\begin{NB}
  Remember that the polarization problem should be discussed later.
\end{NB}%
The decomposition \eqref{eq:decomp2} is mapped to
\begin{equation}
    \label{eq:47}
    U^{d,B^\sigma}_{T^\sigma,G} = U^{d,B^\sigma\cap L_i^\sigma}_{T^\sigma, L_i^\sigma}\oplus
    (U^{d,B^\sigma\cap L_i^\sigma}_{T^\sigma, L_i^\sigma})^\perp.
\end{equation}

Suppose $\sigma\in L$. We have $L^\sigma = L$, $P^\sigma = P$, $i_\sigma = i$, $j_\sigma =
j$. Then $i^* j^!  \IC(\Uh{d})$ is an $L$-equivariant perverse sheaf,
and \eqref{eq:44} is the isomorphism induced by the equivariant
structure.

Let us further assume $L=T$. Then $T$ acts trivially on $\Uh[T]{d} =
S^d\AA^2$, and $\varphi_\sigma|_{\Uh[T]d} = \id$. The equivariant structure
of the $T$-equivariant perverse sheaf $i^* j^! \IC(\Uh{d})$ is {\it
  trivial}.
\begin{NB}
In the definition \cite[2.1.3]{BL}, the diagram is
\(
   X \xleftarrow{p} P = X\times E \xrightarrow{q} \overline{P} = X\times B,
\)
where $E\to B$ is an $n$-acyclic free $T$-space. Then a {\it
  trivial\/} equivariant object $F\in D^b_T(X)$ is a triple $(F_X,
\overline{F}, \beta)$ such that $\overline{F}\in D^b(\overline{P})$ is
a pull-back of $F_X$ and $\beta\colon p^*(F_X)\cong q^*(\overline{F})$
is the trivial isomorphism.
\end{NB}%
In particular, the isomorphism \eqref{eq:44} is the identity.
\begin{NB}
    This argument does not work if we do not know that $i^* j^!
    \IC(\Uh{d})$ is perverse. So, for example, I do not know what we
    can say on $i^!j^!\IC(\Uh{d})$.
\end{NB}%
Therefore \eqref{eq:44} is well-defined for $\sigma\in \Aut(G)/(T/Z(G))$,
where $Z(G)$ is the center of $G$.

Note that chambers of hyperbolic restrictions for $L=T$ are Weyl
chambers. They appear as a subfamily for $W = N_G(T)/T$ in
$\Aut(G)/(T/Z(G))$.

Let us take $\sigma = w_0$, the longest element of the Weyl group. Then
$B^{w_0} = B_-$. We come back to $B$ via \eqref{eq:41}, and hence we
get
\begin{equation}
  \la u,v\ra_B = \la\varphi_{w_0}(u), \varphi_{w_0}(v)\ra_{B_-}
  = \la\varphi_{w_0}(v), \varphi_{w_0}(u)\ra_B
\end{equation}
for $u\in U^{d,B}_{T,G}$,$v\in U^{d,B_-}_{T,G}$.
\begin{NB}
  Suppose that $\g$ is of type $B_2$ or $G_2$. Then
  $\varphi_{w_0}(\Tilde\alpha_i^d) = \Tilde\alpha_i^{d,-}$
  ($i=1,2$). Therefore
  \begin{equation}
    \la\Tilde\alpha_1^d, \Tilde\alpha_2^{d,-}\ra_B
    = \la\varphi_{w_0}(\Tilde\alpha_2^{d,-}),
    \varphi_{w_0}(\Tilde\alpha_1^d)\ra_B
    = \la\Tilde\alpha_2^d, \Tilde\alpha_1^{d,-}\ra_B.
  \end{equation}
  This argument does not work for $A_2$. So I need to use the outer
  automorphism.
\end{NB}

We can take $\sigma\in\Aut(G)$, which preserves $T$ and the set of positive
roots, and induces a Dynkin diagram automorphism. Then $B^\sigma =
B$. Hence $U^{d,B}_{T,G}$ is a representation of the group of Dynkin
diagram automorphisms. The inner product is preserved.

We have $L_i^\sigma = L_{\sigma(i)}$, where $\sigma(i)$ is the vertex of the Dynkin
diagram, the image of $i$ under the corresponding Dynkin diagram
automorphism. From \eqref{eq:46} $\varphi_\sigma (\tilde\alpha^d_i)$ is
equal to $\tilde\alpha^d_{\sigma(i)}$ up to scalar. We will prove the
following in \subsecref{sec:autG-invar}.
\begin{Lemma}\label{lem:autg}
  We have
  \begin{equation}
    \varphi_\sigma (\tilde\alpha^d_i) = \pm \tilde\alpha^d_{\sigma(i)},
  \end{equation}
  where $\pm$ is the ratio of the polarizations for $\Bun[L_i]d$ and
  $\Bun[L_{\sigma(i)}]d$, compared under $\varphi_\sigma$.
\end{Lemma}

\section{Hyperbolic restriction in type \texorpdfstring{$A$}{A}}\label{sec:typeA}

We shall study the case $G = SL(r)$ in detail in this chapter.

We have the moduli space $\Gi{r}^d$\index{UhTilder@$\Gi{r}^d$} of
framed torsion free sheaves $(E,\varphi)$ of rank $r$, second Chern
class $d$ over $\proj^2$. It is called the {\it Gieseker space}.
We have a projective morphism $\pi$ (the {\it Gieseker-Uhlenbeck
  morphism\/}\index{pai@$\pi$ (Gieseker-Uhlenbeck morphism)}) from
$\Gi{r}^d$ to the corresponding Uhlenbeck space $\Uh{d}$.
It is known that $\Gi{r}^d$ is smooth and $\pi$ is a semi-small
resolution of singularities. Therefore we can study $\IC(\Uh{d})$ via
the constant sheaf $\cC_{\Gi{r}^d}$ over $\Gi{r}^d$.

If $r=1$, we understand $\Gi{1}^d$ as the Hilbert scheme
$\operatorname{Hilb}^{d}(\AA^2)$ of $d$ points on $\AA^2$, while
$\Uh[{SL(1)}]d$ is the symmetric power $S^d\AA^2$.

\subsection{Gieseker-Uhlenbeck}

Let us first explain the relation between $\IC(\Uh{d})$ and
$\cC_{\Gi{r}^d}$ in more detail.

\begin{Theorem}[\protect{\cite[\S3]{Baranovsky}}]\label{thm:GU}
    The Gieseker-Uhlenbeck morphism $\pi\colon \Gi{r}^d\to \Uh{d}$ is
    semi-small with respect to the standard
    stratification~\eqref{eq:strat}. All strata are relevant and
    fibers are irreducible. Therefore
\begin{equation}\label{eq:GU}
  \pi_! \cC_{\Gi{r}^d} \cong \bigoplus_{d_1+|\lambda|=d}
  H_{\topdeg}(\pi^{-1}(x^{d_1}_\lambda))\otimes
  \IC(\BunGl{d_1}),
\end{equation}
where $x^{d_1}_\lambda$ is a point in the stratum $\BunGl{d_1}$.
\end{Theorem}

(See also \cite[Ch.~3,5,6]{Lecture}, where $\Gi{r}^d$, $\Uh{d}$ are
denoted by $\mathcal M(n,r)$, $\mathcal M_0(n,r)$ respectively. See
also \cite[Ch.3]{MR2095899} for the detail on the irreducibility
of fibers.)

\begin{NB}
We identify $H_{\topdeg}(\pi^{-1}(x^{d_1}_{\lambda}))\cong \CC$ by the
fundamental class $[\pi^{-1}(x^{d_1}_{\lambda})]$.
\end{NB}%
Since $\IC(\UhGl{d_1})$ is isomorphic to the pushforward of
$\IC(\Uh{d_1})\boxtimes \cC_{\overline{S_\lambda\AA^2}}$ under the
finite morphism \eqref{eq:finite},
we have
\begin{equation}\label{eq:IH}
  H_{\TT}^{[*]}(\Gi{r}^d) \cong
  \bigoplus \IH_{\TT}^{[*]}(\Uh{d_1})\otimes
  H_{\topdeg}(\pi^{-1}(x^{d_1}_\lambda))\otimes
  H_{\TT}^{[*]}(\overline{S_\lambda\AA^2}).
\end{equation}
We also have the corresponding isomorphism for the cohomology with
compact support.

\subsection{Heisenberg operators}\label{sec:BaraOp}

For $r = 1$, the third author and Grojnowski independently constructed
operators acting on the direct sum of homology groups of $\Gi{1}^d$
satisfying the Heisenberg relation (see \cite[Ch.~8]{Lecture}).
It was extended by Baranovsky to higher rank case \cite{Baranovsky}.
Let us review his construction in this subsection.

We consider here both $H_\TT^{[*]}(\Gi{r}^d)$ and
$H_{\TT,c}^{[*]}(\Gi{r}^d)$, the equivariant cohomology with arbitrary
and compact support, which is Poincar\'e dual to Borel-Moore and the
ordinary equivariant homology
groups. \index{HTc@$H_{\TT,c}^{[*]}(\Gi{r}^d)$}
To save the notation, we use the notation
$H_{\TT(,c)}^{[*]}(\Gi{r}^d)$ meaning either of cohomology groups.

For $n > 0$ we consider subvariety
\begin{equation}
  P_n \subset \bigsqcup_d \Gi{r}^d \times \Gi{r}^{d+n} \times \AA^2,
\end{equation}
consisting of triples $(E_1,E_2,x)$ such that $E_1\supset E_2$ and
$E_1/E_2$ is supported at $x$.
We have
\begin{Proposition}
  $P_n$ is half-dimensional in $\Gi{r}^d \times \Gi{r}^{d+n} \times
  \AA^2$ for each $d$.
\end{Proposition}

Let us denote the projection to the third factor by $\Pi$. For a
cohomology class $\alpha\in H^{[*]}_{\TT(,c)}(\AA^2)$, we consider
$P_{-n}^\Delta(\alpha) = [P_n]\cap
\Pi^*(\alpha)$\index{P@$P_n^\Delta(\alpha)$ (diagonal Heisenberg
  generator)} as a correspondence in $\Gi{r}^d\times\Gi{r}^{d+k}$.
Then we have the convolution product
\begin{equation}
    P_{-n}^\Delta(\alpha) \colon H^{[*]}_{\TT(,c)}(\Gi{r}^d) \to
    H^{[* + \deg\alpha]}_{\TT(,c)}(\Gi{r}^{d+n}).
\end{equation}
Thanks to the previous proposition, the shift of the degree is simple
in our perverse degree convention.
\begin{NB}
    I am still wondering for the convention for cohomology
    degree. Here I take the notation from \cite[\S8.9]{CG}.
    However, it is a little annoying as the Heisenberg algebra we
    consider is not of degree $0$. In particular, its representation,
    the homology group of the fiber, is not of degree $0$.
\end{NB}%
The reason why we put $\Delta$ in the notation will be clear later.

We define $P_n^\Delta(\alpha)$ as the adjoint operator
\begin{equation}
  P_n^\Delta(\alpha) \colon H^{[*]}_{\TT(,c)}(\Gi{r}^{d+n}) \to
  H^{[*+\deg\alpha]}_{\TT(,c)}(\Gi{r}^{d}).
\end{equation}
Here we have two remarks. First we follow the sign convention in
\cite[3.1.3]{MO} for the intersection pairing
\begin{equation}\label{eq:73}
  \la\bullet,\bullet\ra = (-1)^{\dim X/2} \int_X \bullet\cup \bullet.
\end{equation}
Second, we take $\alpha\in H^{[*]}_{\TT,c}(\AA^2)$ for
$H^{[*]}_\TT(\Gi{r}^d)$ and $\alpha\in H^{[*]}_{\TT}(\AA^2)$ for
$H^{[*]}_{\TT,c}(\Gi{r}^d)$. Then the operators are well-defined,
though various projections are not proper. (See
\cite[\S8.3]{Lecture}.)

We have the commutator relation
\begin{equation}
  [P_m^\Delta(\alpha), P_n^\Delta(\beta)]
  = \la\alpha,\beta\ra m\delta_{m+n,0}\, r.
\end{equation}
If $m+n=0$, one of $\alpha$ or $\beta$ is in $H^{[*]}_{\TT}(\AA^2)$
and another is in $H^{[*]}_{\TT,c}(\AA^2)$. Hence $\la\alpha,\beta\ra$
is well-defined.

Since the construction is linear over $H^*_\TT(\mathrm{pt})$, and
$H^{[*]}_{\TT,c}(\AA^2)$, $H^{[*]}_\TT(\AA^2)$ are free of rank $1$,
we can choose $\alpha$ to be their generators, i.e., the Poincar\'e
dual of $[0]$ for $H^{[*]}_{\TT,c}(\AA^2)$, and $1$ (dual of $[\AA^2]$)
for $H^{[*]}_{\TT}(\AA^2)$.
\begin{NB}
    $\deg [0] = 2$, $\deg 1 = -2$.
\end{NB}%
We assume these choices hereafter until \secref{sec:w-algebra-repr}.
\begin{NB}
    For the integral form of the Heisenberg algebra, we choose a
    different convention.
\end{NB}%
Note also that $\la[0],1\ra = -1$ in our sign convention.

We take the direct sum over $d$ in \eqref{eq:IH}:
\begin{equation}\label{eq:IH2}
  \bigoplus_d H_\TT^{[*]}(\Gi{r}^d)
  \cong \bigoplus_{d} \IH_\TT^{[*]}(\Uh{d})
  \otimes \bigoplus_{\lambda}
  H_{\topdeg}(\pi^{-1}(x^{d}_\lambda))\otimes
  H_\TT^{[*]}({\overline{S_\lambda\AA^2}}).
\end{equation}
Note that $H_\TT^{[*]}({\overline{S_\lambda\AA^2}}) \cong
H_{\TT}^*(\mathrm{pt}) \cdot 1$, as $\overline{S_\lambda\AA^2}$ is
equivariantly contractible. Here $1\in
H^{0}_\TT({\overline{S_\lambda\AA^2}})
= H^{[-2l(\lambda)]}_\TT({\overline{S_\lambda\AA^2}})$.

From the definition of the Heisenberg operators, it acts only on the
second factor of \eqref{eq:IH2}: $\lambda=\emptyset$ are killed by
$P_k^\Delta([0])$ ($k>0$) and the summand for $\lambda =
(1^{n_1}2^{n_2}\cdots)$ is spanned by the monomial in
$P_{-1}(1)^{n_1}/n_1!  \cdot P_{-2}(1)^{n_2}/n_2! \cdots$. The second
factor is isomorphic to the Fock space.

Let us give another representation of the Heisenberg algebra. Let $0$
denote the point $d\cdot 0\in S_{(d)}\AA^2$, and consider the inverse
image $\pi^{-1}(0)\subset\Gi{r}^{d}$, and denote it by
$\Gi{r,0}^{d}$.\index{UhTilder0@$\Gi{r,0}^d$} It is the Quot scheme
parametrizing quotients of $\shfO_{\proj^2}^{\oplus r}$ of length $d$
whose support is $0$.

Let us restate \thmref{thm:GU} in a different form:
\begin{Proposition}\label{prop:quot}
  $\Gi{r,0}^{d}$ is an irreducible $(dr - 1)$-dimensional subvariety
  in $\Gi{r}^d$, unless $d=0$.
\end{Proposition}

It is needless to say that we have $\Gi{r,0}^0 = \Gi{r}^0 =
\mathrm{pt}$.

The convolution product by $P_{\pm k}^\Delta(\alpha)$ sends
$H_{[*]}^\TT(\Gi{r,0}^d)$ to $H_{[*-\deg\alpha]}^\TT(\Gi{r,0}^{d\pm
  k})$, where $\alpha\in H^*_{\TT,c}(\AA^2)$ for $k < 0$, $\alpha\in
H^*_{\TT}(\AA^2)$ for $k > 0$. Therefore
\begin{equation}
    \label{eq:19}
    \bigoplus_d H^\TT_{[*]}(\Gi{r,0}^d)
\end{equation}
is a representation of the Heisenberg algebra.
It is known that $\Gi{r,0}^d$ is homotopy equivalent to $\Gi{r}^d$,
hence $H^\TT_{[*]}(\Gi{r,0}^d)$ is isomorphic to the ordinary homology
group of $\Gi{r}^d$, and hence to $H^{[-*]}_{\TT,c}(\Gi{r}^d)$ by the
Poincar\'e duality.

\begin{NB}
  It seems that the following paragraphs should be put in
  Introduction, and is not relevant here.

Maulik-Okounkov \cite{MO} and Schiffmann-Vasserot \cite{SV} extended
the representation to $W(\gl_r)$, the $\scW$-algebra of $\gl_r$, which is
the tensor product of $W(\algsl_r)$, the $\scW$-algebra of $\algsl_r$ and
the Heisenberg algebra.

Since $\bigoplus_{d} \IH_\TT^{[*]}(\Uh{d})$ consists of highest weight
vectors of the Heisenberg algebra, it is a module of
$W(\algsl_r)$. This is the special case of our main result for $G =
SL_r$.
In order to generalize it to arbitrary $G$, we will reformulate the
argument of \cite{MO} in terms of the hyperbolic restriction functor
on the Uhlenbeck space after \subsecref{sec:sheaf}.
\end{NB}

\subsection{Fixed points and polarization}\label{sec:Gifixed}

Let us take a decomposition $r = r_1 + r_2 + \cdots + r_N$. We have
the corresponding $(N-1)$-dimensional torus, which is the connected
center $A = Z(L)^0$ of the Levi subgroup $L = S(GL(r_1)\times \cdots\times
GL(r_N))\subset SL(r)$. We have the corresponding parabolic subgroup
$P$ consisting of block upper triangular elements.

Let us consider the fixed point set $\Gi{L}^d =
(\Gi{r}^d)^{A}$.\index{UhTildeL@$\Gi{L}^d$} It consists of framed
sheaves, which is a direct sum of sheaves of rank $r_1$, $r_2$, \dots,
$r_N$. Thus we have
\begin{equation}\label{eq:product}
  \Gi{L}^d =
  \bigsqcup_{d=d_1+\dots+d_N} \Gi{r_1}^{d_1}\times\cdots\times
  \Gi{r_N}^{d_N}.
\end{equation}

We omit the superscript $d$, when there is no fear of confusion.

\begin{NB}
  Sep. 16 : Following Sasha's suggestion, I explain the polarization
  in more detail.
\end{NB}

Following \cite[Ex.3.3.3]{MO}, we choose a polarization $\pol$ for
each component of $\Gi{L}$, as a quiver variety associated with the
Jordan quiver.
Let us review the construction quickly. See the original paper for
more detail:
We represent $\Gi{L}$ as the space of quadruples $(B_1,B_2,I,J)$
satisfying certain conditions.
We decide to choose pairs, say $(B_1,I)$, from quadruples. The choice
gives us a decomposition of the tangent bundle of $\Gi{r}$ as
\begin{equation}
  T\Gi{r} = T^{1/2} + (T^{1/2})^\vee
\end{equation}
in the equivariant $K$-theory with respect to the $A$-action on
$\Gi{r}^d$. We also have the decomposition of $T\Gi{L}$, and hence
also of the normal bundle.
Then we choose a polarization $\pol$ of $\Gi{L}$ in $\Gi{r}$ as product
of weights in the normal bundle part of $(T^{1/2})^\vee$.

Let us also explain another description of the polarization
$\pol$\index{delta@$\pol$ (polarization)} given in
\cite[\S12.1.5]{MO}. We consider the following Quot scheme
\begin{equation}
  Q_r = \{ (E,\varphi) \mid x_2 \shfO^{\oplus r}\subset E\subset
  \shfO^{\oplus r} \} \subset \Gi{r},
\end{equation}
where $x_2$ is one of coordinates of $\AA^2$.
This is a fixed point component of a certain $\CC^*$-action, and is a
smooth lagrangian subvariety in $\Gi{r}$. In the ADHM description, it
is given by the equation $B_2 = 0 = J$.
Now $(T^{1/2})^\vee$ is the normal direction to $Q_r$ at a point in
$Q_r$. Since any component of $\Gi{L}$ intersects with $Q_r$, and the
intersection is again a smooth lagrangian subvariety, $Q_r$ gives us
the polarization.

Note that the polarization is invariant under the action of $G =
SL(r)$ on $\Gi{G}$, as we promised in \subsecref{sec:autg-invariance}.

We calculate the sign $\pm$ of the ratio of this polarization $\pol$
and the repellent one $\pol_{\text{rep}}$, of $\Gi{2}^d\times\Gi{1}^0$
and $\Gi{1}^0\times\Gi{2}^d$ in $\Gi{3}^d$ for a later purpose. Here
$L = S(GL(2)\times GL(1))$ in the first case and $L= S(GL(1)\times
GL(2))$ for the latter case.

\begin{Lemma}\label{lem:polSL3}
  We have $\pol_{\mathrm{rep}}/\pol = 1$ for $\Gi{2}^d\times\Gi{1}^0$,
  $\pol_{\mathrm{rep}}/\pol = (-1)^d$ for $\Gi{1}^0\times\Gi{2}^d$.
\end{Lemma}

\begin{NB}
  Consider the attracting sets $\Gi{P,0}^d$ corresponding to two Levi
  subgroups. By \propref{prop:Un} the first one is collapsed under the
  Hilbert-Chow morphism $\pi$. It seems that the only ratio of two
  polarizations is relevant for our later purpose, but still need a
  clarification.

  The ratio is much simpler to compute. We consider $g\in SL(3)$
  corresponding to the longest element of the Weyl group. Then the
  first fixed point component is mapped to the second by $g$. The
  repellent direction is mapped to the attracting one. And $\pol$ is
  invariant under $g$. Therefore the ratio is
  $\pol_{\mathrm{rel}}\pol_{\mathrm{att}} = (-1)^{\codim/2} \pol_{\mathrm{rel}}^2
  = (-1)^{\codim/2} = (-1)^d$.
\end{NB}

\begin{proof}
  Both components $\Gi{2}^d\times\Gi{1}^0$, $\Gi{1}^0\times\Gi{2}^d$
  intersect with the open set $\pi^{-1}_{a,G}(S_{(1^d)}\AA^1)$, the
  inverse image of the open stratum under the factorization morphism.
  Since the normal bundle decomposes according to the factorization,
  the polarization is of the form $(\pm 1)^d$. Hence it is enough to
  determine the case $d=1$.

  We factor out $\AA^2$ in $\Gi{r}^1$ and
  consider the centered Gieseker
  spaces. \index{UhTildercG@$\cGi{r}^d$|textit} We have
\begin{gather}
\cGi{3}^1 \cong T^* \proj^{2},
\\
\cGi{2}^1\times\cGi{1}^0 \cong T^*(z_2=0), \quad
\cGi{1}^0\times\cGi{1}^2 \cong T^*(z_0=0),
\end{gather}
where $[z_0:z_1:z_2]$ is the homogeneous coordinate system of
$\proj^2$. The polarization $\pol$ above is given by the base
direction of the cotangent bundle.

On the other hand, the repellent directions are base in the first
case and fibers in the second case. Therefore we have
$\pol_{\text{rep}}/\pol = 1$ in the first case and $-1$ in the second
case.
\end{proof}

\subsection{Stable envelope}\label{sec:envelop}

Recall we considered the attracting set $\UhP{}$ in the Uhlenbeck
space $\Uh{}$.  Let us denote its inverse image $\pi^{-1}(\UhP{})$ in
$\Gi{r}$ by $\Gi{P}$.\index{UhTildeP@$\Gi{P}^d$} This is the {\it tensor
  product variety}, denoted by $\fT$ in \cite{tensor2}, where $\UhP{}$
is denoted by $\fT_0$.
(In \cite{Na-Tensor} $\fT$ was denoted by $\mathfrak Z$.)

We have the following moduli theoretic description:
\begin{equation}
  \Gi{P} = \left\{ (E,\varphi)\in \Gi{r}\, \middle|\,
  \begin{minipage}{0.6\linewidth}
    $E$ admits a filtration
$0 = E_0 \subset E_1 \subset \cdots \subset E_N = E$ with
$\rank E_i/E_{i-1} = r_i$, compatible with $\varphi$.
  \end{minipage}
  \right\}.
\end{equation}
See \subsecref{sec:hyperb-restr}.

We consider the fiber product $Z_P$ of $\Gi{P}$ and $\Gi{L}$
over $\UhL{}$:
\begin{equation}
  Z_P = \Gi{P}\times_{\UhL{}} \Gi{L}, \index{ZP@$Z_P$}
\end{equation}
where the map from $\Gi{L}$ to $\UhL{}$ is the restriction of $\pi$,
and the map from $\Gi{P}$ to $\UhL{}$ is the composition of the
restriction
\(
  \Gi{P}\to \UhP{}
\)
of $\pi$ and the map $p$ in \subsecref{sec:hyperb-restr}.  In the
above description of $\Gi{P}$, it is just given as the direct sum
$\bigoplus (E_i/E_{i-1})^{\vee\vee}$ plus the sum of singularities of
$E_i/E_{i-1}$.
One can show that $Z_P$ is a lagrangian subvariety in $\Gi{r}\times
\Gi{L}$. See \cite[Prop.~1]{tensor2}. (There are no lower dimensional
irreducible components, as all strata are relevant for the semismall
morphism $\pi\colon \Gi{r}\to \Uh{}$.)

Maulik-Okounkov stable envelope is a `canonical' lagrangian cycle class
$\cL$\index{L@$\cL$ (stable envelope)} in $Z_P$:
\begin{equation}
  \cL\in H_{[0]}(Z_P).
\end{equation}
See \cite[\S3.5]{MO}. Note that $\cL$ depends on the choice of the
parabolic subgroup $P$ as well as the polarization $\pol$. Since they
are canonically chosen, we suppress them in the notation $\cL$.

The convolution by $\cL$ defines a homomorphism
\begin{equation}\label{eq:conv}
  \cL\ast - = p_{1*}(p_2^*(-)\cap \cL)
  \colon H_{[*]}(\Gi{L}) \to H_{[*]}(\Gi{P}).
\end{equation}
It is known that $\cL\ast-$ is an isomorphism (see
\cite[\S4.2]{tensor2}), and it {\it does\/} also make sense for
equivariant homology groups, as $H_{[0]}(Z_P) \cong H_{[0]}^\TT(Z_P)$

We have $H_{[*]}(\Gi{L}) \cong H^{[*]}(\Gi{L})$ by the Poincar\'e duality.
Then we have
\begin{equation}\label{eq:stable}
  H^{[*]}_\TT(\Gi{L}) \to H^{[*]}_\TT(\Gi{r})
\end{equation}
as the composite of $\cL\ast -$ and the pushforward with respect to
the inclusion $\Gi{P}\subset \Gi{r}$.
This is the original formulation of {\it stable envelope\/} in
\cite[Ch.~3]{MO}, and properties of $\cL$ are often stated in terms of
this homomorphism there.

Let $x\in\UhL{}$. Let $\Gi{L,x}$ denote the inverse image of $x$ under
the Gieseker-Uhlenbeck morphism $\Gi{L}\to\UhL{}$. Similarly let
$\Gi{P,x}$ denote the inverse image of $x$ under the composition
$\Gi{P}\to \UhP{}\to \UhL{}$.\index{UhTildePx@$\Gi{P,x}^d$} Then the
convolution $\cL\ast -$ also defines
\begin{equation}\label{eq:convx}
  \cL\ast-\colon H^{\TT_x}_{[*]}(\Gi{L,x})\to H^{\TT_x}_{[*]}(\Gi{P,x}),
\end{equation}
where $\TT_x$ is the stabilizer of $x$.

\begin{NB}
  The degree is the perverse degree as before. Therefore
  $H_{[*]}(\Gi{P,x}) = H_{\dim\Gi{r}+*}(\Gi{P,x})$. If $x=0$ as in the
  next subsection, $\dim \Gi{P,0} = \dim\Gi{r}/2 - 1$, so $*=-2$ is
  the top degree, unless $d=0$.
\end{NB}

\begin{NB}
  We can also consider $\Gi{r,x}$ the inverse image of $x$ under
  $\pi\colon \Gi{r}\to \Uh{}$, where $x\in\UhL{}$ is considered as a
  point in $\Uh{}$ by $j\circ i\colon \UhL{}\to\Uh{}$. Then we have
  $\Gi{r,x}\subset \Gi{P,x}$ by definition. But $\Gi{P,x}$ is larger than
  $\Gi{r,x}$ in general. We have the pushforward homomorphism
  $H^\TT_*(\Gi{r,x})\to H^\TT_*(\Gi{P,x})$.
\end{NB}

\subsection{Tensor product module}\label{sec:tensor}

Let $0 = d\cdot 0$ as before and consider the inverse image $\Gi{P,0}^d$
of $0$ under $\Gi{P}^d\to \UhL{d}$ as in the previous subsection.

We consider the direct sum
\begin{equation}
    \label{eq:21}
 \bigoplus_d H^\TT_{[*]}(\Gi{P,0}^d).
\end{equation}
The Heisenberg algebra acts on the sum: This follows from a general
theory of the convolution algebra: it is enough to check that
$\Gi{P,0}^d \circ (P_n\cap\Pi^{-1}(0))\subset \Gi{P,0}^{d+n}$ (for
$k>0$). If $(E_1,E_2,x)\in P_n\cap\Pi^{-1}(0)$, then $\pi(E_2) =
\pi(E_1) + n\cdot 0$. Therefore the assertion follows.

The stable envelope $\cL\ast-$ gives an isomorphism $\bigoplus_d
H^\TT_{[*]}(\Gi{L,0}^d)\cong \bigoplus_d H^\TT_{[*]}(\Gi{P,0}^d)$,
where the left hand side is the tensor product
\begin{equation}
    \label{eq:20}
    \bigoplus_{d_1, \cdots, d_N} H^\TT_{[*]}(\Gi{r_1,0}^{d_1})
    \otimes\cdots\otimes
    H^\TT_{[*]}(\Gi{r_N,0}^{d_N})
\end{equation}
by \eqref{eq:product}.
This is a representation of $N$ copies of Heisenberg algebras. Under
the stable envelope, $P_{-k}^\Delta([0])$ on \eqref{eq:21} is mapped to
\begin{equation}
    \label{eq:22}
    \sum_{i=1}^N 1\otimes\cdots\otimes
    \underbrace{P_{-k}^\Delta([0])}_{\text{$i^{\mathrm{th}}$ factor}}
    \otimes\cdots\otimes 1.
\end{equation}
This is \cite[Th.~12.2.1]{MO}. Our Heisenberg generators are diagonal
in this sense, and hence we put $\Delta$ in the notation.
This result is compatible with the decomposition $\scW(\mathfrak{gl}_r) =
\scW(\algsl_r)\otimes\Heis$, where $\scW(\algsl_r)$ is contained in the
tensor product of the remaining $(N-1)$ copies of Heisenberg algebras,
orthogonal to the diagonal one.

\subsection{Sheaf theoretic analysis}\label{sec:sheaf}

By \cite[\S4, Lem.~4]{tensor2} we have a natural isomorphism
\begin{equation}\label{eq:isom}
  H_{[0]}(Z_P) \cong \Hom_{\Perv(\UhL{})}
  \left(p_! j^* \pi_! \cC_{\Gi{r}}, \pi_! \cC_{\Gi{L}}\right)
\end{equation}
where $j$, $p$ are as in \subsecref{sec:hyperb-restr} and we use the
same symbol $\pi$ for Gieseker-Uhlenbeck morphisms for $\Gi{r}$ and
$\Gi{L}$.
\begin{NB}
  The proof will be given in \subsecref{subsec:stHom}.
\end{NB}%

The Verdier duality gives us an isomorphism
\begin{equation}
  \Hom(p_!j^* \pi_!\cC_{\Gi{r}}, \pi_! \cC_{\Gi{L}})
  \cong
  \Hom(\pi_* \cC_{\Gi{L}}, p_*j^! \pi_*\cC_{\Gi{r}}).
\end{equation}
Therefore the stable envelope gives us the canonical isomorphism
\begin{equation}\label{eq:stac}
  \pi_! \cC_{\Gi{L}} \xrightarrow[\cong]{\cL}
  \Phi_{L,G}(\pi_! \cC_{\Gi{r}}) =
  p_* j^! \pi_! \cC_{\Gi{r}},
\end{equation}
as $\pi_! = \pi_*$.
This is nothing but \thmref{nakajima-symplectic}(2) in Introduction.

Let $x\in \UhL{}$ and $i_x$ denote the inclusion of $x$ in
$\UhL{}$. Then
\(
   \cL\in \Hom(\pi_! \cC_{\Gi{L}}, p_*j^! \pi_!\cC_{\Gi{r}})
\)
defines an operator
\begin{equation}\label{eq:stalk}
  \begin{CD}
  H^*(i_x^! \pi_! \cC_{\Gi{L}}) @>>>
  H^*(i_x^! p_*j^! \pi_!\cC_{\Gi{r}})
\\
  @| @|
\\
  H_{[-*]}(\Gi{L,x}) @.
  H_{[-*]}(\Gi{P,x}).
  \end{CD}
\end{equation}
This is equal to $\cL\ast -$ in \eqref{eq:convx} under the isomorphism
% \(
%    \Hom(p_!i^* \pi_!\cC_{M(n,r)}, \pi^A_! \cC_{M(n,r)^A})
%    \cong H_{\topdeg}(Z_P(n))
% \)
\eqref{eq:isom}. See \cite[\S4.4]{tensor2}.

\begin{NB}
\subsection{Proof of \texorpdfstring{\eqref{eq:isom}}{}}\label{subsec:stHom}

We give a proof of \eqref{eq:isom}.
We have the diagram
\begin{equation}
  \begin{CD}
    Z_P = \Gi{P} \times_{\UhL{}} \Gi{L}
    @>{\tilde \delta}>> \Gi{P} \times \Gi{L} @>{\tilde j\times\id}>>
    \Gi{r} \times \Gi{L}
\\
  @V{\alpha}VV @VV{\tilde p\times\pi}V @.
\\
  \UhL{} @>\delta>> \UhL{}\times \UhL{},
  \end{CD}
\end{equation}
where $\tilde p$ is the composition of $\pi\colon \Gi{P}\to \UhP{}$ and
$p\colon\UhP{}\to \UhL{}$. Therefore we have
\begin{equation}
  \begin{split}
    H_{[-*]}(Z_P) &= H^*(Z_P,((\tilde j\times\id)\circ\tilde \delta)^!
    \cC_{\Gi{r}\times \Gi{L}})
    \\
    &= H^*(\UhL{}, \alpha_* \tilde \delta^! (\tilde j\times\id)^!
    (\cC_{\Gi{r}}\boxtimes \cC_{\Gi{L}}))
\\
    &= H^*(\UhL{}, \delta^! (\tilde p\times\pi)_*
    (\tilde j^! \cC_{\Gi{r}}\boxtimes \cC_{\Gi{L}}))
\\
    &= H^*(\UhL{}, \delta^! (\tilde p_* \tilde j^!
    \cC_{\Gi{r}}\boxtimes \pi_* \cC_{\Gi{L}}))
\\
    &= H^*(\UhL{}, (\tilde p_! \tilde j^* \cC_{\Gi{r}})^\vee \overset{!}{\otimes}
    \pi_* \cC_{\Gi{L}})
\\
    &= H^*(\UhL{}, {\mathcal Hom}(\tilde p_! \tilde j^* \cC_{\Gi{r}},
    \pi_* \cC_{\Gi{L}}))
\\
    &= \Ext^*(\tilde p_!\tilde j^* \cC_{\Gi{r}}, \pi_* \cC_{\Gi{L}}).
  \end{split}
\end{equation}
We have $\pi_* \cC_{\Gi{L}} = \pi_! \cC_{\Gi{L}}$ as $\pi$ is proper.
We also have a Cartesian diagram
\begin{equation}
  \begin{CD}
    \Gi{P} @>{\tilde j}>> \Gi{r}
\\
    @V{\pi}VV @VV{\pi}V
\\
    \UhP{} @>j>> \Uh{}.
  \end{CD}
\end{equation}
Therefore
\begin{equation}
  \tilde p_! \tilde j^* \cC_{\Gi{r}}
  = p_! \pi_! \tilde j^* \cC_{\Gi{r}}
  = p_! j^* \pi_! \cC_{\Gi{r}}.
\end{equation}
This finishes the proof of \eqref{eq:isom}.
\end{NB}

\subsection{The associativity of stable envelopes}

Let us take parabolic subgroups $Q\subset P \subset G$ and the
corresponding Levi subgroup $M\subset L$ as in
\subsecref{sec:ass}. ($G$ is still $SL(r)$.) Let $Q_L$ be the image of
$Q$ in $L$.

Let us denote by $\cL_{L,G}$ the isomorphism given by the stable
envelope in \eqref{eq:stac}:
\begin{equation}\label{eq:st}
    \pi_! \cC_{\Gi{L}}\xrightarrow[\cong]{\cL_{L,G}}
    \Phi_{L,G}(\pi_! \cC_{\Gi{r}}).
\end{equation}
We similarly have isomorphisms
\begin{equation}
  \pi_! \cC_{\Gi{M}}\xrightarrow[\cong]{\cL_{M,G}}
    \Phi_{M,G}(\pi_!\cC_{\Gi{r}}),\quad
  \pi_! \cC_{\Gi{M}} \xrightarrow[\cong]{\cL_{M,L}}
    \Phi_{M,L}(\pi_!\cC_{\Gi{L}}).
\end{equation}

Then stable envelopes are compatible with the associativity
\eqref{eq:5} of the hyperbolic restriction:
\begin{Proposition}
We have a commutative diagram
\begin{equation}\label{eq:ass2}
  \begin{CD}
  \pi_!\cC_{\Gi{M}}@>\cong>{\cL_{M,G}}>  \Phi_{M,G}(\pi_! \cC_{\Gi{r}})
\\
  @V{\cong}V{\cL_{M,L}}V \eqref{eq:5}@|
\\
  \Phi_{M,L}(\pi_!\cC_{\Gi{L}})
 @>{\Phi_{M,L}(\cL_{L,G})}>\cong>
  \Phi_{M,L} \Phi_{L,G}(\pi_! \cC_{\Gi{r}}).
  \end{CD}
\end{equation}
\end{Proposition}

Let us check that this follows from the proof of
\cite[Lemma~3.6.1]{MO}.
(To compare the following with the original paper, the reader should
note that the tori $A\supset A'$ were used in \cite{MO}, which
correspond to $Z(M)^0\supset Z(L)^0$ respectively in our situation.)

We consider
\begin{equation}
  Z_{P} = \Gi{P}\times_{\Uh[{L}]{}} \Gi{L},
\quad
  Z_{Q} = \Gi{Q}\times_{\Uh[{M}]{}} \Gi{M},
\quad
  Z_{Q_L} = \Gi{Q_L}\times_{\Uh[M]{}} \Gi{M}.
\end{equation}
The stable envelopes $\cL_{L,G}$, $\cL_{M,G}$, $\cL_{M,L}$ are classes in
$H_{[0]}(Z_{P})$, $H_{[0]}(Z_{Q})$, $H_{[0]}(Z_{Q_L})$
respectively.
We consider the convolution product
\begin{equation}
  \cL_{L,G}\ast \cL_{M,L}\in
  H_{[0]}(Z_{P}\circ Z_{Q_L}).
\end{equation}
Note that $Z_{P}\circ Z_{Q_L}$ consists of $(x_1,x_3)\in
\Gi{P}\times\Gi{M}$ such that there exists $x_2\in \Gi{Q_L}\subset
\Gi{L}$ with $(x_1,x_2)\in Z_{P}$, $(x_2,x_3)\in Z_{Q_L}$ by
definition. This is nothing but $Z_{Q}$. Therefore
$\cL^P\ast\cL^{Q_L}$ is a class in $H_{[0]}(Z_{Q})$.
The proof in \cite[Lemma~3.6.1]{MO} actually gives $\cL_{L,G}\ast
\cL_{M,L} = \cL_{M,G}$.

Therefore the commutativity of \eqref{eq:ass2} follows, once we check
that the convolution product corresponds to the composition of
homomorphisms (Yoneda product) under the isomorphism \eqref{eq:isom}.
This is not covered by \cite[Prop.~8.6.35]{CG}, as the base spaces of
fiber products are different: $\UhL{}$ and $\Uh[M]{}$. But we can
easily modify its proof to our situation.
\begin{NB}
  Really ?
\end{NB}

\subsection{Space \texorpdfstring{$V^d$}{Vd} and its base given by irreducible components}\label{sec:spaceVd}

Let us write $d$ for the instanton number again. Similarly to
\eqref{eq:Ud} we define
\begin{equation}
    \label{eq:13}
    \begin{split}
    V^d_{L,G} \equiv V^d \index{VLG@$V^d_{L,G}$}
    &\defeq \Hom(\cC_{S_{(d)}\AA^2}, \Phi_{L,G}(\pi_!\cC_{\Gi{r}^d})) \\
    & = H^{-2}(S_{(d)}\AA^2, \xi^!\Phi_{L,G}(\pi_! \cC_{\Gi{r}^d})),
    \end{split}
\end{equation}
where $\xi\colon S_{(d)}\AA^2\to \UhL{d}$ is as before. We denote by
$V^{d,P}_{L,G}$ or $V^{d,P}$ when we want to emphasize
$P$. \index{VLGdP@$V^{d,P}_{L,G}$}

As in \lemref{lem:Ud} we have
\begin{equation}\label{eq:tv}
  V^d \cong H_{[0]}(\Gi{P,0}^d),
\end{equation}
and $V^d$ has a base given by $(dh^\vee-1)$-dimensional irreducible
components of $\Gi{P,0}^d$.

\begin{NB}
We have
\begin{equation}
    \label{eq:14}
    \begin{split}
    V^d &= H^{-2}(S_{(d)}\AA^2, \xi^!p_*j^{!} \pi_! \cC_{\Gi{r}^d})
\\
    &\cong H_{[2]}(\tilde p^{-1}(S_{(d)}\AA^2)) \cong
    H_{[-2]}(\Gi{P,0}^d),
    \end{split}
\end{equation}
where $\tilde p = p \circ\pi\colon \Gi{P_-}\to \UhL{d}$, $0$ is
the point $d\cdot 0$ in $S_{(d)}\AA^2$, and
$\Gi{P,0}^d$ is the inverse image $\tilde p^{-1}(d\cdot 0)$.
The last isomorphism comes from $\tilde p^{-1}(S_{(d)}\AA^2)\cong
\Gi{P,0}\times\AA^2$. In fact, consider the fiber square
\begin{equation}
    \begin{CD}
        p^{-1}(S_{(d)}\AA^2) @>{\tilde \xi}>> \UhP{d}
\\
       @V{p}VV @VV{p}V
\\
       S_{(d)}\AA^2 @>{\xi}>> \UhL{d}.
    \end{CD}
\end{equation}
We have $\xi^! p_* = p_* \tilde \xi^!$. Hence
\begin{equation}
    H^{-2}(S_{(d)}\AA^2, \xi^! p_* j^{!} \pi_!\cC_{\Gi{r}^d})
    = H^{-2}(p^{-1}(S_{(d)}\AA^2),
    (j\circ\tilde \xi)^{!} \pi_!\cC_{\Gi{r}^d}).
\end{equation}

Next we consider
\begin{equation}
    \label{eq:16}
    \begin{CD}
        \tilde p^{-1}(S_{(d)}\AA^2) @>\alpha>> \Gi{r}^d
\\
     @V{\pi}VV      @V{\pi}VV
\\
    p^{-1}(S_{(d)}\AA^2) @>>{j\circ \tilde \xi}> \Uh{d}.
    \end{CD}
\end{equation}
Therefore the above is equal to
\begin{equation}
    \label{eq:18}
    H^{-2-2dh^\vee}(\tilde p^{-1}(S_{(d)}\AA^2), \alpha^! \mathbb D_{\Gi{r}^d}).
\end{equation}
This is the Borel-Moore homology of $\tilde p^{-1}(S_{(d)}\AA^2)$.
\end{NB}

On the other hand, $H_{[0]}(\Gi{P,0}^d)$ is isomorphic to
$H_{[0]}(\Gi{L,0}^d)$ by the stable envelope. In the description
\eqref{eq:product}, note that the fiber $\Gi{r_i,0}^{d_i}$ has $\dim =
\dim \Gi{r_i}^{d_i}/2 - 1$ by \propref{prop:quot} unless $d_i =
0$. Therefore we can achieve the degree $[0]=\dim\Gi{L}^d-2 = \sum
\dim\Gi{r_i}^{d_i}-2$ only when all $d_i = 0$ except one. There are $N$
choices $i=1$, \dots, $N$. Therefore $\dim V^d = N$.

Let us study $V^d = H_{[0]}(\Gi{P,0}^d)$ in more detail. This will
give the detail left over from \subsecref{sec:irred-comp}. By
\cite[\S3]{Na-Tensor} we have a decomposition
\begin{equation}
  \Gi{P,0}^d = \bigsqcup_{d_1+\cdots+d_N = n}
  \fT(d_1,\dots,d_N)_0,
\end{equation}
where
\begin{equation}\label{eq:tenstr}
  \fT(d_1,\dots,d_N)_0 =
  \left\{ (E,\varphi)\, \middle|\,
  \begin{minipage}{0.55\linewidth}
    $E$ admits a filtration $0 = E_0 \subset E_1 \subset \cdots
    \subset E_N = E$ with $E_i/E_{i-1} \in \Gi{r_i,0}^{d_i}$
    compatible with $\varphi$.
  \end{minipage}
  \right\}.
\end{equation}
\index{Td1@$\fT(d_1,\dots,d_N)_0$}
We have a projection
\begin{equation}
  \fT(d_1,\dots,d_N)_0 \to \Gi{r_1,0}^{d_1}\times\cdots\times
  \Gi{r_N,0}^{d_N},
\end{equation}
which is a vector bundle of rank $dr - \sum d_i r_i$.
Note that
\begin{equation}
  \dim \Gi{r_i,0}^{d_i} =
  \begin{cases}
    0 & \text{if $d_i = 0$},
\\
    d_i r_i - 1 & \text{if $d_i\neq 0$}.
  \end{cases}
\end{equation}
(See \propref{prop:quot}.)
Therefore
\begin{equation}
  \dim \fT(d_1,\dots,d_N)_0
  = dr - \# \{ i \mid d_i \neq 0\} \le dr - 1.
\end{equation}
The equality holds if and only if there is only one $i$ with $d_i \neq
0$. Therefore $H_{[0]}(\Gi{P_-,0}^d)$ is spanned by fundamental
cycles
\begin{equation}\label{eq:fT_0}
  [\overline{\fT(d,0,\dots,0)_0}], \quad \cdots \quad,
  [\overline{\fT(0,\dots,0,d)_0}].
\end{equation}
Thus it is $N$-dimensional, as expected.

In the remainder of this subsection, we study the corresponding space
$U^d = H_{[0]}(\Uh[P,0]d)$ for the Uhlenbeck space. Note that we have
projective morphism $\pi\colon \Gi{P,0}^d\to \Uh[P,0]d$, and the
Quot scheme $\pi^{-1}(d\cdot 0) = \Gi{r,0}^d$ is contained in
$\Gi{P,0}^d$.
The class of fiber $\Gi{r,0}^d$ is given by $P_{-d}([0])[\Uh{0}]$, and
$H_{[0]}(\Uh[P,0]d)$ is killed by Baranovsky's Heisenberg operators
by the construction.

\begin{Proposition}\label{prop:Un}
  Among $N$ cycles in \eqref{eq:fT_0}, the first one
  $[\overline{\fT(d,0,\dots,0)_0}]$ is $[\Gi{r,0}^d]$. The remaining
  cycles give a base of $U^d = H_{[0]}(\Uh[P,0]d)$ under $\pi_*$.
\end{Proposition}

From the definition, this description of irreducible components
\begin{NB}
    $\pi(\overline{\fT(0,d,0,\dots,0)_0})$, \dots,
    $\pi(\overline{\fT(0,\dots,0,d)_0})$
\end{NB}%
of $\Uh[P,0]d$ is the same as one in \propref{prop:irred-comp} when
$G = SL(r)$, $P = B$.

\begin{NB}
I first thought that we need the following. But it seems not.....

\begin{Claim}
  Suppose that $F$ is a torsion free sheaf on a nonsingular surface
  and $E$ is its subsheaf. Then the natural homomorphism
  $E^{\vee\vee}\to F^{\vee\vee}$ is injective.
\end{Claim}

\begin{proof}
  Since the statements are local, we may assume $E$, $F$ are modules
  over the local ring $A$ at a point. Let $K$ be the quotient field of
  $A$. Then $E\otimes_A K\to F\otimes_A K$ is injective as $K$ is flat
  over $A$. Since $E$, $F$ are torsion free, $E\to E\otimes_A K$,
  $F\to F\otimes_A K$ are injective.

  We have $E^\vee = \{ \varphi\in (E\otimes_A K)^\vee \mid
  \varphi(E)\subset A\}$, and similarly for $E^{\vee\vee}$. Thus we
  have $E\subset E^{\vee\vee}\subset E\otimes_A K$ and the same for
  $F$. The assertion is clear.
\end{proof}
\end{NB}

\begin{proof}
Suppose that $E\in\fT(d,0,\dots,0)_0$. Then we have a short exact
sequence
\begin{equation}
  0 \to E_1 \to E \to \shfO^{\oplus r_2 + \cdots + r_N} \to 0
\end{equation}
with $E_1\in \Gi{r_1,0}^d$.
\begin{NB}
  We have
  \begin{equation*}
    \begin{split}
    & 0 \to E_2/E_1 \to E_3/E_1 \to E_3/E_2 \to 0,
\\
    & 0 \to E_3/E_1 \to E_4/E_1 \to E_4/E_3 \to 0,
\\
    & \qquad\qquad\qquad \cdots
\\
    & 0 \to E_{N-1}/E_1 \to E/E_1 \to E_N/E_{N-1} \to 0.
    \end{split}
  \end{equation*}
  From our assumption we have $E_i/E_{i-1} \cong \shfO^{\oplus r_i}$
  for $i\neq 1$. By the induction, we see that $E_i/E_1$ is
  trivial. In particular, $E/E_1$ is trivial.
\end{NB}%
Consider
\begin{equation}
  0 \to \shfO^{\oplus r_2 + \cdots + r_N} \to E^\vee \to E_1^\vee \to
  \mathcal Ext^1(\shfO^{\oplus r_2 + \cdots + r_N}, \shfO)
\end{equation}
Since $\mathcal Ext^1(\shfO^{\oplus r_2 + \cdots + r_N}, \shfO) = 0$,
the last homomorphism $E^\vee\to E_1^\vee$ is surjective. Therefore
this is a short exact sequence. Dualizing again, we get
\begin{equation}
  0 \to E_1^{\vee\vee} \to E^{\vee\vee} \to \shfO^{\oplus r_2 + \cdots + r_N} \to 0.
\end{equation}
The last homomorphism $E^{\vee\vee} \to \shfO^{\oplus r_2 + \cdots +
  r_N}$ is surjective as $E\to \shfO^{\oplus r_2 + \cdots + r_N}$ is
so. (Or, we observe $E_1^{\vee\vee} \cong \shfO^{\oplus r_1}$ as
$E_1\in \Gi{r_1,0}^d$, and $\mathcal Ext^1(\shfO^{\oplus r_1},
\shfO) = 0$.)
Therefore this is also exact. We have $E_1^{\vee\vee} = \shfO^{\oplus
  r_1}$ as $E_1\in \Gi{r_1,0}^{d_1}$. Since the extension between the
trivial sheaves is zero on $\proj^2$, we have $E^{\vee\vee} =
\shfO^{\oplus r}$. Therefore $E\in \Gi{r,0}^d$.

Thus we have $\overline{\fT(d,0,\dots,0)_0}\subset \Gi{r,0}^d$. Since
both are $(dr-1)$-dimensional, and $\Gi{r,0}^d$ is irreducible, they
must coincide. This shows the first claim.

The second claim follows as we have already shown $\dim U^d = N-1$ in
\lemref{lem:dimU}, hence other classes cannot be killed by $\pi_*$.

Let us directly check that any of $\fT(0,d,0,\dots,0)_0$, \dots,
$\fT(0,\dots,0,d)_0$ contains a locally free sheaf for
definiteness. (It gives us another proof of \lemref{lem:dimU},
which does not dependent on \cite[Theorem~7.10]{BFG}.)
Then it it is enough to consider the case $N=2$ and check that
$\fT(0,d)_0$ contains a locally free sheaf, as
an extension of a locally free sheaf by a locally free sheaf is again
locally free.
\begin{NB}
  We have
  \begin{equation*}
    \begin{split}
    & 0 \to E_{N-1} \to E \to E_N/E_{N-1} \to 0,
\\
    & 0 \to E_{N-2} \to E_{N-1} \to E_{N-1}/E_{N-2} \to 0,
\\
    & \qquad\qquad\qquad \cdots
\\
    & 0 \to E_1 \to E_2 \to E_2/E_1 \to 0.
    \end{split}
  \end{equation*}
\end{NB}%
Furthermore we may assume $r=2$ and $r_1 = r_2 = 1$.
\begin{NB}
For $n=1$, we
consider the Koszul resolution
\begin{equation*}
  0 \to \Omega^2_{\AA^2} \to \Omega^1_{\AA^2} \to \Omega^0_{\AA^2}
  \to \CC_0 \to 0,
\end{equation*}
where $\CC_0$ is the skyscraper sheaf at the origin $0$. We replace the
last part by the maximal ideal $\mathfrak m_0$, we get
\begin{equation*}
  0 \to \Omega^2_{\AA^2} \to \Omega^1_{\AA^2} \to \mathfrak m_0 \to 0.
\end{equation*}
This exact sequence does not give us an example, as the first Chern
classes does not vanish. But it gives a local example.
\end{NB}%

We use the ADHM description. Let
\begin{equation*}
  \begin{split}
  & B_1 =  \begin{pmatrix}
   0      & 1        & 0       & \hdots & 0          \\
          & 0        & 1       & \hdots  & 0   \\
          &          & \ddots  & \ddots & \vdots     \\
          &          &         & 0      & 1          \\
   \bigzerol &       &         &         & 0
  \end{pmatrix}, \;
  B_2 = 0, \;
\\
  & I =
\begin{pmatrix}
  0 & 0 \\ \vdots & \vdots \\ 0 & 0 \\ 1 & 0
\end{pmatrix},
\;
  J =
  \begin{pmatrix}
    0 & 0 & \hdots & 0 \\
    1 & 0 & \hdots & 0
  \end{pmatrix}.
  \end{split}
\end{equation*}
We have $[B_1,B_2]+IJ = 0$. We see $(B_1,B_2,I,J)$ is stable, i.e., a
subspace $S\subset\CC^d$ containing the image of $I$ and invariant
under $B_1$, $B_2$ must be $S = \CC^d$. We also see that
$(B_1,B_2,I,J)$ is costable, i.e., a subspace $S\subset\CC^d$
contained in the kernel of $J$ and invariant under $B_1$, $B_2$ must
be $S = 0$. Therefore $(B_1,B_2,I,J)$ defines a framed locally free
sheaf $(E,\varphi)$, i.e., an element in $\Bun[SL(2)]d$. We consider
a subspace $\{
\begin{pmatrix}
  0 \\ *
\end{pmatrix}\}\subset \CC^2$, which is the kernel of $a$. Taking $0$
as a subspace in $\CC^d$, we have a subrepresentation of a
quiver. Therefore $E$ contains the trivial rank $1$ sheaf
$\shfO_{\proj^2}$ correspondingly. The quotient $E/\shfO_{\proj^2}$ is
given by the data
\begin{equation*}
  I =
\begin{pmatrix}
  0 \\ \vdots  \\ 0 \\ 1
\end{pmatrix}
\end{equation*}
and $J = 0$, $B_1$, $B_2$ the same as above. This is the ideal sheaf
$(x^d,y)$, and hence in $\Gi{1,0}^d$. Thus $E$ is a point in $\fT(0,d)_0$.
\end{proof}

\begin{NB}
  The following is a copy of an old computation. It is kept as a
  record. Note that some notation is different.

\subsection{Study of the space \texorpdfstring{$U_n$}{Un}}

We give a similar description for $U_n^{L,G} =
\Hom(\Phi_{L,G}(\IC(\Uh{n})), \cC_{S_{(n)}\AA^2})$.

Consider the diagram
\begin{equation}
  \begin{CD}
    \UhP{n}
    @>{\id\times p}>>
    \UhP{n} \times \UhL{n}
    @>{i\times \id}>> \Uh{n}\times \UhL{n}
\\
  @V{p}VV @VV{p\times\id}V @.
\\
  \UhL{n} @>j>> \UhL{n}\times \UhL{n}.
  \end{CD}
\end{equation}
Then
\begin{equation}
  \begin{split}
    & \Ext^*(p_! i^* \IC(\Uh{n}), \cC_{S_{(n)}\AA^2})
\\
    =\; &
    H^*(\UhL{n}, {\mathcal Hom}(p_! i^* \IC(\Uh{n}),
    \cC_{S_{(n)}\AA^2}))
\\
    =\; & H^*(\UhL{n}, (p_! i^* \IC(\Uh{n}))^\vee
    \overset{!}{\otimes}
    \cC_{S_{(n)}\AA^2})
\\
    =\; & H^*(\UhL{n}, j^! (p_* i^! \IC(\Uh{n})
    \boxtimes \cC_{S_{(n)}\AA^2}))
\\
    =\; & H^*(\UhL{n}, j^! (p\times\id)_*
    (i^! \IC(\Uh{n}) \boxtimes \cC_{S_{(n)}\AA^2}))
\\
    =\; & H^*(\UhL{n}, p_* (\id\times p)^!  (i\times\id)^!
    \IC(\Uh{n}) \boxtimes \cC_{S_{(n)}\AA^2}
    )
\\
    =\; & H^*(\UhP{n},(i\times p)^!
    \IC(\Uh{n}) \boxtimes \cC_{S_{(n)}\AA^2}
    ).
  \end{split}
\end{equation}

We next consider the diagram
\begin{equation}
  \begin{CD}
   (\UhP{n})_0 \times\AA^2 @>\cong>> p^{-1}(S_{(n)}\AA^2)
   @>{i\times p}>> \Uh{n}\times S_{(n)}\AA^2
\\
  @. @V{\beta}VV @VV{\id\times\gamma}V
\\
  @. \UhP{n} @>{i\times p}>> \Uh{n}\times \UhL{n},
  \end{CD}
\end{equation}
where $\beta$, $\gamma$ are inclusions. We have
\begin{equation}
  \begin{split}
  & H^*(\UhP{n},(i\times p)^!
    \IC(\Uh{n}) \boxtimes \cC_{S_{(n)}\AA^2}
    )
\\
  = \; &
  H^*(\UhP{n},(i\times p)^! (\id\times\gamma)_*
    \IC(\Uh{n}) \boxtimes \cC_{S_{(n)}\AA^2}
    )
\\
  = \; &
  H^*(\UhP{n},\beta_* (i\times p)^!
    \IC(\Uh{n}) \boxtimes \cC_{S_{(n)}\AA^2}
    )
\\
  = \; &
  H^*((\UhP{n})_0\times\AA^2, i^!
    \IC(\Uh{n}) \boxtimes \cC_{\AA^2})
\\
  = \; &
  H^*((\UhP{n})_0, i^!\IC(\Uh{n})).
\end{split}
\end{equation}
By considering the degree, we have
\begin{equation}
  \Hom(p_!i^* \IC(\Uh{n}), \cC_{S_{(n)}\AA^2})
  \cong H^{\topdeg}((\UhP{n})_0, i^! \IC(\Uh{n})).
\end{equation}
\end{NB}

\subsection{A pairing on \texorpdfstring{$V^d$}{Vd}}\label{sec:pairV}

In the same way as \subsecref{sec:pairing}, we can define a
nondegenerate pairing between $V^{d,P}$ and $V^{d,P_-}$.

\begin{NB}
For a later purpose, we do not use the trivialization of the local
system over $G/L_0$ and keep $i_-$, $j_-$. Then
\end{NB}
We have an isomorphism
\begin{equation}
  H^0(\xi_0^! i^* j^! \pi_* \cC_{\cGi{r}^d})
  \xrightarrow{\cong} H^0(\xi_0^* i_-^! j_-^* \pi_! \cC_{\cGi{r}^d}),
\end{equation}
where we also used $\pi_* = \pi_!$ as $\pi$ is proper. By the base
change and the replacement $i^*$, $i_-^!$ to $p_*$, $(p_-)_{!}$, we
can identify this with
\begin{equation}
  H_{[0]}( \Gi{P,0}^d) \to H_c^{[0]}(\Gi{P_-,0}^d),
\end{equation}
and we have a pairing
\begin{equation}
  \la\ ,\ \ra\colon
  H_{[0]}( \Gi{P,0}^d) \otimes H_{[0]}(\Gi{P_-,0}^d) \to \CC.
\end{equation}

Note that we also have the intersection pairing in the centered
Gieseker space $\cGi{r}^d$.\index{UhTildercG@$\cGi{r}^d$} As the
intersection $\Gi{P,0}^d\cap \Gi{P_-,0}^d$ consists of a {\it
  compact\/} space $\Gi{L,0}^d$, the pairing is well-defined, and
takes values in $\CC$. We multiply the sign $(-1)^{\dim \cGi{r}^d/2} =
(-1)^{dr-1}$ as before.

\begin{Lemma}\label{lem:inter=pair}
  The pairing is equal to the intersection pairing.
\end{Lemma}

\begin{proof}
  \begin{NB}
    Dec. 22 : The following idea does not seem to work, as the square
    \begin{equation}
      \begin{CD}
        \Gi{L}^d @>>> \Gi{P}^d
\\
         @VVV @VVV
\\
        \Uh[L]{d} @>>> \Uh[P]{d}
      \end{CD}
    \end{equation}
    is {\it not\/} a fiber square. Therefore we cannot move to
    $\Gi{L}^d$. Moreover, we do not have a natural homomorphism
    $\cC_{M^T}\to i^* j^! \cC_M$. In fact, the stable envelop gives us
    a homomorphism.

    Oct. 24 : Misha says that it is a consequence of $\cC_{M^T} \cong
    i^* j^! \cC_M \to i_-^! j_-^* \cC_M = \cC_{M^T}$ is an identity
    for a smooth space $M$. I need to check it.

  \end{NB}%
The pairing is the restriction of that on equivariant cohomology groups:
\begin{equation}
    \la\ ,\ \ra\colon
    H^{*}_\TT(\xi_0^!i^*j^!\pi_*\cC_{\cGi{r}^d})
    \otimes
    H^{*}_\TT(\xi_0^!i_-^*j_-^!\pi_*\cC_{\cGi{r}^d})
   \to H^*_\TT(\mathrm{pt}).
\end{equation}
By the localization theorem, natural homomorphisms
\begin{gather}
     H^{*}_\TT(\xi_0^!i^!j^!\pi_*\cC_{\cGi{r}^d})\to
     H^{*}_\TT(\xi_0^!i^*j^!\pi_*\cC_{\cGi{r}^d}),
\\
    H^{*}_\TT(\xi_0^!i_-^*j_-^!\pi_*\cC_{\cGi{r}^d})\to
    H^{*}_\TT(\xi_0^*i_-^*j_-^*\pi_*\cC_{\cGi{r}^d})
\end{gather}
become isomorphisms over the fractional field of
$H^*_\TT(\mathrm{pt})$. Then the pairing between
$H^{*}_\TT(\xi_0^!i^!j^!\pi_*\cC_{\cGi{r}^d})$ and
$H^{*}_\TT(\xi_0^*i_-^*j_-^*\pi_*\cC_{\cGi{r}^d})$ is equal to the
intersection pairing by \cite[\S8.5]{CG}. Therefore we only need to
show that the composition
\begin{equation}
    i^! j^! \to i^* j^! \to i_-^! j_-^*
\end{equation}
is equal to $i^! j^! = i_-^! j_-^! \to i_-^! j_-^*$. This is a
consequence of the following general statement :
\begin{NB}
    See Sasha's message on May 13, 2014.
\end{NB}%
Let $T$ be a torus
action on $X$ and $Y = X^T$ (more generally, it can be a closed
invariant subset containing $X^T$). Let $a\colon Y\to X$ be the
embedding. Let $F$ be a functor from $D_T(X)$ to $D_T(Y)$. Assume that
we have two morphisms of functors $\alpha$, $\beta\colon a^!\to
F$. Then $\alpha=\beta$ if and only if it is so on the image of
$a_!\colon D_T(Y)\to D_T(X)$. We apply this claim to $a=ji$, $F=i_-^!
j_-^*$. In our case, $\alpha = \beta$ on $a_!D_T(Y)$ is evident, as all
the involved morphisms are identities on the fixed point set $Y$.

Let us give the proof of the claim. We consider a natural map $a_!a^!
\mathcal F \to\mathcal F$ for $\mathcal F\in D_T(X)$. It becomes an
isomorphism if we apply $a^!$ by the base change. We set $\mathcal G =
a_! a^!\mathcal F$. We have $\alpha_{\mathcal G} = \beta_{\mathcal
  G}$, as homomorphisms $a^!\mathcal G\to F(\mathcal G)$, from
the assumption. Then we have $\alpha_{\mathcal F} = \beta_{\mathcal
  F}$ as the composition of $\alpha_{\mathcal G} = \beta_{\mathcal G}$
and $F(\mathcal F)\to F(\mathcal G)$.
\begin{NB}
Record of the original text:

This directly follows from the definition of Braden's isomorphism $i^*
j^! \to i_-^!  j_-^*$.
\begin{NB2}
  For the global case, $i^!j^!\to i^* j^!$, $i_-^!j_-^!\to i_-^!
  j_-^*$ are identified with
  \begin{equation}
    H^*(X,X\setminus X^A,\scF) \to H^*(X, X\setminus
    \mathcal A_X,\scF),
    \quad
    H^*(X,X\setminus X^A,\scF) \to
    H^*(\mathcal R_X, \mathcal R_X\setminus X^A,\scF)
  \end{equation}
respectively. Braden's isomorphism is the restriction
\begin{equation}
  H^*(X, X\setminus \mathcal A_X,\scF)\to
  H^*(\mathcal R_X, \mathcal R_X\setminus X^A,\scF).
\end{equation}
Therefore the assertion is clear.

Recall the definition of Braden's morphism:
$$i^!j^!\to i^*j^!\to i^*j^!j_{-*}j_-^*\to i^*i_*i_-^!j_-^*\to i_-^!j_-^*$$
where the third arrow comes from the base change, and the fourth arrow
comes from the adjointness of $i^*$ and $i_*$.

The desired equality follows from the fact that the composition
$$i^!j^!=i^*j^!j_{-*}j_-^!\to i^*i_*i_-^!j_-^!\to i_-^!j_-^!$$ is the identity
automorphism of $(j\circ i)^!=(j_-\circ i_-)^!$, where the first arrow
comes from the base change, while the second arrow comes from the adjointness
of $i^*$ and $i_*$.
\end{NB2}%
\end{NB}%
\end{proof}

\subsection{Another base of \texorpdfstring{$V^d$}{Vd}}\label{sec:comp}

Recall we have the canonical isomorphism
\(
 \pi_! \cC_{\Gi{L}^d}
  \xrightarrow[\cong]{\cL}\Phi_{L,G}(\pi_! \cC_{\Gi{r}^d})
\)
in \eqref{eq:stac}.
Thanks to the decomposition \eqref{eq:product},
the morphism $\pi\colon \Gi{L}^d\to \UhL{d}$ is the composite of
\begin{equation}\label{eq:pipi}
\pi\times\cdots\times\pi \colon \Gi{r_1}^{d_1}\times\cdots\times
  \Gi{r_N}^{d_N}
\to
\Uh[{SL(r_1)}]{d_1}\times\cdots\times \Uh[{SL(r_N)}]{d_N}
\end{equation}
with the sum map
\begin{equation}
  \kappa\colon
  \Uh[{SL(r_1)}]{d_1}\times\cdots\times \Uh[{SL(r_N)}]{d_N}
  \to \UhL{d}.
\end{equation}
The latter is a finite birational morphism.
Then $\cC_{\Gi{L}^d}$ decomposes under \eqref{eq:pipi} as in \eqref{eq:GU}:
\begin{equation}\label{eq:piA}
  \pi_! \cC_{\Gi{L}^d}
  \cong
  \begin{aligned}[t]
      & \bigoplus
      H_{\topdeg}(\pi^{-1}(x^{d_1}_{\lambda_1})\times\dots\times
      \pi^{-1}(x^{d_N}_{\lambda_N})) \\ & \quad \otimes \kappa_!
      \IC(\Bun[{SL(r_1),\lambda_1}]{d_1}\times \cdots\times
      \Bun[{SL(r_N),\lambda_N}]{d_N}),
  \end{aligned}
\end{equation}
where $\lambda_1$, \dots, $\lambda_N$ are partitions with
$d = d_1 + |\lambda_1| + \cdots +d_N + |\lambda_N|$. (These $d_1$, \dots, $d_N$ are different from above.)
The image of the closure of $\Bun[{SL(r_1),\lambda_1}]{d_1}\times
  \cdots\times
  \Bun[{SL(r_N),\lambda_N}]{d_N}$ under $\kappa$ is the closure of
\begin{equation}
   \Bun[{SL(r_1)}]{d_1}\times \cdots\times \Bun[{SL(r_N)}]{d_N}\times S_\mu\AA^2,
\end{equation}
where $\mu = \lambda_1\sqcup\cdots\sqcup\lambda_N$. Let us denote this
stratum by $\Bun[{L,\mu}]{d_1,\dots,d_N}$. Then as $\kappa$ is a
finite morphism, we have
\begin{equation}\label{eq:kappa}
    \kappa_! \IC(\Bun[{SL(r_1),\lambda_1}]{d_1}\times
  \cdots\times
  \Bun[{SL(r_N),\lambda_N}]{d_N})
  \cong
  \IC(\Bun[{L,\mu}]{d_1,\dots,d_N},\rho),
\end{equation}
where $\rho$ is the local system corresponding to the covering
\begin{equation}\label{eq:cover}
  S_{\lambda_1}\AA^2\times
  \cdots\times S_{\lambda_N}\AA^2\setminus\text{diagonal}
  \to
  S_{\mu}\AA^2
\end{equation}
and $\mu = \lambda_1\sqcup\cdots\sqcup\lambda_N$. Taking sum over
$\lambda_1$, $\lambda_2$, \dots which give the same $\mu$, we get
\begin{multline}
  \bigoplus_{\lambda_1\sqcup\cdots\sqcup\lambda_N = \mu}
  \kappa_! \IC(\Bun[{SL(r_1),\lambda_1}]{d_1}\times
  \cdots\times
  \Bun[{SL(r_N),\lambda_N}]{d_N})
\\
  \cong \IC(\Bun[{L,\mu}]{d_1,\dots,d_N},\rho),
\end{multline}
where $\rho$ is now given by the permutation representation
\begin{equation}\label{eq:rho}
  ({}'V^1)^{\otimes n_1}\otimes ({}'V^2)^{\otimes n_2}\otimes \cdots
\end{equation}
of $S_{n_1}\times S_{n_2}\times \cdots$ if $\mu = (1^{n_1}
2^{n_2}\cdots)$ with $\dim {}'V^d = N$. 
\begin{NB}
Let us understand that this $V^d$ is temporarily different from the
previous one \eqref{eq:13}.
\end{NB}%
Here we define ${}'V^d$ as the cohomology of the union of the fibers of
\eqref{eq:cover} for the special case when $\mu$ is the partition
$(d)$ with the single entry $d$, where the union runs over $\lambda_1$, \dots,
$\lambda_N$:
\begin{equation}\label{eq:sigma}
  \sigma \colon
  \bigsqcup_{\lambda_1\sqcup\cdots\sqcup\lambda_N = (d)} S_{\lambda_1}\AA^2
  \times \cdots\times S_{\lambda_N}\AA^2 \setminus\text{diagonal}
%  S_{(d)}\AA^2\times S_{\emptyset}\AA^2\times\cdots
%  \times S_{\emptyset}\AA^2\sqcup\cdots\sqcup
%  S_{\emptyset}\AA^2\times\cdots
%  \times S_{\emptyset}\AA^2\times S_{(d)}\AA^2
  \to S_{(d)}\AA^2,
\end{equation}
and
\begin{equation}
  {}'V^d = H_0( \sigma^{-1}(d\cdot 0)).
\end{equation}
Since $\mu = (d)$, one of $\lambda_1$, \dots, $\lambda_N$ is $(d)$ and
others are the empty partition $\emptyset$. Therefore the fiber
$\sigma^{-1}(d\cdot 0)$ consists of $N$ distinct points, hence we have
$\dim {}'V^d = N$.

Moreover
\(
  \Hom_{\Perv(\UhL{d})}(\cC_{S_{(d)}\AA^2},
  \pi_! \cC_{\Gi{L}^d})
\)
is given by the component $\IC(\Bun[{L,(d)}]{0,\dots,0},\rho)$, where
$\rho$ is the trivial representation of $S_1$ on ${}'V^d$. Therefore
we have a canonical isomorphism
\begin{equation}\label{eq:Vd}
  \begin{split}
      {}'V^d 
      &\cong
      \Hom_{\Perv(\UhL{d})}(\cC_{S_{(d)}\AA^2},
      \pi_! \cC_{\Gi{L}^d})\\
      &\cong \Hom_{\Perv(\UhL{d})}(\cC_{S_{(d)}\AA^2},
      \Phi_{L,G}(\pi_! \cC_{\Gi{r}^d})),
  \end{split}
\end{equation}
where the first isomorphism is via
\(
   H_{\topdeg}(\pi^{-1}(x^{d_1}_{\lambda_1})\times\dots\times
   \pi^{-1}(x^{d_N}_{\lambda_N})) \cong \CC
\)
($d_1=\dots=d_N=0$, one of $\lambda_1$,\dots, $\lambda_N$ is $(d)$ and
others are the empty partition) given by the fundamental class,
and the second isomorphism is given by the stable envelope $\cL$.
Thus our ${}'V^d$ is isomorphic to $V^d$ in 
\eqref{eq:13}. We will identify ${}'V^d$ with $V^d$ hereafter.

We have just shown
\begin{equation}
  \pi_!\cC_{\Gi{L}^d} \cong \bigoplus \IC(\BunLl{d_1,\dots,d_N},\rho).
\end{equation}
This is similar to \propref{prop:hyperbolic}, where we used the
factorization argument to construct an isomorphism.
Our argument looks slightly different, as we have not used the
projection $a\colon\AA^2\to \AA^1$. But the isomorphism is the same as
one given by the factorization argument from the above construction,
together with the observation that $a$, $i$, $j$ commute with the
projection $\pi^d_{a,?}$ ($? = G, P, L$).

Note that we have $e_i\in V^d = H_0(\sigma^{-1}(d\cdot 0))$ corresponding to
the component of $\sigma^{-1}(d\cdot 0)$ in
\begin{equation}
    S_{\emptyset}\AA^2\times\cdots\times
    \underbrace{S_{(d)}\AA^2}_{\text{$i^{\mathrm{th}}$ factor}}
    \times\cdots\times S_{\emptyset}\AA^2.
\end{equation}
Then $e_1,\dots, e_N$ gives a base of $V^d$.
\index{ei@$e_i$ (base element of $V^d$)}

If we view $V^d$ as $\Hom_{\Perv(\UhL{d})}(\cC_{S_{(d)}\AA^2}, \pi_!
\cC_{\Gi{L}^d})$, $e_i$ is the composite of homomorphisms
\begin{equation}
    \label{eq:25}
    \cC_{S_{(d)}\AA^2}\to \pi_! \cC_{\Gi{r_i}^d} \to
    \pi_! \cC_{\Gi{L}^d},
\end{equation}
where the left homomorphism is given by the fundamental class
$[\pi^{-1}(d\cdot 0)]$, and the right one is given by the inclusion
of the component $d_i = d$, $d_j = 0$ ($j\neq i$) in the decomposition
\eqref{eq:product}.

\begin{Example}
For $d=1$, $\Gi{r}^1$ (resp.\ $\Uh{1}$) is isomorphic to the product
of $\AA^2$ and the cotangent bundle of $\proj^{r-1}$ (resp.\ the
closure of the minimal nilpotent orbit of $\algsl_r$). Further suppose
$N=r$ and $r_1=\cdots=r_N = 1$.
Then \cite[Remark~3.5.3]{MO} gives us the relation:
\begin{equation}
    \label{eq:23}
    [\overline{\fT(0,\cdots,0,
    \underbrace{1}_{\text{$k^{\mathrm{th}}$ factor}},
    0,\cdots, 0)}]
  = (-1)^{k-1}(e_k + e_{k+1} + \cdots + e_r).
\end{equation}
Here the sign $(-1)^{k-1}$ comes from the polarization, mentioned
in~\subsecref{sec:Gifixed}.
\end{Example}

\begin{NB}
Let us give the computation.

This is taken from my message on Dec.~17, 2012.

Let us consider the special case $n=1$. Then it is well-known that
$\Mreg_0(1,r)$ is the product of $\AA^2$ and the minimal nilpotent orbit
of $\algsl_r$. And $M(1,r)$ is the product of $\AA^2$ and the
cotangent bundle of $\proj^{r-1}$.

The torus fixed points are numbered $1$, \dots, $r$, and given a
positive chamber $\mathfrak C$, we have the corresponding
Bialynicki-Birula stratification
$\proj^0\subset\proj^1\subset\cdots\subset\proj^{r-1}$.
The fundamental classes of the conormal bundles
$T^*_{\proj^{k-1}}\proj^{r-1}$ ($k=1,\dots,r$) gives a basis of $V_1 =
H_{\topdeg}(\fT(1)_0)$. We have relations
\begin{equation*}
  \begin{gathered}
  T^*_{\proj^{r-1}}\proj^{r-1} = \overline{\fT(1,0,\dots,0)}, \quad
  T^*_{\proj^{r-2}}\proj^{r-1} = \overline{\fT(0,1,0, \dots,0)},
  \quad\cdots\quad
\\
  T^*_{\proj^0}\proj^{r-1} = \overline{\fT(0,\dots,0,1)}
  \end{gathered}
\end{equation*}
in the notation in the previous subsection.

According to \cite[Remark~3.5.3]{MO}, the stable envelop component
corresponding to the torus-fixed point is the union of two conormal
bundles
\begin{equation*}
  T^*_{\proj^k}\proj^{r-1}\cup T^*_{\proj^{k-1}}\proj^r.
\end{equation*}
Let us give a little more detail:
\begin{NB2}
  Misha's message on Nov.~29, 2012.
\end{NB2}%
Since all the strata closures are smooth in our situation, it is easy
to compute $j_{k!}(\CC)$: it is an extension $0\to \CC_{\proj^{k-1}}\to
j_{k!}(\CC)\to \CC_{\proj^{k}}\to 0$. Now the characteristic cycles are
additive in short exact sequences, and the characteristic cycles of
constant sub- and quotient- sheaves are just the conormal bundles to
their supports. This gives the above.

There is a sign issue here, called a {\it polarization\/} in
\cite[3.3.2]{MO}. For quiver varieties, in particular, Gieseker
spaces, they choose a canonical choice of a polarization, which is
explained in \cite[2.2.7]{MO}.

In our case, we have
\begin{equation}\label{eq:ek}
  \begin{split}
  e_k &= (-1)^{k-1} ([T^*_{\proj^{r-k}}\proj^{r-1}]
  + [T^*_{\proj^{r-k-1}}\proj^{r-1}])
\\
  &= (-1)^{k-1}
  \begin{aligned}[t]
  & \big([\overline{\fT(0,\cdots,0,
    \underbrace{1}_{\text{$k^{\mathrm{th}}$ factor}},
    0,\cdots, 0)}]
\\
  & \quad +
  [\overline{\fT(0,\cdots,0,
    \underbrace{1}_{\text{$(k+1)^{\mathrm{th}}$ factor}},
    0,\cdots, 0)}]\big)
  \end{aligned}
  \end{split}
\end{equation}
for $k=1,\dots,r$. Then gives a basis of for the vector space $V_1$
from the hyperbolic restriction from $G = SL_r$ to $T$, applied to the
pushforward of the constant sheaf for the Gieseker space. We have
\begin{equation}
  \sum_{k=1}^{r} e_k = [T^*_{\proj^{r-1}}\proj^{r-1}].
\end{equation}
The right hand side is $\beta_{-1}[M(0,r)]$, where $\beta_{-1}$ is the
Baranovsky (creation) operator.
This means that the diagonal gives the Baranovsky's Heisenberg.
It is compatible with \cite[Th.~12.2.1]{MO}.
More generally, we have
\begin{equation}
  [T^*_{\proj^{r-k}}\proj^{r-1}]
  = (-1)^{k-1}(e_k + e_{k+1} + \cdots + e_r).
\end{equation}
\end{NB}%

\begin{Example}\label{ex:base}
  We know that $[\overline{\fT(d,0,\dots,0)_0}] = [\Gi{r,0}^d]$
  (\propref{prop:Un}), and hence
  \begin{equation}
    [\overline{\fT(d,0,\dots,0)_0}] = e_1+\cdots+e_N
  \end{equation}
  by \eqref{eq:22}.

  On the other hand, the opposite extreme
  $[\overline{\fT(0,0,\dots,d)_0}]$ is equal to $e_N$ up to sign by
  the support property of the stable envelope \cite[Th.~3.3.4
  (i)]{MO}. The polarization is opposite, therefore the sign is the
  half of the codimension of the corresponding fixed point
  component. We get
  \begin{equation}
    [\overline{\fT(0,0,\dots,d)_0}] = (-1)^{d(r-r_N)} e_N.
  \end{equation}

  If $N=2$, two elements exhaust the base.
\end{Example}

The transition matrix between two bases for $d > 1$, $N > 2$ can be
calculated from \eqref{eq:111} together with \eqref{eq:109}
below. Though \eqref{eq:111} determines $[\overline{\fT(0,\cdots,0,
  {d}, 0,\cdots, 0)}]$ ($d$ is in the {$k^{\mathrm{th}}$ entry) up to
  $\CC[\overline{\fT(d,0,\dots,0)_0}]$, it is a linear span of $e_k$,
  \dots, $e_N$ thanks to the support property of the stable envelope.
  Therefore we can fix the ambiguity.

\begin{NB}
    Let me describe $\Uh[B,0]1$ even more explicitly. It consists of
    rank $1$ upper triangular matrices $X$. Let
    \begin{equation}
        \label{eq:49}
        X = \left(
          \begin{array}{c:c}
              0 & X' \\
              \hdashline
              0 & 0
          \end{array}
        \right),
    \end{equation}
    where $X'$ is a $(k-1)\times(r-k+1)$ matrix of rank $\le 1$. Such
    matrices $X$ is the irreducible component in \eqref{eq:23}.
\end{NB}

\subsection{Computation of the pairing}

Let us relate the pairing in \subsecref{sec:pairV} to the pairing
defined on $\Gi{L}^d$ using the stable envelope
\begin{equation}
  \cL \colon H_{[0]}(\Gi{L,0}^d) \xrightarrow{\cong} H_{[0]}(\Gi{P,0}^d).
\end{equation}

Let us temporarily denote the stable envelope with respect to the
opposite parabolic by $\cL^-$. Then we want to compute
\begin{equation}
  \la\cL(\alpha),\cL^-(\beta)\ra,
\end{equation}
which is equal to the intersection pairing times
$(-1)^{\dim\cGi{r}^d}$ by \lemref{lem:inter=pair}.

Suppose that $\alpha$, $\beta$ are classes on a component $Z$ of
$\Gi{L,0}^d$.
Let us take equivariant lifts of $\alpha$, $\beta$ to
$Z(L)^0$-equivariant cohomology.
Since the supports of $\cL(\alpha)$ and $\cL^-(\beta)$ intersect
along $Z$ by one of characterizing properties of the stable envelope
\cite[Th.~3.3.4(i)]{MO}, we need to compute the restriction of the
(Poincar\'e dual of) $\cL(\alpha)$, $\cL^-(\beta)$ to the fixed point
component $Z$. Again by a property of the stable envelop
\cite[Th.~3.3.4(ii)]{MO}, we have $\cL(\alpha)|_{Z} =
(\pol_{\text{rep}}/\pol) e(N^-)\cup \alpha$ and $\cL^-(\beta)|_{Z} =
(\pol_{\text{att}}/\pol) e(N^+)\cup \beta$, where $\pol_{\text{rep}}$,
$\pol_{\text{att}}$ are the polarizations given by attracting and
repellent directions. Then we have
\begin{equation}
  \int_{\cGi{r}^d} \cL(\alpha)\cup \cL^-(\beta)
  = \frac{\pol_{\text{rep}} \pol_{\text{att}}}{\pol^2}
  \int_{\cGi{r}^d} e(N) \cup \alpha\cup\beta
  = (-1)^{\codim Z/2} \int_Z \alpha\cup\beta
\end{equation}
by the fixed point formula.
\begin{NB}
  Note that $\pol_{\text{rep}} \pol_{\text{att}} = (-1)^{\codim Z/2}
  \pol_{\text{rep}}^2 = (-1)^{\codim Z/2}\pol^2$.
\end{NB}%
Therefore if we multiply $(-1)^{\dim \cGi{r}^d/2}$, we get
$(-1)^{\dim Z/2}\int_X \alpha\cup\beta = \la\alpha,\beta\ra$.

If $\alpha$, $\beta$ are supported on different components $Z$, $Z'$
of $\Gi{L,0}^d$ respectively, we use a property \cite[Th.~3.7.5]{MO},
which says the restrictions of $\cL(\alpha)$, $\cL(\beta)$ to
components other than $Z$, $Z'$ are zero. Then it is clear that
$\la\cL(\alpha),\cL^-(\beta)\ra = 0$.

\begin{NB}
  I probably need to explain the identification by the trivialization
  over $G/P_0$ at some point...... I think that it does not affect the
  sign.
\end{NB}

As an application of this formula, we compute $\la
e_i,e^-_j\ra$, where $e_i\in V^d$ as in the previous subsection, and
$e_j^-\in V^{d,P_-}$ is defined in the same way using the opposite
hyperbolic restriction $\cL^-$.\index{eiminus@$e_i^-$}
This is reduced to the computation of the self-intersection number of
the punctual Quot scheme $\Gi{r_i,0}^d$ in the centered Gieseker
space $\cGi{r_i}^d$. This is given by $(-1)^{r_id-1}d r_i = (-1)^{\dim
  \cGi{r_i}^d/2} d r_i$ (\cite[\S4]{Baranovsky}).
Therefore we get
\begin{Proposition}\label{prop:pairinge}
  We have
  \begin{equation}\label{eq:109}
    \la e_i,e_j^-\ra = d r_i \delta_{ij}.
  \end{equation}
\end{Proposition}

\begin{NB}
  This is a comment for the previous formulation.

  I am no longer sure about this formula. We need to use the
  identification via the trivialization of the local system. I am not
  sure that the following is commutative.
  \begin{equation}
    \begin{CD}
      H_{[0]}(\Gi{L,0}^d) @>\cL>> H_{[0]}(\Gi{P,0}^d)
\\
      @| @VV{\cong}V
\\
      H_{[0]}(\Gi{L,0}^d) @>\cL^->> H_{[0]}(\Gi{P_-,0}^d)
    \end{CD}
  \end{equation}
\end{NB}

\subsection{Relation between \texorpdfstring{$V^d$}{Vd} and \texorpdfstring{$U^d$}{Ud}}\label{subsec:VU}
Let us apply the decomposition \eqref{eq:GU} to \eqref{eq:Vd}. We have
\begin{equation}
  \begin{split}
  V^d &= \Hom(\cC_{S_{(d)}\AA^2}, \Phi_{L,G}(\pi_!\cC_{\Gi{r}^d}))
\\
  &= \bigoplus_{d_1+|\lambda|=d}
  H_{\topdeg}(\pi^{-1}(x^{d_1}_\lambda)) \otimes
  \Hom(\cC_{S_{(d)}\AA^2}, \Phi_{L,G}(\IC(\BunGl{d_1}))).
  \end{split}
\end{equation}
Then
\(
  \Hom(\cC_{S_{(d)}\AA^2}, \Phi_{L,G}(\IC(\BunGl{d_1})))
\)
is nonzero only in either of the following cases:
\begin{enumerate}
\item $d_1=d$ and $\lambda=\emptyset$,
\item $d_1=0$ and $\lambda=(d)$.
\end{enumerate}
In the first case, it is $U^{d}$ by definition. And in the second
case, it is
\begin{equation}
  \Hom(\cC_{S_{(d)}\AA^2}, \Phi_{L,G}(\cC_{S_{(d)}\AA^2}))
  =
  \Hom(\cC_{S_{(d)}\AA^2},
  \cC_{S_{(d)}\AA^2}) \cong \CC\id.
\end{equation}
Thus
\begin{equation}\label{eq:VU}
  V^d \cong \left(H_{\topdeg}(\pi^{-1}(x^d_{\emptyset})) \otimes U^d\right)
  \oplus H_{\topdeg}(\pi^{-1}(x^0_{(d)})).
\end{equation}
Note that $\pi^{-1}(x^d_{\emptyset})$ is a single point. Therefore we
have the canonical isomorphism
$H_{\topdeg}(\pi^{-1}(x^d_{\emptyset}))\cong \CC$.
Now the homomorphism $\pi_*\colon V^d \cong H_{[0]}(\Gi{P,0}^d) \to
U^d \cong H_{[0]}(\Uh[P,0]d)$ is identified with the projection to
the first component in \eqref{eq:VU}. In particular, bases of $U^d$
and $V^d$ given by irreducible components (see \lemref{lem:Ud} and
\propref{prop:Un}) are related by the projection.
\begin{NB}
The following is not quite true. See \eqref{eq:51}.

Now it is clear that the isomorphism \eqref{eq:VU} is compatible with
bases of $U^d$ and $V^d$ given by irreducible components (see
\lemref{lem:Ud} and \propref{prop:Un}).
\end{NB}%

The subspace $H_{\topdeg}(\pi^{-1}(x^0_{(d)}))$ is $1$-dimensional
space spanned by the fundamental class $[\pi^{-1}(x^0_{(d)})]$, or
equivalently $P_{-d}([0])\cdot [\Uh{0}]$ where $P_{-d}([0])$ is the
Heisenberg operator, and $[\Uh{0}] = 1\in H^0_{\TT}(\Uh{0})$.
Recall that the Baranovsky's Heisenberg operator is mapped to the
diagonal operator under the stable envelope, see \subsecref{sec:tensor}.
\begin{NB}
  I do not check that this is true even for general $L\neq T$. It is
  stated only for $L = T$. Probably the compatibility implies the
  general case.
\end{NB}%
It means that $[\pi^{-1}(x^0_{(d)})]$ is equal to
\begin{equation}
  e_1 + \cdots + e_N,
\end{equation}
where $\{ e_i\}$ is the base of $V^d$ in the previous subsection.

And $U^d$ is the subspace killed by the Heisenberg operator
$P_d(1)$. Therefore
\begin{equation}\label{eq:baseofU}
  U^d
  \cong
  \left\{ \lambda_1 e_1 + \cdots + \lambda_N e_N\, \middle|\,
  \lambda_1 + \cdots + \lambda_N = 0 \right\}.
\end{equation}
We have a base $\{ e_i - e_{i+1} \}_{i=1,\dots, N-1}$ of $U^d$.

It is also clear that the decomposition \eqref{eq:VU} is orthogonal
with respect to the pairing in \subsecref{sec:pairV}. And the
restriction of the pairing to $U^d$ is equal to one in
\subsecref{sec:pairing}. Therefore we can calculate the pairing
between $U^{d,P}$ and $U^{d,P_-}$. Let us consider the case $P = B$
for brevity. We have
\begin{equation}\label{eq:Cartan}
  \la e_i - e_{i+1}, e_j^- - e_{j+1}^-\ra =
  \begin{cases}
    2d & \text{if $i=j$}, \\
    -d & \text{if $|i-j|=1$}, \\
    0 & \text{otherwise}
  \end{cases}
\end{equation}
by \propref{prop:pairinge}.
Thus the pairing between $U^{d,B}$ and $U^{d,B_-}$ is identified with
the natural pairing on the Cartan subalgebra $\h$ of $\algsl_r$
multiplied by $d$, under the identification $e_i - e_{i+1}$ and
$e_i^--e_{i+1}^-$ with the simple coroot $\alpha_i^\vee$.

\begin{NB}
  Let us denote \eqref{eq:23} by $t_{k-1}$ ($k=1,2,\dots,r$). Its
  projection to $U^{d=1}$ is given by
    \begin{equation}\label{eq:51}
        \begin{split}
        & \frac{(-1)^k}r\left(
          (r-k+1) (e_1 + \cdots + e_{k-1})
          - (k-1) (e_k + \cdots + e_r) \right)
\\
   =\; & \frac{(-1)^k}r\left(
     \begin{aligned}[c]
          & (r-k+1) (e_1 - e_2) + 2(r-k+1)(e_2 - e_3) + \cdots
          + (k-1)(r-k+1) (e_{k-1} - e_k)
\\
          & \quad
          + (k-1)(r-k)(e_k - e_{k+1}) + \cdots + (k-1) (e_{r-1} - e_r)
     \end{aligned}\right).
        \end{split}
    \end{equation}
For $k=r$, we have
\begin{equation}
  \frac{(-1)^r}r (e_1 + \cdots + e_{r-1} - (r-1) e_r).
\end{equation}
We have
\begin{equation}\label{eq:112}
  (-1)^{k}\la pr_{U^1} t_{k-1}, e^-_i - e^-_{i+1}\ra =
  \begin{cases}
    1 & \text{if $i=k-1$,} \\ 0 & \text{otherwise}.
  \end{cases}
\end{equation}
\end{NB}

\subsection{Compatibility}\label{subsec:comp}

Let us take $L=T$. We shall show that the base $\{ e_i - e_{i+1} \}_i$
of $U^d$ is compatible with the construction in
\subsecref{sec:anotherbase} in this subsection.

We fix the Borel subgroup $B$ consisting of upper triangular matrices,
and let $P_i$ be the parabolic subgroup corresponding to a simple root
$\alpha_i$ and $L_i$ be the Levi subgroup ($i=1,\dots r-1$).
Recall that we have taken
\begin{equation}
  1^d_{L_i,G}\in
  \Hom_{\Perv(\Uh[{L_i}]{d})}(\IC(\Uh[{L_i}]{d}), \Phi_{L_i,G}(\IC(\Uh{d}))).
\end{equation}
(See \eqref{eq:1LG}.)

Let us consider the corresponding fixed point set $\Gi{L_i}^d =
(\Gi{r}^d)^{Z(L_i)}$ in the Gieseker space. The decomposition
\eqref{eq:product} in our case is
\begin{equation}\label{eq:121}
\bigsqcup_{d_1+\dots+\widehat{d_{i+1}}
    + \dots + d_r = d} \Gi{1}^{d_1}\times\cdots \times \Gi{1}^{d_{i-1}}
  \times \Gi{2}^{d_i}\times \Gi{1}^{d_{i+2}}\times\cdots \times\Gi{1}^{d_r}.
\end{equation}
There is a distinguished connected component, isomorphic to $\Gi{2}^d$
with $d_i = d$, $d_j = 0$ for $j\neq i$. Let us denote it by $Z$.

Recall that $\Uh[{L_i}]d$ is equal to $\Uh[SL(2)]d$ as a topological
space and the open subvariety $\Bun[{L_i}]d$ is equal to
$\Bun[SL(2)]d$. The connected component $Z$ is characterized among all
components of $\Gi{L_i}^d$, as it contains
$\Bun[{L_i}]d$.

We denote by $\pol$ the polarization of $Z$ in $\Gi{r}^d$ in
\subsecref{sec:Gifixed}. We understand it is $\pm 1$, according to
whether it is equal to the polarization given by attracting directions
or not, as in \subsecref{sec:polarization}.
We correct $1^d_{L_i,G}$ by $\pol 1^d_{L_i,G}$ so that it will be
compatible with the stable envelope.

Let us consider the diagram
\begin{equation}\label{eq:CD}
  \begin{CD}
  \IC(\Uh[L_i]{d})@>{\pol 1^d_{L_i,G}}>> \Phi_{L_i,G}(\IC(\Uh{d}))
\\
@AAA @AAA
\\
  \pi_! \cC_{\Gi{L_i}^d} @>\cong>{\cL_{L_i,G}}>
  \Phi_{L_i,G}(\pi_!\cC_{\Gi{r}^d}).
  \end{CD}
\end{equation}
The upper arrow is given just above, and the bottom arrow is the
stable envelope.
The right vertical arrow comes from the natural projection to the
direct summand
\(
   \pi_!(\cC_{\Gi{r}^d})\to \IC(\Uh{d})
\)
in \eqref{eq:GU},
which is the identity homomorphism on the open subset $\Bun{d}$ of
$\Uh{d}$.
The left vertical arrow is defined as follows. We have the
distinguished component $Z$ of $\Gi{L_i}^d$ isomorphic to
$\Gi{2}^d$. We have $\IC(\Uh[{L_i}]d) = \IC(\Uh[SL(2)]d)$, and hence
have a natural projection $\pi_! \cC_{\Gi{2}^d}\to \IC(\Uh[L_i]{d})$,
as for the right vertical arrow. Composing with the restriction to the
distinguished component $\pi_!\cC_{\Gi{L_i}^d}\to \pi_!  \cC_{Z}$ , we
define the left vertical arrow.

\begin{Proposition}
  The diagram \eqref{eq:CD} is commutative.
\end{Proposition}

\begin{proof}
  From the construction of the diagram, it is clear that we need to
  check the commutativity on the open subset $\Bun[{L_i}]{d}$. Then
  the commutativity is clear, as two constructions $\pol 1^d_{L_i,G}$
  and $\mathcal L_{L_i,G}$ are the same: Both are given by the Thom
  isomorphism corrected by polarization. See \cite[Th.~3.3.4(ii)]{MO}
  for the stable envelope.
\end{proof}

Recall also that we have proposed that there exists a canonical
element
\begin{equation}
  \begin{split}
  1^d_{L_i}\in & \Hom(\cC_{S_{(d)}\AA^2}, \Phi_{T,L_i}(\IC(\Uh[L_i]{d})))\\
  & \cong \Hom(\cC_{S_{(d)}\AA^2},\Phi_{\CC^*,SL_2}(\IC(\Uh[{SL_2}]{d})))
  \end{split}
\end{equation}
in \subsecref{sec:anotherbase}. \index{1Li@$1^d_{L_i}$}
We define it so that the following diagram is commutative:
\begin{equation}
    \label{eq:24}
    \begin{CD}
     \cC_{S_{(d)}\AA^2} @>1^d_{L_i}>> \Phi_{\CC^*,SL(2)}(\IC(\Uh[SL(2)]{d}))
\\
  @VVV @AAA
\\
    \pi_! \cC_{\Gi{\CC^*}^d}
    @>\cong>\cL_{\CC^*,SL(2)}> \Phi_{\CC^*,SL(2)}(\pi_!\cC_{\Gi{2}^d}),
    \end{CD}
\end{equation}
where we choose the parabolic subgroup in $SL(2)\cong [L_i,L_i]$
corresponding to the chosen Borel subgroup $B$ to define the
hyperbolic restriction $\cL_{\CC^*,SL(2)}$. The right vertical arrow
is the projection to the direct summand as before.
The left vertical arrow is $e_i - e_{i+1}$, where $\{e_i, e_{i+1}\}$
is the base of $V^d_{\CC^*,SL(2)} \cong \Hom(\cC_{S_{(d)}\AA^2},
\pi_!\cC_{\Gi{\CC^*}^d})$, i.e., $e_i$ corresponds to
$\Gi{1}^d\times\Gi{1}^0\subset \Gi{\CC^*}^d = \bigsqcup
\Gi{1}^{d_1}\times \Gi{1}^{d_2}$, and $e_{i+1}$ corresponds
to $\Gi{1}^0\times\Gi{1}^d$.

We enlarge the bottom row as
\begin{equation}\label{eq:CD2}
  \begin{CD}
  \cC_{S_{(d)}\AA^2} @>1^d_{L_i}>>   \Phi_{T,L_i}(\IC(\Uh[L_i]{d}))
\\
  @VV{e_i-e_{i+1}}V @AAA
\\
     \pi_! \cC_{\Gi{T}^d}
    @>\cong>\cL_{T,L_i}>\Phi_{T,L_i}(\pi_!\cC_{\Gi{L_i}^d}).
  \end{CD}
\end{equation}
Here we identify $\Gi{2}^d$ with the distinguished component $Z$.
We similarly consider $\Gi{\CC^*}^d$ as a union of components
of $\Gi{T}^d$, putting it in $i$ and $(i+1)^{\mathrm{th}}$ components.
The left vertical arrow is $e_i - e_{i+1}$, where $\{ e_1,\dots,
e_r\}$ is the base of $V^d_{T,G}$.
Two bases are obviously compatible, so it is safe to use the same
notation.

\begin{NB}
  Remark that this shows that our guess that $1_{L_i}$ is given by an
  irreducible component of the relevant variety is WRONG. See Misha's
  message on Nov.~29, 2012. Recall (see \eqref{eq:ek})
  \begin{equation}
    e_1 - e_2 = [T^*_{\proj^1}\proj^1] + 2[T^*_{\proj^0}\proj^1].
  \end{equation}
Under the projection $\pi\colon \Gi{P}^1\to \UhP{1}$, this is mapped to
$2[T^*_{\proj^0}\proj^1]$. If we divide this by $2$, we need to change
$1_{L_i}$ to
  \begin{equation}
    \frac12 (e_1 - e_2) = \frac12 [T^*_{\proj^1}\proj^1] + [T^*_{\proj^0}\proj^1].
  \end{equation}
  This seems unnatural to me.
\end{NB}

We apply $\Phi_{T,L_i}$ to the commutative diagram \eqref{eq:CD} and
combine it with \eqref{eq:CD2}:
\begin{equation}\label{eq:CD3}
  \begin{CD}
  \cC_{S_{(d)}\AA^2} @>1^d_{L_i}>>
  \Phi_{T,L_i}(\IC(\Uh[L_i]{d}))
  @>{\Phi_{T,L_i}(\pol 1_{L_i,G})}>> \Phi_{T,G}(\IC(\Uh{d}))
\\
 @VV{e_i-e_{i+1}}V @AAA @AAA
\\
  \pi_! \cC_{\Gi{T}^n}@>\cong>\cL_{T,L_i}>
  \Phi_{T,L_i}(\pi_! \cC_{\Gi{L_i}^d})
  @>\cong>{\Phi_{T,L_i}(\cL_{L_i,G})}>
  \Phi_{T,G}(\pi_!\cC_{\Gi{r}^d}).
  \end{CD}
\end{equation}
The composite of lower horizontal arrows is $\cL_{T,G}$ by the
commutativity \eqref{eq:ass2}.
Recall we made an identification of $V^d$ by $\cL_{T,G}$ (see
\eqref{eq:Vd}). Therefore $e_i - e_{i+1}\in V^d$ considered as a
homomorphism in $\Hom(\cC_{S_{(d)}\AA^2},
\Phi_{T,G}(\pi_!\cC_{\Gi{r}^d}))$ is the composition of arrows from
the upper left corner to the lower right corner.

It is also clear that the homomorphism $V^d\to U^d$ given by the
composition of the rightmost upper arrow coincides with the projection
in \eqref{eq:VU}.

We thus see that $\{ \Tilde\alpha^d_i = \Phi_{T,L_i}(\pol
1^d_{L_i,G})\circ 1^d_{L_i}\}_i$ coincides with the base $\{ e_i -
e_{i+1}\}$ of $U^d$. This gives the construction promised in
\subsecref{sec:anotherbase} when $G$ is of type $A$.

\begin{Remark}\label{rem:SL(2)}
    Suppose $G=SL(2)$. Thanks to Example~\ref{ex:base}, we have
    $[\overline{\fT(0,d)_0}] = (-1)^{d} e_2^d$. (Here $r_1=r_2=1$.)
    Therefore we have
    \begin{equation*}
        \la [\overline{\fT(0,d)_0}], \Tilde\alpha^{d,-}_1\ra
        = (-1)^d \la e_2^d, e_1^{d,-} - e_2^{d,-}\ra
        = (-1)^{d+1}d
    \end{equation*}
    by \propref{prop:pairinge}. This completes the proof of
    \eqref{eq:111}.
\end{Remark}

\subsection{\texorpdfstring{$\Aut(G)$}{Aut(G)} invariance}
\label{sec:autG-invar}

Recall that we have studied $\Aut(G)$ invariance of various
constructions for $\Uh{d}$ in \subsecref{sec:autg-invariance}. The
same applies also to the Gieseker space $\Gi{r}^d$, if we restrict to
the inner automorphism $\operatorname{Inn}(G)$.
This is because $\operatorname{Inn}(G)$ acts on $\Gi{r}^d$, and hence
the same applies.

Let us consider $\Aut(G)/\operatorname{Inn}(G)$. It is $\{ \pm 1\}$
for type $A$, and is the Dynkin diagram automorphism given by the
reflection at the center. It is represented modulo inner automorphisms
by a group automorphism $g\mapsto {}^tg^{-1}$.
In terms of $\Bun{d}$, it corresponds to taking the {\it dual\/}
vector bundle. In particular, it does not extend to an action on the
Gieseker space $\Gi{r}^d$, as the second Chern class may drop when we
take the dual of a sheaf.

In the ADHM description, the diagram automorphism is given by
\begin{equation}
  [(B_1,B_2,I,J)]\mapsto [(B_1^t, B_2^t, -J^t, I^t)].
\end{equation}
This does not preserve the stability condition. Therefore we must be
careful when we study what happens under this automorphism.

Nevertheless we give

\begin{proof}[Proof of Lemma~{\rm\ref{lem:autg}}]
  Recall $\sigma\in\Aut(G)$ preserves $T$, $B$, and corresponds to a Dynkin
  diagram automorphism. Recall also $\tilde\alpha_i^d =
  \Phi_{T,L_i}(\delta 1^d_{L_i,G})\circ 1^d_{L_i}$.

  It is clear that $\Phi_{T,L_i}(1^d_{L_i,G})$ is sent to
  $\Phi_{T,L_{\sigma(i)}}(1^d_{L_{\sigma(i)},G})$ under $\varphi_\sigma$ from its
  definition.

  Next consider $1^d_{L_i} \in U^d_{T,L_i}$. In view of
  \lemref{lem:Ud}, $U^d_{T,L_i}$ is $H_{[0]}(\Uh[B\cap L_i,0]d)$,
  which is $1$-dimensional space spanned by the irreducible component
  $[\Uh[B\cap L_i,0]d]$. The class $[\Uh[B\cap L_i,0]d]$ is sent to
  $[\Uh[B\cap L_{\sigma(i)},0]d]$ under $\varphi_\sigma$, as it is induced from
  the isomorphism $\Uh[B\cap L_i,0]d\to \Uh[B\cap L_{\sigma(i)},0]d$.

  On the other hand, $[\Uh[B\cap L_i,0]d]$ is the image of
  $[\overline{\fT(0,d)_0}]$ under $\pi_*\colon H_{[0]}(\Gi{B\cap
    L_i,0}^d)\to H_{[0]}(\Uh[B\cap L_i,0]d)$. We have
  $[\overline{\fT(0,d)_0}] = (-1)^d e_2$ by Example~\ref{ex:base}.
  \begin{NB}
    Its projection is $\frac{(-1)^{d+1}}2 (e_1 - e_2)$.
  \end{NB}%
  Hence $[\Uh[B\cap L_i,0]d] = {(-1)^{d+1}} 1^d_{L_i}/2$. Combining
  with the above observation, we deduce the assertion.
\end{proof}

\begin{NB}
    Let us consider the case $d=1$. Recall that $\cUh[SL(r)]1$ is the
    space of nilpotent matrices of rank $\le 1$. Let
    \begin{equation}
\def\hsymbl#1{\smash{\hbox{\huge$#1$}}}
\def\hsymbu#1{\smash{\lower1.7ex\hbox{\huge$#1$}}}
        g_0 =
        (-1)^{r(r-1)/2}
        \begin{pmatrix}
            \hsymbu{0} & & & 1
\\
           & & \reflectbox{$\ddots$} &
\\
           & \reflectbox{$\ddots$} & &
\\
          1 & & & \hsymbl{0}
         \end{pmatrix}.
    \end{equation}
    Then the map
    \begin{equation}
        \label{eq:54}
        X \longmapsto - g_0 {}^t X g_0
    \end{equation}
    is an involution on $\cUh[SL(r)]1$. It is the composite of the
    morphism induced by $g\mapsto {}^t g^{-1}$ and the action of
    $g_0\in SL(r)$.

    From the description \eqref{eq:49}, we see that the
    $k^{\mathrm{th}}$ irreducible component $t_k$ is mapped to
    $t_{r-k+2}$, as $(k-1)$ and $(r-k-1)$ are swapped.
    So consider the involution $t_k\mapsto t_{r-k+2}$
    ($k=2,3,\dots,r$) in \eqref{eq:51}. It is given by
    \begin{equation}
        \label{eq:52}
         % \begin{bmatrix}
         %  e_1 - e_2 \mapsto (-1)^r(e_{r-1} - e_r)
         %  \\
         %  e_2 - e_3 \mapsto (-1)^r(e_{r-2} - e_{r-1})
         %  \\
         %  \vdots
         %  \\
         %  e_{r-1} - e_r \mapsto (-1)^r(e_1 - e_2)
         %  \end{bmatrix}.
        e_k - e_{k+1} \mapsto (-1)^r (e_{r-k} - e_{r-k+1})
        \qquad (k=1,\dots,r-1).
    \end{equation}
\end{NB}

\section{\texorpdfstring{$\scW$}{W}-algebra representation on
  localized equivariant cohomology}\label{sec:w-algebra-repr}

The goal of this chapter is to define a representation of the
$\scW$-algebra $\scW_k(\g)$\index{Wkg@$\scW_k(\g)$|textit} on the direct sum of equivariant intersection
cohomology groups $\IH^*_\TT(\Uh{d})$ over $d$, isomorphic to the Verma
module with the level and highest weight, given by the equivariant
variables by
\begin{equation}\label{eq:level}
  k + h^\vee = -\frac{\ve_2}{\ve_1}, \quad
  \lambda = \frac{\ba}{\ve_1} - \rho, \quad
  \text{where $\ba = (a^1,\dots,a^\ell)$}
    \index{a_vec@$\ba= (a^1,\dots,a^\ell)$}
\end{equation}
respectively. Here $\ba$ is a collection of variables, but will be
regarded also as a variable in the Cartan subalgebra $\h$ so that $a^i
= \alpha_i(\ba)$ for a simple root $\alpha_i$.
\begin{NB}
  So $\ba = \sum_i a^i \varpi_i^\vee$.
\end{NB}%

\begin{NB}
  I need to check the normalization of the highest weight. Sometimes,
  it is simply given by invariant polynomials in $a^i$. Sometimes one
  shift them by some functions in level.
\end{NB}

Since the level is a rational function in $\ve_1$, $\ve_2$, we
must be careful over which ring the representation is defined.
In geometric terms, it corresponds to that we need to consider {\it
  localized\/} equivariant cohomology groups.
The equivariant cohomology group $H^*_\TT(\ )$ is a module over
$H^*_\TT(\mathrm{pt}) = \CC[\operatorname{Lie}\TT] =
\CC[\ve_1,\ve_2,\ba]$. Let us denote this polynomial ring by $\bA_T$
and its quotient field by $\bF_T$.
\index{AAT@$\bA_T = \CC[\ve_1,\ve_2,\ba]$}
\index{FFT@$\bF_T = \CC(\ve_1,\ve_2,\ba)$}
In algebraic terms, it means that our $\scW$-algebra is defined over
$\CC(\ve_1,\ve_2)$. Then the level $k$ is a {\it generic point\/} in
$\mathbb A^1$. Moreover we consider a Verma module whose highest
weight is in $\h^*\otimes\bF_T$. This means that the highest weight is
also generic.
More precisely, we regard $\ba$ as a canonical element in
$\h^*\otimes\bF_T = \h^*\otimes
\operatorname{Frac}(S(\h^*)[\ve_1,\ve_2])$ given by the inner product
on $\h$. Here we have used the Langlands duality implicitly : we first
consider $\ba$ as the identity element in $\h\otimes\h^*\subset\h\otimes\bF_T$.
Then we regard the first $\h$ as the dual of the
Cartan subalgebra of the Langlands dual of $\g$. But the Langlands dual
is $\g$ itself as we are considering ADE cases.
\begin{NB}
    More detail was given in my message on Apr.3,~2014.
\end{NB}%

We will construct a representation on
\begin{equation}\label{eq:locIC}
  \bigoplus_d \IH^*_\TT(\Uh{d})\otimes_{\bA_T}\bF_T =
  \bigoplus_d H^*_\TT(\Uh{d}, \IC(\Uh{d}))\otimes_{\bA_T}\bF_T.
\end{equation}
By the localization theorem and \lemref{lem:centralfiber}, natural
homomorphisms
\begin{equation}\label{eq:4sp}
  \begin{split}
  & \IH^*_{\TT,c}(\Uh{d})%\otimes_{\bA_T}\bF_T
  \cong H^*_{\TT,c}(\Uh[T]{d}, i^!j^! (\IC(\Uh{d})))%\otimes_{\bA_T}\bF_T
\\
  & \to H^*_{\TT,c}(\Uh[T]{d}, \Phi_{T,G}(\IC(\Uh{d})))%\otimes_{\bA_T}\bF_T
\\
  & \to H^*_{\TT}(\Uh[T]{d}, \Phi_{T,G}(\IC(\Uh{d})))%\otimes_{\bA_T}\bF_T
\\
  & \to H^*_{\TT}(\Uh[T]{d}, i^*j^* (\IC(\Uh{d})))%\otimes_{\bA_T}\bF_T
  \cong \IH^*_\TT(\Uh{d})%\otimes_{\bA_T}\bF_T
  \end{split}
\end{equation}
all become isomorphisms over $\bF_T$. Thus over $\bF_T$, we could use
any of these four spaces. Let us denote its direct sum by $M_\bF(\ba)$:
\begin{equation}\label{eq:localhyper}
  M_\bF(\ba) = \bigoplus_d \IH^*_{\TT,c}(\Uh{d})\otimes_{\bA_T}\bF_T.
\end{equation}
\index{MF@$M_\bF(\ba)$}
\begin{NB}
  I denote it by $M$, as it will be the Verma module. On the other
  hand, other three spaces (nonequivariant ones) will be identified
  with Wakimoto modules, dual Wakimoto modules, and dual Verma module
  respectively.
\end{NB}

In fact, we will construct representations of {\it integral forms\/}
(i.e., $\bA_T$-forms) of Heisenberg and Virasoro algebras on {\it
  non-localized\/} equivariant cohomology groups
\(
   \bigoplus_d H^*_{\TT,c}(\Uh[T]{d},\Phi_{T,G}(\IC(\Uh{d})))
\)
of hyperbolic restrictions in this chapter. This construction will be
the first step towards a construction of the $\scW$-algebra
representation on {\it non-localized\/} equivariant cohomology groups.
To follow the remaining argument, the reader needs to read our
definition of an integral form of the $\scW$-algebra given in
\secref{sec:intW}. Therefore the whole construction will be postponed
to \subsecref{sec:univ-verma-modul}.

Let us denote the fundamental class $1\in \IH^0_\TT(\Uh{0}) =
\IH^0_{\TT,c}(\Uh{0}) = H^0_\TT(\mathrm{pt})$ by $|\ba\rangle$.
\index{a_vecr@$\vert\ba\rangle$}
It will be identified with the highest weight vector (or the vacuum
vector) of the Verma module. See \propref{prop:lambda} below.

We also use the following notation:
\begin{equation*}
  \bA = \CC[\ve_1,\ve_2], \quad \bF = \CC(\ve_1,\ve_2).
\end{equation*}
\index{AA@$\bA = \CC[\ve_1,\ve_2]$}\index{FF@$\bF = \CC(\ve_1,\ve_2)$}

\subsection{Freeness}

\begin{Lemma}\label{lem:free}
  Four modules appearing in \eqref{eq:4sp} are free over $\bA_T$.
\end{Lemma}

\begin{proof}
  By \lemref{lem:centralfiber} all four modules are pure, as
  $(\Uh{d})^{\TT}$ is a single point, and they are stalks at the
  point. Now freeness follows as in \cite[Th.~14.1(8)]{GKM}.

  Or we have odd cohomology vanishing by \cite[Th.~7.10]{BFG}. So
  it also follows from \cite[Th.~14.1(1)]{GKM}.
\end{proof}

\begin{NB}
   Since $\Uh{d}$ is not projective,
  it is not clear whether we can apply \cite[Th.~14.1(7)]{GKM}. We
  must explain this somewhere.

  \begin{NB2}
    If you know the purity of $\IH^*(\Uh{d})$, it is also enough.
  \end{NB2}
\end{NB}

In particular, homomorphisms in \eqref{eq:4sp} are all injective.

\subsection{Another base of \texorpdfstring{$U^d$}{Ud}, continued}\label{sec:anotherbase2}

Let $U^d = U^d_{T,G}$ be as in \subsecref{sec:Ud}. Let $L_i$ be the
Levi subgroup corresponding to a simple root
$\alpha_i$\index{alphai@$\alpha_i$} and consider $U^d_{T,L_i}$ as in
\subsecref{sec:anotherbase}.
We identify $\IC(\Uh[L_i]{d})$ with $\IC(\Uh[SL(2)]{d})$ by the
bijective morphism $\Uh[SL(2)]d\to \Uh[L_i]{d}$ (see
\propref{prop:LL}).
We have a maximal torus and a Borel subgroup induced from those of
$G$.
Then $U^d_{T,L_i}$ has the base in \eqref{eq:baseofU}, where it
consists of a single element as $N=2$. Let us denote the element by
$1^d_{L_i}$, as we promised in \subsecref{sec:anotherbase}.

Next consider $1^d_{L_i,G}$ given by the Thom isomorphism as in
\subsecref{sec:open}. We have the repellent polarization
$\pol_{\text{rep}}$ of $\Bun[L_i]{d}$ in $\Bun[G]{d}$. We modify it to
$\pol$\index{delta@$\pol$ (polarization)} according to
\lemref{lem:polSL3}.
We choose and fix a bipartite coloring of the vertices of the Dynkin
diagram, i.e., $o\colon I\to \{ \pm 1\}$ such that $o(i) = -o(j)$ if
$i$ and $j$ are connected in the diagram. Then we set
\begin{equation}\label{eq:polarization}
  \pol = o(i)^d \pol_{\text{rep}}.
\end{equation}
This is our polarization, which was promised in \eqref{eq:anotherbase}.
Let us write
\begin{equation}
  \tilde\alpha^d_i \defeq \Phi_{T,L_i}(\pol 1^d_{L_i,G})\circ 1^d_{L_i}.
\end{equation}
\index{alphatildai@$\Tilde\alpha^d_i$}%
This gives us a collection $\{\tilde\alpha^d_i\}_i$ of elements in
$U^d$ labeled by $I$.
Thanks to \eqref{eq:111}, it is a base of $U^d$.
This will follow also from \propref{prop:HeisRel}.
\begin{NB}
We will prove that it is a base of $U^d$ in \corref{cor:Ubase}.
\end{NB}

\subsection{Heisenberg algebra associated with the Cartan subalgebra}\label{sec:heis-algebra-assoc}

We construct a representation of the Heisenberg algebra associated
with the Cartan subalgebra $\mathfrak h$ of $\g$ on the direct sum of
\eqref{eq:locIC} in this subsection. It will be the first step towards
the $\scW$-algebra representation.

Let us first review the construction of the Heisenberg algebra
representation in \subsecref{sec:tensor} for the case $r=2$ and $L =
S(GL(1)\times GL(1)) = \CC^*$. We consider Heisenberg operators
$P_n^\Delta\equiv P_n^\Delta(1)$ associated with the cohomology class
$1\in H^{[*]}_\TT(\AA^2)$. We omit $(1)$ hereafter.
They are not well-defined on $\bigoplus_d H^\TT_{[*]}(\Gi{P}^d)$ if
$d > 0$, but are well-defined on the localized equivariant homology
group
\(
   \bigoplus_d H^\TT_{[*]}(\Gi{P}^d)\otimes_{\bA_T}\bF_T,
\)
and satisfy the commutation relations
\begin{equation}
  [P_m^\Delta, P_n^\Delta] = - 2m\delta_{m,-n}\frac1{\ve_1\ve_2}.
\end{equation}
\begin{NB}
  In fact, we must invert $\ve_1$, $\ve_2$, but not $\ba$. Therefore
  we can work on $\otimes_{\bA}\bF$. Does this make any difference ?
\end{NB}

Via the stable envelope, we have the isomorphism
\begin{equation}
  \bigoplus_d H^\TT_{[*]}(\Gi{P}^d) \cong \bigoplus_d H^\TT_{[*]}(\Gi{L}^d)
  = \bigoplus_{d_1,d_2}H^\TT_{[*]}(\Gi{1}^{d_1})\otimes
  H^\TT_{[*]}(\Gi{1}^{d_1}),
\end{equation}
and we have the representation of the tensor product of two copies of
Heisenberg algebras, given by $P_n^{(1)} = P_n\otimes 1$ and
$P_n^{(2)} = 1\otimes P_n$ on the localized equivariant
homology group, where $P_n$ is the Heisenberg generator for $r=1$.
The above Heisenberg generator $P_n^{\Delta}$ is the
diagonal $P_n^{(1)} + P_n^{(2)}$. See \subsecref{sec:tensor}.
\index{P@$P_n^{(i)}$}

We have
\begin{equation}
  H^\TT_{[-*]}(\Gi{P}^d)
  \cong H^{*}_{\TT}(\Uh[\CC^*]d, p_* j^! \pi_!\cC_{\Gi{2}^d})
  = H^{*}_{\TT}(\Uh[\CC^*]d, \Phi_{\CC^*,SL(2)}(\pi_!\cC_{\Gi{2}^d}))
\end{equation}
by \subsecref{sec:sheaf} and $\pi_! = \pi_*$. This homology group
contains
\begin{equation}
  H^{*}_\TT(\Uh[\CC^*]d, \Phi_{\CC^*,SL(2)} (\IC(\Uh[SL(2)]d)))
\end{equation}
as a direct summand, and the {\it anti-diagonal\/} Heisenberg algebra
generated by $P_n^{(1)} - P_n^{(2)}$ acts on its direct sum over
$d$. (See \subsecref{subsec:VU}.)

Let us return back to general $G$.
Let $L_i$ be the Levi subgroup as in the previous subsection. We
identify $\IC(\Uh[SL(2)]d)$ with $\IC(\Uh[L_i]d)$ as before, and we have
a(n anti-diagonal) Heisenberg algebra representation on
\begin{equation}
  \bigoplus_d H^*_\TT(\Uh[T]d, \Phi_{T,L_i}(\IC(\Uh[L_i]d)))
  \otimes_{\bA_T}\bF_T.
\end{equation}

Using the decomposition \eqref{eq:decomp} and $\Phi_{T,L_i}\Phi_{L_i,G}
= \Phi_{T,G}$, we have an induced Heisenberg algebra representation on
$M_\bF(\ba)$ in \eqref{eq:localhyper}.
Let us denote the Heisenberg generator by $P^i_n$.
\index{P@$P_n^i$ (Heisenberg generator)}

By \lemref{lem:Fock}, the space $M_\bF(\ba)$ is isomorphic to
\begin{equation}\label{eq:symprod}
  % \operatorname{Sym}(U^1\otimes_\CC\bF_T))\otimes
  % \operatorname{Sym}(U^2\otimes_\CC\bF_T)
  % \otimes\cdots =
  \Sym((U^1\oplus U^2\oplus\cdots)\otimes_\CC\bF_T),
\end{equation}
where $\Sym$ denotes the symmetric power.  ($U^d = U^d_{T,G}$ as
before.)

Let us describe $P^i_n$ in this space. Recall that we have the
orthogonal decomposition $U^d = U^d_{T,L_i}\oplus
(U^d_{T,L_i})^\perp$ in \eqref{eq:decomp2}.
Then we have the factorization
\begin{multline}\label{eq:53}
  \operatorname{Sym}((U^1\oplus U^2\oplus\cdots)\otimes_\CC\bF_T)
\\
  \cong \operatorname{Sym}((U^1_{T,L_i}\oplus U^2_{T,L_i}\oplus\cdots)
  \otimes_\CC\bF_T)  \otimes_{\bF_T}
\\
  \operatorname{Sym}(((U^1_{T,L_i})^\perp\oplus (U^2_{T,L_i})^\perp\oplus\cdots)
  \otimes_\CC\bF_T)
\end{multline}
The first factor of the right hand side is the usual Fock space
associated with the Cartan subalgebra $\h_{\algsl_2}$ of
$\algsl_2$. In fact, using $U^d_{T,L_i} \cong \CC 1^d_{L_i}$, we
identify $U^d_{T,L_i}$ with $\h_{\algsl_2}$.  The pairing is
multiplied by $-1/\ve_1\ve_2$ from the natural one. Then the factor is
$\operatorname{Sym}(z^{-1}\h_{\algsl_2}[z^{-1}])$ and the Heisenberg
algebra acts in the standard way.
From its definition, our Heisenberg operator $P^i_n$ is given by the
tensor product of the Heisenberg operator for
$\operatorname{Sym}(z^{-1}\h_{\algsl_2}[z^{-1}])$, and the identity.

The following means that the operators $P^i_n$ define the Heisenberg
algebra $\Heis(\mathfrak h)$\index{Heis@$\Heis(\mathfrak h)$}
associated with the Cartan subalgebra $\mathfrak h$ of $\g$.

\begin{Proposition}\label{prop:HeisRel}
  Heisenberg generators satisfy commutation relations
  \begin{equation}
    [P^i_m, P^j_n] = - m\delta_{m,-n}(\alpha_i,\alpha_j)\frac1{\ve_1\ve_2}.
  \end{equation}
\end{Proposition}

If we normalize the generator by $\widehat h^i_n = \ve_2 P^i_n$, the
relations match with a standard convention with level $-\ve_2/\ve_1 =
k+h^\vee$. See \eqref{eq:hhatcomm}.

From the construction, $P^i_{-d}$ applied to the vacuum vector
$|\ba\ra\in H^0_{\TT}(\Uh[T]{0}, \Phi_{T,G}(\IC(\Uh{0})))$ is equal to
$\Phi_{T,L_i}(\pol 1^d_{L_i,G})\circ 1^d_{L_i}\in U^d$ divided by
$\ve_1\ve_2$, considered as an element in \eqref{eq:symprod}.

\begin{NB}
    The following is commented out as we already knew
    $\{ \Tilde\alpha^d_i\}$ is a base.

Hence we deduce the following, as promised in
\subsecref{sec:anotherbase} and \subsecref{sec:anotherbase2}.

\begin{Corollary}\label{cor:Ubase}
\(
  \{ \tilde\alpha^d_i = \Phi_{T,L_i}(\pol 1^d_{L_i,G})\circ 1^d_{L_i}
  \}_i
\)
is a base of $U^d$.
\end{Corollary}
\end{NB}

From the construction, \eqref{eq:symprod} is the Fock space of the
Heisenberg algebra associated with the Cartan subalgebra $\mathfrak
h$. It is $\operatorname{Sym}(z^{-1}\h[z^{-1}])$, where the base field
is $\bF_T$. The element $\tilde\alpha^d_i
%= \Phi_{T,L_i}(\pol 1^d_{L_i,G})\circ 1^d_{L_i}
$ is a linear function, living on $z^{-d}\mathfrak h$.

\begin{proof}[Proof of Proposition~{\rm\ref{prop:HeisRel}}]
  The case $(\alpha_i,\alpha_j) = 2$, i.e., $i=j$ is obvious from the
  construction.

  Next consider the case $(\alpha_i,\alpha_j) = -1$. Then $i$ and $j$
  are connected by an edge in the Dynkin diagram. Let us take the
  parabolic subgroup $P$ corresponding to the subset consisting of two
  vertices $i$ and $j$, and the corresponding Levi subgroup $L$. We
  have $[L,L] \cong SL(3)$. Then from our construction and the
  compatibility of the stable envelope with the hyperbolic restriction
  functor in \subsecref{subsec:comp}, the assertion follows from the
  $SL(3)$-case, which is clear as Heisenberg algebra generators are
  given by
  \begin{equation}
    P^i_n = P_n\otimes 1\otimes 1 - 1\otimes P_n\otimes 1,
\quad
    P^j_n = 1\otimes P_n\otimes 1 - 1\otimes 1\otimes P_n.
  \end{equation}
  Note also that our polarization $\pol$ in \eqref{eq:polarization} was
  chosen so that it is the same as the polarization for $\Gi{SL(3)}^d$
  via \lemref{lem:polSL3} up to overall sign independent of $d$.
  \begin{NB}
    Here the only ambiguity comes from the choice of a bijection $\{
    i,j\} \leftrightarrow \{ 1,2\}$.
  \end{NB}%

  Finally consider the case $(\alpha_i,\alpha_j) = 0$. We argue as
  above by taking the corresponding Levi subgroup $L$ with $[L,L]\cong
  SL(2)\times SL(2)$. Then it is clear that Heisenberg generators
  commute.

  If a reader would wonder that $SL(2)\times SL(2)$ is not considered
  in \secref{sec:typeA}, we instead take a type $A_k$ subdiagram
  containing $i$, $j$ and take the corresponding Levi subgroup $L$
  with $[L,L]\cong SL(k+1)$. Then it is clear that the Heisenberg
  generators $P^i_m$, $P^j_n$ commute for $SL(k+1)$. Therefore they
  commute also for $G$.
\end{proof}

Let us consider Heisenberg operators
\(
  P^i_n([0]) = \ve_1\ve_2 P^i_n,
\)
coupled with the Poincar\'e dual of $[0]\in H^\TT_0(\AA^2)$, and
denote them by $\widetilde{P}^i_n$.
\index{P@$\widetilde P_n^i$}
Then they are well-defined on non-localized equivariant cohomology
groups
\begin{equation}
  \bigoplus_d H^*_\TT(\Uh[T]d, \Phi_{T,G}(\IC(\Uh{d}))),
\end{equation}
and satisfy the commutation relations
\begin{equation}\label{eq:mPrel}
  [\widetilde{P}^i_m, \widetilde{P}^j_n]
  = - m\delta_{m,-n}(\alpha_i,\alpha_j){\ve_1\ve_2}.
\end{equation}
The same is true for the non-localized equivariant cohomology with
compact supports.

We define the $\bA$-form $\Heis_\bA(\mathfrak h)$
\index{HeisA@$\Heis_\bA(\mathfrak h)$} of the Heisenberg vertex
algebra as the vertex $\bA$-subalgebra of $\Heis(\mathfrak h)$
generated by $\widetilde{P}^i_m$.

\subsection{Virasoro algebra}\label{sec:Vir}

Let us introduce $0$-mode operators $P^i_0$. In \subsecref{sec:BaraOp}
we did not introduce them. Since they commute with all other
operators, we can set them any scalars. We follow the convention in
\cite[\S13.1.5, \S14.3.1]{MO}, that is
\begin{equation}
  P^i_0 = \frac{a^i}{\ve_1\ve_2}.
\end{equation}
Here $a^i$ is the $i^{\mathrm{th}}$ simple root, and should be
identified with $a_i - a_{i+1}$ in \cite{MO} in the Fock space
$\mathbf F(a_1)\otimes \cdots\otimes \mathbf F(a_r)$ corresponding to
the equivariant cohomology of Gieseker spaces for rank $r$ sheaves.
\begin{NB}
Note also that the convention in \cite[\S13.1.5]{MO} is
  \begin{equation}
    \alpha_0(\gamma)|\eta\rangle
    = - (\gamma,\eta)|\eta\rangle,
  \end{equation}
and $-(1,\eta) = \int_{\AA^2}\eta = \eta/\ve_1\ve_2$.
\end{NB}
We also set $\widetilde{P}_0^i = \ve_1\ve_2 P_0^i = a^i$.

We then introduce Virasoro generators by
\begin{equation}\label{eq:Vira}
  {L}^i_n = - \frac14 \ve_1\ve_2 \sum_m \normal{P^i_m P^i_{n-m}}
  - \frac{n}2 (\ve_1+\ve_2) P^i_n
  + \frac{(\ve_1+\ve_2)^2}{4\ve_1\ve_2}\delta_{n,0}.
\end{equation}
See \cite[(13.10),(14.10)]{MO}.\index{Lin@$L^i_n$ (Virasoro generator)}
Let us briefly explain how to derive the above expression from \cite{MO}:
The Virasoro field $T(\gamma,\kappa) = \sum L_n(\gamma,\kappa) z^{-n}$
in \cite[(13.10)]{MO} is given by
\begin{equation}
    T(\gamma,\kappa) = \frac12 \normal{\boldsymbol\alpha^2}(\gamma)
    + \partial\boldsymbol\alpha(\gamma\kappa) - \frac12 \tau(\gamma\kappa^2),
\end{equation}
where $\boldsymbol\alpha(\gamma) = \sum \alpha_n(\gamma) z^{-n}$ is
the free field. Note that $T$ and $\boldsymbol\alpha$ are different
from the usual convention, as the exponents are not $-n-1$, $-n-2$
respectively. Also $\partial = z\partial_z$.

We take $\gamma = 1$, the fundamental class of $H^0_\TT(\AA^2)$.
Next note that $\boldsymbol\alpha = {\boldsymbol\alpha^-}/{\sqrt{2}}$
\cite[(14.8)]{MO}, and our $P^i$ is identified with
$\boldsymbol\alpha^-$. This is the reason we have $1/4$ instead of
$1/2$.
  The remaining factor $-\ve_1\ve_2$ comes from
\(
    1^\Delta = - 1\otimes\mathrm{pt}
\)
in \cite[\S13.3.2]{MO}.

For the second term, note $\kappa = \hbar/\sqrt{2}$ (see
\cite[(14.8)]{MO}), $\hbar = -t_1-t_2$ (see
\cite[\S17.1.1,(18.10)]{MO} for example). We denote their $t_1$, $t_2$
by $\ve_1$, $\ve_2$ instead.

For the last constant term, we have
\(
  -\gamma\kappa^2 = - (\ve_1+\ve_2)^2/2
\)
and
\(
   \tau(1) = -\int_{\AA^2} 1 = - 1/\ve_1\ve_2.
\)

\begin{NB}
  Lehn's formula \cite[3.10]{Lehn} says
  \begin{equation}
    \frac1n [c_1(\shfO^{[n]}_X), q_n(\gamma)]
    = L_n(\gamma) + \frac{|n|-1}2 q_n(K\gamma).
  \end{equation}
  I should compare it with \cite[(14.1), \S15.1]{MO}.
\end{NB}

The Virasoro algebra commutation relations are
\begin{equation}
  [L^i_m, L^i_n]
    = (m-n) L^i_{m+n} +
    \left(1 + \frac{6(\ve_1+\ve_2)^2}{\ve_1\ve_2}\right)
    \delta_{m,-n}\frac{m^3 - m}{12}.
\end{equation}
See \cite[\S13.3.2]{MO}. And the highest weight is given by
\begin{equation}\label{eq:92}
  L_0^i |\ba\rangle = -\frac14 \left( \frac{(a^i)^2}{\ve_1\ve_2} -
    \frac{(\ve_1+\ve_2)^2}{\ve_1\ve_2}
    \right)|\ba\rangle.
\end{equation}
See \cite[\S13.3.5]{MO}.
\begin{NB}
  Here the highest weight is shifted by a function in level. I must be
  careful to check that this is a correct formulation.

  In \cite[(13.15)]{MO}, the highest weight is
  \begin{equation}
      \frac12\tau(\eta^2 - \kappa^2)
      = -\frac1{4\ve_1\ve_2}\left( (a_1 - a_2)^2 - (\ve_1+\ve_2)^2\right).
  \end{equation}

  \cite[(285)]{Arakawa2007} says
  \begin{equation}
      \label{eq:60}
      L(0) |\gamma_{\overline{\lambda}}\rangle
      = \Delta_{\overline{\lambda}} |\gamma_{\overline{\lambda}}\rangle
  \end{equation}
with
\begin{equation}
    \label{eq:61}
    \Delta_{\overline{\lambda}} =
    \frac{|\overline{\lambda}+\overline{\rho}|^2}{2(k+h^\vee)}
    - \frac{\rank\overline\g}{24} + \frac{c(k)}{24},
\end{equation}
where
\begin{equation}
      \label{eq:59}
      \begin{split}
      c(k) &= \rank\overline\g - 12\left(
        (k+h^\vee)|\overline\rho^\vee|^2
        - 2\langle\overline\rho,\overline\rho^\vee\rangle
        + \frac1{k+h^\vee}|\overline\rho|^2\right),
\\
     &= \rank\overline\g - 12|\overline\rho^\vee|^2\left(
        (k+h^\vee)
        - 2
        + \frac1{k+h^\vee}\right)\quad
     \text{for $ADE$}.
      \end{split}
\end{equation}
Hence
\begin{equation}
    \label{eq:62}
    \Delta_{\overline\lambda}
    = \frac{|\overline{\lambda}+\overline{\rho}|^2}{2(k+h^\vee)}
    - \frac{|\overline\rho^\vee|^2}2
    \left(
        (k+h^\vee)
        - 2
        + \frac1{k+h^\vee}\right).
\end{equation}

Let us also note that highest weight modules are invariant under
the shifted action \cite[(57)]{Arakawa2007}:
\begin{equation}
    \begin{split}
    & \gamma_{\overline{\lambda}} = \gamma_{\overline{w}\circ\overline{\lambda}},
\\
    & \overline{w}\circ\overline\lambda
    = \overline{w}(\overline{\lambda}+\overline\rho)
    - \overline\rho.
    \end{split}
\end{equation}
In particular, $\Delta_{\overline\lambda} =
\Delta_{\overline{w}\circ\overline\lambda}$ as
\begin{equation}
    |\overline{w}\circ\overline\lambda + \overline\rho|^2
    = |\overline{w}(\overline\lambda+\overline\rho)|^2
    = |\overline\lambda+\overline\rho|^2.
\end{equation}
\end{NB}

In order to apply the result of Feigin-Frenkel to our situation later,
we shift $P^i_n$ in \eqref{eq:Vira} as $P^i_n -
(\ve_1+\ve_2)/\ve_1\ve_2 \delta_{n,0}$ (see \cite[\S19.2.5]{MO}) so
that
\begin{equation}
    \label{eq:57}
      {L}^i_n = - \frac14 \ve_1\ve_2 \sum_m \normal{P^i_m P^i_{n-m}}
  - \frac{n+1}2 (\ve_1+\ve_2) P^i_n.
\end{equation}
This is a standard embedding of the Virasoro algebra in the Heisenberg
algebra, given as the kernel of the screening operator (see
\cite[\S15.4.14]{F-BZ}). We have
\begin{equation}
    \label{eq:58}
  P^i_0 = \frac1{\ve_1\ve_2}\left( a^i - (\ve_1+\ve_2)\right)
\end{equation}
in this convention.

We modify \eqref{eq:Vira} as
\begin{equation}
    \begin{split}
  {L}^i_n &=
  \begin{aligned}[t]
    & - \frac14 \ve_1\ve_2 \sum_m
  \normal{(P^i_m - \frac{\ve_1+\ve_2}{\ve_1\ve_2} \delta_{m,0})
    (P^i_{n-m}- \frac{\ve_1+\ve_2}{\ve_1\ve_2} \delta_{m,n})
  }
  \\
  & \qquad - \frac12 (\ve_1+\ve_2) P^i_n
  + \frac{(\ve_1+\ve_2)^2}{4\ve_1\ve_2} \delta_{n,0}
  \\
  & \quad - \frac{n}2 (\ve_1+\ve_2) P^i_n
  + \frac{(\ve_1+\ve_2)^2}{4\ve_1\ve_2}\delta_{n,0}
  \end{aligned}
\\
  &=
  \begin{aligned}[t]
    & - \frac14 \ve_1\ve_2 \sum_m
  \normal{(P^i_m - \frac{\ve_1+\ve_2}{\ve_1\ve_2} \delta_{m,0})
    (P^i_{n-m}- \frac{\ve_1+\ve_2}{\ve_1\ve_2} \delta_{m,n})
  }
  \\
  & \quad - \frac{n+1}2 (\ve_1+\ve_2)
  (P^i_n - \frac{\ve_1+\ve_2}{\ve_1\ve_2} \delta_{n,0}).
  \end{aligned}
    \end{split}
\end{equation}
Therefore if we replace $P^i_n$ by $P^i_n - (\ve_1+\ve_2)/\ve_1\ve_2
\delta_{n,0}$, we get the above expression.

\begin{NB}
  We have
  \begin{equation}
    \ve_2 P^i_0 = \frac{a^i}{\ve_1} - 1 - \frac{\ve_2}{\ve_1}
    = (\alpha_i,\frac{\ba}{\ve_1} - \bar\rho) + k + h^\vee.
  \end{equation}
  Therefore we should make an identification of the highest weight
  vector as
  \begin{equation}
    \bar\lambda = \frac{\ba}{\ve_1} - \bar\rho
  \end{equation}
  in \cite[\S5.1]{Arakawa2007}.
\end{NB}

\begin{NB}
  Let me check the central charge formula again. We have
  (\cite[\S2.5.9]{F-BZ})
  \begin{equation}
    \begin{split}
    & Y(\omega_\lambda,z) = \frac12 \normal{b(z)^2} + \lambda\partial_z b(z)
\\
  =\; & \sum_n \left(\frac12 \sum_{m} \normal{b_m b_{n-m}}
    - \lambda(n+1) b_n\right) z^{-n-2}.
    \end{split}
  \end{equation}
  From the Heisenberg commutator relation we should make an
  identification
  \begin{equation}
     b_n = \frac{\sqrt{-\ve_1\ve_2}P_n}{\sqrt{2}}.
  \end{equation}
  Therefore the above is
  \begin{equation}
    \sum_n\left(-\frac{\ve_1\ve_2}4 \sum_{m} \normal{P_m P_{n-m}}
    - \lambda(n+1) \frac{\sqrt{-\ve_1\ve_2}P_n}{\sqrt{2}} \right) z^{-n-2}.
  \end{equation}
Comparing with \eqref{eq:57}, we find
\begin{equation}
  \lambda = \frac{\ve_1+\ve_2}{\sqrt{-2\ve_1\ve_2}}.
\end{equation}
Therefore the central charge is
\begin{equation}
  c_\lambda = 1 - 12 \lambda^2 = 1 + \frac{6(\ve_1+\ve_2)^2}{\ve_1\ve_2}.
\end{equation}
\end{NB}

We denote by $\Vir_i$ the Virasoro vertex subalgebra of $\Heis(\h)$
generated by $L^i_n$. \index{Vir@$\Vir_i$}

Let us introduce a modified Virasoro generator $\mL^i_n = \ve_1\ve_2
L^i_n$. \index{Linm@$\mL^i_n$}
We have
\begin{equation}\label{eq:mVir}
    \mL^i_n = - \frac14 \sum_m \normal{\widetilde{P}^i_m \widetilde{P}^i_{n-m}}
  % - \frac{n}2 (\ve_1+\ve_2) \widetilde{P}^i_n
  % + \frac{(\ve_1+\ve_2)^2}4\delta_{n,0}.
  - \frac{n+1}2 (\ve_1+\ve_2) \widetilde{P}^i_n.
\end{equation}
Hence $\mL^i_n$ is an element in $\Heis_\bA(\mathfrak h)$. We denote
the corresponding vertex $\bA$-subalgebra by $\Vir_{i,\bA}$.\index{ViriA@$\Vir_{i,\bA}$}

Note that the central charge $1 + {6(\ve_1+\ve_2)^2}/{\ve_1\ve_2}$ is
equal to that of Virasoro algebras, appearing in the construction of the
$\scW$-algebra $\scW_k(\g)$\index{Wkg@$\scW_k(\g)$|textit} as the BRST
reduction of the affine vertex algebra at level $k$, if we have the
relation
\begin{equation}
    -\frac{(\ve_1+\ve_2)^2}{\ve_1\ve_2}
  = k + h^\vee - 2 + \frac1{k+h^\vee},
\end{equation}
see \cite[\S15.4.14]{F-BZ} and \corref{cor:duality} below.
In other words,
\(
  k + h^\vee = -{\ve_2}/{\ve_1}
\)
or
\(
   -{\ve_1}/{\ve_2}.
\)
It is known that the $\scW$-algebra for type $ADE$ has a symmetry under
$k+h^\vee\leftrightarrow (k+h^\vee)^{-1}$
\cite[Prop.~15.4.16]{F-BZ}. Therefore either choice gives the same
result. We here take
\(
  k + h^\vee = -{\ve_2}/{\ve_1},
\)
see \eqref{eq:level}.
It is remarkable that the symmetry $k+h^\vee\leftrightarrow
(k+h^\vee)^{-1}$ corresponds to a trivial symmetry
$\ve_1\leftrightarrow\ve_2$ in geometry.
\begin{NB}
  Let us give more detail. The central charge is given by $c(\beta) =
  1 - 6(\beta/\sqrt{2}-\sqrt{2}/\beta)^2$ for the vertex operator
  $V_\beta(z)$. This follows from the computation in
  \cite[\S15.4.14]{F-BZ}, where we substitute $\beta =
  -\sqrt{2/(k+2)}$ to $c(k) = 1- 6(k+1)^2/(k+2)$. Or the formula
  $c(\lambda) = 1 - 12\lambda^2$ in \cite[(5.3.10)]{F-BZ} and $\lambda
  = \beta/2 - 1/\beta$ in \cite[\S15.4.14]{F-BZ}.
  Since the $i^{\mathrm{th}}$ screening operator for the $\scW$-algebra
  is given with $\beta = -\alpha_i/\sqrt{k+h^\vee}$, we get
  \begin{equation}
    c(\beta) = 1 - 6\left(\frac{2(k+h^\vee)}{(\alpha_i,\alpha_i)}
    - 2 + \frac{(\alpha_i,\alpha_i)}{2(k+h^\vee)}\right).
  \end{equation}
  This is the central charge for the $i^{\mathrm{th}}$ Virasoro
  algebra, given as the kernel of the $i^{\mathrm{th}}$ screening
  operator. Since we assume $\g$ is of type $ADE$, we get the above
  relation.
\end{NB}%

\subsection{The first Chern class of the tautological bundle}

Let us explain a geometric meaning of the Virasoro generators in the
previous subsection. It was obtained in \cite[Th.~14.2.3]{MO},
based on an earlier work by Lehn \cite{Lehn} for the rank $1$ case.
Let us first consider the rank $2$ case.

Consider the Gieseker space $\Gi{2}^d$ of rank $2$ framed sheaves on
$\proj^2$ with $c_2 = d$. For $(E,\varphi)\in \Gi{2}^d$, consider
$H^1(\proj^2, E(-\linf))$. Other cohomology groups vanish, and hence
it has dimension equal to $d$ by the Riemann-Roch formula. In the ADHM
description, it is identified with the vector space $V$. When we vary
$E$, it forms a vector bundle over $\Gi{2}^d$, which we denote by
$\mathcal V$. Its first Chern class $c_1(\mathcal
V)$\index{c1V@$c_1(\mathcal V)$} can be considered as an operator on
$H^*_\TT(\Gi{2}^d)$ acting by the cup product. Then its commutator
with the diagonal Heisenberg generator, restricted to
$\IH^*_{\TT}(\Uh[SL_2]{d})$, is the Virasoro generator up to constant:
\begin{equation}\label{eq:32}
  \left.[c_1(\mathcal V), P_n^\Delta]\right|_{\IH^*_{\TT}(\Uh[SL_2]{d})}
  = n L_n,
\end{equation}
where we denote $L_n^i$ in the previous subsection by $L_n$ since
$G=SL_2$. (See \cite[Th.~14.2.3]{MO}.)
\begin{NB}
    Corrected from $2nL_n$. Feb.~18, See Misha's messages on Feb.~17.
\end{NB}

\begin{NB}
  The following was not correct. If I do not restrict to
  $\IH^*_{\TT}(\Uh[SL_2]{d})$, there are additional terms written by
$P_n^{\Delta}$. See Note 2013-12-03:

Then its commutator with the diagonal Heisenberg generator is the
Virasoro generator up to constant:
\begin{equation}%\label{eq:32}
  [c_1(\mathcal V), P_n^\Delta] = 2n L_n.
\end{equation}
See \cite[Th.~14.2.3]{MO}.
\begin{NB2}
  I have calculated the commutator of \cite[(14.6)]{MO} with
  $P_n^\Delta$. It is better to check it again.
\end{NB2}
\end{NB}

Let us remark that $c_1(\mathcal V)$ is defined on {\it
  non-localized\/} equivariant cohomology groups
$\IH^*_{\TT,c}(\Gi{2}^d)$. Therefore $\mL_n = \ve_1\ve_2 L_n$
is also well-defined on non-localized equivariant cohomology groups
for $n\neq 0$.
The operator $\mL_0 = \ve_1\ve_2 L_0$ is also well-defined as it is
the grading operator (\cite[Lem.~13.1.1]{MO}).

Returning back to general $G$, we see that $\mL^i_n$ is well-defined on
\begin{equation}\label{eq:71}
  \bigoplus_d H^*_{\TT,c}(\Uh[L_i]d, \Phi_{L_i,G}(\IC(\Uh{d})))
\end{equation}
thanks to the decomposition \eqref{eq:decomp}. Namely this space is a
module over $\Vir_{i,\bA}$.
It lies in between the first two spaces in \eqref{eq:4sp}:
\begin{equation}
  \IH^*_{\TT,c}(\Uh{d}) \to
  H^*_{\TT,c}(\Uh[L_i]d, \Phi_{L_i,G}(\IC(\Uh{d})))
  \to
  H^*_{\TT,c}(\Uh[T]d, \Phi_{T,G}(\IC(\Uh{d}))).
\end{equation}
The formula \eqref{eq:mVir} relates operators $\mL^i_n$ and
$\widetilde{P}^i_n$ acting on the middle and right spaces respectively
via the second homomorphism.

\subsection{\texorpdfstring{$\scW$}{W}-algebra representation}

Let us consider the vertex algebra associated with the Heisenberg
algebra, and denote it by the same notation $\Heis(\mathfrak h)$ for
brevity.
It is regarded as a vertex algebra over $\bF$.

We have the Virasoro vertex subalgebra $\Vir_i$ corresponding to each
simple root $\alpha_i$ as in \subsecref{sec:Vir}.
\begin{NB}
  I must consider whether the zero mode must be corrected so that we
  have correct highest weight or not.

  Zero mode is corrected, and the notation $\Vir_i$ is already
  introduced above.
\end{NB}%
Consider the orthogonal complement $\alpha_i^\perp$ of $\CC\alpha_i$
in $\mathfrak h$, and the corresponding Heisenberg vertex algebra
$\Heis(\alpha_i^\perp)$. It commutes with $\Vir_i$, and the tensor
product $\Vir_i\otimes \Heis(\alpha_i^\perp)$ is a vertex subalgebra
of $\Heis(\mathfrak h)$.

By a result of Feigin-Frenkel (see \cite[Th.~15.4.12]{F-BZ}), the
$\scW$-algebra $\scW_k(\g)$\index{Wkg@$\scW_k(\g)$} is identified with the intersection
\begin{equation}\label{eq:W=Vir}
  \bigcap_i \Vir_i\otimes \Heis(\alpha_i^\perp)
\end{equation}
in $\Heis(\mathfrak h)$ when the level $k$ is generic. More precisely,
$\Vir_i\otimes \Heis(\alpha_i^\perp)$ is given by the kernel of a
screening operator on $\Heis(\mathfrak h)$, and $\scW_k(\g)$ is the
intersection of the kernel of screening operators.

Now $\scW_k(\g)$ has a representation on the direct sum of localized
equivariant cohomology groups $M_\bF(\ba)$ (see \eqref{eq:localhyper}),
as a vertex subalgebra of $\Heis(\mathfrak h)$.

\begin{NB}
  At this stage, we cannot state much on the $\scW_k(\g)$-module
  structure on $M_\bF(\ba)$. It is known that the Fock space
  $M_\bF(\ba)$ is irreducible as a $\scW_k(\g)$-module, when the central
  charge and the highest weight is generic.

  I would like to say that Fourier modes of vertex operators in
  $\scW_k(\g)$ are in something like the direct product
  \begin{equation}
    \prod_{d_1,d_2} \Ext(\IC(\Uh{d_1}), \IC(\Uh{d_2})),
  \end{equation}
  probably localized one at this stage. I do not consider seriously,
  but probably we embed $\Uh{d_1}$ and $\Uh{d_2}$ into a larger space
  $\Uh{d}$ adding singularities concentrated at $0$, then consider
  $\Ext_{D^b(\Uh{d})}(\IC(\Uh{d_1}), \IC(\Uh{d_2}))$. This space should
  be independent of $d$.

  If we want to go further, we need to have a characterization of this
  kind of Ext algebra in the Ext algebra for the Heisenberg algebra. I
  do not know how to do at this moment. Therefore I only have a very
  weak statement at this point.
\end{NB}

\subsection{Highest weight}\label{sec:highest-weight}

In this subsection we explain that we can identify $\ba$ with the
highest weight of the $\scW_k(\g)$-module $M_\bF(\ba)$, where the highest
weight vector is $|\ba\rangle$.

Let us first briefly review the definition of Verma modules of the
$\scW$-algebra to set up the notation. See \cite[\S5]{Arakawa2007} for
detail.

Let $\fU(\scW_k(\g))$ be the current algebra of the $\scW$-algebra as in
\cite[\S4]{Arakawa2007}. (The finite dimensional Lie algebra is
denoted by $\bar\g$, while $\g$ is the corresponding untwisted affine
Lie algebra in \cite{Arakawa2007}.) We denote the current algebra of
the Heisenberg algebra by $\fU(\Heis(\h))$. It is a completion of the
universal enveloping algebra of the Heisenberg Lie algebra. The
embedding $\scW_k(\g)\subset\Heis(\h)$ induces an embedding
$\fU(\scW_k(\g))\to\fU(\Heis(\h))$.
\index{UWkg@$\fU(\scW_k(\g))$}
\index{UHh@$\fU(\Heis(\h))$}

We have decompositions
\(
   \fU(\scW_k(\g)) = \bigoplus_d \fU(\scW_k(\g))_d,
\)
\(
   \fU(\Heis(\h)) = \bigoplus_d \fU(\Heis(\h))_d
\)
by degree. Two decompositions are compatible under the embedding.
Let
\begin{equation}
   \fU(\scW_k(\g))_{\ge 0} \defeq \bigoplus_{d \ge 0} \fU(\scW_k(\g))_d,
   \quad
   \fU(\scW_k(\g))_{>0} \defeq \bigoplus_{d > 0} \fU(\scW_k(\g))_d.
\end{equation}
The {\it Zhu algebra\/} of $\fU(\scW_k(\g))$ is given by
\begin{equation}
  \mathfrak{Zh}(\scW_k(\g)) \defeq \fU(\scW_k(\g))_0/
  \overline{\sum_{r>0} \fU(\scW_k(\g))_{-r}\fU(\scW_k(\g))_r}.
\end{equation}
Then it is isomorphic to the center $Z(\g)$ of the universal
enveloping algebra $U(\g)$ of $\g$ (\cite[Th.~4.16.3]{Arakawa2007}).
\index{ZhuWkg@$\mathfrak{Zh}(\scW_k(\g))$}
We further identify it with the Weyl group invariant part of the
symmetric algebra of $\h$ (\cite[(55)]{Arakawa2007}):
\begin{equation}\label{eq:76}
  \mathfrak{Zh}(\scW_k(\g))\cong Z(\g) \cong S(\h)^W.
\end{equation}

We have an induced embedding $\mathfrak{Zh}(\scW_k(\g))\to
\mathfrak{Zh}(\fU(\Heis(\h)))$, where the latter is the subalgebra
generated by zero modes. We have
\begin{equation}
  \label{eq:75}
  \mathfrak{Zh}(\fU(\Heis(\h)))\cong S(\h).
\end{equation}

\begin{Lemma}\label{lem:shift}
  Under the identifications \textup{(\ref{eq:76}, \ref{eq:75})}, the
  embedding $\mathfrak{Zh}(\scW_k(\g))\to \mathfrak{Zh}(\fU(\Heis(\h)))$
  is induced by
  \begin{equation}
    h^i \mapsto h^i + (k+h^\vee),
  \end{equation}
  where $h^i$ is a simple coroot of $\h$.
\end{Lemma}

\begin{proof}
  The assertion follows from \cite[Th.~4.16.4]{Arakawa2007}, together
  with an isomorphism $\widehat t_{-\bar\rho^\vee}$ which sends the
  old Zhu algebra, denoted by
  $H^0(\mathfrak{Zh}(C_k(\bar\g)''_{\mathrm{old}}))$ there, to a new
  one $H^0(\mathfrak{Zh}(C_k(\bar\g)''_{\mathrm{new}}))$. The zero
  mode is written as $\widehat{J_i}(0)$ there.
  We can calculate $\widehat t_{-\bar\rho^\vee}(\widehat{J_i}(0)) =
  \widehat{J_i}(0) + k+h^\vee$ by formulas in \cite[bottom of
  p.276]{Arakawa2007}.
  \begin{NB}
    See note 2013-11-18.
  \end{NB}%
\end{proof}

We regard $\lambda\in\h^*$ as a homomorphism $S(\h)^W\to \CC$ by the
evaluation at $\lambda+\rho$, where $\rho$ is the half sum of positive
roots of $\g$.
(It is denoted by $\gamma_{\bar{\lambda}}$ in
\cite[\S5]{Arakawa2007}.)
We further regard $\CC$ as a $\mathfrak{Zh}(\scW_k(\g))$-module by the
above isomorphism, and denote it by $\CC_\lambda$. We extend it to a
$\fU(\scW_k(\g))_{\ge 0}$-module on which $\fU(\scW_k(\g))_{>0}$ acts
trivially. Then we define
\begin{equation}
  M(\lambda) \defeq \fU(\scW_k(\g))\otimes_{\fU(\scW_k(\g))_{\ge 0}}\CC_\lambda.
\end{equation}
This is called the {\it Verma module with highest weight\/} $\lambda$.
\index{Mlambda@$M(\lambda)$}

Now we turn to our $\scW_k(\g)$-module $M_\bF(\ba)$. We identify $\h =
\operatorname{Lie}T$ with $\h^*$ by the invariant bilinear form $(\ ,\
)$. Then we have an identification
\begin{equation}
  S(\h)^W \cong
  \CC[\operatorname{Lie}T]^W = H^*_T(\mathrm{pt})^W
  \begin{NB}
    = H^*_G(\mathrm{pt})
  \end{NB}%
  .
\end{equation}
We regard the collection $\ba = (a^1,\dots,a^\ell)$ as a variable in
$\operatorname{Lie}T$ by considering $a^i$ its coordinate. Hence $\ba$
has value in $\h^*$ by the above identification.

Recall that $|\ba\ra$ is the fundamental class $1\in
\IH^0_\TT(\Uh{0})$. Since the degree $d$ corresponds to an instanton
number, $\fU(\scW_k(\g))_{\ge 0}$ acts via a homomorphism
$\fU(\scW_k(\g))_{\ge 0}\to\bF_T$ induced from $\mathfrak{Zh}(\scW_k(\g))\to
\bF_T$ on $\bF_T|\ba\ra = \IH^*_\TT(\Uh{0})\otimes_{\bA_T}\bF_T$.
Hence we have a $\scW_k(\g)$-homomorphism $M(\lambda)\to M_\bF(\ba)$, sending
$1\in\CC_\lambda\subset M(\lambda)$ to $|\ba\ra\in M_\bF(\ba)$. Here we
generalize the above definition to $\lambda\colon
\mathfrak{Zh}(\scW_k(\g))\to\bF_T$.

\begin{Proposition}\label{prop:lambda}
  \textup{(1)} The highest weight $\lambda$ is given by
  \begin{equation}
    \lambda = \frac{\ba}{\ve_1} - \rho.
  \end{equation}

  \textup{(2)} $M_\bF(\ba)$ is irreducible as a $\scW_k(\g)$-module, and
  isomorphic to $M(\lambda)$.
\end{Proposition}

Note that the Weyl group action on $\ba$ corresponds to the
dot action on $\lambda$,
\(
    w\circ\lambda = w(\lambda+\rho) - \rho.
\)

\begin{proof}
  (1) Recall that our Heisenberg generators and standard generators
  are related by $\widehat h^i_n = \ve_2 P^i_n$. Then the zero mode
  acts by
  \begin{equation}
    \frac{a^i}{\ve_1} - 1 - \frac{\ve_2}{\ve_1}
    = (\alpha_i,\frac{\ba}{\ve_1} - \rho) + k + h^\vee
  \end{equation}
  thanks to \eqref{eq:58}.

  We compare this formula with a realization of $M(\lambda)$ in
  \cite[\S5.2]{Arakawa2007}. Our $\widehat h^i_n$ is $\widehat
  t_{-\bar\rho^\vee}(\widehat{J_i}(n))\in
  \fU(C_k(\bar\g)''_{\mathrm{old}})$, and $\widehat
  t_{-\bar\rho^\vee}(\widehat{J_i}(0)) = \widehat{J_i}(0) + k+h^\vee$
  as in \lemref{lem:shift}.
  Since $\widehat{J_i}(0)$ acts by $\lambda(J_i)$ on $M(\lambda)$, we
  obtain $\lambda = \ba/\ve_1 - \rho$.

  (2) It is well-known that $M(\lambda)$ is irreducible when $\lambda$
  is generic. It follows, for example, from the fact that the determinant of the
  Kac-Shapovalov form is a nonzero rational function, hence the form
  is nondegenerate if $\lambda$ is neither a zero nor  a pole. (See below
  for the Kac-Shapovalov form.)  It also means that the form is
  nondegenerate when one views $\lambda$ as a rational function like
  us. Therefore $M(\lambda)\to M_\bF(\ba)$ is injective.

  Now we compare the graded characters. The character of $M(\lambda)$
  is the same as the character of $S(t{\mathfrak h}[t])$ where $\deg(t)=1$.
  We have $M_\bF(\ba)=\bigoplus_{d\in{\mathbb N}}\IH^*(\Uh{d})\otimes\bF_T$.
  According to~\cite[Theorem~7.10]{BFG}, the character of $M_\bF(\ba)$
  (with grading by the instanton number) is the same as the character of
  $S(t{\mathfrak g}^f[t])$ where $f$ is a principal nilpotent. Since
  $\dim{\mathfrak g}^f=\dim{\mathfrak h}$, the graded characters of
  $M(\lambda)$ and $M_\bF(\ba)$ coincide.
\end{proof}

\begin{NB}
  The argument should be used again, probably in more refined manner,
  in \propref{prop:span}.
\end{NB}

\subsection{Kac-Shapovalov form}\label{sec:kac-shapovalov-form}

We shall identify the Kac-Shapovalov form on $M(\lambda)$ with a
natural pairing on $M_\bF(\ba)$ given by the Verdier duality in this
subsection.

\begin{NB}
The following paragraph added on Sep.14, 2014.
\end{NB}

Let $\sigma$ be the Dynkin diagram automorphism given by
$-\alpha_{\sigma(i)} = w_0(\alpha_i)$. We denote the corresponding
element in $\Aut(G)$ also by $\sigma$. We have an induced isomorphism
\begin{equation*}
    \varphi_\sigma\varphi_{w_0}\colon 
    \IH^*_{\TT,c}(\Uh{d})\to \IH^*_{\TT,c}(\Uh{d}),
\end{equation*}
which is $\bA_T = H^*_\TT(\mathrm{pt})$-linear if we twist the
$\bA_T$-structure on the second $\IH^*_{\TT,c}(\Uh{d})$ by composing
the automorphism $\ba\mapsto -\ba$ of $\bA_T$. This is explained in
the paragraph after \eqref{eq:44}.

Let us denote the natural perfect pairing by
\begin{equation}
  \la\ ,\ \ra\colon \IH^*_{\TT,c}(\Uh{d})\otimes_{\bA_T}
  \IH^*_{\TT}(\Uh{d})\to\bA_T,
\end{equation}
where we compose the above $\varphi_\sigma\varphi_{w_0}$ for the first
factor.
\begin{NB}
    Added, on Sep. 14, 2014.
\end{NB}%
We also multiply it by $(-1)^{dh^\vee}$ as in \eqref{eq:73}. The
notation conflicts with the pairing between $U^{d,P}$ and $U^{d,P_-}$
in \subsecref{sec:pairing}. But the two pairings are closely related,
so the same notation does not give us any confusion. (See
\subsecref{sec:univ-verma-modul} for a more precise relation.)
\index{< , >@$\la\ ,\ \ra$}

By the localization theorem and \lemref{lem:centralfiber} we extend it
to a perfect pairing
\begin{equation}\label{eq:pairing}
    \la\ ,\ \ra\colon M_\bF(-\ba)\otimes M_\bF(\ba)\to\bF_T.
\end{equation}
(cf.\ \cite[\S2.6]{BraInstantonCountingI}.)
Here the highest weight of the first factor is $-\ba$ since we compose
the automorphism $\ba\mapsto -\ba$.
\begin{NB}
    The first factor is changed from $M_\bF(\ba)$ to $M_\bF(-\ba)$ on
    Sep.\ 14, 2014.
\end{NB}%
\begin{NB}
We multiply it by
$(-1)^{\dim \Uh{d}/2}$ as in \eqref{eq:73}.
\end{NB}%

When we localize the equivariant cohomology groups, there is no
distinction between compact support and arbitrary support. We then see
that \eqref{eq:pairing} is symmetric in the sense as in \eqref{eq:41}.
\begin{NB}
    Strictly speaking, we have two pairings one on $M_\bF(\ba)\otimes
    M_\bF(-\ba)$ and the other on $M_\bF(-\ba)\otimes
    M_\bF(\ba)$. Analog of \eqref{eq:41} is a more precise statement.
\end{NB}

We also have the pairing
\begin{equation}
  \la\ ,\ \ra\colon H^*_{\TT,c}(\Uh[T]{d}, \Phi_{T,G}(\IC(\Uh{d})))
  \otimes_{\bA_T}
  H^*_{\TT}(\Uh[T]{d}, \Phi_{T,G}(\IC(\Uh{d})))
  \to\bA_T,
\end{equation}
where we compose $\varphi_\sigma\varphi_{w_0}$ on the
first factor as above.
Since $\sigma w_0$ sends $B$ to the opposite Borel $B_-$, the above is
coming from the pairing between $H^*_{\TT,c}(\Uh[T]{d},
\Phi^{B_-}_{T,G}(\IC(\Uh{d})))$ and $H^*_{\TT}(\Uh[T]{d},
\Phi_{T,G}(\IC(\Uh{d})))$. Therefore it is a perfect pairing
thanks to Braden's isomorphism \eqref{eq:Braden}.
\begin{NB}
    Editted on Sep.~14, 2014.
\end{NB}%
This pairing also extends to a pairing \eqref{eq:pairing}, which is
the same as defined above thanks to the compatibility between Braden's
isomorphism and $i^!j^!\to i^* j^!$ as in the proof of
\lemref{lem:inter=pair}.

The Heisenberg generator $P^i_n$ satisfies
\begin{equation}\label{eq:83}
  \la u, P^i_n v\ra = \la \theta(P^i_n) u, v\ra,
\end{equation}
where $\theta$ is an anti-involution on the Heisenberg algebra given
by
\begin{equation}\label{eq:81}
  \theta(P^i_n) = - P^i_{-n} - \frac{2(\ve_1+\ve_2)}{\ve_1\ve_2} \delta_{n0}.
\end{equation}
\begin{NB}
We can rewrite this as
  \begin{equation}
    \theta(\ve_2 P^i_n) = - \ve_2 P^i_{-n} -
    \frac{2(\ve_1+\ve_2)}{\ve_1} \delta_{n0}
  \end{equation}
in more standard convention. We have
\begin{equation}
  \frac{(\ve_1+\ve_2)}{\ve_1} = 1 + \frac{\ve_2}{\ve_1} =
  1 - (k+h^\vee).
\end{equation}
\end{NB}%
Let us explain the reason for this formula of $\theta$. Thanks to a
standard property of convolution algebras, the diagonal Heisenberg
generator $P^\Delta_n$ in \subsecref{sec:BaraOp} was defined so that
$P^\Delta_n$ is adjoint to $P^\Delta_{-n}$. Since the intersection
pairing \eqref{eq:73} is compatible with the above one, we change $n$
to $-n$. Moreover, since $P^i_n$ is defined via the stable envelope and
we must use the opposite Borel as in \subsecref{sec:pairing}, we need
to swap $P^{(1)}_n$ and $P^{(2)}_n$ in
\subsecref{sec:heis-algebra-assoc}. Therefore we need to change the
sign of $P^i_{-n}$.
The zero mode $P^i_0$ was defined by hand as \eqref{eq:58}. We must
also change the sign of $a^i$, as the $\bA_T$-module structure is
twisted by $\ba\mapsto-\ba$ on the first factor.  
\begin{NB}
    The reason is corrected. Sep.~14, 2014
\end{NB}%
Then we must correct $-P^i_0$ by $-2(\ve_1+\ve_2)/\ve_1\ve_2$.

\begin{NB}
    We need to change the sign as $a^i\to -a^i$, since the stable
    envelop is naturally replaced by its opposite. Also the sign of
    $P^i_n$ must be changed, as it is the anti-diagonal.
\end{NB}

The Virasoro generator $L^i_n$ is mapped to $L^i_{-n}$ by
$\theta$. This is clear from \eqref{eq:32}: $c_1(\mathcal V)$ is self
adjoint and $\theta(P^\Delta_n) = P^\Delta_{-n}$ as we have just
explained. It can be also checked by the formula \eqref{eq:57}.
\begin{NB}
  See note on 2013-11-27.
\end{NB}%

Therefore $\theta$ preserves $\scW_k(\g)$, more precisely the associated
Lie algebra $\mathfrak L(\scW_k(\g))$ and the current algebra
$\fU(\scW_k(\g))$, thanks to \eqref{eq:W=Vir}. We have
\begin{equation}\label{eq:98}
  \la u, x v\ra = \la \theta(x)u, v\ra
\end{equation}
for $x\in\mathfrak L(\scW_k(\g))$, $u,v\in M_\bF(\ba)$.
On the other hand, $\mathfrak L(\scW_k(\g))$ has an anti-involution as in
\cite[\S5.5]{Arakawa2007}, denoted also by $\theta$.

\begin{Proposition}
  Our $\theta$ coincides with one in \cite[\S5.5]{Arakawa2007}.
\end{Proposition}

\begin{proof}
\begin{NB}
  Let us first check that the assertion is true for Verma modules.

  By \cite[Lem.~5.5.2]{Arakawa2007}, the highest weight is changed as
  $\lambda\mapsto -w_0(\lambda)$, where the action of $w_0$ is the
  ordinary one, instead of the shifted one. Then under the correspondence
  in \propref{prop:lambda}(1), we have
  \begin{equation}
    -w_0(\lambda) = -w_0(\frac{\ba}{\ve_1}) - \rho,
  \end{equation}
  as $w_0(\rho) = -\rho$. This means that the equivariant variable
  $\ba$ is replaced by $-w_0(\ba)$. Since the highest weight module is
  invariant under the Weyl group action, we can omit $w_0$, so we get
  $-\ba$. This is the same as ours. This checks that the assertion is
  true for $\mathfrak{Zh}(\scW_k(\g))$.
\end{NB}%
We use the formula \cite[Prop.~3.9.1]{Arakawa2007} for the Heisenberg
vertex algebra. We follow various notation in \cite{Arakawa2007}.

Since $\widehat J_i(n)$ is a Fourier mode of the vertex operator
$Y(v,z) = \sum \widehat J_i(n) z^{-n-1}$ with $v= \widehat J_i({-1})|0
\ra$, we have
\begin{equation}
  \theta(\widehat J_i(n)) = - (e^{T^*} v)_{-n}.
\end{equation}
Here $T^*$ must be substituted by $T^*_{\mathrm{new}}$ in
\cite[(173)]{Arakawa2007}. Using
\begin{equation}
  v = \widehat J_i(-1)|0\ra = J_i({-1})|0\ra -
  \sum_{\alpha\in\Delta} \alpha(h^i) \psi_{-\alpha}(0) \psi_\alpha(-1) |0\ra
\end{equation}
(see \cite[the beginning of \S4.8]{Arakawa2007}), we can check
\begin{equation}
  e^{T^*} v = \widehat J_i(-1)|0\ra + 2(1 - (k+h^\vee))|0\ra.
\end{equation}
\begin{NB}
  See Note 2013-12-02.
\end{NB}%
Therefore we get the same formula as \eqref{eq:81} under the
identification $\widehat J_i(-1) = \ve_2 P^i_n$. (This $\widehat
J_i(-1)$ is in $\fU(C_k(\bar\g)''_{\mathrm{new}})$ and we do not need
to apply $\widehat t_{-\bar\rho^\vee}$ in the proof of
\propref{prop:lambda}, as it is in $T^*_{\mathrm{new}}$.)
\end{proof}

\begin{Remark}\label{rem:Dual}
  We can identify the graded dual $D(M_\bF(\ba))$ of $M_\bF(\ba)$ with
  $M_\bF(-\ba)$
  \begin{NB}
      The second factor is changed from $M_\bF(\ba)$ to $M_\bF(-\ba)$
      on Sep.\ 14, 2014.
  \end{NB}%
  via $\la\ ,\ \ra$. The graded dual has a $\scW_k(\g)$-module
  structure via $\theta$ and the formula \eqref{eq:98}. This is the
  duality functor $D$ in \cite[\S5.5]{Arakawa2007}. The isomorphism
  $D(M_\bF(\ba))\cong M_\bF(-\ba)$ respects $\scW_k(\g)$-module
  structures.

  When $\lambda$ is generic and $M(\lambda)$ is irreducible, the dual
  module $D(M(\lambda))$ is isomorphic to $M(-w_0(\lambda))$, where
  $w_0$ is the longest element in the Weyl group by
  \cite[Th.~5.5.4]{Arakawa2007}. Under the correspondence
  in \propref{prop:lambda}(1), we have
  \begin{equation}
    -w_0(\lambda) = -w_0(\frac{\ba}{\ve_1}) - \rho,
  \end{equation}
  as $w_0(\rho) = -\rho$. This means that the equivariant variable
  $\ba$ is replaced by $-w_0(\ba)$. Since the highest weight module is
  invariant under the Weyl group action, we can omit $w_0$. So the
  equivariant variable is $-\ba$ for $D(M_\bF(\ba))$. Therefore we
  have $D(M_\bF(\ba))\cong M_\bF(-\ba)$. This is what we already
  observed in a geometric way above.

  The pairing $\la \cdot,\cdot \ra$ is uniquely determined from \eqref{eq:98}
  and the normalization $\la -\ba|\ba\ra = 1$ for generic $\ba$.
  It is called the {\it Kac-Shapovalov form\/}. We thus see that the
  Poincar\'e pairing twisted by $\varphi_\sigma\varphi_{w_0}$ on
  $M_\bF(\ba)$ coincides with the Kac-Shapovalov form.
  \begin{NB}
      Corrected, Sep.~14, 2014.
  \end{NB}%
\end{Remark}

\section{\texorpdfstring{$R$}{R}-matrix}\label{sec:R-matrix}

Recall that our hyperbolic restriction $\Phi_{L,G}$ depends on the
choice of a parabolic subgroup $P$. Following \cite[Ch.~4]{MO} (see
also \cite[\S1.3]{dynamicalW}), we introduce $R$-matrices giving
isomorphisms between various hyperbolic restrictions, and study their
properties. They are defined as rational functions in equivariant
variables, and their existence is an immediate corollary to
localization theorem in the previous chapter.

As for the usual $R$-matrices for Yangians, they satisfy the
Yang-Baxter equation and are ultimately related to the $\scW$-algebra.

As an application, we give a different proof of the Heisenberg
commutation relation (\propref{prop:HeisRel}) up to sign, which does
not depend on Gieseker spaces for $SL(3)$. We hope that this proof
could be generalized to other rank $2$ cases $B_2$, $G_2$.

Since the dependence on a parabolic subgroup is important, we denote
the hyperbolic restriction by $\Phi_{L,G}^P$ in this chapter.

\subsection{Definition}

Let us consider the diagram \eqref{eq:1} with respect to a parabolic
subgroup $P$. Let us consider the homomorphism in \eqref{eq:48}
\begin{equation}\label{eq:30}
  \hR_P\colon
  H^*_{\TT}(\Uh[L]d, \Phi^P_{L,G}(\scF))  \to
  H^*_{\TT}(\Uh[L]d, i^* j^*\scF)
  \cong H^*_{\TT}(\Uh{d}, \scF)
\end{equation}
for $\scF\in D^b_\TT(\Uh{d})$. This is an isomorphism over the
quotient field $\bF_T$ of $\bA_T = \CC[\operatorname{Lie}(\TT)]$.
When we want to emphasize $\scF$, we write $\hR^\scF_P$.

\begin{Definition}
  Let $P_1$, $P_2$ be two parabolic subgroups compatible with
  $(G,L)$. Let us introduce the {\it $R$-matrix\/}
  \begin{multline}\label{eq:R}
    \index{RPA@$R_{P_1,P_2}^\scF$}
    R_{P_1,P_2} = (\hR_{P_1})^{-1} \hR_{P_2}  \colon
    H^*_{\TT}(\Uh[L]d, \Phi_{L,G}^{P_2} (\scF))\otimes_{\bA_T}\bF_T
\\
    \to
    H^*_{\TT}(\Uh[L]d, \Phi_{L,G}^{P_1} (\scF))\otimes_{\bA_T}\bF_T
  \end{multline}
\end{Definition}

When we want to view $R_{P_1,P_2}$ as a rational function in
equivariant variables, we denote it by $R_{P_1P_2}(\ba)$. Dependence
on $\ve_1$, $\ve_2$ are not important, so they are omitted.
When we want to emphasize $\scF$, we write $R^\scF_{P_1,P_2}$.

From the definition, we have
\begin{equation}\label{eq:27}
  R_{P_1,P_2} R_{P_2,P_3} = R_{P_1,P_3}.
\end{equation}

\begin{NB}
The following subsection seems unnecessary.

\subsection{Adjoint operators}(cf.\ \protect\cite[\S4.4]{MO})
We have a nondegenerate pairing between
\(
    H^*_{\TT}(\Uh{d}, \scF)
\)
and
\(
    H^*_{\TT,c}(\Uh{d}, \scF^\vee).
\)

Similarly we have a natural nondegenerate pairing between
\(
  H^*_{\TT}(\Uh[L]d, \Phi^P_{L,G}(\scF))
\)
and
\(
  H^*_{\TT,c}(\Uh[L]d, \Phi^{P_-}_{L,G}(\scF^\vee)),
\)
where $P_-$ is the opposite parabolic to $P$.

Therefore the adjoint of $\hR^\scF_P$ is an operator
\begin{equation}
    \label{eq:35}
    (\hR^\scF_P)^t\colon
    H^*_{\TT,c}(\Uh{d}, \scF^\vee)
    \to
    H^*_{\TT,c}(\Uh[L]d, \Phi^{P_-}_{L,G}(\scF^\vee)).
\end{equation}
We consider $\otimes_{\bA_T}\bF_T$, then we do not need to distinguish
$H^*_{\TT,c}$ and $H^*_\TT$. Therefore it goes in the opposite
direction to \eqref{eq:30}.

\begin{Proposition}
    We have $\hR^\scF_P \circ (\hR^{\scF^\vee}_{P_-})^t = \operatorname{id}$.
\end{Proposition}

\begin{proof}
Note that \eqref{eq:35} is given by
\begin{equation}
    H^*_{\TT,c}(\Uh{d}, \scF^\vee)
    \cong H^*_{\TT,c}(\Uh[L]d, i^! j^! \scF^\vee)
    \to
    H^*_{\TT,c}(\Uh[L]{d}, i^! j^* \scF^\vee)
    \cong
    H^*_{\TT,c}(\Uh[L]{d}, i_-^* j_-^! \scF^\vee).
\end{equation}
We replace $P$ by $P_-$, $\scF$ by $\scF^\vee$ to get
\begin{equation}
    H^*_{\TT,c}(\Uh{d}, \scF)
    \cong H^*_{\TT,c}(\Uh[L]d, i_-^! j_-^! \scF)
    \to
    H^*_{\TT,c}(\Uh[L]{d}, i_-^! j_-^* \scF)
    \cong
    H^*_{\TT,c}(\Uh[L]{d}, i^* j^! \scF).
\end{equation}
Therefore the composite is given by
\begin{equation}
  (j_-\circ i_-)^! \to i_-^! j_-^* \to i^* j^! \to (j\circ i)^*.
\end{equation}
From the definition of $i_-^! j_-^* \to i^* j^!$ in \cite{Braden},
this is equal to the natural homomorphism
$(j_-\circ i_-)^! = (j\circ i)^!\to (j\circ i)^*$.
Hence we have the assertion.
\end{proof}

\begin{Corollary}
$(R^\scF_{P_1,P_2})^t = R^{\scF^\vee}_{(P_2)_-, (P_1)_-}$.
\end{Corollary}
\end{NB}

\subsection{Factorization}\label{sec:factorization-1}

Suppose that $Q_1\subset P$ be a pair of parabolic subgroups as in
\subsecref{sec:ass}. Let $M\subset L$ be the corresponding Levi
subgroups. We have $\Phi^{Q_{1,L}}_{M,L}\circ \Phi^P_{L,G} =
\Phi^{Q_1}_{M,G}$ by \propref{prop:trans}.

We further suppose that there is another parabolic subgroup $Q_2$
contained in $P$ such that the corresponding Levi subgroup is also
$M$:
\begin{equation}
    M \subset Q_1, Q_2 \subset P.
\end{equation}
Then we also have the factorization $\Phi^{Q_{2,L}}_{M,L}\circ
\Phi^P_{L,G} = \Phi^{Q_2}_{M,G}$. It is clear from the definition that
we have
\begin{equation}\label{eq:28}
  R_{Q_1,Q_2}^\scF = R_{Q_{1,L}, Q_{2,L}}^{\Phi^P_{L,G}(\scF)}.
\end{equation}

\begin{NB}
Consider the diagram
\begin{equation}
  \begin{CD}
    \Uh[Q_a]{} @>{j_a''}>> \UhP{} @>{j}>> \Uh{}
    \\
    @A{i''_a}AA @AA{i}A @.
    \\
    \Uh[Q_{a,L}]{} @>{j_a'}>> \UhL{} @.
    \\
    @A{i'_a}AA @. @.
    \\
    \Uh[M]{} @. @.
  \end{CD}
\end{equation}
Let $\Phi'_{M,G}$ denote the naive restriction $(j\circ j_a'' \circ
i_a'' \circ i_a')^*$, which is independent of the choice $a=1,2$. We
similarly write $\Phi'_{L,G}$, $\Phi'_{M,L}$.

We have
\begin{equation}
  \begin{CD}
    H^*_{\TT}(\Uh[M]d, \Phi^{Q_a}_{M,G}(\scF)) @>>>
    H^*_{\TT}(\Uh[G]d,\scF)\\
    @| @|
    \\
    H^*_{\TT}(\Uh[M]d, \Phi^{Q_{a,L}}_{M,L} \Phi^{P}_{L,G}(\scF))
    @>>> H^*_{\TT}(\Uh[M]d, \Phi'_{M,L}\Phi'_{L,G}(\scF))
    \\
    @| @AAA
    \\
    H^*_{\TT}(\Uh[M]d, \Phi^{Q_{a,L}}_{M,L} \Phi^{P}_{L,G}(\scF))
    @>>> H^*_{\TT}(\Uh[M]d, \Phi'_{M,L}\Phi^P_{L,G}(\scF)).
  \end{CD}
\end{equation}
Now the assertion is clear.
\end{NB}

Consider the case $L = T$.
Note that Borel subgroups containing a fixed torus $T$ are
parametrized by the Weyl group $W$. Let us denote by $B^w$ the Borel
subgroup corresponding to $w\in W$, where $B^e = B$ is one which we
have fixed at the beginning.
From \eqref{eq:27} $R^\scF_{B^w,B^y}$ factors to a composition of
$R$-matrices for two Borel subgroups related by a simple reflection,
i.e.,
\(
  y = w s_i.
\)
Then we choose $P = P_i^w\supset B^w, B^{w s_i}$ for the parabolic
subgroup to use \eqref{eq:28}. We have
\begin{equation}
  R^\scF_{B^w, B^{ws_i}} = R^{\Phi^{P}_{L,G}(\scF)}_{B_{1,L}, B_{2,L}},
\end{equation}
where $L$ is the Levi subgroup of $P$ and $B_{1,L}$, $B_{2,L}$ are
images of $B^w$, $B^{ws_i}$ in $L$ respectively. As $[L,L] \cong
SL(2)$, we are reduced to study the $SL(2)$ case. The $R$-matrix for
$SL(2)$ was computed in \cite[Th.~14.3.1]{MO} and will be explained in
\subsecref{sec:sl2-case}.

\subsection{Intertwiner property}\label{sec:intertwiner-property}

Let $\scF\in D^b_\TT(\Uh{d})$. We have representations of the Ext algebra
$\Ext_{D^b_\TT(\Uh{d})}(\scF,\scF)$ on two cohomology groups in
\eqref{eq:30}. This is thanks to (\ref{eq:31}, \ref{eq:38}).
Since $\hR^\scF_P$ is defined by a natural transformation of
functors, it is a homomorphism of the Ext algebra. Therefore
\begin{Proposition}\label{prop:intertwiner}
    The $R$-matrix $R^\scF_{P_1,P_2}$ is a homomorphism of modules
    over the Ext algebra $\Ext_{D^b_\TT(\Uh{d})}(\scF,\scF)$.
\end{Proposition}

\subsection{Yang-Baxter equation}

Take $L=T$ and $\scF = \IC(\Uh{d})$ in this subsection.

By \eqref{eq:43} we can map all cohomology groups in \eqref{eq:R} to
the fixed one
\(
   H^*_{\TT}(\Uh[T]d, \Phi_{T,G}^{B}(\IC(\Uh{d})))\otimes_{\bA_T}\bF_T
   \cong M_\bF(\ba)
\)
by $\varphi_w$. We conjugate the $R$-matrix as
\begin{equation}\label{eq:50}
    \varphi_{w_1}^{-1} R_{B^{w_1}, B^{w_2}} \varphi_{w_2}
    \in \End(% M_\bF(\ba)
    H^*_{\TT}(\Uh[T]d, \Phi_{T,G}^{B}(\IC(\Uh{d})))\otimes_{\bA_T}\bF_T
    ).
\end{equation}
Remark that $H^*_\TT(\mathrm{pt})$-structures are twisted by
isomorphisms $w_1, w_2\colon \TT\to \TT$, as mentioned after
\eqref{eq:44}. In practice, we change the equivariant variable $\ba$
according to $w_1$ ,$w_2$.

Since $\hR_P$ is $\varphi_w$-equivariant, \eqref{eq:50} depends only
on $w_1 w_2^{-1}$. Moreover by \eqref{eq:27} it is enough to consider
the case $w_1 w_2^{-1}$ is a simple reflection $s_i$. Therefore we
define
\begin{equation}
  \Check R_i \defeq \varphi_{s_i}^{-1} R_{B^{s_i}, B} \varphi_{e}.
  \index{Rcheck@$\Check R_i$}
\end{equation}

By the factorization (\subsecref{sec:factorization-1}), this is the
$R$-matrix for $SL(2)$. Since we only have two chambers, \eqref{eq:27}
implies
\begin{equation}
  \Check R_i(s_i\ba) \Check R_i(\ba) = 1.
\end{equation}
We change the equivariant variable to $s_i\ba$, as it is the
$R$-matrix from the opposite Borel to the original Borel.
In the conventional notation for the $R$-matrix, we write $u = \langle
\alpha_i,\ba\rangle$ for the variable. Then $\langle \alpha_i,
s_i\ba\rangle = -u$, so this equation means the unitarity of the
$R$-matrix.

Consider $R$-matrices $\Check R_i$, $\Check R_j$. By the factorization
(\subsecref{sec:factorization-1}), we consider them as the
$R$-matrices for the rank $2$ Levi subgroup $L$ containing $SL(2)$ for
$i$ and $j$. We compute the $R$-matrix from a Borel subgroup of $L$ to
the opposite Borel by \eqref{eq:27} in two ways to get

\begin{Theorem}
  \begin{gather}
    \Check R_i(s_j\ba) \Check R_j(\ba) = \Check R_j(s_i\ba)\Check R_i(\ba)
    \quad\text{if $(\alpha_i,\alpha_j) = 0$},
\\
\label{eq:42}
    \Check R_j(s_i s_j\ba) \Check R_i(s_j\ba) \Check R_j(\ba)
    = \Check R_i(s_j s_i \ba) \Check R_j(s_i\ba)\Check R_i(\ba)
    \quad\text{if $(\alpha_i,\alpha_j) = -1$}.
  \end{gather}
\end{Theorem}

\subsection{\texorpdfstring{$SL(2)$}{SL(2)}-case}\label{sec:sl2-case}

As we mentioned earlier, it is enough to compute the $R$-matrix for
$SL(2)$, which was given in \cite[Th.~14.3.1]{MO}. We briefly recall
the result, and point out a slight difference for the formulation.

By \propref{prop:intertwiner} and the observation that the left hand
side of the formula \eqref{eq:32} is contained in the Ext algebra, we
deduce that the $R$-matrix is an intertwiner of the Virasoro
algebra. This is a fundamental observation due to Maulik-Okounkov
\cite{MO}.

The highest weight is generic, since we work over $\bF_T$. Therefore
the intertwiner is unique up to scalar, and we normalize it so that it
preserves the highest weight vector $|\ba\rangle$.

In \cite{MO} the $R$-matrix is given as an endomorphism of the
localized equivariant cohomology group of the fixed point set via the
stable envelop. On the other hand, our $R$ is an endomorphism of
$H^*_\TT(\Uh[T]d, \Phi^B_{T,G}(\IC(\Uh{d})))$. Concretely
\begin{equation}
  \Check R = \left.P_{12} R^{\text{MO}} \right|_{\text{anti-diagonal part}},
\end{equation}
where $P_{12}$ is the exchange of factors of the Fock space $F\otimes
F$, as ${s_i} = P_{12}$.

By \cite[Prop.~4.1.3]{MO} we have
\begin{equation}\label{eq:55}
  \Check R = -1 + O(\ba^{-1}),
  \qquad \ba\to \infty.
\end{equation}

\subsection{\texorpdfstring{$\GG$}{G}-equivariant cohomology}
\label{subsec:G-equiv-coh}

Recall that a larger group $\GG = G\times\CC^*\times\CC^*$ acts on
$\Uh{d}$ so that $\IC(\Uh{d})$ is a $\GG$-equivariant perverse
sheaf. Therefore we can consider $\IH^*_\GG(\Uh{d}) =
H^*_\GG(\Uh{d},\IC(\Uh{d}))$. It is related to the $\TT$-equivariant
cohomology $\IH^*_\TT(\Uh{d})$ as follows.

Let $N(\TT)$ (resp.\ $N(T)$) be the normalizer of $\TT$ (resp.\ $T$)
in $\GG$ (resp.\ $G$). Then we have forgetful homomorphisms
$\IH^*_\GG(\Uh{d})\to \IH^*_{N(\TT)}(\Uh{d})\to \IH^*_\TT(\Uh{d})$. It is
well-known that the first homomorphism is an isomorphism, as the
cohomology of $\GG/N(\TT) = G/N(T)$ is $1$-dimensional (see e.g.,
\cite{MR0423384}).
The Weyl group $W = N(T)/T$ acts naturally on $\IH^*_\TT(\Uh{d})$,
induced from the $N(T)$-action on $\Uh{d}$.
Moreover we have
\begin{equation}\label{eq:105}
    \IH^*_\GG(\Uh{d})\xrightarrow{\cong} \IH^*_{N(\TT)}(\Uh{d})
    \xrightarrow{\cong} \IH^*_\TT(\Uh{d})^W.
\end{equation}

Let us consider the following diagram
\begin{equation}\label{eq:108}
    \begin{CD}
        H^*_\TT(\Uh[T]d, \Phi^B_{T,G}(\IC(\Uh{d})))\otimes_{\bA_T}\bF_T
        @>{\hR_B}>{\cong}> \IH^*_\TT(\Uh{d})\otimes_{\bA_T}\bF_T
\\
    @V{\Check R_i}VV @VV{s_i}V
\\
        H^*_\TT(\Uh[T]d, \Phi^B_{T,G}(\IC(\Uh{d})))\otimes_{\bA_T}\bF_T
        @>{\hR_B}>{\cong}> \IH^*_\TT(\Uh{d})\otimes_{\bA_T}\bF_T,
    \end{CD}
\end{equation}
where $s_i$ is a simple reflection of the above $W$-action.

\begin{Lemma}
    The diagram \eqref{eq:108} is commutative.
    \begin{NB}
    The $R$-matrix $\Check R_i$, viewed as an endomorphism of
    $\IH^*_\TT(\Uh{d})\otimes_{\bA_T}\bF_T$, coincides with the simple
    reflection $s_i$ of the above $W$-action.
    \end{NB}%
\end{Lemma}

\begin{proof}
    We have $\Check R_i = \varphi_{s_i}^{-1} \hR_{B_i}^{-1}\hR_{B}
    \varphi_e$. As an endomorphism of
    $\IH^*_\TT(\Uh{d})\otimes_{\bA_T}\bF_T$, it is replaced by
    $\hR_{B}\varphi_{s_i}^{-1} \hR_{B_i}^{-1}$, as
    $\varphi_e=\operatorname{id}$.
    \begin{NB}
        Viewing an endomorphism of
        $\IH^*_\TT(\Uh{d})\otimes_{\bA_T}\bF_T$ means that we take the
        conjugation by
        \begin{equation}
            \hR_B\colon H^*_\TT(\Uh[T]{d},\Phi^{B}_{T,G}(\IC(\Uh{d})))
            \otimes_{\bA_T}\bF_T
            \xrightarrow{\cong} \IH^*_\TT(\Uh{d})\otimes_{\bA_T}\bF_T.
        \end{equation}
    \end{NB}%

    From the definition of $\hR_{B_i}$ and the commutativity of the
    diagram \eqref{eq:39}, we have $\hR_{B_i} \varphi_{s_i} =
    \varphi_{s_i} \hR_{B}$, where $\varphi_{s_i}$ in the right hand
    side is the action on $\Uh{d}$, the rightmost arrow in
    \eqref{eq:39}. Since the $W$-action is induced from
    $\varphi_{\sigma}$, the assertion follows.
\end{proof}

\begin{Proposition}\label{prop:commute}
    The Weyl group action on $M_\bF(\ba) = \bigoplus_d
    \IH^*_\TT(\Uh{d})\otimes_{\bA_T}\bF_T$ commutes with the $\scW_k(\g)$
    action.
    Hence $\scW_k(\g)$ acts on the $W$-invariant part
\(
    M_\bF(\ba)^W = \bigoplus_d
    \IH^*_\GG(\Uh{d})\otimes_{\bA_G}\bF_G.
    \index{FFG@$\bF_G = \CC(\ve_1,\ve_2,\ba)^W$}
\)
\end{Proposition}

\begin{proof}
    Since $\scW_k(\g)$ is the intersection of $\Vir_i\otimes
    \Heis(\alpha_i^\perp)$ (see \eqref{eq:W=Vir}), it is enough to
    show that $\Vir_i\otimes \Heis(\alpha_i^\perp)$ commutes with
    $s_i$. By the previous lemma, $s_i$ is given by the
    $R$-matrix.

    Let us first factorize the hyperbolic restriction functors
    $\Phi_{T,G}^B$, $\Phi_{T,G}^{B^{s_i}}$ as
    \begin{equation*}
        \Phi_{T,G}^B = \Phi_{T,L_i}^{B_{L_i}} \Phi_{L_i,G}^{P_i},\qquad
        \Phi_{T,G}^{B^{s_i}} = \Phi_{T,L_i}^{B^{s_i}_{L_i}} \Phi_{L_i,G}^{P_i}
    \end{equation*}
    by \propref{prop:trans}. Then the same argument as in
    \propref{prop:intertwiner} shows that $\Check R_i$ commutes with
    the action of the Ext algebra of
    $\Phi_{L_i,G}^{P_i}(\IC(\Uh{d}))$.  Since the Virasoro generators
    $\mL^i_n$ are in this Ext algebra, the first assertion follows.

    For the second assertion, we only need to check
    \begin{equation*}
        (\IH^*_\TT(\Uh{d})\otimes_{\bA_T}\bF_T)^W \cong
        \IH^*_\GG(\Uh{d})\otimes_{\bA_G}\bF_G.
    \end{equation*}
    By \eqref{eq:105} we have a natural injective homomorphism form
    the right hand side to the left. On the other hand, if $m/f$
    ($f\in\bA_T$, $m\in \IH^*_\TT(\Uh{d})$) is fixed by $W$, we have
    \begin{equation*}
        \frac{m}f = \frac1{|W|}\sum_{\sigma\in W} \frac{\sigma m}{\sigma f}
        = \frac1{|W|} \left(\prod_{\sigma\in W} \sigma f\right)^{-1}
      \sum_{\sigma\in W} \sigma m \prod_{\tau\neq\sigma} \tau f.
    \end{equation*}
This is contained in the right hand side. Therefore the above follows.
\end{proof}

\subsection{A different proof of the Heisenberg commutation relation}

We give a different proof of \propref{prop:HeisRel}.

Let $\Tilde\alpha^{d,-}_i$ be the element defined as in
$\Tilde\alpha^{d}_i$ for the opposite Borel. Since the pairing can be
computed from the $SL(2)\cong [L_i,L_i]$ case, we already know that
\begin{equation}
  \la\Tilde\alpha^d_i, \Tilde\alpha^{d,-}_i\ra = 2d.
\end{equation}

We generalize this to
\begin{Proposition}\label{prop:diff_proof}
  \begin{equation}
    \la\Tilde\alpha^d_i, \Tilde\alpha^{d,-}_j\ra = \pm d(\alpha_i,\alpha_j).
  \end{equation}
\end{Proposition}

The following proof does not determine $\pm$, though we know that it
is $+$ by the reduction to the $SL(3)$ case and the formula
\eqref{eq:Cartan}, which has been proved via Gieseker spaces.

\begin{NB}
    Fortunately, the linear independence now follows from
    \eqref{eq:111}, and hence the proof is completed.

\textcolor{red}{
  The following proof contains a gap, as I implicitly assume that
  $\Tilde\alpha_i^d$, $\Tilde\alpha_j^d$ are linearly independent.
}
\end{NB}

\begin{proof}
  We consider the case $(\alpha_i,\alpha_j) = -1$. The proof for the
  case $(\alpha_i,\alpha_j) = 0$ is similar (and simpler).

  Let us study the leading part of Yang-Baxter equation
  \eqref{eq:42}. We consider $R$-matrices as endomorphisms of the
  space \eqref{eq:symprod}.
  By the factorization \eqref{eq:28}, we can use the expansion
  \eqref{eq:55} for $SL(2)$. Then `$-1$' in \eqref{eq:55} is replaced
  by the direct sum of $(-1)$ on $U^d_{T,L_i}\cong \h_{\algsl_2}$
  and the identity on $(U^d_{T,L_i})^\perp$ in \eqref{eq:53}. Let us
  denote it by $\Tilde s_i$.

  Since $(U^d_{T,L_i})^\perp$ is the orthogonal complement of
  $\CC\Tilde\alpha_i^{d,-}$, we have
  \begin{equation}
    \Tilde s_i(x)
    \begin{NB}
      =
      x - \frac{2\la x,\Tilde\alpha_i^{d,-}\ra}
      {\la\Tilde\alpha_i^d,\Tilde\alpha_i^{d,-}\ra} \Tilde\alpha_i^{d}
    \end{NB}%
    =
    x - \la x,\Tilde\alpha_i^{d,-}\ra
    \frac{\Tilde\alpha_i^d}d, \quad
    \text{for $x\in U^d$}.
  \end{equation}
  From the Yang-Baxter equation, we have the braid relation
  \begin{equation}
    \Tilde s_i \Tilde s_j \Tilde s_i = \Tilde s_j \Tilde s_i \Tilde s_j.
  \end{equation}

  Since we are considering the $SL(3)$-case, there is the diagram
  automorphism $\sigma$ exchanging $i$ and $j$. By \lemref{lem:autg}, we
  have $\varphi_\sigma(\Tilde\alpha^d_i) = (-1)^d \Tilde\alpha^d_j$. Since
  $\varphi_\sigma$ preserves the inner product, we get
  \begin{equation}
    \la\Tilde\alpha^d_i, \Tilde\alpha^{d,-}_j\ra =
    \la\Tilde\alpha^d_j, \Tilde\alpha^{d,-}_i\ra.
  \end{equation}

  Now $\Tilde s_i$ is the usual reflection with respect to the
  hyperplane $\Tilde\alpha^{d,-}_i = 0$. Hence we conclude
  $\la\Tilde\alpha^d_i,\Tilde\alpha^{d,-}_j\ra = \pm d$.

  Note that $\Tilde\alpha_i^d = \pm \Tilde\alpha_j^d$ are excluded
  thanks to \eqref{eq:111}, which has been proved without using
  Gieseker spaces for $SL(3)$.
\end{proof}

\begin{NB}
  Let us give a little more detail. We may assume $d=1$. We know
  \begin{equation}
    \la\tilde\alpha^d_1,\tilde\alpha^{d,-}_1\ra = 2 =
    \la\tilde\alpha^d_2, \tilde\alpha^{d,-}_2\ra.
  \end{equation}
  Set
  \begin{equation}
    a \defeq \la\tilde\alpha^d_1, \tilde\alpha^{d,-}_2\ra, \qquad
    b\defeq \la\tilde\alpha^d_2, \tilde\alpha^{d,-}_1\ra.
  \end{equation}
Then reflections $\Tilde s_1$, $\Tilde s_2$ are
\begin{equation}
  \Tilde s_1(\tilde\alpha^d_1) = -\tilde\alpha^d_1, \quad
  \Tilde s_1(\tilde\alpha^d_2) = \tilde\alpha^d_2 - a\tilde\alpha^d_1, \quad
  \Tilde s_2(\tilde\alpha^d_1) = \tilde\alpha^d_1 - b \tilde\alpha^d_2, \quad
  \Tilde s_2(\tilde\alpha^d_2) = -\tilde\alpha^d_2.
\end{equation}
In the matrix form, we have
\begin{equation}
  \Tilde s_1 =
  \begin{pmatrix}
    -1 & -a \\ 0 & 1
  \end{pmatrix},\quad
  \Tilde s_2 =
  \begin{pmatrix}
    1 & 0 \\
    -b & -1
  \end{pmatrix}.
\end{equation}
Then
\begin{equation}
  \Tilde s_1 \Tilde s_2 \Tilde s_1 =
  \begin{pmatrix}
    1 - ab & 2a - a^2 b \\ b & ab-1
  \end{pmatrix},\quad
  \Tilde s_2 \Tilde s_1 \Tilde s_2 =
  \begin{pmatrix}
    ab - 1 & a \\ 2b - ab^2 & 1 - ab
  \end{pmatrix}.
\end{equation}
We deduce $ab = 1$ from $\Tilde s_1 \Tilde s_2 \Tilde s_1 = \Tilde s_2
\Tilde s_1 \Tilde s_2$.

Since we know $a = b$, we get $a = \pm 1$.

Consider the case $(\alpha_i,\alpha_j) = 0$. We have
\begin{equation}
  \Tilde s_1 \Tilde s_2 =
  \begin{pmatrix}
    ab - 1 & a \\ -b & -1
  \end{pmatrix},
\quad
  \Tilde s_2 \Tilde s_1 =
  \begin{pmatrix}
    - 1 & -a \\ b & ab - 1
  \end{pmatrix}.
\end{equation}
Therefore $a=b=0$.
\end{NB}

Once we compute the inner product, the Heisenberg relation is a
consequence of the factorization \eqref{eq:53}.
The generator $P^i_n$ is the tensor product of the Heisenberg
generator for the first factor and the identity in \eqref{eq:53}.

\begin{NB}
  The following subsection is my personal note.

\subsection{}

Let us consider the case $G=SL(3)$ and $L = T$ the maximal torus.
Then $\h = \{ (a_1,a_2,a_3) \in\CC^3 \mid a_1+a_2+a_3=0\}$ and root
hyperplanes are $a_1 = a_2$, $a_2 = a_3$, $a_1 = a_3$. We have 6
chambers. The $R$-matrix between the fundamental chamber $a_1 > a_2 >
a_3$ and the opposite one factorizes into $R$-matrices of $SL(2)$-type
in two ways, and gives the Yang-Baxter equation
\begin{multline}
    \label{eq:33}
  R_{12}(a_1-a_2) R_{13}(a_1-a_3) R_{23}(a_2-a_3)
\\
  = R_{23}(a_2-a_3) R_{13}(a_1-a_3) R_{12}(a_1-a_2).
\end{multline}

*****************************************

In \cite{MO}, the Yang-Baxter equation is given by the following form:
\begin{equation}
  R_{12}(u) R_{13}(u+v) R_{23}(v) = R_{23}(v) R_{13}(u+v) R_{12}(u),
\end{equation}
where $R_{\alpha\beta}$ is the $R$-matrix acting of the
$(\alpha,\beta)$-factors in the triple product $F_1\otimes F_2\otimes
F_3$ of the Fock spaces, by acting the identity on the remaining
factor. And it is normalized so that
\begin{equation}\label{eq:classicalr}
  R_{12}(u) = 1 + \frac{\hbar}u \br_{12} + \cdots,
\end{equation}
where $\br_{12}$ is the classical $R$-matrix. Similar for $13$,
$23$. Note that it starts with the identity $1$.

This formulation does not have an immediate generalization to other
root systems,
\begin{NB2}
  We have three terms $R_{12}$, $R_{13}$, $R_{23}$, but the braid
  relation should involve only two terms.
\end{NB2}%
so let us put the exchange matrix $P_{\alpha\beta}\colon
F_\alpha\otimes F_\beta\to F_\beta\otimes F_\alpha$ and introduce
$\Check R_{\alpha\beta} = P_{\alpha\beta}R_{\alpha\beta}$.

Note first that
\begin{equation}
    \label{eq:34}
    P_{12} R_{12}(u) P_{12} = R_{12}(-u)^{-1}.
\end{equation}
See \cite[(4.11)]{MO}. Therefore
\begin{equation}
    \label{eq:36}
    \Check R_{12}(s_{12} u) \Check R_{12}(u) = 1,
\end{equation}
as $s_{12}u = -u$.

The Yang-Baxter equation is written as
\begin{equation}
  \Check R_{12}(v) \Check R_{23}(u+v) \Check R_{12}(u)
  = \Check R_{23}(u) \Check R_{12}(u+v) \Check R_{23}(v).
\end{equation}
Two equations are equivalent: The left hand side is
\begin{equation}
  \begin{split}
  & (P_{13} P_{23} \Check R_{23}(v) P_{13} P_{12})
  (P_{12}\Check R_{13}(u+v) P_{12})\Check R_{12}(u)
\\
  = \; &P_{13} R_{23}(v) R_{13}(u+v) R_{12}(u),
  \end{split}
\end{equation}
and similarly the right hand side is equal to
\begin{equation}
  P_{13} R_{12}(u) R_{13}(u+v) R_{23}(v).
\end{equation}

Now we set $\Check R_{1} = \Check R_{12}$, $\Check R_{2} = \Check
R_{23}$ as usual, and we write the Yang-Baxter equation as
\begin{equation}
  \Check R_1 (u) \Check R_2(u+v) \Check R_1(v)
  =   \Check R_2 (v) \Check R_1(u+v) \Check R_2(u).
\end{equation}
Taking the leading term of this equation in the expansion
\eqref{eq:classicalr}, we see the braid relation
\begin{equation}
  s_1 s_2 s_1 = s_2 s_1 s_2.
\end{equation}
Also note that if we replace $v=a_1 - a_2$, $u=a_2 - a_3$ as above, we
can further write
\begin{multline}
    \label{eq:37}
  \Check R_1 (\langle \alpha_1, s_2 s_1(\ba)\rangle)
  \Check R_2( \langle \alpha_2, s_1(\ba)\rangle)
  \Check R_1( \langle \alpha_1, \ba\rangle)
\\
  =   \Check R_2(\langle \alpha_2, s_1 s_2(\ba)\rangle)
      \Check R_1(\langle \alpha_1, s_2(\ba)\rangle)
      \Check R_2(\langle \alpha_2, \ba\rangle).
\end{multline}
This looks more compatible with the root system.

In our setting $s_1$, $s_2$ appear as reflections with respect to the
orthogonal decomposition \eqref{eq:decomp2}. Therefore this equation
means that two lines $U^d_{T,L_i}$ ($i=1,2$) intersect at angle
$2\pi/3$. This implies that
\begin{equation}
  (\Tilde\alpha_1^d,\Tilde\alpha_2^d) = \pm d.
\end{equation}
In order to see that this is really equal to $-d$, we need to be a
little more careful, as it involves the sign for $\Tilde\alpha_i^d$.

Let us proceed a little further. The classical $R$-matrix is given in
\cite[Th.~12.4.4]{MO}
\begin{equation}
  \begin{split}
  & \br_{12} = - \ve_1\ve_2 \sum_{n > 0}
  (P_{-n}^{(1)} - P_{-n}^{(2)})(P_n^{(1)} - P_n^{(2)}), \\
&  \br_{23} = - \ve_1\ve_2 \sum_{n > 0}
  (P_{-n}^{(2)} - P_{-n}^{(3)})(P_n^{(2)} - P_n^{(3)}).
  \end{split}
\end{equation}
Let us switch to the root system notation and write
\begin{equation}
  \br_i = -\ve_1\ve_2 \sum_{n > 0} P^i_{-n} P_n^i
  =
  \begin{cases}
      \br_{12} & \text{if $i=1$}
\\
      \br_{23} & \text{if $i=2$}
  \end{cases}
  \quad
  %\text{for $i=1,2$}, \quad
  \Check\br_i = s_i \br_i.
\end{equation}
We also have
\begin{equation}
  \br_{13} = -\ve_1\ve_2 \sum_{n>0} (P^1_{-n}+P^2_{-n})(P^1_{n}+P^2_{n}).
\end{equation}

Then the classical $R$-matrix satisfies
\begin{equation}
  \begin{split}
    & s_1 \Check\br_2 \Check\br_1 + \Check\br_1 s_2\Check\br_1
    = \Check\br_2 s_1\Check\br_2 + \Check\br_2\Check\br_1 s_2,
\\
    & \Check\br_1 s_2 \Check\br_1 + \Check\br_1 \Check\br_2 s_1
    = s_2 \Check\br_1 \Check\br_2 + \Check\br_2 s_1 \Check\br_2.
  \end{split}
\end{equation}
\begin{NB2}
    We also have $\Check\br_1 s_2 s_1 = s_2 s_1\Check\br_2$,
    $\Check\br_2 s_1 s_2 = s_1 s_2\Check\br_1$, $s_1\Check\br_2 s_1 =
    s_2\Check\br_1 s_2$. And also $\Check\br_1\Check\br_2\Check\br_1 =
    \Check\br_2\Check\br_1\Check\br_2$.
\end{NB2}

We already know that how $s_i$ acts on $\Check\br_j$ since we already
compute the inner product. We move all $s_i$ to the right and get
\begin{equation}
  \begin{split}
    & (\br_{13} \br_{23} + \br_{12} \br_{23})s_1s_2s_1
  = (\br_{23}\br_{12} + \br_{23}\br_{13}) s_2 s_1 s_2,
\\
    &
    (\br_{12} \br_{23} + \br_{12} \br_{13})s_1s_2s_1
  = (\br_{13}\br_{12} + \br_{23}\br_{12}) s_2 s_1 s_2.
  \end{split}
\end{equation}
Removing $s_1 s_2 s_1 = s_2 s_1 s_2$ and get
\begin{equation}
  [\br_{13} + \br_{12}, \br_{23}] = 0 =
  [\br_{12}, \br_{23} + \br_{13}].
\end{equation}
This is the equation in \cite[(5.15)]{MO}.
\begin{NB2}
  It seems that their classical Yang-Baxter equation \cite[(4.22)]{MO}
  is not quite correct, and \cite[(5.15)]{MO} can be directly deduced
  from the Yang-Baxter equation.
\end{NB2}

We have
\begin{equation}
  \begin{split}
    & \frac1{(\ve_1\ve_2)^4}[\br_{12}, \br_{23}] =
    \sum_{n>0} (P^1_{-n} [P^1_n, P^2_{-n}] P^2_n
    - P^2_{-n} [P^2_n, P^1_{-n}] P^1_n) ,
\\
    & \frac1{(\ve_1\ve_2)^4}[\br_{13}, \br_{23}]
    \begin{aligned}[t]
    &=
    \sum_{n>0} ((P^1_{-n} + P^2_{-n})[P^1_n + P^2_n, P^2_{-n}] P^2_n
    - P^2_{-n} [P^2_n, P^1_{-n} + P^2_{-n} ](P^1_n + P^2_n))
\\
    &=
    \sum_{n>0} (P^1_{-n} [P^1_n + P^2_n, P^2_{-n}] P^2_n
    - P^2_{-n} [P^2_n, P^1_{-n} + P^2_{-n} ]P^1_n ),
    \end{aligned}
  \end{split}
\end{equation}
as $[P^1_n+P^2_n, P^2_{-n}] = [P^2_n, P^1_{-n}+P^2_{-n}]$.
Therefore the sum is
\begin{equation}
  \sum_{n > 0}
   (P^1_{-n} [2P^1_n + P^2_n, P^2_{-n}] P^2_n
    - P^2_{-n} [P^2_n, 2P^1_{-n} + P^2_n]P^1_n).
\end{equation}
If $P^i_n, P^j_n$ have the correct commutation relation, this
vanishes.
\end{NB}

%\input{old_integral}

% Time-stamp: <2016-10-25 10:59:46 nakajima>
\section{Whittaker state}\label{sec:whittaker-state}

\subsection{Universal Verma/Wakimoto modules}\label{sec:univ-verma-modul}

Let us denote the direct sum of four $\bA_T$-modules over $d\in\Z_{\ge
  0}$ in \eqref{eq:4sp} by $M_\bA(\ba)$, $N_\bA(\ba)$,
$D(N_\bA(-\ba))$, $D(M_\bA(-\ba))$ respectively. Thus we have
\begin{equation}\label{eq:89}
  M_\bA(\ba) \subset N_\bA(\ba) \subset D(N_\bA(-\ba))\subset D(M_\bA(-\ba)).
\end{equation}
\begin{NB}
    The notation changed on Sep.~14, 2014
\end{NB}%
The reason for notation will be clear
shortly. \index{MA@$M_\bA(\ba)$}\index{NA@$N_\bA(\ba)$}

The pairing \eqref{eq:pairing} restricts to a perfect pairing
\begin{equation}
    \label{eq:63}
    \la\ ,\ \ra\colon M_\bA(-\ba)\otimes D(M_\bA(-\ba)) \to \bA_T,
\end{equation}
given by the Verdier duality, where the $\bA_T$-structure is twisted
by the automorphism $\ba\mapsto-\ba$ as in
\subsecref{sec:kac-shapovalov-form}, and hence the notation is changed
to $M_\bA(-\ba)$.
\begin{NB}
    Added: Sep.~14, 2014.
\end{NB}%
Then $D(M_\bA(-\ba))$ is identified with the graded dual of
$M_\bA(-\ba)$ by \eqref{eq:63}, hence our notation is compatible with
the convention in Remark~\ref{rem:Dual}.
\begin{NB}
    Modified : Sep.~15, 2014.
\end{NB}%
Similarly if we twist $N_\bA(\ba)$, we have an isomorphism
\begin{equation}\label{eq:86}
    \varphi_\sigma\varphi_{w_0}\colon N_\bA(-\ba)\xrightarrow{\cong}
    \bigoplus_d H^*_{\TT,c}(\Uh[T]d, \Phi_{T,G}^{B_-}(\IC(\Uh{d}))),
\end{equation}
where $\Phi_{T,G}^{B_-}$ is the hyperbolic restriction with respect to
the opposite Borel $B_-$. Then we have a perfect pairing
\begin{equation}
    \label{eq:64}
    \la\ ,\ \ra\colon N_\bA(-\ba)\otimes D(N_\bA(-\ba))\to\bA_T.
\end{equation}
Recall that $N_\bA(\ba)$, $D(N_\bA(-\ba))$ are modules over the
integral form of the Heisenberg algebra $\Heis_\bA(\mathfrak h)$, as
we remarked at the end of \subsecref{sec:heis-algebra-assoc}.

Using \lemref{lem:Fock}, we make an identification
\begin{equation}\label{eq:102}
      N_\bA(\ba) 
      % = \bigoplus_d H^*_{\TT,c}(\Uh[T]d, \Phi_{T,G}(\IC(\Uh{d}))
      % \\
      % &
      \cong \bigoplus_\lambda
      \Sym^{n_1}U^{1}\otimes \Sym^{n_2} U^{2}\otimes\cdots
      \otimes H^*_{\TT,c}(\overline{S_\lambda\AA^2}),
\end{equation}
where $U^d=U^{d,B}_{T,G}$ and $\lambda = (1^{n_1}2^{n_2}\cdots)$. We
also have an identification for the opposite Borel $B_-$:
\begin{equation}
  D(N_\bA(\ba))
  \cong \bigoplus_\lambda \Sym^{n_1}U^{1,-}\otimes \Sym^{n_2} U^{2,-}\otimes\cdots
      \otimes H^*_{\TT}(\overline{S_\lambda\AA^2}),
\end{equation}
where $U^{d,-} = U^{d,B_-}_{T,G}$.
Then the pairing \eqref{eq:64} is the product of the pairing between
$U^d$ and $U^{d,-}$ in \subsecref{sec:pairing} and one between
$H^*_{c,\TT}(\overline{S_\lambda\AA^2})$ and
$H^*_{\TT}(\overline{S_\lambda\AA^2})$.
\begin{NB}
  The latter is the intersection pairing multiplied by
  $(-1)^{\dim S_\lambda\AA^2/2} = (-1)^{n_1+n_2+\cdots}$.
\end{NB}

Moreover, two pairings \eqref{eq:63}
\begin{NB}
    The reference corrected. Sep.~14, 2014.
\end{NB}%
and \eqref{eq:64} are compatible
with the embeddings \eqref{eq:89}.
\begin{NB}
    This requires an explanation.
\end{NB}

Let $\scW_\bA(\g)$ be the $\bA$-form of the $W$-algebra in
\secref{sec:intW}.
\index{WAg@$\scW_\bA(\g)$|textit}

\begin{Proposition}
    $M_\bA(\ba)$, $D(M_\bA(\ba))$ are $\scW_\bA(\g)$-modules.
\end{Proposition}

\begin{proof}
  Note that $D(M_\bA(\ba))$ is characterized as
    \begin{equation}
        \{ m \in M_\bF(\ba) \mid \la m, M_\bA(\ba)\ra\in \bA_T \}.
    \end{equation}
    Therefore it is enough to show the assertion for $M_\bA(\ba)$.

    We consider $M_\bA(\ba)$ as a subspace of $N_\bA(\ba)$. The latter
    is a module over $\Heis_\bA(\mathfrak h)$, and hence over
    $\Vir_{i,\bA}$.
    By \thmref{thm:screening}, it is enough to check that $M_\bA(\ba)$ is
    invariant under the intersection of $\Vir_{i,\bA}$ for all $i$.
    Recall that we know that \eqref{eq:71} is a $\Vir_{i,\bA}$-module,
    as $\mL^i_n$ is well-defined. Therefore it is enough to show that
    \begin{equation}
        \label{eq:72}
        \IH^*_{\TT,c}(\Uh{d}) =
        \bigcap_i H^*_{\TT,c}(\Uh[L_i]d, \Phi_{L_i,G}(\IC(\Uh{d}))).
    \end{equation}

    \begin{NB}
        Here we follow Sasha's message on June 9, 2013.
    \end{NB}%
    By \thmref{thm:nonlocal} we have
    \begin{multline}
        \label{eq:74}
        H^*_{\TT,c}(\Uh[L_i]d, \Phi_{L_i,G}(\IC(\Uh{d})))
\\
        = H^*_{\TT,c}(\Uh[T]d, \Phi_{T,G}^B(\IC(\Uh{d})))\cap
        H^*_{\TT,c}(\Uh[T]d, \Phi_{T,G}^{B^{s_i}}(\IC(\Uh{d}))),
    \end{multline}
    where $B^{s_i}$ is the Borel subgroup corresponding to a simple
    reflection $s_i$. Therefore it is enough to show that the
    intersection of the right hand side of \eqref{eq:74} for all $i$
    is $\IH^*_{\TT,c}(\Uh{d})$. This is proved in a similar manner as
    \thmref{thm:nonlocal}. The only thing we need to use is the fact
    for any non-zero dominant $\lambda$ there exists $i\in I$ such
    that $s_i(\lambda)$ is not dominant.
\end{proof}

\begin{Proposition}\label{prop:span}
  The $\scW_\bA(\g)$-submodule of $M_\bF(\ba)$ generated by $|\ba\rangle$ is
  $M_\bA(\ba)$, i.e.,
  \begin{equation}
    M_\bA(\ba) = \scW_\bA(\g)|\ba\rangle.
  \end{equation}
\end{Proposition}

\begin{proof}
  Comparison of bigraded dimensions: $\scW_\bA(\g)|\ba\rangle$ is bigraded
  by the usual degree and $\cohdeg$, so that the bidegree of
  $\widetilde{W}{}^{(\kappa)}_n$ is $(n,d_\kappa+1)$, see~\subsecref{subsec:gen}.
  According to {\em loc.\ cit.}, $\scW_\bA(\g)|\ba\rangle$ is a
  free $\bA$-module (the bidegree of $\varepsilon_1,\varepsilon_2,
  {\mathfrak h}$ equals $(0,1)$) with the space of generators
  $S(t{\mathfrak w}[t])$ where ${\mathfrak w}=
  \oplus_{\kappa=1}^\ell{\mathfrak w}^{(\kappa)}$ with the bidegree of
  ${\mathfrak w}^{(\kappa)}$ equal to $(0,d_\kappa+1)$, and the bidegree of $t$
  equal to $(1,0)$.

  On the other hand, $M_\bA(\ba)$ is bigraded by the instanton number and
  half the cohomological degree. It is a free $\bA$-module with the space
  of generators equal to $\bigoplus_{d\in{\mathbb N}}\IH^*_c(\Uh{d})$.
  According to~\cite[Theorem~7.10]{BFG}, $\bigoplus_{d\in{\mathbb N}}\IH^*_c(\Uh{d})
  \simeq S(t{\mathfrak g}^f[t])$ where 
  ${\mathfrak g}^f=\oplus_{\kappa=1}^\ell{\mathfrak g}^f_{(\kappa)}$ with the
  bidegree of ${\mathfrak g}^f_{(\kappa)}$ equal to $(0,d_\kappa+1)$, 
  and the bidegree of $t$ equal to $(1,0)$. 
\end{proof}

On the other hand, it is clear from \eqref{eq:102} that
\begin{equation}\label{eq:103}
  N_\bA(\ba) = \Heis_\bA(\h)|\ba\ra
\end{equation}

For a homomorphism $\chi\colon\bA_T\to\CC\equiv\CC_\chi$, the
specialization
\begin{equation}
  M_\bA(\ba)\otimes_\bA\CC_\chi
\end{equation}
is a module over $\scW_k(\g)$ with level $k = \chi(-\ve_2/\ve_1) -
h^\vee$. It is a Verma module with highest weight $\chi(\ba/\ve_1) -
\rho$, see \subsecref{sec:highest-weight}. Here $\chi$ is regarded as
the assignment of variables $\ba$, $\ve_1$, $\ve_2$, or more concretely
$\chi(\ba) = \sum \chi(a^i)\varpi_i$ for fundamental weights
$\varpi_i$.

\begin{Definition}
  We call $M_\bA(\ba)$ the {\it universal Verma module}.
\end{Definition}

Similarly $N_\bA(\ba)$ is specialized to the Fock representation of the
Heisenberg algebra by $\chi$. We call $N_\bA(\ba)$ the {\it universal
  Wakimoto module}. Similarly $D(M_\bA(\ba))$ is the universal dual Verma
module, and $D(N_\bA(\ba))$ the universal dual Wakimoto module.

\subsection{\texorpdfstring{$\GG$}{G}-equivariant cohomology}

Let us consider the $\GG$-equivariant intersection cohomology groups
as in \subsecref{subsec:G-equiv-coh}. We have
\begin{equation}\label{eq:106}
    \bigoplus_d \IH^*_{\GG,c}(\Uh{d}) = M_\bA(\ba)^W, \quad
    \bigoplus_d \IH^*_{\GG}(\Uh{d}) = D(M_\bA(-\ba))^W
\end{equation}
by \eqref{eq:105}. Since the $W$-action commutes with the
$\scW_k(\g)$-action by \propref{prop:commute}, we see that both of
\eqref{eq:106} are modules over $\scW_\bA(\g)$.

\subsection{Whittaker condition}

Let $\mW{\kappa}{}$ be as in \subsecref{subsec:gen}, which generates 
$\scW_\bA(\g)$ in the sense of the reconstruction theorem.
Let $|1^d\rangle \defeq [\Uh{d}]\in
\IH_\TT^{0}(\Uh{d})$\index{1d@$\vert 1^d\rangle$} be the fundamental class.
It conjecturally satisfies the following {\it Whittaker conditions}
\begin{Conjecture}\label{conj:Whit}
Let $d\ge 1$, $n > 0$. We have
\begin{equation}\label{eq:Whittaker}
  \mW{\kappa}n |1^d\rangle =
  \begin{cases}
     |1^{d-1}\rangle & 
    \text{if $\kappa=\ell$ and $n=1$},
    \\
    0 & \text{otherwise}.
  \end{cases}
\end{equation}
\end{Conjecture}

\begin{NB}
(1) In \cite[Prop.~9.4]{SV} the condition is written as
  \begin{equation}
    \W{\ell}1 |1_n\rangle = \ve_1^{-1} \ve_2^{-h^\vee+1}  |1_{n-1}\rangle
  \end{equation}
  under $\ve_1 = x$, $\ve_2 = y$. Therefore their $W$-algebra
  generator $\W{\ell}1 = \ve_1^{-1}\ve_2^{-h^\vee+1} \mW{\ell}1$. This
  normalization is natural in the sense that $\W{\ell}1$ preserves
  the perverse degree, or the cohomological degree shifted by the
  complex dimension.

(2)  I am not yet sure about the sign issue. In the formulation of
  \cite{MO}, one need into an account the polarization.
\end{NB}%

Since $\mW{\kappa}n$ is contained in $\scW_\bA(\g)$, it is a well-defined
operator on $D(M_\bA(-\ba)) = \bigoplus \IH^*_\TT(\Uh{d})$.
\begin{NB}
  Note that $\mW{\kappa}m$ is the operator coupled with the fundamental
  class $[0] = \ve_1\ve_2$ of the point $0$ in $\AA^2$. Therefore it
  is well-defined on $\IH^\GG_*(\Uh{d}) = H^{2h^\vee n-*}_{\GG}(\Uh{d},
  \IC_{\Uh{d}})$.
\end{NB}%
Since $\mW{\kappa}n$ has $\cohdeg = d_\kappa + 1$ ($d_\kappa$ is an
exponent as in \secref{sec:intW}), it sends $|1^d\ra\in
\IH^0_\TT(\Uh{d})$ into $\IH^{2(d_\kappa + 1 - nh^\vee)}_\TT(\Uh{d-n})$.
\begin{NB}
  $0 = 2dh^\vee - 2dh^\vee$ is sent to
  \begin{equation}
    2(d-n)h^\vee - 2dh^\vee + 2(d_\kappa+1) = 2(d_\kappa + 1 - nh^\vee).
  \end{equation}
\end{NB}%
Since $d_\kappa \le d_\ell = h^\vee - 1$, we have $\mW{\kappa}n |
1^d\ra = 0$ unless $n=1$, $\kappa = \ell$.
Also we see that $\mW{\ell}1 | 1^d\ra$ is a multiple of $|1^{d-1}\ra$
with the multiple constant of degree $0$, i.e., a complex number.
Moreover, if the multiple constant would be $0$, it is a highest
weight vector and generates a nontrivial submodule.
\begin{NB}
  Here I used that $\scW_\bA(\g)$ is generated by $\mW{\kappa}n$.
\end{NB}%
Since $M_\bF$ is irreducible, it is a contradiction. Therefore the
constant cannot be zero. In particular, if we divide $|1^d\ra$ by the
constant, it satisfies the Whittaker condition \eqref{eq:Whittaker}.

Let $|w^d\ra$ be the vector determined by with the normalization
$|w^0\rangle = |1^0\rangle = |\ba\ra\in \IH_\TT^0(\Uh{0}) =
\IH^0_\TT(\mathrm{pt})$. Its existence and uniqueness will follow from
the discussion in \subsecref{subsec:Part} below. 
(However it is not {\it a priori\/} clear that $|w^d\ra\in D(M_\bA(-\ba))$,
as for $|1^d\ra$.)
Therefore we already know that $|1^d\ra = c_d |w^d\ra$ for some
$c_d\in\CC$ by the above observation. The goal of this chapter is to
prove a slightly weaker version of \eqref{eq:Whittaker}.
\begin{Theorem}\label{thm:cd}
  Conjecture~\ref{conj:Whit} holds up to sign.
\end{Theorem}

Our strategy of the proof is as follows.
To determine $c_d$ up to sign,
\begin{NB}
added on June 15  
\end{NB}%
it is enough to compare pairings $\la 1^d|1^d\ra$ with $\la w^d|
w^d\ra$. Moreover, as $c_d$ is a complex number, we may do it after
specifying equivariant variables $\ve_1$, $\ve_2$.
\begin{NB}
  The following has been checked in
  \subsecref{sec:kac-shapovalov-form}.

  Here I implicitly assume that the pairing $\langle\ |\ \rangle$ is
  corresponding to the natural intersection pairing on equivariant
  homology groups.
\end{NB}%
We will show that
\begin{equation}\label{eq:91}
  \begin{split}
  & \left.(\ve_1\ve_2)^d \la 1^d| 1^d\ra\right|_{\ve_1,\ve_2=0}
    = \frac1{d !}
      \left(\left.\ve_1\ve_2 \la 1^1|1^1\ra\right|_{\ve_1,\ve_2=0}\right)^d,
  \\
  & \left.(\ve_1\ve_2)^d \la w^d| w^d\ra\right|_{\ve_1,\ve_2=0} = \frac1{d !} 
  \left(\left.\ve_1\ve_2 \la w^1|w^1\ra\right|_{\ve_1,\ve_2=0}\right)^d.
  \end{split}
\end{equation}
It implies that
\begin{equation*}
    c_d^2 = c_1^{2d}.
\end{equation*}

Recall that the top degree field $\mW{\ell}{}$ in
\subsecref{subsec:gen} is well-defined only up to nonzero multiple
even ignoring lower degree terms, as we just take it as a highest
weight vector of a certain $\algsl_2$ representation. Therefore if we
divide $\mW{\ell}{}$ by $c_1$, \eqref{eq:Whittaker} holds up to sign.

Since $|1^d\ra$ is canonically determined from geometry, it means that
the top degree generator $\mW{\ell}{}$ is fixed without constant
multiple ambiguity (up to sign).
In particular, when we applied $\mW{\ell}0$ to the highest weight
vector $|\ba\ra$, we get an invariant polynomial in $\ba$ of degree
$h^\vee$. (See \ref{sec:highest-weight}.)
We do not study what this natural choice of the highest degree
generator of the invariant polynomial $S(\h)^W$ is in general.
But we will check that it is indeed a natural one for $\g =
\algsl_{\ell+1}$ in \subsecref{sec:Whit_typeA}.

\subsection{Whittaker vector and Kac-Shapovalov form}
\label{subsec:Part}

In this subsection, we shall prove that the Whittaker vector exists
and is unique in the localized equivariant cohomology $M_\bF(\ba)$,
which we think of Verma module with generic highest weight by
\propref{prop:lambda}.
The argument is more or less standard (see e.g., \cite{ABCDEFG}), but
we give the detail, as we will use similar one later in
\subsecref{sec:proof}.

We have a nondegenerate Kac-Shapovalov form $\la\ ,\ \ra$ on $M_\bF(\ba)$.
Let $\theta$ denote the anti-involution on $\fU(\scW_\bF(\g))$ as in
\subsecref{sec:kac-shapovalov-form}. We have
\begin{equation}\label{eq:87}
  \theta(\mW{\kappa}{n}) = (-1)^{d_\kappa+1} \mW{\kappa}{-n}.
\end{equation}
See \cite[\S5.5]{Arakawa2007}. In particular, $\fU(\scW_\bA(\g))$ is
invariant under $\theta$.

Let us denote the highest weight vector of $D(M_\bF(\ba))$ by
$\la-\ba|$. See Remark~\ref{rem:Dual} to see that its highest weight
is $-\ba$.
\begin{NB}
    Corrected, Sep.~14, 2014, Corrected again, Sep.~15, 2014.
\end{NB}%

\begin{NB}
This pairing is symmetric. Namely we have
\begin{equation}
    \langle m_1 | m_2\rangle = \langle m_2 | m_1 \rangle,
\end{equation}
where $m_1\in M(-\ba)$, $m_2\in M(\ba)$. The left hand side is defined
via $M(-\ba)\cong D(M(\ba))$. The right hand side is defined via
$M(\ba) \cong D(M(-\ba))$. This is true as $D(D(M(\ba))) = M(\ba)$.
\end{NB}

Let $\blam = (\lambda^1,\dots,\lambda^\ell)$ be an $\ell$-partition,
i.e., it is an $\ell$-tuple of partitions $\lambda^i =
(\lambda^i_1,\lambda^i_2,\dots)$. We consider the corresponding
operator
\begin{equation}
  \widetilde W[\blam] \defeq
  \mW{1}{-\lambda^1_1} \mW{1}{-\lambda^1_2} \cdots
  \mW{\ell}{-\lambda^\ell_1} \mW{\ell}{-\lambda^\ell_2} \cdots
\end{equation}
in the current algebra of the $W$-algebra. Then
\begin{equation}
  \widetilde W[\blam]|\ba\rangle
\end{equation}
form a PBW base of $M_\bF(\ba)$. We define the Kac-Shapovalov form
\begin{equation}
  K \equiv K^d \defeq (\langle -\ba|
  \theta(\widetilde W[\blam]) \widetilde W[\bmu] | \ba\rangle)_{\blam\bmu},
\end{equation}
where $\blam$, $\bmu$ runs over $\ell$-partitions whose total sizes
are $d$.
We consider it as a matrix, and an entry is denoted by $K_{\blam\bmu}$.

Let $(1^d)=(1,\dots,1)$ be the partition of $n$ whose all entries are
$1$. Let $\blam_0 = (\emptyset,\dots,\emptyset,(1^d))$ be the
$\ell$-partition where the first $(\ell-1)$ partitions are all
$\emptyset$ and the last one is $(1^d)$.
The corresponding operator $\widetilde W[\blam_0]$ is
$(\mW{\ell}{-1})^d$.

We have
\begin{equation}\label{eq:Wh}
  \langle -\ba | \theta (\widetilde W[\blam]) | w^d\ra = 
  \begin{cases}
    1 & \text{if $\blam=\blam_0$},
    \\
    0 & \text{otherwise}
  \end{cases}
\end{equation}
from \eqref{eq:Whittaker} by the induction on $d$. Note that $|w^0\ra =
|\ba\ra$, and hence $\la-\ba|\ba\ra = 1$.

Let us write the Whittaker vector $|w^d\rangle$ in the PBW base as
\begin{equation}\label{eq:wn}
  |w^d\rangle = \sum_{\bmu} a_\bmu \widetilde W[\bmu]|\ba\rangle.
\end{equation}
By \eqref{eq:Wh} we have
\begin{equation}
  \sum_\bmu K_{\blam\bmu} a_\bmu = 
  \begin{NB}
  \sum_\bmu a_\bmu \langle -\ba| \theta(\widetilde W[\blam]) 
  \widetilde W[\bmu] 
  | \ba\rangle 
  =     
  \end{NB}%
  \delta_{\blam\blam_0}.
\end{equation}
In other words,
\begin{equation}
  a_\bmu = K^{\bmu\blam_0},
\end{equation}
where $K^{-1} = (K^{\bmu\blam})$ is the inverse of $K$. In particular,
the existence and the uniqueness of $|w^d\rangle$ follow.

We also get
\begin{equation}
  \langle w^d| w^d\rangle = K^{\blam_0\blam_0}.
\end{equation}

\subsection{Lattices}\label{subsec:int}

Let 
\begin{equation}
  \bW{\kappa}n = (\ve_1\ve_2)^{-1} \mW{\kappa}n
\end{equation}
for $\kappa=1,\dots,\ell$, $n\in\Z$.

\begin{NB}
  The following is a record of an old manuscript. If it will be
  unnecessary, I will delete it. Also note that the notation
  $D(M(\ba))$ was written ${}^\bA M(-\ba)$ here.

Let
\begin{equation}
  \widetilde W[\blam] \defeq
  \mW{1}{-\lambda^1_1} \mW{1}{-\lambda^1_2} \cdots
  \mW{\ell}{-\lambda^\ell_1} \mW{\ell}{-\lambda^\ell_2} \cdots.
\end{equation}
This is equal to $\ve_1^{N(\lambda)} W[\blam]$, where $N(\lambda)
= \sum_i (d_i + 1) |\lambda^i|.
$

We also need
\begin{equation}
  \begin{split}
    & \bW{\kappa}m = (\ve_1\ve_2)^{-1} \mW{\kappa}m,
\\
   &\widehat W[\blam] \defeq 
     \bW{1}{-\lambda^1_1} \bW{1}{-\lambda^1_2} \cdots
  \bW{\ell}{-\lambda^\ell_1} \bW{\ell}{-\lambda^\ell_2} \cdots
  =
   (\ve_1\ve_2)^{-l(\blam)} \widetilde W[\blam],
  \end{split}
\end{equation}
where $l(\blam)$ is the sum of lengths $l(\lambda^i)$. This is
supposed to be an operator coupled with the fundamental class
$[\AA^2]$ of $\AA^2$.

Let ${}_\bA M(\ba)$ be the
$\bA_T$-submodule of $M(\ba)$ spanned by
\begin{equation}
  \widetilde W[\blam]|\ba\rangle.
\end{equation}
It should be true that ${}_\bA M(\ba)\otimes_{\bA_T}\bF_T = M(\ba)$.
It should be also clear that ${}_\bA M(\ba)$ is invariant under
$\scW_\bA(\g)$. These assertions should follow from the appropriate
definition of the integral form of Verma modules.

We define ${}_\bA D(M(\ba))$ in the same way. We have ${}_\bA
D(M(\ba)) \cong {}_\bA M(-\ba)$. Then ${}_\bA M(\ba)$ should be
invariant under operators $\mW{\kappa}m$.
\begin{NB2}
  This is supposed to be an algebraic definition of the ordinary
  (equivariant intersection) homology group
  $\IH^{\mathrm{ord},\GG}_*(\Uh{d}) = H^{*-2h^\vee n}(\Uh{d},
  \IC_{\Uh{d}})^\vee = H^{2h^\vee n-*}_c(\Uh{d},
  \IC_{\Uh{d}})$. Here we identify $\bA_T$ with
  $H^*_{\TT}(\mathrm{pt})$.

  The reason is the operator $\mW{\kappa}m$ is coupled with the
  point class $\ve_1\ve_2 = [0]$, and hence is well-defined on the
  ordinary (equivariant intersection) homology group
  $\IH^{\mathrm{ord},\GG}_*(\Uh{d}) = H^{*-2h^\vee n}(\Uh{d},
  \IC_{\Uh{d}})^\vee$.
\end{NB2}%

Let ${}^\bA M(\ba)$ be the $\bA_T$-submodule of $M(\ba)$ consisting of
elements $m$ such that $\langle f|m\rangle\in\bA_T$ for all
$f\in{}_\bA M(-\ba)$.
\begin{NB2}
  This is supposed to be an algebraic definition of the Borel-Moore
  (equivariant intersection) homology group
  $\IH^{\GG}_*(\Uh{d}) = H^{2h^\vee n-*}(\Uh{d},
  \IC_{\Uh{d}})$.

  However we do not prove ${}^\bA M(\ba) \cong \IH_*^\GG(\Uh{d})$. At
  least we need $|1_n\rangle \in {}^\bA M(\ba)$ to guarantee that
  $|1_n\rangle$ and $|w_n\rangle$ are equal up to a multiple of a
  rational number. For this statement, it is enough to check that
  ${}_\bA M(\ba)\subset \IH^{\mathrm{ord},\GG}_*(\Uh{d})$.
\end{NB2}%

We have a pairing
\begin{equation}
  \langle\ |\ \rangle \colon
  {}_\bA M(-\ba)\otimes_{\bA} {}^\bA M(\ba)\to \bA_T.
\end{equation}

From \eqref{eq:Wh} we have
\begin{equation}
  \langle -\ba | \theta(\widetilde W[\blam]) \ve_1^{-h^\vee n} | w_n\rangle
  =
    \begin{cases}
    1 & \text{if $\blam=\blam_0$},
    \\
    0 & \text{otherwise}.
  \end{cases}
\end{equation}
Therefore $\ve_1^{-h^\vee n} | w_n\rangle \in {}^\bA M(\ba)$.

The following is a working hypothesis. Once the definition of the
$\bA$-form of the Verma module is settled, it should follow
automatically.
\begin{Conjecture}
  Let ${}_\bA F(\ba)$ be the $\bA_T$-form of the Fock module of the
  Heisenberg algebra generated by $|\ba\rangle$ and $\widetilde P^i_n$.
  Then we have an embedding ${}_\bA M(\ba) \to {}_\bA F(\ba)$,
  compatible with the embedding $\scW_\bA(\g)\to \widetilde H^0_\bA(\g)$.
\end{Conjecture}
\end{NB}

\begin{Lemma}\label{lem:inv}
  $M_\bA(\ba)$ is invariant under $\bW{\kappa}m$ with $m >
  0$. Equivalently $D(M_\bA(\ba))$ is invariant under
  $\bW{\kappa}{-m}$ with $m > 0$.
\end{Lemma}

\begin{NB}
  This is probably because $\bW{\kappa}m$ is an annihilation operator
  coupled with $[\AA^2]$, and hence is well-defined on
  $\IH^{\mathrm{ord},\GG}_*(\Uh{d})$. But at this moment, I do not know
  how to prove this assertion geometrically.

  Let us prove the assertion algebraically.
\end{NB}

\begin{proof}
  Recall that $M_\bA(\ba)$ is graded by the instanton number $d$: $M_\bA(\ba) =
  \bigoplus_d M_{d,\bA}$. In algebraic terms, it is the grading by
  $L_0$.
  \begin{NB}
    $M_{d,\bA} =  \{x\in M_\bA(\ba)\mid L_0 x = d x\}$.
  \end{NB}
Let us take $\bW{\kappa}m$ with $m > 0$. We show
\begin{equation}
  \bW{\kappa}{m} x \in M_{d-m,\bA}
\end{equation}
for any $x\in M_{d,\bA}$ by an induction on $d$. If $d = 0$, we have
$\bW{\kappa}m x = 0$. Therefore the assertion is true.

Suppose that the statement is true for $d' < d$. We may assume $x =
\mW{\kappa'}{-n} x'$ with $n > 0$, $x'\in M_{d-n,\bA}$ by
\propref{prop:span}. Since $\bW{\kappa}m x'\in M_\bA(\ba)$ by the
induction hypothesis, it is enough to show that $[\bW{\kappa}m,
\mW{\kappa'}{-n}] x'\in M_\bA(\ba)$.
In the Heisenberg algebra, we have $[a, b] \in \ve_1 \ve_2 \widetilde
H^0_\bA(\g)$ for $a$, $b\in \widetilde H^0_\bA(\g)$ from the relation
\eqref{eq:mPrel}.
Since $\scW_\bA(\g)\to\widetilde H^0_\bA(\g)$ is an embedding, we have
the same assertion for $\scW_\bA(\g)$. Therefore the assertion follows.
\end{proof}

Let $\bR\subset\bF = \Q(\ve_1,\ve_2)$ be the local ring of regular
functions at $\ve_1 = \ve_2 = 0$.
Let $\bR_T = \bR(\ba)$.
We set
\begin{equation}\label{eq:90}
  \begin{split}
  & M_\bR(\ba) = M_\bA(\ba)\otimes_{\bA_T} \bR_T,\quad
  D(M_\bR(-\ba)) = D(M_\bA(-\ba))\otimes_{\bA_T} \bR_T,
\\
  & N_\bR(\ba) = N_\bA(\ba)\otimes_{\bA_T} \bR_T,\quad
  D(N_\bR(-\ba)) = D(N_\bA(-\ba))\otimes_{\bA_T} \bR_T.
  \end{split}
\end{equation}
These modules are the localization with respect to the ideal
\begin{equation}
  \Ker\left(\bA_T = \CC[\ve_1,\ve_2,\ba] = \CC[\operatorname{Lie}\TT]
    \to
    \CC[\ba]
    = \CC[\operatorname{Lie}T]
  \right)
\end{equation}
consisting of polynomials vanishing on $\operatorname{Lie}T$.

From the definition, operators $\mW{\kappa}n$ are well-defined on 
four modules in \eqref{eq:90}.
Moreover operators $\bW{\kappa}n$ and $\bW{\kappa}{-n}$ are
well-defined on $M_\bR(\ba)$ and $D(M_\bR(\ba))$ respectively if $n >
0$ by \lemref{lem:inv}.
\begin{NB}
  Therefore $\mW{\kappa}n$ and $\mW{\kappa}{-n}$ vanish on
  $M_\bR(\ba)$ and $D(M_\bR(\ba))$ respectively if $n > 0$.
\end{NB}

By the localization theorem, the first and the third homomorphisms
in \eqref{eq:4sp} become isomorphisms over $\bR_T$. Therefore
\begin{equation}\label{eq:85}
%\begin{multline}
%  \IH^*_{\TT,c}(\Uh{d})\otimes_{\bA_T}\bR_T
  M_\bR(\ba)
  \xrightarrow{\cong}
  N_\bR(\ba),
%  H^*_{\TT,c}(\Uh[T]{d}, \Phi_{T,G}(\IC(\Uh{d})))\otimes_{\bA_T}\bR_T,
%\\
\qquad
  D(N_\bR(-\ba))
%  H^*_{\TT}(\Uh[T]{d}, \Phi_{T,G}(\IC(\Uh{d})))\otimes_{\bA_T}\bR_T
  \xrightarrow{\cong}
  D(M_\bR(-\ba)).
%  \IH^*_\TT(\Uh{d})\otimes_{\bA_T}\bR_T.
%\end{multline}
\end{equation}

Recall that we have Heisenberg operators $P^{i}_n =
(\ve_1\ve_2)^{-1}\widetilde P^{i}_n$, coupled with the fundamental
class $1\in H^{0}_\TT(\AA^2)$. Let
\begin{equation}
  P[\blam] = 
  P^{1}_{-\lambda^1_1} P^{1}_{-\lambda^1_2} \cdots
  P^{\ell}_{-\lambda^\ell_1}
  P^{\ell}_{-\lambda^\ell_2} \cdots
\end{equation}
for $i=1,\dots,\ell$, $n\in\Z$, and an $\ell$-partition $\blam =
(\lambda^1,\dots,\lambda^\ell)$.
\begin{NB}
  For the consistency of notations, it is better to write
  $\widehat P[\blam]$.
\end{NB}%
It is a well-defined operator on $D(M_\bR(-\ba))$ by the proof of
\lemref{lem:inv}.

Replacing $P^i_m$ by $\widetilde{P}^i_m$, we introduce similar
operators $\widetilde P[\blam]$.

\begin{Proposition}\label{prop:SpanP}
  We have
  \begin{equation}
    D(M_\bR(-\ba)) = \operatorname{Span}_{\bR_T}\{ P[\blam]|\ba\rangle\},
  \end{equation}
where $\blam$ runs all $\ell$-partitions.
\end{Proposition}

\begin{proof}
  Thanks to \eqref{eq:85}, it is enough to show the assertion for
  $D(N_\bR(-\ba))$. We shall prove that $D(N_\bA(-\ba))$ is spanned by
  $P[\blam]$ over $\bA$.

  Recall that $N_\bA(\ba) = \operatorname{Span}_{\bA_T}\{ \widetilde
  P[\blam]|\ba\ra\}$, see \eqref{eq:103}.
  From the commutation relation
  \begin{equation}
    [P^i_m, \widetilde{P}^j_n] 
  = - m\delta_{m,-n}(\alpha_i,\alpha_j),
  \end{equation}
  we clearly have a perfect pairing between $N_\bA(-\ba)$ and
  $\operatorname{Span}_{\bA_T}\{ P[\blam]|\ba\ra\}$. The assertion follows.
\end{proof}

\begin{NB}
  The original formulation was the following, but I think that the
  above could be proved first and then next the following might
  follow. But I do not have the precise argument right now.
\end{NB}

\begin{NB}
  From now on, I will present a speculation. First I assume

\begin{Conjecture}\label{span}
  We have
  \begin{equation}
    {}^\bR M(\ba) = \operatorname{Span}_{\bR_T}\widehat W[\blam]|\ba\rangle.
  \end{equation}
\end{Conjecture}
\end{NB}

\begin{NB}
  It is probably enough to check this for $\ve_1 = \ve_2 = 0$, $\ba =
  0$ for our purpose. Then it might be possible...... But I am not sure.

  Also I need to warn you as $\bW{\kappa}m$ ($m < 0$) does not form a
  subalgebra. We do not have the triangular decomposition for the
  $W$-algebra.

  The following statement follows if we consider ${}^\bA M(\ba)$
  instead of ${}^\bR M(\ba)$. Therefore it is not true.
  \begin{NB2}
  And it seems to me that this conjecture implies that \( |1_n\rangle
  \) is a multiple of $(\bW{\ell}{-1})^n$ by the degree reason. This
  is fine at $\ve_1, \ve_2 = 0$, but looks too strong to me for
  generic $\ve_1$, $\ve_2$.
  \end{NB2}
\end{NB}

\subsection{Pairing at \texorpdfstring{$\ve_1, \ve_2 = 0$}{epsilon1, epsilon2=0}}

\begin{NB}
    I correct that $\la\ ,\ \ra$ is a pairing between $M_\bF(-\ba)$
    and $M_\bF(\ba)$. Sep.~15, 2014
\end{NB}

We consider the pairing $\langle\ , \ \rangle$ on
$M_\bF(-\ba)\otimes_{\bF_T} M_\bF(\ba)$ in
\subsecref{sec:kac-shapovalov-form}, and restrict it to
$D(M_\bR(\ba))\otimes_{\bR_T} D(M_\bR(-\ba))$.

\begin{Lemma}
  We decompose $D(M_\bR(\pm\ba))$ as $\bigoplus D(M^\pm_{d,\bR})$ by the
  instanton number $d$ as before.

  \textup{(1)}
  $(\ve_1\ve_2)^d \langle\ ,\ \rangle$ takes values in $\bR_T$ on
  $D(M^-_{d,\bR})\otimes D(M^+_{d,\bR})$.

  \textup{(2)} Let $\la\ ,\ \ra_0$ be its specialization at $\ve_1 =
  \ve_2 = 0$. For $m > 0$, we have
\begin{equation}\label{eq:vanish}
  \la x , \bW{\kappa}{-m} y\ra_0 = 
  \begin{cases}
  (-1)^{d_\kappa+1} \la\mW{\kappa}{m} x,  y\ra_0 & \text{if $m= 1$},
\\
    0 & \text{otherwise}.
  \end{cases}
\end{equation}
\end{Lemma}

Since $\la\ ,\ \ra$ is symmetric, (2) remains true when we exchange
the first and second entries.
\begin{NB}
    Added, Sep.~17, 2014.
\end{NB}

\begin{NB}
  $\bW{\kappa}0$ is not well-defined at $\ve_1,\ve_2 = 0$. On the
  other hand $\mW{\kappa}0$ is the invariant polynomial in $\ba$ of
  degree $d_\kappa+1$.
\end{NB}

\begin{proof}
  (1) Thanks to \eqref{eq:85}, it is enough to show the assertion for
  $D(N_{\bR}(\ba))\otimes D(N_\bR(-\ba))$. By \eqref{eq:86}
  \begin{NB}
    More precisely,
    $D(N_\bA(-\ba)) = \bigoplus_d 
    H^*_\TT(\Uh[T]d, \Phi_{T,G}^{B_-}(\IC(\Uh{d})))$
  \end{NB}
  and $\Uh[T]d = S^d\AA^2$, it is enough to show that the intersection
  pairing $\la\ ,\ \ra$ on $H^*_\TT(S^d\AA^2)$ satisfies the same
  property. Note that $S^d\AA^2$ is a smooth orbifold. Since we only
  have a single fixed point $d\cdot 0$ in $S^d\AA^2$ and the weight of
  the tangent space there is $\ve_1,\ve_2,\ve_1,\ve_2,\dots$ ($d$
  times), the fixed point formula implies the assertion.

  (2) Suppose $x\in D(M^+_{d,\bR})$, $y\in D(M^-_{d-m,\bR})$ with $m >
  0$. Then
  \begin{equation}
    (\ve_1\ve_2)^d \langle x, \bW{\kappa}{-m} y\rangle
    =  (-1)^{d_i+1}(\ve_1\ve_2)^{m-1} (\ve_1\ve_2)^{d-m}
    \langle \mW{\kappa}{m} x, y\rangle
  \end{equation}
  by \eqref{eq:87}. Now we specialize $\ve_1, \ve_2 = 0$ to get the
  assertion.
  \begin{NB}
    Here is an original algebraic argument. If we use
    \propref{prop:SpanP} instead of Conjecture~\ref{span}, the
    argument works. But I do not give the proof of
    \propref{prop:SpanP} yet, so I have used a geometric argument above.

  Let $M(\ba) = \bigoplus_{n\ge 0} M(\ba)_n$ as before. This is an
  orthogonal decomposition. We prove (1) by an induction on $n$. The
  assertion is true at $n=0$. We assume that the assertion is true for
  $n' < n$.

  Suppose $x\in {}^\bR M(\ba)_n$, $y\in {}^\bR M(\ba)_{n-m}$ with $m >
  0$. Then 
  \begin{equation}
    (\ve_1\ve_2)^n \langle x | \bW{\kappa}{-m} y\rangle
    =  (-1)^{d_i+1}(\ve_1\ve_2)^{m-1} (\ve_1\ve_2)^{n-m}
    \langle \mW{\kappa}{m} x | y\rangle.
  \end{equation}
  Since $\mW{\kappa}{m} x\in {}^\bR M(\ba)_{n-m}$, the right hand side is
  in $\bR_T$ by the induction hypothesis. Since we assume
  Conjecture~\ref{span}, the assertion is true also for ${}^\bR
  M(\ba)_n$. Therefore (1) is proved.

  Evaluating at $\ve_1, \ve_2 = 0$, we obtain (2).
  \end{NB}%
\end{proof}

Let us consider $M_0(\pm\ba) \defeq D(M_\bR(\pm\ba))\otimes_{\bR_T}
\CC/\operatorname{Rad}\la\ ,\ \ra_0$, where $\bR_T\to\CC$ is the
evaluation at $\ve_1 = \ve_2 = 0$, and $\operatorname{Rad}\la\ ,\ \ra_0$
is the radical of $\la\ ,\ \ra_0$. Then \eqref{eq:vanish} implies that
$\bW{\kappa}{-m} = 0$ if $m > 1$, and $\bW{\kappa}{-1}$, $\mW{\kappa}1$
are well-defined on $M_0(\pm\ba)$.
\begin{NB}
  But $\mW{\kappa}m$ with $m > 1$ may not be well-defined on
  $M_0(\ba)$, as we cannot say anything on $(\mW{\kappa}m x|y)_0$ for
  $m > 1$.
\end{NB}%

\begin{NB}
  SUMMARY : As operators on $M_0(\ba)$, 
  \begin{equation}
    \text{$\bW{\kappa}{-n}$ is}
    \begin{cases}
      \text{well-defined} & \text{if $n = 1$,}\\
      0 & \text{if $n > 1$,} \\
      \text{not defined, at least so far} & \text{if $n\le 0$.}
    \end{cases}
  \end{equation}
  
\end{NB}

\begin{Proposition}\label{prop:Pat0}
  \textup{(1)} $P^{(i)}_{-m} = 0$ if $m > 1$, and $P^{(i)}_{-1}$,
  $\widetilde P^{(i)}_1$ are well-defined on $M_0(\pm\ba)$. And we have
  \begin{equation}\label{eq:88}
    \la x, P^{(i)}_{-1}y\ra_0 =
    - \la \widetilde P^{(i)}_1 x, y\ra_0.
  \end{equation}

\textup{(2)} We have commutation relations
\begin{equation}
  [P^{(i)}_{-1}, P^{(j)}_{-1}] = 0,\quad
  [\widetilde P^{(i)}_{1}, \widetilde P^{(j)}_{1}] = 0,\quad
  [\widetilde P^{(i)}_1, P^{(j)}_{-1}] = -(\alpha_i, \alpha_j).
\end{equation}

\textup{(3)} $M_0(\pm\ba)$ is isomorphic to the polynomial ring in
$P^{(i)}_{-1}$ \textup($i=1,\dots,\ell$\textup). The pairing $\la\ ,\
\ra_0$ is the induced pairing on the symmetric power from the pairing
  \begin{equation}
  \la -\ba | P^{(i)}_{1} P^{(j)}_{-1}| \ba\ra_0 
  = (\alpha_i,\alpha_j).
  \end{equation}
  \begin{NB}
      The above is modified from 
  \begin{equation}
  \la P^{(i)}_{-1} \ba, P^{(j)}_{-1} \ba\ra_0 
  = -(\alpha_i,\alpha_j)
  \end{equation}
  on Sep.~17, 2014.
  \end{NB}
\end{Proposition}

\begin{proof}
The same argument as above shows (1).
\begin{NB}
  In fact, I must be a little more careful as the pairing on
  $N_\bA(\ba)$ is defined naturally on $N_\bA(\ba)$ and
  $D(N_\bA(\ba))$, where the latter is defined by the hyperbolic
  restriction with respect to the opposite Borel. But it is just
  given by $-1$, I guess. So this is OK.
\end{NB}
   
By \propref{prop:SpanP} and (1), $M_0(\ba)$ is spanned by monomials in
$P^{(i)}_{-1}$ applied to $|\ba\ra$.

(2) follows from \propref{prop:HeisRel}.

(3) Let us replace $P^{(i)}_{-1}$, $\widetilde P^{(i)}_1$ by
$Q^{(i)}_{-1}$, $\widetilde Q^{(i)}_1$ corresponding to an orthonormal
basis of $\h$
\begin{NB}
  i.e., $Q^{(i)}_{-1} = \sqrt{C}^{ij} P^{(i)}_{-1}$, and
  the same for $\widetilde Q^{(i)}_1$, where
  $\sqrt{C}^{ij}$ is the square root of the inverse
  of the matrix $C_{ij} = (\alpha_i,\alpha_j)_{ij}$,
\end{NB}%
so that the commutation relation is $[\widetilde Q^{(i)}_1,
Q^{(j)}_{-1}] = -\delta_{ij}$. Then \eqref{eq:88} implies that
monomials in $Q^{(i)}_{-1}$ are orthogonal. More precisely, the
pairing is the standard one on $\CC[Q^{(i)}_{-1}]$
\begin{equation}
  \la -\ba | (Q^{(i)}_{1})^n (Q^{(i)}_{-1})^m | \ba\ra_0 = n! \delta_{mn},
\end{equation}
\begin{NB}
Modified on Sep.~14, 2014 : 
up to sign:    
\begin{equation}
  \la (Q^{(i)}_{-1})^n \ba, (Q^{(i)}_{-1})^m \ba\ra_0 = (-1)^n n! \delta_{mn},
\end{equation}
\end{NB}%
and the pairing factors on $M_0(\ba) = \CC[Q^{(1)}_{-1}]\otimes\cdots\otimes
\CC[Q^{(\ell)}_{-1}]$.
\begin{NB}
\begin{equation*}
  \begin{split}
  & \la (Q^{(i)}_{-1})^n \ba, (Q^{(i)}_{-1})^n \ba\ra_0 = 
  \la (Q^{(i)}_{-1})^{n-1} \ba, \widetilde Q^{(i)}_1
     (Q^{(i)}_{-1})^n \ba\ra_0 
\\
=\; & 
  - \la (Q^{(i)}_{-1})^{n-1} \ba, (Q^{(i)}_{-1})^{n-1} \ba\ra_0 
  + \la (Q^{(i)}_{-1})^{n-1} \ba, Q^{(i)}_{-1} \widetilde Q^{(i)}_1
  (Q^{(i)}_{-1})^{n-1} \ba\ra_0 
\\
=\; & \cdots
\\
=\; & 
  - n \la (Q^{(i)}_{-1})^{n-1} \ba, (Q^{(i)}_{-1})^{n-1} \ba\ra_0.
  \end{split}
\end{equation*}
\end{NB}%
This proves the assertion.
\end{proof}

\begin{NB}
  This paragraph (May 22) was wrong.

Therefore we may assume that $\ve_1^{-h^\vee n}|w_n$ is a polynomial
in $\bW{\kappa}{-1}$ ($i=1,\dots,\ell$) in $M(\ba)_0$. Then the Whittaker
vector condition implies that it is a multiple of
$(\bW{\ell}{-1})^n$. The absolute constant can be computed from the
commutator $[\bW{\ell}{-1}, \mW{\ell}{1}]$ at $M(\ba)_0$. This is what
expected from the geometry. See \propref{prop:two}(2) below.
\end{NB}

\subsection{Proof, a geometric part}

\begin{Lemma}\label{lem:geomind}
  The first equality of \eqref{eq:91} is true.
\end{Lemma}

\begin{proof}
  We have a natural homomorphism $\IH^*_\TT(\Uh{d})\to
  H_*^{\TT}(\Uh{d})$ and the image of $1^d$ is the fundamental class
  $[\Uh{d}]$. Then $\la 1^d| 1^d\ra$ is equal to
  $\iota_*^{-1}[\Uh{d}]$, where $\iota\colon \{ d\cdot 0\}\to \Uh{d}$
  is the embedding of the $\TT$-fixed point $d\cdot 0$, and we use the
  localization theorem to invert $\iota_*\colon H^\TT_*(\{d\cdot
  0\})\to H^\TT_*(\Uh{d})$ over $\bF_T$.

  Let us consider the embedding $\xi\colon (\Uh{d})^T = S^d\AA^2\to
  \Uh{d}$ of the $T$-fixed point set. Then
  \begin{equation}\label{eq:100}
    \xi_*\colon H_*^\TT(S^d\AA^2)\to H^\TT_*(\Uh{d})
  \end{equation}
  is an isomorphism over $\bR_T$. Since $H_*^\TT(S^d\AA^2)\cong
  \bA_T[S^d\AA^2]$, we have
  \begin{equation}
    \xi_*^{-1}[\Uh{d}] = f_d(\ba,\ve_1,\ve_2)[S^d\AA^2]
  \end{equation}
  for $f_d(\ba,\ve_1,\ve_2)\in \bR_T$.
  
  We have $\iota_* = \xi_* \zeta_*$ for
  $\zeta\colon \{d\cdot 0\}\to S^d\AA^2$, and $\zeta_*^{-1}[S^d\AA^2]
  = (\ve_1\ve_2)^{-d}/d!$. Therefore
  \begin{equation}
    d! \left.(\ve_1\ve_2)^d \la 1^d|1^d\ra\right|_{\ve_1,\ve_2=0}
    = f_d(\ba,0,0).
  \end{equation}
  We replace the group $\TT$ by $T$ in \eqref{eq:100} and denote the
  homomorphism by $\xi^T_*$, i.e.,
\(
  \xi^T_*\colon H_*^T(S^d\AA^2)\to H^T_*(\Uh{d}).
\)
It is an isomorphism over $\CC(\ba)$. Then we have
\begin{equation}\label{eq:101}
  (\xi^T_*)^{-1}[\Uh{d}] = f_d(\ba,0,0)[S^d\AA^2],
\end{equation}
where $[\Uh{d}]$, $[S^d\AA^2]$ are considered in $T$-equivariant
homology groups.

Let us take the projection $a\colon \AA^2\to \AA^1$ and the
factorization morphism $\pi^d_{a,G}\colon \Uh{d}\to S^d\AA^1$. 
Let $S^da\colon S^d\AA^2\to S^d\AA^1$ denote the induced projection.
Let $(S^d\AA^1)_0$ be the open subset of $S^d\AA^1$ consisting of
distinct $d$ points. Then $\xi$ induces a morphism between inverse
images $(S^d a)^{-1}(S^d\AA^1)_0$ and
$(\pi^d_{a,G})^{-1}(S^d\AA^1)_0$. We get
\begin{equation}
  (\xi^T_*)^{-1}[(\pi^d_{a,G})^{-1}(S^d\AA^1)_0] 
  = f_d(\ba,0,0)[(S^d a)^{-1}(S^d\AA^1)_0]
\end{equation}
by restricting \eqref{eq:101} to open subsets. Now by the
factorization we deduce $f_d(\ba,0,0) = f_1(\ba,0,0)^d$.
\end{proof}

\begin{Remark}
  This result is also a simple consequence of a property of Nekrasov's
  partition function
\begin{equation}
  Z^{\text{inst}}(\ve_1,\ve_2,\ba,\Lambda) \defeq
  \sum_{d=0}^\infty \la 1^d|1^d\ra \Lambda^{2h^\vee d}
\end{equation}
stating that
\begin{equation}
  {\ve_1\ve_2} \log Z^{\text{inst}}(\ve_1,\ve_2,\ba,\Lambda)
  = F_0^{\text{inst}}(\ba,\Lambda) + o(\ve_1,\ve_2)
\end{equation}
at $\ve_1=\ve_2=0$. This property was proved by
\cite{MR2199008,NekrasovOkounkov} for type $A$ and by
\cite{BraInstantonCountingII} for general $G$.
\end{Remark}

\begin{NB}
  Let $F_0^{\text{inst}}(\ba)_{1}$ be the $1$-instanton contribution
  (i.e., the coefficient of $\Lambda^{2h^\vee}$) in
  $F_0^{\text{inst}}(\ba,\Lambda)$. Then we have
\begin{equation}
  \left.(\ve_1\ve_2)^d \la 1^d|1^d\ra\right|_{\ve_1,\ve_2=0}
  = \frac1{d!} \left(F_0^{\text{inst}}(\ba)_{1} \right)^d.
\end{equation}
This proves the assertion.
\end{NB}%

\begin{NB}
Here is a record of an old paragraph:  

From a geometric side, we have
\begin{Proposition}\label{prop:two}
  \textup{(1)} $(\ve_1\ve_2)^n \langle 1_n| 1_n\rangle$ is regular
at $\ve_2 = 0$.

\textup{(2)} We have the following relation
\begin{equation}\label{eq:factor}
  \left.(\ve_1\ve_2)^n \langle 1_n| 1_n\rangle\right|_{\ve_2=0}
  = \frac1{n!} 
  \left(\left.\ve_1\ve_2\langle 1_1| 1_1\rangle\right|_{\ve_2=0}
    \right)^n.
\end{equation}
\end{Proposition}

This is true even at $\ve_1\neq 0$.
\end{NB}

\subsection{Proof, a representation theoretic part}\label{sec:proof}

We shall complete the proof of the second equation in \eqref{eq:91} in
this subsection.

Let $F^{(\kappa)}\in S(\h)^W$ be one of generators as in
\subsecref{sec:oppos-spectr-sequ}. It has degree $d_\kappa+1$.

\begin{NB}
The following was conjecture, but now proved in \lemref{lem:c1}.
\begin{Conjecture}\label{c1}
  Let $f^{(\kappa)}(x_1,\dots,x_\ell)$ be the $\kappa^{\mathrm{th}}$
  generator of the invariant polynomial ring $S(\h)^W$. Then the
  following relation holds on $M_0(\ba)$:
  \begin{equation}
    \mW{\kappa}1 = \sum_i f^{(\kappa)}(\widetilde P^{(1)}_0,\dots,
    \widetilde P^{(i)}_{1},\dots, \widetilde P^{(\ell)}_0)
  \end{equation}
  \textup($i^{\mathrm{th}}$ entry is $\widetilde P^{(i)}_{1}$, other
  entries are $\widetilde P^{(j)}_0$\textup)
\end{Conjecture}
\end{NB}

\begin{Lemma}\label{lem:c1}
  Following relations hold as operators on
  $D(M_\bR(-\ba))\otimes_{\bR_T}\CC$\textup:
\begin{align}
  & \mW{\kappa}{1} = \sum_i F^{(\kappa)} (a^1,\dots,
      \underbrace{\widetilde P^{(i)}_{1}}_{\text{\rm $i^{\mathrm{th}}$ factor}},
      \dots, a^\ell),\label{eq:96}
\\
  & \bW{\kappa}{-1}
  = \sum_i F^{(\kappa)} (a^1,\dots,
      \underbrace{P^{(i)}_{-1}}_{\text{\rm $i^{\mathrm{th}}$ factor}},
      \dots, a^\ell).\label{eq:97}
\end{align}
\end{Lemma}

\begin{proof}
  At first sight, the formula \eqref{eq:93} seems to imply
  $\mW{\kappa}{-1} = 0$, and hence also $\mW{\kappa}{1} = 0$ thanks to
  the anti-involution $\theta$.
  But \eqref{eq:93} is the formula in the $W$-algebra at
  $\ve_1=\ve_2=0$, and we want to consider $\mW{\kappa}{1}$ on
  $D(M_\bR(-\ba))$. Since the highest weight $\lambda = \ba/\ve_1 -
  \rho$ cannot be specialized at $\ve_1=0$, it could be
  nontrivial.

  Let $\mW{\kappa}{}$ be the state corresponding to the field
  $Y(\mW{\kappa}{},z) = \sum \mW{\kappa}n z^{-n-d_\kappa-1}$ as in
  \eqref{eq:94}. By \eqref{eq:93} we have
\begin{equation}
  \mW{\kappa}{} = \mW{\kappa}{-d_\kappa-1}|0\ra
  \begin{NB}
    =\left.F^{(\kappa)}(\sum_{n<0} \widetilde P^{(i)}_n z^{-n-1})|0\ra
      \right|_{z=0}
  \end{NB}
  = F^{(\kappa)}(\widetilde P^{(i)}_{-1})|0\ra
\end{equation}
at $\ve_1=\ve_2=0$. It implies that
\begin{equation}\label{eq:99}
  Y(\mW{\kappa}{},z) = \normal{F^{(\kappa)}(\widetilde P^{(i)}(z))}
  + o(\ve_1,\ve_2),
\end{equation}
where $o(\ve_1,\ve_2)$ is a field in $\scW_\bA(\g)$ which vanishes at
$\ve_1=\ve_2=0$.

Let the field act on $M_\bR(\ba)$ and specialize at
$\ve_1=\ve_2=0$. The point is that $\widetilde P^{(i)}_0$ acts on
$M_\bR(\ba)$ by $a^i$ at $\ve_1=\ve_2=0$. Therefore the field $
\widetilde P^{(i)}(z) = \sum_n \widetilde P^{(i)}_n z^{-n-1}$ is
specialized to
\begin{equation}\label{eq:95}
  a^i z^{-1} + \sum_{n < 0} \widetilde P^{(i)}_n z^{-n-1}
\end{equation}
on $M_\bR(\ba)$.

Let us specialize \eqref{eq:99} at $\ve_1=\ve_2=0$. Then $\widetilde
P^{(i)}(z)$ is replaced by \eqref{eq:95}, and the normal ordering by
the usual multiplication. Therefore we obtain
\begin{equation}
  Y(\mW{\kappa}{},z) 
  = F^{(\kappa)}(a^i z^{-1} + \sum_{n < 0} \widetilde P^{(i)}_n z^{-n-1}).
\end{equation}
\begin{NB}
  From this, we see that $\mW{\kappa}{n} = 0$ for $n > 0$. This is
  what we have already observed. We also see that $\mW{\kappa}{n}$ for
  $n < -1$ can be nontrivial, but is also possible to compute.
\end{NB}%
Taking coefficients of $z^{-d_\kappa}$ and then applying $\theta$, we
obtain \eqref{eq:96}.

Next we study the action of $Y(\mW{\kappa}{},z)$ on $D(M_\bR(-\ba))$.
Let us consider $\mW{\kappa}{-1}$ in \eqref{eq:99}. So we take
coefficients of $z^{-d_\kappa}$. The term $o(\ve_1,\ve_2)$ can be
represented as a linear combination of monomials in $\widetilde
P^{(i)}_m$ with coefficients in the maximal ideal of $\bR$. We have at
least one $\widetilde P^{(i)}_m$ with $m < 0$ in each monomial. It can
be divided, as an operator on $D(M_\bR(-\ba))$, by $\ve_1\ve_2$ thanks
to \lemref{lem:inv}. Therefore $o(\ve_1,\ve_2)/\ve_1\ve_2$ still
specialized to $0$ at $\ve_1=\ve_2=0$. Therefore \eqref{eq:99} implies
\eqref{eq:97}.
\begin{NB}
  Let us consider $\bW{\kappa}n$ with $n < -1$. We must have at least
  two $\widetilde P^{(i)}_n$ in the coefficient in $F^{(\kappa)}(a^i
  z^{-1} + \sum_{n < 0} \widetilde P^{(i)}_n z^{-n-1})$. Therefore at
  least one $\widetilde P^{(i)}_n$ remains even after we divide it by
  $\ve_1\ve_2$. Therefore $\bW{\kappa}n = 0$ at $\ve_1=\ve_2=0$.
\end{NB}%
\end{proof}

\begin{Lemma}
  The determinant of the matrix
  \begin{equation*}
    \left(\frac{\partial F^{(\kappa)}(a^i)}{\partial a^i}\right)_{
      i,\kappa=1,\dots,\ell}
  \end{equation*}
  is a nonzero constant multiple of the discriminant $\Delta(\ba)$.
\end{Lemma}

\begin{proof}
  Consider $F = (F^{(1)},\dots,F^{(\ell)})$ as the morphism from $\h$
  to $\h/W$, written in a coordinate system on $\h/W$. Then the matrix
  in question is the differential of $F$. Since $\h\to\h/W$ is a
  covering branched along root hyperplanes, we deduce that a) its
  determinant is nonzero, and b) it is divisible by $\Delta(\ba)$. The
  degree of the determinant is the sum $\sum d_\kappa$, which is equal
  to the number of positive roots. Therefore we get the assertion.
\end{proof}

Since $\Delta(\ba)$ is invertible in $\CC(\ba)$, we deduce
\begin{Lemma}
  $M_0(\ba)$ is isomorphic to the polynomial ring in
  $\bW{\kappa}{-1}$ \textup($\kappa=1,\dots,\ell$\textup).
\end{Lemma}

Now the specialization of the Whittaker vector $|w^d\ra$ in $M_0(\ba)$
is characterized by the conditions
\begin{equation}
  \mW{\kappa}{1} | w^d\ra =
  \begin{cases}
    |w^{d-1}\ra & \text{if $\kappa=\ell$,}\\
    0 & \text{if $\kappa\neq l$.}
  \end{cases}
\end{equation}
The existence and the uniqueness in $M_0(\ba)$ are proved exactly as
in \subsecref{subsec:Part}.
Moreover the pairing $\la w^d| w^d\ra_0$ is an entry of the inverse of
the matrix
\begin{equation}
  K_0^d \defeq (\la -\ba | \widetilde W[\bm] \widehat W[\bn] | 
  \ba\ra_0)_{\bm,\bn},
\end{equation}
where $\bm = (m_1,\dots,m_\ell)$, $\bn = (n_1,\dots,n_\ell)\in\Z_{\ge
  0}^\ell$ and
\begin{equation}
  \begin{split}
  & \widetilde W[\bm] := (\bW{1}{-1})^{m_1} \cdots (\bW{\ell}{-1})^{m_\ell},
\\
  & \widehat W[\bn] := (\mW{1}{1})^{n_1} \cdots (\mW{\ell}{1})^{n_\ell}.
  \end{split}
\end{equation}
Here multi-indices $\bm$, $\bn$ runs over $\sum m_\kappa = \sum
n_\kappa = d$ for each $d$.

Now the matrix $K_0^d$ is the $d^{\mathrm{th}}$ symmetric power of
$K_0^{d=1}$, and hence we complete the proof of \eqref{eq:91}.

\subsection{Type \texorpdfstring{$A$}{A}}\label{sec:Whit_typeA}

Let us consider the special case $\g=\algsl_r$ in this section.
Let us switch to the notation for $\gl_r$. We have standard generators
of the invariant polynomial ring:
\begin{equation}
  F^{(\kappa)} = \sum_{i_1 < i_2 < \cdots < i_\kappa}
  h^{i_1} h^{i_2} \cdots h^{i_p},
\end{equation}
where $(h^1,\dots,h^r)$ is the standard coordinate system of the
Cartan subalgebra of $\gl_r$ such that $(h^i,h^j) = \delta_{ij}$.

Let us denote by $\widetilde Q^{(i)}_n$, $Q^{(i)}_n$ the Heisenberg
algebra generators corresponding to $\widetilde P^{(i)}_n$,
$P^{(i)}_n$. Then
\begin{equation}
  \begin{split}
  \widehat W^{(\kappa)}_{-1}|\ba\rangle
  &= \sum_{i_1 < i_2 < \cdots < i_\kappa} \sum_{l=1}^p \widetilde
  Q^{(i_1)}_{0} \widetilde Q^{(i_2)}_{0} \cdots Q^{(i_l)}_{-1} \cdots
  \widetilde Q^{i_\kappa}_{0} |\ba\rangle
  \\
  &= \sum_{i_1 < i_2 < \cdots < i_\kappa} \sum_{l=1}^p a_{i_1} a_{i_2}
  \cdots \widehat{a_{i_l}} \cdots a_{i_\kappa} Q^{i_l}_{-1}|\ba\rangle.
  \end{split}
\end{equation}

We use the Heisenberg algebra commutation relation
\begin{equation}
  [\widetilde Q^i_1, Q^j_{-1}] = \delta_{ij} 
\end{equation}
to get
\begin{equation}
  \begin{split}
  \widetilde Q^i_1 W^{(p)}_{-1}|\ba\rangle
  &= \sum_{\substack{i_1 < i_2 < \cdots < i_p\\
    i_l = i}} 
    a_{i_1} a_{i_2} \cdots \widehat{a_{i_l}} \cdots a_{i_p} |\ba\rangle
\\
  &= \frac{\partial}{\partial a_i}
  e_p(\ba) |\ba\rangle,
  \end{split}
\end{equation}
where $e_p(\ba)$ is the $p^{\mathrm{th}}$ elementary symmetric
polynomial in $\ba$.

The determinant of the $r\times r$-matrix $(\partial e_p(\ba)/\partial
a_i )_{i,p=1,\dots, r}$ is equal to $\prod_{i < j} (a_i -
a_j)$. Therefore the matrix is invertible. This, in particular,
implies that $\{ W^{(p)}_{-1}|\ba\rangle \}_{p=1,\dots,r}$ form a
basis of $(M(\ba)_0)_1$.
\begin{NB}
  Degree $1$ part.
\end{NB}%

\begin{Proposition}
  The Whittaker vector $|w_1\rangle$ at the instanton number $1$ is
  given by
  \begin{equation}
    \sum_i \frac{Q^i_{-1}|\ba\rangle}{\prod_{j:j\neq i} a_j - a_i}.
  \end{equation}
\end{Proposition}

\begin{proof}
  We have
  \begin{equation}
    W^{(p)}_1 Q^i_{-1}|\ba\rangle
    = \frac{\partial}{\partial a_i} e_p(\ba)|\ba\rangle 
  \end{equation}
as above. Now it is elementary to check that
\begin{equation}
  \sum_i \frac{\frac{\partial}{\partial a_i} e_p(\ba)}
  {\prod_{j:j\neq i} a_j - a_i} = 0
\end{equation}
if $p < r$. If $p=r$, we have
\begin{equation}
  \sum_i \frac{\frac{\partial}{\partial a_i} e_p(\ba)}
  {\prod_{j:j\neq i} a_j - a_i} = 
  \sum_i \prod_{j:j\neq i} \frac{a_j}{a_j - a_i}
  = 1.
\end{equation}
\end{proof}

Now we have
\begin{equation}
  (w_1|w_1)_0 = \sum_i \prod_{j:j\neq i} \frac1{(a_j - a_i)^2}.
\end{equation}
This coincides with what is known from geometry.
\begin{NB}
  I am not careful enough for the sign.
\end{NB}%

\begin{NB}
\subsection{Virasoro case}

Let us consider the $\algsl_2$ case, i.e., when the $W$-algebra is the
Virasoro algebra. In this case we have the commutation relation
\begin{equation}\label{eq:VirComm}
  [L_n, L_{m}] = (n-m)L_{m+n} + \frac{n^3 - n}{12} \delta_{n,-m}
  (1 + \frac{6(\ve_1+\ve_2)^2}{\ve_1\ve_2}).
\end{equation}
Therefore
\begin{equation}
  \begin{split}
    & L_1 L_{-1}^n |\ba\rangle
\\
=\; & 2 L_0 L_{-1}^{n-1} | \ba\rangle 
    + L_{-1} L_1 L_{-1}^{n-1} |\ba\rangle 
\\
=\; & 2 (\Delta + n-1)
   L_{-1}^{n-1} | \ba\rangle 
    + L_{-1} L_1 L_{-1}^{n-1} |\ba\rangle 
= \cdots 
\\   
=\; & 
 2 \sum_{k=0}^{n-1} (\Delta + k) L_{-1}^{n-1} |\ba\rangle
= n(2\Delta+n-1) L_{-1}^{n-1}|\ba\rangle, 
  \end{split}
\end{equation}
where $L_0|\ba\rangle = \Delta|\ba\rangle$. Furthermore,
\begin{equation}
  \langle -\ba| L_1^n L_{-1}^n |\ba\rangle
  = n (2\Delta+n-1) \langle-\ba | L_1^{n-1} L_{-1}^{n-1}|\ba\rangle
  = \cdots
  = n! \prod_{k=0}^{n-1}(2\Delta+k).
\end{equation}
Hence we get
\begin{equation}
  (\ve_1\ve_2)^n \langle -\ba| L_1^n L_{-1}^n |\ba\rangle
  = n! \prod_{k=0}^{n-1} (2\ve_1\ve_2\Delta + k\ve_1\ve_2).
\end{equation}
We have
\begin{equation}\label{eq:L0}
  \Delta = \frac1{\ve_1\ve_2}(2a^2 - \frac12 (\ve_1+\ve_2)^2).
\end{equation}
\begin{NB2}
  I need to check the normalization.
\end{NB2}%
Therefore we have
\begin{equation}
  \left.(\ve_1\ve_2)^n \langle -\ba| L_1^n L_{-1}^n
    |\ba\rangle\right|_{\ve_2=0}
  = n! \left({4a^2} - \ve_1^2\right)^n.
\end{equation}
In the $\algsl_2$ case, we only have a single $\blam$ with $l(\blam) =
n$, i.e., $\blam = (1^n)$. Therefore $\left.(\ve_1\ve_2)^{-n}
  K^{\blam_0\blam_0}\right|_{\ve_2 = 0} = \left.(\ve_1\ve_2)^n \langle
  -\ba| L_1^n L_{-1}^n |\ba\rangle\right|_{\ve_2=0}^{-1}$.
  Thus the above computation show that the properties in
  \propref{prop:two} hold for $|w_n\rangle$.
  And $|w_1\rangle = |1_1\rangle$ follows from a direct computation.

\subsection{Poisson algebra}

We have hoped $\mW{\kappa}m$ can be defined in the $\bA$-form, though we do
not check it yet.
The corresponding Heisenberg algebra becomes commutative at $\ve_2 = 0$
as we have observed above from the commutation relation
\begin{equation}
   [\widetilde{P}^i_m, \widetilde{P}^i_n]
 = -m(\alpha_i,\alpha_j) \delta_{m,-n}\ve_1\ve_2.
\end{equation}
Therefore we can define the Poisson bracket at $\ve_2 = 0$ by 
\begin{equation}
  \{ \sW{\kappa}m, \sW{\lambda}n \} 
  = \left.\frac1{\ve_1\ve_2} [\mW{\kappa}m, \mW{\lambda}n]
  \right|_{\ve_2 = 0},
\end{equation}
where $\sW{\kappa}m = \left.\mW{\kappa}m\right|_{\ve_2 = 0}$. The strange
division $1/\ve_1\ve_2$ was explained before.

\begin{NB2}
  Here we must be a little careful. If we consider operators in the
  space of endomorphisms of the Verma module, $\sW{\kappa}m$ vanishes for
  $m > 0$, as it commutes with $\sW{\lambda}n$ and $\sW{\kappa}m|\ba\rangle =
  0$. Therefore the Poisson bracket does not make sense. We only have
  the well-defined operator $\left.\W{\kappa}m\right|_{\ve_2 =
    0}$. However then it is not clear why for example
  $[\left.\W{\kappa}m\right|_{\ve_1,\ve_2=0}, \W{\lambda}n]$ commutes with other
  operators.

  I think that a clearer treatment can be formulated in the framework
  of the vertex Poisson algebra.
\end{NB2}

Now we have
\begin{equation}
  \begin{split}
    &
    \left.\ve_1^n \ve_2^n 
      \langle -\ba| (\W{\ell}{1})^n (\W{\ell}{-1})^n | \ba\rangle
    \right|_{\ve_2=0}
\\
  =\; &
  \left.
  \langle - \ba |
  (\W{\ell}{1})^{n-1} 
  [\W{\ell}{1}, (\ve_1\ve_2 \W{\ell}{-1})^n] 
  | \ba\rangle\right|_{\ve_2=0} = \cdots
\\
  =\; &
  \left.
  \langle - \ba |
  [\W{\ell}{1}, [\cdots, [\W{\ell}{1}, 
  (\ve_1\ve_2 \W{\ell}{-1})^n]\cdots ]]
  | \ba\rangle\right|_{\ve_2=0}
\\
  =\; &
  \langle - \ba |
   \{ \sW{\ell}1, \{ \cdots, \{ \sW{\ell}1, 
   (\sW{\ell}{-1})^n\}\cdots \}\}
  | \ba\rangle.
  \end{split}
\end{equation}

Now
\begin{equation}
  \begin{split}
  & \{ \sW{\ell}1, \{ \sW{\ell}1, (\sW{\ell}{-1})^n \}\}
  = n \{ \sW{\ell}1, (\sW{\ell}{-1})^{n-1} \{ \sW{\ell}1, \sW{\ell}{-1}\}\}
\\
  = \; & n(n-1) (\sW{\ell}{-1})^{n-2} (\{ \sW{\ell}1, \sW{\ell}{-1}\})^2
  + (\sW{\ell}{-1})^{n-1} A_{> 0},
  \end{split}
\end{equation}
where $A_{>0} = \{ \sW{\ell}{1}, \{ \sW{\ell}1, \sW{\ell}{-1}\}\}$ is an
operator which decreases the energy, and hence acts by $0$ on
$|\ba\rangle$.
By induction, we get
\begin{equation}
     \{ \sW{\ell}1, \{ \cdots, \{ \sW{\ell}1, 
   (\sW{\ell}{-1})^n\}\cdots \}\}
   = n! (\{ \sW{\ell}1, \sW{\ell}{-1}\})^n
   + \sW{\ell}{-1} B_{>0}, 
\end{equation}
where $B_{>0}$ is an operator which acts by $0$ on $|\ba\rangle$.
Therefore
\begin{equation}
  \left.(\ve_1\ve_2)^n \langle -\ba| (\W{\ell}{1})^n (\W{\ell}{-1})^n | \ba\rangle
  \right|_{\ve_2=0}
  = n! \langle-\ba | (\{ \sW{\ell}1, \sW{\ell}{-1}\})^n | \ba\rangle.
\end{equation}
Since the energy $0$ space is $1$-dimensional, we have
\begin{equation}
  \{ \sW{\ell}1, \sW{\ell}{-1}\} | \ba\rangle
  = \Delta | \ba\rangle
\end{equation}
for some scalar $\Delta\in \Q(\ba)$. Therefore
\begin{equation}
  n! \langle-\ba | (\{ \sW{\ell}1, \sW{\ell}{-1}\})^n | \ba\rangle
  = n! \Delta^n .
\end{equation}

More generally a similar computation shows
\begin{equation}
  \ve_1^n\ve_2^n \langle -\ba|
  (\W{1}{1})^{n_1}\cdots (\W{\ell}{1})^{n_\ell}
  (\W{1}{-1})^{m_1}\cdots (\W{\ell}{-1})^{m_\ell}|\ba\rangle
\end{equation}
is written in terms of $\langle -\ba|\{ \sW{\kappa}1, \sW{\lambda}{-1}\} |
\ba\rangle$. The relation is the same as the relation between a linear
operator $A$ on $V$ and the induced operator on the symmetric power
$S^n V$. The inverse of $A$ and the induced operator on the symmetric
power is related in the same way. Therefore the properties in
\propref{prop:two} follow.
\end{NB}

%%% Local Variables: 
%%% mode: latex
%%% TeX-master: "BFN"
%%% End: 

\appendix

\section{Appendix: exactness of hyperbolic restriction}
\label{sec:append-exactn}
\newcommand\SBun{\operatorname{Bun}}
\newcommand\tcalA{\widetilde{\calA}}
\newcommand\NN{\mathbb N}
\newcommand\calU{\mathcal U}
\newcommand\calR{\mathcal R}
\newcommand\calS{\mathcal S}
\newcommand\del{\delta}
\newcommand\tilT{\widetilde T}
\newcommand\tilPhi{\widetilde\Phi}
\newcommand\CA{\mathcal A}
\newcommand\CU{\mathcal U}
\newcommand\lam{\lambda}
\newcommand\BA{\mathbb A}
\newcommand\calE{\mathcal E}
\newcommand\oS{\overline{S}}

\subsection{Zastava spaces}
Let us denote by
$\SBun_{G,B}$ the moduli space of $G$-bundles endowed with the following
structures:

a) A trivialization at the infinite line $\PP^1_{\infty} = \linf$.
\begin{NB}
    Added by H.N.
\end{NB}%

b) A $B$-structure on the horizontal line $\PP^1_h = \{ y=0\}$.
\begin{NB}
    Added by H.N.
\end{NB}%

\noindent
These two structures are required to be compatible at the intersection of $\PP^1_{\infty}$ and $\PP^1_h$ in the obvious way.

The connected components of $\SBun_{G,B}$ are numbered by positive elements of the coroot lattice of $G_{\aff}$ (cf.\ \cite[\S9]{BFG}); for such element
$\alp$ we denote by $\SBun_{G,B}^{\alp}$ the corresponding connected component.

We will also denote by $\calZ^{\alp}_G$\index{ZalphaG@$\calZ^\alpha_G$} the corresponding ``Zastava" space (a.k.a.\ ``flag Uhlenbeck space")
defined in \cite{BFG}. We are going to need the following properties of $\calZ^{\alp}_G$.
(Some of them are proved for the space
$\operatorname{QMap}(\proj^1_h,\mathscr G_{\g,\mathfrak p})$ of based
quasi-maps to a flag scheme $\mathscr G_{\g,\mathfrak p}$ of a
Kac-Moody Lie algebra $\g$ associated with its parabolic $\mathfrak
p$. Since $\calZ^{\alp}_G$ is the fiber product
$\operatorname{QMap}(\proj^1_h,\mathscr G_{\g,\mathfrak b})
\times_{\operatorname{QMap}(\proj^1_h,\mathscr G_{\g,\mathfrak p})}
\Uh{d}$ for a Borel subalgebra $\mathfrak b$ of an affine Lie
algebra $\g$ and a maximal parabolic $\mathfrak p$, we can deduce
assertions for $\mathcal Z^\alpha_G$ from those for $\operatorname{QMap}(\proj^1_h,\mathscr G_{\g,\mathfrak p})$.)
\begin{NB}
(all proved in \cite{BFG}):
\end{NB}

(Z1) $\calZ^{\alp}_G$ is an irreducible affine scheme of dimension
$2|\alpha|$
\begin{NB}
    $2d|\alp|$ ?
\end{NB}%
endowed with an action of $T\times \CC^*\times \CC^*$ which contains $\SBun_{G,B}^{\alp}$ as an open subset
(here we set $|\alp|=\sum a_i$ if $\alp=\sum a_i\alp_i$ where $\alp_i$ are the simple coroots of $G_{\aff}$).

(Z2) There is a (factorization) map $\pi^{\alp}_{\calZ}\colon
\calZ^{\alp}_G\to S^{\alp}(\AA^1_h)$. This map is $T\times \CC^*\times
\CC^*$-equivariant if we let $T\times \CC^*\times \CC^*$ act on
$S^{\alp}(\AA^1_h)$ just through the horizontal $\CC^*$ (denoted by
$\CC^*_h$) and it admits a $T\times \CC^*\times \CC^*$-equivariant
section $\iota^{\alp}$. In particular, the fibers of
$\pi^{\alp}_{\calZ}$ are stable under $T\times \CC^*_v$ where the
$\CC^*_v = \CC^*$-action comes from the vertical action on $\AA^2$.
All of these fibers have dimension $|\alpha|$.
\begin{NB}
$d|\alp|$ ?
\end{NB}%
(See Conjecture~2.27, which is reduced to Conjecture~15.3 and proved
for affine Lie algebras in \S15.6 in \cite{BFG}.)
\begin{NB}
    Added by H.N.
\end{NB}%

(Z3) Let set $\calF^{\alp}=(\pi^{\alp}_{\calZ})^{-1}(\alp\cdot 0)$. Let $\rho\colon\mathbb C^*\to\tilT=T\times \CC^*_v$ be any one-parameter subgroup which is a regular dominant coweight of $G_{\aff}$ (i.e.\ such that $\langle \rho,\beta\rangle>0$ for any affine positive
root $\beta$).
Then the corresponding $\CC^*$-action contracts $\calZ^{\alp}_G$ to
$\iota^{\alp}(S^\alp(\AA_h^1))$, and hence $\mathcal F^\alpha$ to
$\iota^\alpha(\alpha\cdot 0)$.
\begin{NB}
    Originally, it was written
    \begin{quote}
Then the corresponding $\CC^*$-action contracts $\calF^{\alp}$ to the point $\iota^{\alp}(\alp\cdot 0)$.
    \end{quote}
    This formulation, a priori, claims only $\mathcal F^\alpha\subset
    \mathcal A_{\iota^\alpha(\alpha\cdot 0)}$, instead of $\mathcal
    F^\alpha = \mathcal A_{\iota^\alpha(\alpha\cdot 0)}$. But the
    latter is proved (and used later), as explained in Misha's message
    on May 19.
\end{NB}%
(cf.\ Proposition 2.6
and Corollary 10.4
\begin{NB}
    added by H.N.
\end{NB}%
in \cite{BFG}).

(Z4) Let $\alp_0$ denote the affine simple coroot and let $d$ be the coefficient of $\alp_0$ in $\alp$ (in other
words, $d=\langle \alp,\omega_0\rangle$ where $\omega_0$ denotes the corresponding fundamental weight of $G_{\aff}$). Then
there is a (``forgetting the $B$-structure") $T\times \CC^*\times\CC^*$-equivariant map $f_{\alp}\colon \calZ^{\alp}_G\to \calU^d_G$ which fits into
a commutative diagram
$$
\begin{CD}
\calZ^{\alp}_G @> f_{\alp}>> \calU^d_G\\
@V\pi^{\alp}_{\calZ}VV @VV\pi^d_G V\\
S^{\alp}\AA^1_h @>>> S^d\AA^1_h
\end{CD}
$$
where the bottom horizontal map sends a divisor $\sum \beta_i x_i$ to $\sum \langle \beta_i,\omega_0\rangle x_i$.

%-------------------------------------------------------------------------------------------------

%---------------------------------------------------------------
\subsection{Plan of the proof}
Let us discuss our strategy for proving Theorem \ref{thm:perverse}.
As we have explained in \subsecref{subsec:hypsemismall}, 
it follows from dimension estimates of attracting and repelling
sets by using arguments similar to those of \cite{MV2}.
\begin{NB}
Original : edited on Oct.5, 2014.

Using arguments similar to those of \cite{MV2}
it is not difficult to see that Theorem \ref{thm:perverse} would follow
if we could prove that for any $x\in \calU^d_L$ the dimension of $\calA_x$ and $\calR_x$ is equal to
$\frac{\dim \calU^d_G-\dim\calU^d_L}{2}$.     
\end{NB}%
However, at the moment we do not know how to prove estimates
\begin{NB}
this     
\end{NB}%
directly. So, our actual strategy will
be slightly different. First, recall that we have
$$
(\calU^d_G)^T=S^d(\AA^2),
$$
and that we denote by $\Uh[B]d$, $\Uh[B_-]d$ the corresponding attracting and
repelling sets. Also we denote by $p\colon \calU_B^d\to S^d(\AA^2)$ the corresponding map (sometimes we shall
denote it by $p^d$ when dependence on $d$ is important). Then we are going to proceed in the following way:

1) Prove that the preimage of $S^d(\AA^1)\subset S^d(\AA^2)$ under the map
$p\colon \calU^d_B\to S^d(\AA^2)=\calU^d_{T,G}$ has dimension $\frac{\dim \calU^d_G}{2}$ (here $\AA^1\subset \AA^2$ is any line). The proof will involve some facts about the Zastava spaces from \cite{BFG}.

2) Deduce Theorem \ref{thm:perverse} for $L=T$ from 1).

3) Using Proposition \ref{prop:trans} deduce Theorem \ref{thm:perverse} for arbitrary $L$ from the case $L=T$.
%--------------------------------------------------------------------------------------------------------------------------------
\subsection{Attractors and repellents on the Uhlenbeck space: maximal torus case}
Let us first look more closely at the case when $P=B$: a Borel subgroup of $G$.
In this case $L=T$: a maximal torus of $G$.

Let us also define the set $\calS^d\subset \calU^d_G$
to be the attracting set in $\calU^d_G$ with respect to the torus $T$ to
$S^d(\AA^1_v\backslash 0)$ where $\AA^1_v$ is the vertical line. In other words,
$\calS^d=p^{-1}(S^d(\AA^1_v\backslash 0))$.

\begin{Proposition}\label{prop:sd}
We have
$$
\dim \calS^d\leq dh^{\vee}=\frac{\dim \calU^d_G}{2}.
$$
\end{Proposition}
\begin{Corollary}\label{cor:dim-cor}
Let $\AA^1\hookrightarrow \AA^2$ be any linear embedding. Then
$$
\dim p^{-1}(S^d(\AA^1))\leq\frac{\dim \calU^d_G}{2}.
$$
\end{Corollary}
Corollary \ref{cor:dim-cor} clearly follows from \ref{prop:sd}. Indeed, first of all, it is clear that it is enough to prove Corollary \ref{cor:dim-cor} when $\AA^1=\AA^1_v$. In this case, $\calS^d$ is open in $p^{-1}(S^d(\AA^1_v))$,
hence we have $\dim\calS^d\leq \dim p^{-1}(S^d(\AA^1_v))$. On the other hand, (the vertical) $\AA^1$ acts naturally on
 $p^{-1}(S^d(\AA^1_v))$ by shifts and any point of $p^{-1}(S^d(\AA^1_v))$ lies in an open subset of the form
$x(\calS^d)$ for some $x\in \AA^1$, hence the opposite inequality follows.

Let us now pass to the proof of Proposition \ref{prop:sd}.
%-------------------------------------------------------------------------------------------
\subsection{The map \texorpdfstring{$f_d$}{fd}}
We have the natural (forgetting the flag)
birational map $f_{d\delta}\colon \calZ^{d\del}_G\to \calU^d_G$, which  we shall simply denote by $f_d$. This map gives an isomorphism
between the open subset of $\calU^d_G$ consisting of (generalized) bundles which are trivial
on the horizontal $\PP^1_h$ and the open subset of $\calZ^{d\delta}_G$ consisting of (generalized) bundles which are trivial
on the horizontal $\PP^1_h$ (and then the $B$-structure on the horizontal $\PP^1_h$ is automatically trivial).
%--------------------------------------------------------------------------------------------------
\subsection{The central fiber}Recall that $\calF^{d\delta}$ denotes the preimage of $d\delta\cdot 0$ under the map
$\pi^{d\delta}_{\mathcal Z}\colon\calZ^{d\delta}_G\to S^{d\delta}(\AA^1_h)$. Again, to simplify the notation, we shall just write $\calF^d$
instead of $\calF^{d\delta}$.
% \begin{NB}
%   This is the horizontal projection.
% \end{NB}
According to (Z2),
\begin{NB}
    \cite[15.6]{BFG},
\end{NB}%
$\dim \calF^d=dh^{\vee}$.

We claim that

1) $\calS^d$ lies in the open subset of $\calU^d_G$ over which $f_d$ is an isomorphism.

2) $f_d^{-1}(\calS^d)\subset \calF^d$.

The first statement is clear, since the image of $\calS^d$ in $S^d(\AA^1_v)$ under the factorization morphism $\pi^d_v$ (to the symmetric product of the vertical line)
% \begin{NB}
%   Oct. 25, HN added an explanation.
% \end{NB}%
must
lie in $S^d(\AA^1_v\backslash 0)$. To prove the second statement, let us note that
$f_d^{-1}(\calS^d)$ %\subset \calF^d$
must lie in the
attracting set in $\calZ^{d\delta}_G$ with respect to the torus $T$ to
$f_d^{-1}(S^d(\AA^1_v\backslash 0))$. It is clear that
$f_d^{-1}(S^d(\AA^1_v\backslash 0))\subset \calF^d$
and thus the statement follows, since every fiber of the map
$\pi^{d\delta}_{\mathcal Z}\colon\calZ_G^{d\delta}\to S^{d\delta}(\AA^1_h)$ is stable under the
action of $T$.

Hence we get $\dim \calS^d\leq dh^{\vee}=\dim \calF^d$.

\subsection{Good coweights}\label{good}
%We have proved that $\dim \calS^d\leq dh^{\vee}$. We now need to show the opposite inequality (this is in fact not really necessary for our purposes but we include the proof for the sake of completeness).
%For this it is enough to show that $f_d(\calF^d)$ lies in the closure of $\calS^d$ which is equal to $p^{-1}(S^d\AA^1_v)$ (since this will imply that $f_d^{-1}(\calS^d)$ is open in $\calF^d$). To do this we would like to introduce some terminology which will also be useful for us later.

Let $X$ be an affine variety endowed with an action of $T\times \CC^*$ (here $T$ can be any torus).
Let $x$ be any $T\x\CC^*$-fixed point (in practice this point will always be unique, but this is not needed formally
for what follows) and let $Y\subset X^T$ be the $\CC^*$-attractor to $x$ inside $X^T$.
Let now $\lam:\CC^*\to T$ be any coweight. Let us denote by $\calA_{\lam}$ the attractor to $Y$ with respect to the
$\CC^*$-action given by $\lambda$. Let us also denote by $\tcalA_{\lam}$ the attractor to $x$ with respect to the
$\CC^*$-action given by the cocharacter $(\lam,1)$ of $T\times \CC^*$.

We say that $\lambda$ is {\em good} if $\calA_{\lam}=\tcalA_{\lam}$.

\begin{Lemma}\label{lem:good}
For any $\lam$ as above, the coweight $n\lam$ is good for $n\in \NN$ large enough.
\end{Lemma}
\begin{proof}
Obviously,  there exists a closed $T$-equivariant embedding of $X$ into a vector space $V$ such that
the action of $T\x\CC^*$ on $V$ is linear and such that $x$ corresponds to $0\in V$.
Then it is clear that if $\lam$ is good for $V$, then it is also good for $X$.
Hence we may assume that $X=V$.

In this case, we see that $n\lambda$
\begin{NB}
    $\lam$
\end{NB}%
is good if and only if for every weight of $T\times \CC^*$ on $V$ of the
form $(\theta,k)$ the following condition is satisfied:

%\medskip
\begin{quote}
$n\langle\lam,\theta\rangle +k>0$
\begin{NB}
$\langle\lam,\theta\rangle +k>0$
\end{NB}%
if and only if either $\langle\lam,\theta\rangle>0$, or
$\langle\lam,\theta\rangle =0$ and $k>0$.
\end{quote}

%\medskip
Now, every $n\in \NN$ such that $n|\langle\lam,\theta\rangle|>|k|$
\begin{NB}
    $|\langle\lam,\theta\rangle|>|k|$
\end{NB}%
for any $(\theta,k)$ as above such that
$\langle\lam,\theta\rangle \neq 0$ will satisfy the conditions of the Lemma.
\end{proof}

Let $\lam$ be as before and assume in addition that

(i) $x$ is the only fixed point of $\CC^*$ acting
by means of the coweight $(\lam,1)$;

(ii) $X^{\lam(\CC^*)}=X^T$

\noindent
(in this case we automatically have $(X^T)^{\CC^*}=\{x\}$).
Let us
 denote by $\tilPhi$ the hyperbolic restriction for $(\lam,1)$ (acting from sheaves on $X$ to sheaves on $\{x\}$), by $\Phi$ the hyperbolic restriction for $\lam\colon\CC^*\to T$ (acting from sheaves on $X$ to sheaves on $X^T$) and by $\Phi_0$ the hyperbolic restriction
 for the action of $\CC^*$ on $X^T$ (from sheaves on $X^T$ to sheaves on $\{x\}$). Then the definition of ``goodness" implies
\begin{Lemma}\label{lem:comp-of-hyp}
Assume that $\lam$ is good and satisfies the conditions (i) and (ii).
Then we have $\tilPhi=\Phi_0\circ\Phi$.
\end{Lemma}
%-----------------------------------------------------------------------------------------------------------------------
\subsection{Exactness of twisted hyperbolic restriction}
Let $\tilT=T\x\CC^*$ and let us make it act on $\calU^d_G$ so that the action of $\CC^*$ comes from the
hyperbolic action of $\CC^*$ on $\AA^2$ of the form $z(x,y)=(z^{-1}x,zy)$.
Note that
$(\calU^d_G)^{\tilT}$ consists of one point.

Let us fix $d$ and
let us choose a dominant regular coweight $\lam:\CC^*\to T$ which is good in the sense of Subsection \ref{good} (such $\lam$
exists because of Lemma \ref{lem:good}).
Then the fact that $\lam$ is regular implies that it satisfies the conditions (i) and (ii). Consider the corresponding functors $\tilPhi,\Phi$ and $\Phi_0$.
Obviously we have $\Phi=\Phi^d_{T,G}$, so we shall write $\tilPhi^d_{T,G}$ instead of $\tilPhi$.
Also, to emphasize the dependence on $d$ we set $\Phi^d_0$ instead of $\Phi_0$.
According to Lemma \ref{lem:comp-of-hyp} we have $\tilPhi^d_{T,G}=\Phi^d_0\circ\Phi^d_{T,G}$.

%-----------------------------------------------------------------------
\begin{Theorem}\label{thm:twistedmain}
The complex of vector spaces $\tilPhi^d_{T,G}(\IC(\calU^d_G))$ is concentrated in degree $0$.
\end{Theorem}
%-----------------------------------------------------------------------------------
\begin{proof}
We will use the same notations as before for $L=T$ replaced with $\tilT$,
such as $i_{\tilT,G}$, $j_{\tilT,G}$, $p_{\tilT,G}$,
$i^-_{\tilT,G}$, $j^-_{\tilT,G}$, $p^-_{\tilT,G}$.
The attracting set is denoted by $\CA_{\lam,\tilT,G}^d$.
According to~\cite[Theorem~1]{Braden}, the natural morphism
$(p^-_{\tilT,G})_*(j^-_{\tilT,G})^!\IC(\calU^d_G)\to
(p_{\tilT,G})_!(j_{\tilT,G})^*\IC(\calU^d_G)=\tilPhi^d_{T,G}(\IC(\calU^d_G))$
is an isomorphism. We will prove that $(p_{\tilT,G})_!(j_{\tilT,G})^*\IC(\calU^d_G)$
is concentrated in nonpositive degrees. A similar (dual) argument proves that
$(p^-_{\tilT,G})_*(j^-_{\tilT,G})^!\IC(\calU^d_G)$ is concentrated in nonnegative
degrees. In other words, we must prove that
$H^\bullet_c(\CA_{\lam,\tilT,G}^d,\IC(\calU^d_G))$ lives in nonpositive cohomological
degrees.

Now $\IC(\CU^d_G)$ is smooth along the stratification
$$
\CU^d_G=\bigsqcup\limits_{m+|\lam|=d}\SBun_G^m\times S_{\lam}(\BA^2),
$$
the dimension of
a stratum being equal to $2l(\lam)+2mh^\vee$.
Here for a partition $\lambda=(\lambda_1,\ldots,\lambda_l)$ we set
$l(\lam)=l$.
The perverse sheaf $\IC(\CU_G^d)$ lives in cohomological degrees
$\leq-2l(\lambda)-2mh^\vee$ on the stratum
$\SBun_G^m\times  S_{\lam}(\BA^2)$.
We have $\CA_{\lam,\tilT,G}^d\cap\left(\SBun_G^m\times S_{\lam}(\BA^2)\right)=
(\CA_{\lam,\tilT,G}^m\cap \SBun_G^m)\times S_{\lam}(\BA^1_v)$.
% \begin{NB}
%    This is not correct. It should be the image of
%    $\CA_{\lam,\tilT,G}^m\times S_{\lam}(\BA^1_v)$ under the finite morphism
%    $\Uh{m}\times S^{|\lambda|}\AA^2\to \Uh{d}$.
% \begin{NB2}
% I have changed the statement slightly.
% \end{NB2}
% \end{NB}%
Now it follows from Corollary \ref{cor:dim-cor} and
%the 2nd assertion of
the goodness assumption on $\lambda$ that
$\dim(\CA_{\lam,\tilT,G}^m)\leq mh^\vee$.
% \begin{NB}
%     Corollary~\ref{cor:dim-cor} is the attracting set for the {\it
%       vertical\/} lines, while Lemma~\ref{lem:twisted} is about the
%     {\it horizontal\/} lines. It does not matter, but you need to
%     change the action to $(x,y)\mapsto (t^{-1}x, ty)$.
% \begin{NB2}
% I have changed the formulation of Corollary~\ref{cor:dim-cor}.
% \end{NB2}
% \end{NB}%
Evidently, $\dim S_{\lam}(\BA^1_v)=l(\lam)$.
So the restriction of $\IC(\CU_G^d)$ to
$\CA_{\lam,\tilT,G}^d\cap\left(\SBun_G^m\times S_{\lam}(\BA^2)\right)$ lives in
degrees $\leq-2\dim\left(\CA_{\lam,\tilT,G}^d\cap
\left(\SBun_G^m\times S_{\lam}(\BA^2)\right)\right)$.
Now an application of
the Cousin spectral sequence for the stratification of $\CA_{\lam,\tilT,G}^d$ finishes
the proof.
\end{proof}

The following corollary is not needed for the rest, but we include it for the sake of completeness.
\begin{Corollary}
$\dim\calS^d=\dim p^{-1}(S^d(\AA^1))=dh^{\vee}$.
\end{Corollary}

\begin{proof}
We need to show that $\dim \CA_{\lam,\tilT,G}^d$ is at least $dh^{\vee}$. By induction on $d$ we may assume that this is true for
all $d'<d$. Assume that $\dim \CA_{\lam,\tilT,G}^d<dh^{\vee}$.
Then repeating the argument from the above proof we see that $\tilPhi^d_{T,G}(\IC(\calU^d_G))$ is concentrated in strictly negative cohomological degrees, which contradicts Theorem \ref{thm:twistedmain}.
\end{proof}

\begin{Remark}\label{rem:dimofirr}
    The above argument only shows that the dimension of the whole of $\calS^d$ is equal to $dh^{\vee}$, but doesn't show that
this is true for each of its irreducible components (however, we believe that this is true).
\end{Remark}
% \noindent
% {\bf Warning.}
%-------------------------------------------------------------------------------------------------------
\subsection{Exactness of \texorpdfstring{$\Phi_{T,G}$}{PhiT,G}}\label{exborel}
We can now show that $\Phi^d_{T,G}(\IC(\calU^d_G))$ is perverse. Indeed, using the factorization argument and induction on $d$,
we may assume that $\Phi^d_{T,G}(\IC(\calU^d_G))$ is perverse away from the main diagonal $\AA^2\subset S^d(\AA^2)$.
Since according to \cite{Braden} the complex  $\Phi^d_{T,G}(\IC(\calU^d_G))$ is semi-simple and since
it is also equivariant with respect to
the action of $\AA^2$ on $S^d(\AA^2)$ by shifts, it follows that we just need to prove that $\Phi^d_{T,G}(\IC(\calU^d_G))$
doesn't have any direct summands which are isomorphic to constant sheaves on $\AA^2$ sitting in cohomological degrees
$\neq-2$. But if such a direct summand existed, it would imply that $\Phi^d_0(\Phi^d_{T,G}(\IC(\calU^d_G)))=\tilPhi^d_{T,G}(\IC(\calU^d_G))$ has
non-zero cohomology in degree $\neq 0$, which contradicts Theorem \ref{thm:twistedmain}.
%--------------------------------------------------------------------------------------------
\subsection{Exactness of \texorpdfstring{$\Phi_{L,G}$}{PhiL,G}}
Let us now show that $\Phi_{L,G}^d(\IC(\calU^d_G))$ is perverse. Indeed, first of all, according to Braden's theorem \cite{Braden},
$\Phi_{L,G}^d(\IC(\calU^d_G))$ is a semi-simple complex, which is constructible with respect to the stratification
(\ref{eq:strat}). In other words, it is a direct sum of (possibly shifted) simple perverse
sheaves, where each such sheaf is isomorphic to the Goresky-MacPherson extension of a local system $\calE$ on
$\SBun_L^{d_1}\x S_{\lam}(\AA^2)$ for some $d_1$ and $\lam$ as in \ref{eq:strat}.

\begin{Lemma}
Any such $\calE$ is necessarily of the form $\CC_{\SBun_L^{d_1}}\boxtimes \calE'$ where $\calE'$ is some local
system on $S_{\lam}(\AA^2)$.
\end{Lemma}
\begin{proof}
To prove this it is enough to show that the restriction of $\Phi_{L,G}(\IC(\calU^d_G))$
 to $\SBun_L^{d_1}\x S^{d_2}(\AA^2)$ (here $d=d_1+d_2$) is isomorphic to the exterior tensor product of the constant sheaf
 of $\SBun_L^{d_1}$ and some complex on $S^{d_2}(\AA^2)$. Moreover, it is enough to
 construct such an isomorphism on some Zariski open subset $U$ of
$\SBun_L^{d_1}\x S^{d_2}(\AA^2)$ (this follows from the fact that a local system which is constant on a Zariski dense subset is constant
everywhere). Let us choose a projection $a:\AA^2\to \AA^1$ and let $\pi_{a,L}^{d_1}:\SBun_L^{d_1}\to S^{d_1}(\AA^1)$ be the corresponding map. Let $U$ be the open subset of $\SBun_L^{d_1}\x S^{d_2}(\AA^2)$ consisting
of pairs $(\calF,x)$ such that $\pi_{a,L}^{d_1}$ is disjoint from the projection of $x$ to $S^{d_2}(\AA^1)$. Then locally in \'etale topology near
every point of $U$ the scheme $\calU^d_G$ looks like the product $\SBun_G^{d_1}\x \calU_G^{d_2}$ and the statement follows.
\end{proof}
Now, we can finish the proof. Indeed, recall that the closure of $\SBun_L^{d_1}\x S_{\lam}(\AA^2)$ admits a finite
birational map from $\calU^{d_1}_L\x\oS^{\lam}(\AA^2)$, where $\oS^{\lam}(\AA^2)$ stands for the closure of
$S_{\lam}$ in $S^{d_2}(\AA^2)$. Thus for any $\calE$ as above we see that $\IC(\calE)$ is the direct image
of $\IC(\calU^{d_1}_L)\boxtimes \IC(\calE')$ under this map. Moreover, the complex
$\Phi_{T,L}(\IC(\calE))$ is equal to the direct image of $\Phi_{T,L}(\IC(\calU^{d_1}_L))\boxtimes \IC(\calE')$. Hence, we see
that it is perverse and non-zero. Thus, if for some $i\neq 0$ the complex $\IC(\calE)[i]$ is a direct summand of $\Phi_{L,G}(\IC(\calU^d_G))$, then $\Phi_{T,L}(\Phi_{L,G}(\calU^d_G))$ is not perverse. Since $\Phi_{T,L}\circ\Phi_{L,G}\simeq \Phi_{T,G}$, this contradicts Subsection \ref{exborel}.
\qed

Recall $\Uh[P,0]d \defeq p^{-1}(d\cdot 0)$, see~(\ref{247}).
\begin{Corollary}\label{cor:dimest}
$\dim\Uh[P,0]d\leq dh^\vee-1$.
\end{Corollary}
\begin{proof}
We will argue by induction in $d$. We assume the claim for all $d'<d$.
We know that the dual space
$(U^d)^*\simeq H_c^\bullet(p^{-1}(d\cdot0),\tilde{j}{}^*\IC(\cUh{d}))$ lives
in degree 0. We consider the Cousin spectral sequence for the stratification
$\Uh[P,0]d=\bigsqcup_{d'\leq d}(\Uh[P,0]{d'}\cap\Bun{d'})$. By the induction
assumption, all the strata for $d'<d$ contribute to nonpositive degrees
of $H_c^\bullet(p^{-1}(d\cdot0),\tilde{j}{}^*\IC(\cUh{d}))$ only.
If we had $\dim\Uh[P,0]d>dh^\vee-1$, the fundamental classes of the top
dimensional components of $\Uh[P,0]d$ would contribute to the strictly
positive degrees in $H_c^\bullet(p^{-1}(d\cdot0),\tilde{j}{}^*\IC(\cUh{d}))$,
and nothing would cancel their contribution. This would contradict to
$H_c^{>0}(p^{-1}(d\cdot0),\tilde{j}{}^*\IC(\cUh{d}))=0$.
\end{proof}

Here is a more direct proof suggested by the referee.
We choose a faithful representation $\varrho\colon G\hookrightarrow\SL(r)$.
It gives rise to a closed embedding 
$\varrho_{\mathcal U}\colon \cUh{d}\hookrightarrow\cUh[r]d$. 
We choose a dominant coweight $\check\chi$ of $T$ such that $L$ is the
centralizer of $\check\chi({\mathbb C}^\times)$. Let 
$L_{\varrho_*\check\chi}\subset P_{\varrho_*\check\chi}\subset\SL(r)$ be the
corresponding Levi and parabolic subgroups. Then
$\varrho_{\mathcal U}(\cUh[P,0]d)
\subset\cUh[P_{\varrho_*\check\chi},0]{d\phi(\varrho)}$, where $\phi(\varrho)$ 
is the Dynkin index of $\varrho$. Now 
$\cUh[r]{d\phi(\varrho)}$ is equipped with a Poisson
structure compatible with the symplectic structure of
$\cGi{r}^{d\phi(\varrho)}$. This Poisson structure has finitely many symplectic
leaves (the strata of the diagonal stratification of 
$\cUh[r]{d\phi(\varrho)}$), and the intersection of 
$\cUh[P_{\varrho_*\check\chi},0]{d\phi(\varrho)}$ with any symplectic leaf is
isotropic since the preimage of $\cUh[P_{\varrho_*\check\chi},0]{d\phi(\varrho)}$
in $\cGi{r}^{d\phi(\varrho)}$ is isotropic. Finally, 
$\varrho_{\mathcal U}\colon \cUh{d}\hookrightarrow\cUh[r]d$ induces a Poisson
structure on $\cUh{d}$ whose symplectic leaves are the strata of the diagonal
stratification of $\cUh{d}$. It follows that the intersection of
$\cUh[P,0]d$ with any symplectic leaf is isotropic, and hence
$\dim\Uh[P,0]d\leq dh^\vee-1$. \qed

\begin{NB}
    Added on Oct.~5, 2014.
\end{NB}

This is the estimate of the attracting set for the most singular point
$d\cdot 0$. The exactness also implies estimates for attracting sets
of other points, more precisely their intersection with the open locus
$\Bun{d}$.
\begin{NB}
    Corrected, Oct. 10, 2014
\end{NB}%
Since any stratum of $\Uh{d}$ is of the form $\Bun{d_1}\times
S_\lambda(\AA^2)$, we have the corresponding dimension estimate for
other strata from the perversity of $\Phi_{L,G}(\IC(\Uh{d_1}))$ for
any $d_1$.
\begin{NB}
    Added, Oct. 10, 2014.
\end{NB}%
Therefore we see that $\Phi_{L,G}$ is hyperbolic semi-small in
the sense of Definition~\ref{def:hypsemismall}.

%%% Local Variables:
%%% mode: latex
%%% TeX-master: "BFN"
%%% End:

\newcommand{\n}{\mathfrak n}
\newcommand{\p}{\mathfrak p}
\newcommand{\fl}{\mathfrak l}
\newcommand{\z}{\mathfrak z}
\section{Integral form of the \texorpdfstring{$\scW$}{W}-algebra}\label{sec:intW}

The purpose of this chapter is to introduce an $\bA$-form of the
$\scW$-algebra, generalizing the $\bA$-form $\Vir_{i,\bA}$ of the
Virasoro algebra in \subsecref{sec:Vir}, where the commutation
relations of integral generators of the Heisenberg algebra and the
Virasoro algebra are (see \eqref{eq:mPrel}, \eqref{eq:mVir})
\begin{gather}\label{eq:mPrelApp}
  [\widetilde{P}^i_m, \widetilde{P}^j_n]
  = - m\delta_{m,-n}(\alpha_i,\alpha_j){\ve_1\ve_2},
\\
  [\mL^i_m, \mL^i_n]
    = \ve_1\ve_2 \left\{(m-n) \mL^i_{m+n} +
    \left(\ve_1\ve_2 + 6(\ve_1+\ve_2)^2\right)
    \delta_{m,-n}\frac{m^3 - m}{12}\right\},\notag
\end{gather}
and they are related by
\begin{equation*}
        \mL^i_n = - \frac14 \sum_m \normal{\widetilde{P}^i_m \widetilde{P}^i_{n-m}}
  % - \frac{n}2 (\ve_1+\ve_2) \widetilde{P}^i_n
  % + \frac{(\ve_1+\ve_2)^2}4\delta_{n,0}.
  - \frac{n+1}2 (\ve_1+\ve_2) \widetilde{P}^i_n.
\end{equation*}

Let $\g$ be a complex simple Lie algebra. We do not assume $\g$ is of
type $ADE$ in this chapter. Let $(\ ,\ )$ be the normalized bilinear
form so that the square length of a long root is $2$. Let $\ell$ be
its rank and $d_1 \le \dots \le d_\ell$ be the exponents of $\g$,
counted with multiplicities. For example, $\g = \algsl_{\ell+1}$, we
have $d_1 = 1, d_2 = 2, \dots, d_\ell = \ell$. We have $d_\ell =
h^\vee - 1$.
The multiplicity of the exponent is equal to $1$, except $d_{\ell/2} =
d_{\ell/2+1} = \ell - 1$ for $D_\ell$ with $\ell$ even.
\begin{NB}
  May 13: I correct the explanation of exponents for
  $D_{\mathrm{even}}$.
\end{NB}%

\subsection{Integral form of the BRST complex}

In order to define an $\bA$-form of the $\scW$-algebra, we need to recall
briefly the BRST complex used in the definition of the $\scW$-algebra in
\cite[Ch.~15]{F-BZ}. We assume that the reader is familiar with
\cite[Ch.~15]{F-BZ}, as we skip details.

Let $\g = \mathfrak n_+ \oplus\mathfrak h\oplus\mathfrak n_-$ be the
Cartan decomposition of $\g$. Let $\Delta_\pm$ denote the set of
positive/negative roots. Let $I$ be the set of simple roots.

We consider the vertex superalgebra $C_k^\bullet(\g)$, which is the
tensor product of the affine vertex algebra $V_k(\g)$ of level $k$ and
the fermionic vertex superalgebra $\Wedge_{\mathfrak n_+}^\bullet$.
We have two anti-commuting differentials $\dst$ and $\chi$ on
$C_k^\bullet$ so that $\scW_k(\g)$ is defined as the $0^{\mathrm{th}}$
cohomology with respect to $d = \dst+\chi$.
\index{dst@$\dst$} \index{x@$\chi$}

We do not need the definition of $\dst$, $\chi$. We start with the
subcomplex $C_k^\bullet(\g)_0$ as the cohomology of $C_k^\bullet(\g)$
is a tensor product of $C_k^\bullet(\g)_0$ and another complex, whose
cohomology is trivial (see \cite[Lem.~15.2.7]{F-BZ}).

We take a basis $\{ J^a\}$ of $\g$ consisting of root vectors and
vectors $h^i$, dual to simple roots $\alpha_i$ with respect to $(\ ,\
)$. Let $c^{ab}_d$ be the structure constants of $\g$ with respect to
the basis $\{ J^a\}$.
Latin indices are used to denote arbitrary basis elements, Latin
indices with bar are used to denote elements in $\mathfrak b_- =
\mathfrak h\oplus\mathfrak n_-$. Therefore $\{ J^{\bara}
\}_{\bara\in\Delta_-\cup I}$ is a basis of $\mathfrak b_-$.
Greek indices are used to denote basis elements of $\mathfrak n_+$.
We also have a basis $\{ \psi^*_\alpha \}_{\alpha\in \Delta_+}$
of $\mathfrak n_+^*$.
We denote the corresponding fields by $\widehat J^{\bara}(z)$ and
$\psi^*_\alpha(z)$, where the former has a correction term (see
\cite[(15.2.1)]{F-BZ}).
The field $\widehat J^{\bara}(z)$ satisfies the commutation relation for
the affine Lie algebra at the level $k+h^\vee$ instead of $k$ because
of the correction terms (cf.\ \cite[(4.8.1)]{Arakawa2007}):
\begin{equation}\label{eq:65}
  [\widehat J^{\bara}(z), \widehat J^{\barb}(w)]
  = \sum_{\bar c} c_{\bar c}^{\bara\barb} \widehat J^{\bar c}(w)\delta(z-w)
  + (k+h^\vee) \partial_w \delta(z-w).
\end{equation}

Now the complex $C_k^\bullet(\g)_0$\index{Ckg@$C_k^\bullet(\g)_0$} is
spanned by monomials of the form
\begin{equation}\label{eq:mono}
  \widehat J^{\bara(1)}_{n_1}\cdots \widehat J^{\bara(r)}_{n_r}
  \psi^*_{\alpha(1),m_1}\cdots \psi^*_{\alpha(s),m_s} | 0\rangle,
\end{equation}
and the action of the differentials is given by the following formulas
{\allowdisplaybreaks
\begin{equation}\label{eq:diff}
  \begin{split}
    & [\chi,\widehat J^{\bara}(z)] = \sum_{i\in I}
    \sum_{\beta\in\Delta_+} c_{\alpha_i}^{\bara\beta} \psi_\beta^*(z),
    \\
    & [\chi,\psi_\alpha^*(z)]_+ = 0,
    \\
    & [\dst,\widehat J^\bara(z)] = \sum_{\barb,\alpha} c^{\alpha
      \bara}_{\barb} \normal{\widehat J^\barb(z)\psi_\alpha^*(z)} +
    k\sum_\alpha (J^\bara, J^\alpha) \partial_z \psi_\alpha^*(z) -
    \sum_{\alpha,\beta,b} c^{\alpha b}_\beta c^{\beta
      \bara}_b \partial_z \psi^*_\alpha(z),
    \\
    & [\dst,\psi^*_\alpha(z)]_+ = -\frac12 \sum_{\beta,\gamma}
    c^{\beta\gamma}_\alpha \psi^*_\beta(z) \psi^*_\gamma(z),
  \end{split}
\end{equation}
together} with $\chi|0\rangle = \dst|0\rangle = 0$. Here the formulas
are copied from \cite[15.2.4]{F-BZ} except that the first one is
simplified as we only consider a field for $J^\bara$ in $\mathfrak
b_-$.

The bidegree is defined by
\begin{equation}\label{bidegree}
  \begin{split}
  & \operatorname{bideg} \widehat J^\bara(z) = (-n,n),
  \\
  & \operatorname{bideg} \psi^*_\alpha(z) = (l,-l+1),
  \end{split}
\end{equation}
where $n$ is the principal gradation of $J^\bara$ and $l$ is the
height of the root $\alpha$. (See \cite[15.1.7]{F-BZ} for definitions
of the principal gradation and the height.)
Therefore $\chi$ has bidegree $(1,0)$, and $\dst$ has bidegree $(0,1)$.
We get the double complex $C_k^\bullet(\g)_0 = \bigoplus_{p,q}
C_k^{p,q}(\g)_0$. From the definition of the bidegree, we see that
$C_k^{p,q}(\g)_0 = 0$ unless $p\ge 0$, $-p\le q \le 0$.
\begin{NB}
  \begin{equation}
    \begin{CD}
      C^{0,0} @>\chi>> C^{1,0} @>\chi>> C^{2,0} @>\chi>> C^{3,0} @>\chi>> \cdots
      \\
      @. @A{\dst}AA  @A{\dst}AA @A{\dst}AA
      \\
      @. C^{1,-1} @>\chi>> C^{2,-1} @>\chi>> C^{3,-1} @>\chi>> \cdots
      \\
      @. @. @A{\dst}AA @A{\dst}AA
      \\
      @. @. C^{2,-2} @>\chi>> C^{3,-2} @>\chi>> \cdots
      \\
      @. @. @. @A{\dst}AA
    \end{CD}
  \end{equation}
\end{NB}

Now we rewrite the complex suitable for our purpose. By
\eqref{eq:level} we replace $k$ by $-(h^\vee + \ve_2/\ve_1)$.

Next let us introduce a modification $\widetilde J^\bara(z)$ of
$\widehat J^\bara(z)$, like $\widetilde P^i_m$ of $P^i_m$ in
\subsecref{sec:heis-algebra-assoc}. There is a simple recipe for
this. Reading formulas in \cite[\S15.4.10]{F-BZ}, we note that
$\widehat J^\bara(z)$ for $\bara\in I$ is denoted by $\widehat h^i(z)$
and satisfies the commutation relation
\begin{equation}\label{eq:hhatcomm}
  [\widehat{h}^i_m, \widehat{h}^j_n]
  = m \delta_{m,-n} (\alpha_i,\alpha_j)(k+h^\vee)
  \begin{NB}
    m \delta_{m,-n} (\alpha_i,\alpha_j)\times (-\frac{\ve_2}{\ve_1})
  \end{NB}%
  .
\end{equation}
See also \eqref{eq:65}.
This Heisenberg operator gives the embedding $\scW_k(\g)\to
\Heis(\mathfrak h)$. Comparing \eqref{eq:mPrelApp} with
\eqref{eq:hhatcomm}, we find that it is natural to set
\begin{equation}\label{eq:Jtilde}
  \widetilde J^\bara(z) = \ve_1 \widehat J^\bara(z).
\end{equation}

We also rescale $\chi$ by a function $\varphi$ in $\ve_1$, $\ve_2$ as
$\widetilde\chi = \varphi\chi$. Unless $\varphi$ vanishes, the
cohomology group is independent of $\varphi$. However we will
specialize $\ve_1$, $\ve_2$ to $0$, the result will be
different. Therefore the choice of $\varphi$ is important.
\begin{NB}
  In \cite[15.4.1]{F-BZ}, $k\chi$ is used. But this choice seems not
  suitable for our purpose.
\end{NB}%
Remember that our goal is to realize a generator $\mW{\kappa}n$ in
geometry.
We want to assign it with the perverse cohomological degree
$2(d_\kappa+1)$, as $\mL^i_n$ in \subsecref{sec:Vir} is of degree
$4$. This generator is a sum of a main term $X_0$ of bidegree
$(d_\kappa,-d_\kappa)$ plus correction terms $X_1$, $X_2$, \dots of
bidegree $(p,-p)$ with $0\le p < d_\kappa$ determined by the condition
$\widetilde\chi X_\kappa = -\dst X_{\kappa-1}$. (See
\cite[15.2.11]{F-BZ}.)
\begin{NB}
Slightly edited from Apr. 29 version:

  Remember that our goal is to realize generators $\mW{i}n$ to have
perverse cohomological degree $2(d_i+1)$. These generators are sum of
main terms $X_0$ of bidegree $(d_i,-d_i)$ plus correction terms $X_1$,
$X_2$, \dots of bidegree $(p,-p)$ with $0\le p < d_i$ determined by
the condition $\chi X_i = -\dst X_{i-1}$. (See \cite[15.2.11]{F-BZ}.)
\end{NB}%
Therefore we want all $X_0$, $X_1$, \dots to have the same (perverse)
cohomological degree. This is achieved if $\varphi$ is of degree
$-2$. We still have ambiguity, but look at the formulas
\eqref{eq:diff} and \eqref{eq:Jtilde}, the simplest solution is to
absorb $1/\ve_1$ in $\widetilde J^\bara(z)$ to $\widetilde\chi$, i.e.,
$\widetilde\chi = \chi/\ve_1$.  \index{xtilde@$\widetilde\chi$}

We thus arrive at the following:
{\allowdisplaybreaks
\begin{equation}\label{eq:diff2}
  \begin{split}
    & [\widetilde\chi,\widetilde J^{\bara}(z)] = \sum_{i\in I}
    \sum_{\beta\in\Delta_+} c_{\alpha_i}^{\bara\beta} \psi_\beta^*(z),
    \\
    & [\widetilde\chi,\psi_\alpha^*(z)]_+ = 0,
    \\
    & [\dst,\widetilde J^\bara(z)] =
    \begin{aligned}[t]
    & \sum_{\barb,\alpha} c^{\alpha
      \bara}_{\barb} \normal{\widetilde J^\barb(z)\psi_\alpha^*(z)}
    - (h^\vee \ve_1 + \ve_2)
    \sum_\alpha (J^\bara, J^\alpha) \partial_z \psi_\alpha^*(z)
\\
    & \qquad - \ve_1
    \sum_{\alpha,\beta,b} c^{\alpha b}_\beta c^{\beta
      \bara}_b \partial_z \psi^*_\alpha(z),
    \end{aligned}
    \\
    & [\dst,\psi^*_\alpha(z)]_+ = -\frac12 \sum_{\beta,\gamma}
    c^{\beta\gamma}_\alpha \psi^*_\beta(z) \psi^*_\gamma(z).
  \end{split}
\end{equation}
}

\begin{Definition}
  We consider an $\bA$-span of monomials of the form \eqref{eq:mono}
  replacing $\widehat J$ by $\widetilde J$. We define the
  differentials $\dst$, $\widetilde\chi$ by \eqref{eq:diff2}. We get a
  double complex $C^\bullet_{\bA}(\g)_0$ defined over
  $\bA$. \index{CAg@$C_{\bA}^\bullet(\g)_0$} Its total cohomology
  group $H^\bullet_{\bA}(\g)$ is a vertex superalgebra defined over
  $\bA$.  \index{Htild-Ag@$H^\bullet_\bA(\g)$}
\end{Definition}

The argument in the proof of \cite[Th.~15.1.9]{F-BZ} goes over $\bA$,
and we get
\begin{equation}
  H^i_{\bA}(\g) = 0\quad\text{for $i \neq 0$}.
\end{equation}
We have
\begin{equation}
  H^0_{\bA}(\g)\otimes_{\bA}\bF
  \cong H^0_{\bF}(\g),
\end{equation}
as the localization is an exact functor.
\begin{NB}
  What is the appropriate reference ?
  \begin{NB2}
    May 4: I change the explanation.
  \end{NB2}
\end{NB}%
Here $H^0_{\bF}(\g)$ is the cohomology group of the complex
$C^\bullet_{\bA}(\g)_0\otimes_{\bA}\bF$. It is isomorphic to
$\scW_k(\g)\otimes_{\CC(k)}\bF$ as $\ve_1\neq 0$ in $\bF$, where $k =
-h^\vee-\ve_2/\ve_1$ as before.

\begin{Proposition}
  $H^0_\bA(\g)$ is free over $\bA$.
\end{Proposition}

\begin{proof}
    Note that the complex $C^\bullet_{\bA}(\g)_0$ is a direct sum of
    its homogeneous components with respect to the $\Z$-gradation.
    Each component forms a subcomplex and is free of finite rank over
    $\bA$. Hence results in the homological algebra can be applied.
    Since only the $0^{\mathrm{th}}$ cohomology survives, a component
    $M$ of $H^0_\bA(\g)$ is quasi-isomorphic to a complex of
    projective modules $P^\bullet$ with $P^i = 0$ for $i < 0$. Then we
    compute $\Ext^\bullet_{\bA}(M,N)$ via $P^\bullet$ to deduce
    $\Ext^{>0}_\bA(M,N) = 0$ for any $N$. Therefore $M$ is projective.
    Since $\bA$ is a polynomial ring, $H^0_{\bA}(\g)$ is free.
    \begin{NB}
        Sasha's mail on Nov.~25.
    \end{NB}%
\end{proof}

Thus $H^0_{\bA}(\g)$ is an $\bA$-form of the $\scW$-algebra.
\begin{Definition}
  We denote $H^0_{\bA}(\g)$ by $\scW_\bA(\g)$. It is called an {\it
    $\bA$-form of the $\scW$-algebra}.
\end{Definition}
\index{WAg@$\scW_\bA(\g)$}

\begin{NB}
Older version on Apr. 29:

Note that $H^0_{\CC[\ve_1,\ve_2]}(\g)$ is a submodule of a free module, it
does not have a torsion. Therefore the homomorphism
$H^0_{\CC[\ve_1,\ve_2]}(\g)\to H^0_{\CC(\ve_1,\ve_2)}(\g)$ is injective.
Therefore $H^0_{\CC[\ve_1,\ve_2]}(\g)$ is an integral form of the
$\scW$-algebra.
\end{NB}

\begin{NB}
Added on May 11:

Let us assign
\begin{equation}
  \begin{split}
    & \operatorname{bideg} \widetilde J^\bara(z) = (-k+1,k+1),
  \\
  & \operatorname{bideg} \ve_1, \ve_2 = (1,1).
  \end{split}
\end{equation}
Then $\dst$ has bidegree $(0,1)$, and $\widetilde\chi$ has bidegree
$(0,-1)$. Therefore the first degree is preserved under $d =
\dst+\widetilde\chi$. So we take the first component to define a new
degree as follows.
\end{NB}

Let us introduce a new degree, which corresponds to the half of the
(perverse) cohomological degree in the geometric side. Let us denote
it by `$\cohdeg$'.
We set $\cohdeg |0\rangle = 0$, $\cohdeg \ve_1 = \cohdeg \ve_2 =
1$. The degree of operators $\widehat J^\bara(z)$ and
$\psi^*_\alpha(z)$ is the first component of the bidegree. Then
we put $\cohdeg \widetilde J^\bara(z) = \cohdeg \widehat J^\bara(z) +
1$ by \eqref{eq:Jtilde}.
For example, $\widetilde P^i_m$ in \subsecref{sec:heis-algebra-assoc}
is a Fourier mode of $\widetilde J^\bara(z)$ for $J^\bara =
h^i$. Therefore $\cohdeg\widetilde P^i_m = 1$.

From the definition \eqref{eq:diff2} we see that both $\widetilde\chi$
and $\dst$ have degree $0$. Therefore this degree descends to the
cohomology group $H^0_{\bA}(\g) = \scW_{\bA}(\g)$. Hence $\scW_{\bA}(\g)$ is
a graded $\bA$-module, where $\bA = \CC[\ve_1,\ve_2]$ is graded in the
same way.

Be warned that $\cohdeg$ is not a $\Z$-grading of the vertex algebra
in the sense of \cite[\S1.3.1]{F-BZ}. All Fourier modes of vertex
operators $Y(A,z)$, say $\widetilde J^\bara(z)$, have the same degree,
which is equal to the degree of the corresponding states $A =
\left.Y(A,z)|0\rangle\right|_{z=0}$. The translation operator $T$ is
of degree $0$.

\subsection{Generators
  \texorpdfstring{$\protect\mW{\kappa}n$}{Wn(k)}}\label{subsec:gen}

The $\scW$-algebra $\scW_k(\g)$ is generated by certain elements $W_\kappa$
($\kappa=1,\dots,\ell)$ in the sense of the reconstruction
theorem. (See \cite[15.1.9]{F-BZ}.) Moreover the subspace spanned by
$W_\kappa$ generates a PBW basis of $\scW_k(\g)$. (See \cite[\S3.6 and
Prop.~4.12.1]{Arakawa2007} for the meaning of this statement.)

We briefly recall the definition of $W_\kappa$ and see that their simple
modifications live in our integral form and generate a PBW base of
$\scW_{\bA}(\g)$. Let us change notation from $W_\kappa$ to $W^{(\kappa)}$ in order
to avoid a possible conflict with Fourier modes.
\index{Wzkappa@$W^{(\kappa)}$}

We have a regular nilpotent element $p_-$ in $\mathfrak n_-$ so that
$\chi$ is given by $(p_-,\bullet) = \chi(\bullet)$. (See
\cite[15.2.9]{F-BZ}.) Let $\mathfrak a_-$ be the kernel of
$\operatorname{ad} p_-$. It is a maximal abelian Lie subalgebra of $\g$.

The cohomology $H^i$ of the complex $C^\bullet_k(\g)_0$ with respect
to $\chi$ vanishes for $i\neq 0$ and $H^0$ is equal to $V(\mathfrak
a_-)$, the vertex algebra associated with $\mathfrak a_-$. It is a
commutative vertex algebra, and isomorphic to the symmetric algebra
$\operatorname{Sym}(\mathfrak a_-\otimes t^{-1}\CC[t^{-1}])$ of
$\mathfrak a_-\otimes t^{-1}\CC[t^{-1}]$. Therefore a basis of
$\mathfrak a_-$ gives a PBW base of $V(\mathfrak a_-)$.

There is a standard choice of a base of $\mathfrak a_-$. We take an
$\algsl_2$-triple $\{ p_+, p_0, p_-\}$ for $p_-$, and decompose $\g$
into a direct sum of $(2d_\kappa+1)$-dimensional representations
$R_\kappa$ ($\kappa=1,\dots,\ell)$.
We choose a decomposition for $\g=D_{\ell}$ with $\ell$ even,
$\kappa=\ell/2$, $\ell/2+1$.
We then choose a lowest weight vector $p^{(\kappa)}_-$ in $R_\kappa$. Then
$\{ p^{(\kappa)}_-\}_{\kappa=1,\dots,\ell}$ is a base of $\mathfrak
a_-$. The vectors $p^{(\kappa)}_-$ are unique up to constant multiple,
and we fix them hereafter.
\begin{NB}
  The highest constant multiple will be fixed from the Whittaker
  condition. An explanation is added here. March 11, 2014.
\end{NB}
In fact, our geometric consideration of the $\scW$-algebra will give us a
canonical choice of $p^{(\kappa)}_-$ for $\kappa = \ell$, at least up
to sign. See several paragraphs after \thmref{thm:cd}.

The same is true over $\bA$. The cohomology of $C^\bullet_\bA(\g)_0$
with respect to $\chi$ vanishes except the degree $0$, and $H^0$ is
equal to $V(\mathfrak a_-)\otimes_\CC \bA$. The PBW base is its
$\bA$-basis.

Let ${}^0\mW{\kappa}{}(z)$ be the linear combination of $\widetilde
J^\bara(z)$ corresponding to $p^{(\kappa)}_-$, and let
${}^0\mW{\kappa}{(-1)}$ be its constant part. Then
${}^0\mW{\kappa}{(-1)}|0\rangle$ is contained in the kernel of
$\widetilde\chi$.
We construct a cocycle $\mW{\kappa}{}$ with respect to $d = \dst +
\widetilde\chi$ which is the main term ${}^0\mW{\kappa}{(-1)}|0\rangle$ of
bidegree $(d_\kappa,-d_\kappa)$ plus a sum of terms of bidegree $(p,-p)$ with
$0\le p < d_\kappa$, as we mentioned above.
It is unique up to an element in $\Ker\widetilde\chi$ of a lower
degree. We fix $\mW{\kappa}{}$ hereafter. We write
\begin{equation}\label{eq:94}
  Y(\mW{\kappa}{},z) = \sum_{n\in\Z} \mW{\kappa}n z^{-n-d_\kappa-1}.
\end{equation}
\index{Wzkappatilde@$\mW{\kappa}{}$}

\begin{NB}
Comment out on May 18:

\begin{NB2}
This construction is the same for the original $\scW_k(\g)$ in
\cite[\S15.2]{F-BZ}, except that the vector $W_\kappa$ was originally
constructed from $P^{(\kappa)}(z)$ corresponding to $\widehat J^\bara(z)$
instead of $\widetilde J^\bara(z)$. Therefore we have
\begin{equation}\label{eq:stdWgen}
  \text{our $\widetilde W_\kappa$} = \text{$\ve_1$ standard $W_\kappa$}.
\end{equation}
\end{NB2}

This identification is wrong as the isomorphism $\scW_k(\g)\otimes\bF
\cong \scW_\bA(\g)\otimes\bF$ involves a normalization for the difference
between $\chi$ and $\widetilde\chi = \chi/\ve_1$. Consider the case
$\g = \algsl_2$ for example. The standard generator $\scW_1$ gives us the
Sugawara field, which satisfies the commutation relation
\begin{equation}
  \begin{split}
    [S_m, S_n] & = ({k+h^\vee}) \left( (m-n) S_{m+n} + c \delta_{m,-n}
      \frac{m^3 - m}{12} \right)
    \\
    & = -\frac{\ve_2}{\ve_1} (m-n) S_{m+n} + \cdots.
  \end{split}
\end{equation}
On the other hand, our Virasoro generators satisfy
\begin{equation}
  [\widetilde L_m, \widetilde L_n] = (m-n)\ve_1\ve_2 \widetilde L_{m+n}
  + \cdots
\end{equation}
Therefore we must have
\begin{equation}
  \widetilde L_m = {\ve_1^2}S_m.
\end{equation}
We correct the identification in the next subsection.
\end{NB}

Let us check that $\cohdeg \mW{\kappa}{} = d_\kappa + 1$. Since $\dst$ and
$\widetilde\chi$ preserve $\cohdeg$, we have $\cohdeg \mW{\kappa}{} =
\cohdeg {}^0\mW{\kappa}{}|0\rangle$.
(Remember that we modify $\chi$ to $\widetilde \chi$ so that this
is achieved.)
Now the latter does not contain $\psi^*_\alpha(z)$, its degree is
equal to the first component of the bidegree plus $1$, i.e., $d_\kappa +
1$. Thus $\cohdeg\mW{\kappa}{} = d_\kappa + 1$. This is what we want from
a geometry side.

\subsection{Grading vs filtration}\label{subsec:grfil}

Let us make the relation between $\scW_k(\g)$ and $\scW_\bA(\g)$ more
precise so that we could easily transfer computation in the literature
to our setting.

Recall that the complexes \eqref{eq:diff} and \eqref{eq:diff2} become
the same if we put $\ve_1 = (k+h^\vee)^{-1}$, $\ve_2 = - 1$ and
identify $\widetilde\chi$ (resp.\ $\widetilde J^\bara(z)$) with
$\chi/\ve_1$ (resp.\ $\ve_1\widehat J^\bara(z)$). As $H^{>0}_\bA(\g) =
0$ and $\scW_\bA(\g)$ is free, the K\"unneth spectral sequence degenerate
at $E_2$, and hence the specialization commutes with the cohomology.
\begin{NB}
    I add an explanation.
\end{NB}%
In particular, the homomorphism $\widehat J^\bara(z)\mapsto \widetilde
J^\bara(z)/\ve_1$ induces an isomorphism
\begin{equation}\label{eq:67}
  \scW_k(\g) \xrightarrow{\cong}
  \scW_\bA(\g)\otimes \bA/(\ve_1-(k+ h^\vee)^{-1}, \ve_2 + 1).
\end{equation}
\begin{NB}
    We also recover that the higher cohomology group vanishes at any
    specialization. But as the vanishing of the higher cohomology
    group over $\bA$ is the same argument as the specialization, it
    cannot be considered as a new proof.
\end{NB}%
Under this isomorphism standard generators $W^{(\kappa)}_n$ and our
$\mW{\kappa}n$ are related by
\begin{equation}\label{eq:std}
  \text{our $\mW{\kappa}n$} =
  \text{$\ve_1^{d_\kappa+1}$ standard $W^{(\kappa)}_n$},
\end{equation}
as they are defined in the same way.

\begin{NB}
  In an older version, we set $\ve_1 = 1$, $\ve_2 = -(k+h^\vee)$. Then
  $\mW{\kappa}n = \W{\kappa}n$ at the specialization, but
  $\ve_1^{d_\kappa+1}$ appears when we consider the Rees algebra.
\end{NB}

From this consideration, we can recover
\(
   \scW_\bA(\g)\otimes_{\bA}\bB_1
\)
with $\bB_1 = \CC[\ve_1] = \bA/(\ve_2 + 1)$ from $\scW_k(\g)$ as
follows.
\index{B1@$\bB_1 = \CC[\ve_1]$}
Let us consider $k$ as a variable and understand that $\scW_k(\g)$ is a
vertex algebra defined over $\CC(k)$. We identify $\CC(k) = \CC(\ve_1)$
via $\ve_1 = (k+h^\vee)^{-1}$.
Then $\scW_\bA(\g)\otimes_\bA\bB_1\otimes_{\bB_1}\CC(k)$ is isomorphic to
$\scW_k(\g)$, the cohomology of the complex over $\CC(k)$ by the K\"unneth spectral sequence as above.
\begin{NB}
    I add an explanation, and change the following slightly.
\end{NB}%
Then we have an embedding $\scW_{\bA}(\g)\otimes_{\bA}\bB_1 \to \scW_k(\g)$,
and the image is the $\bB_1$-submodule generated by
$\ve_1^{d_\kappa+1}W^{(\kappa)}_n$. We denote
$\scW_{\bA}(\g)\otimes_{\bA}\bB_1$ by $\scW_{\bB_1}(\g)$ hereafter.
\index{WBg@$\scW_{\bB_1}(\g) = \scW_{\bA}(\g)\otimes_\bA\bB_1$}

Note further that the entire $\scW_\bA(\g)$ can be recovered from
$\scW_{\bB_1}(\g)$ as follows.
Since $\scW_\bA(\g)$ is graded by $\cohdeg$, we have an induced
filtration $0 = F_{-1} \subset F_0 \subset F_1 \subset\cdots$ on
$\scW_{\bB_1}(\g)$ such that $\ve_1 F_p\subset F_{p+1}$. Then we can
recover $\scW_\bA(g)$ as the associated Rees algebra:
\begin{equation}\label{eq:Rees}
  \scW_\bA(\g) = \bigoplus_p \ve_2^p F_p.
\end{equation}
In fact, we have a natural surjective homomorphism from the left hand
side to the right, and it is also injective as $\scW_\bA(\g)$ is torsion
free over $\bB_2 = \CC[\ve_2]$.\index{B2@$\bB_2 = \CC[\ve_2]$}
\begin{NB}
  Suppose that $x\in \scW_\bA(\g)$ is mapped to zero in $\bigoplus_p
  \ve_2^p F_p$. We may assume that $x$ is of degree $p$. Since $x$ is
  zero at $\ve_2 = -1$, we have $y\in \scW_\bA(\g)$ such that $x = (\ve_2
  + 1)y$. Write $y = \sum y_q$, where $y_q$ is a homogeneous
  element. Then $\ve_2 y_{q-1} + y_q$ is $0$ if $q\neq p$ and $x$ if
  $q=p$. We have $y_q = 0$ for sufficiently large $q$. If $q\neq p$,
  we have $\ve_2 y_{q-1} = 0$, and hence $y_{q-1} = 0$ as $\scW_\bA(\g)$
  is torsion free. Hence we have $y_p = y_{p+1} = \dots = 0$. On the
  other hand, we have $y_{-1} = 0$. Hence $y_0 = 0$ unless $p=0$. We
  repeat the argument until we get $y_{p-1} = 0$. Therefore $y=0$,
  and $x=0$.
\end{NB}%
Note also the specialization at $\ve_2 = 0$ can be also recovered as
the associated graded of the filtration.

The filtration $F_\bullet$ on $\scW_{\bB_1}(\g)$ can be defined directly.
From its definition, we assign $\cohdeg
(\ve_1^{d_\kappa+1}W^{(\kappa)}_n) = d_\kappa + 1$ and $\cohdeg \ve_1
= 1$.
This gives us the filtration on $\scW_{\bB_1}(\g)$.

Let us explain how the formula for $W^{(1)}_n$ given in
\cite[(15.3.1)]{F-BZ} can be understood in our framework, for
example. The field $T(z)$ written there is already divided by
$k+h^\vee$ so that its Fourier modes gives Virasoro generators
$L_n$. Therefore $W^{(1)}_n = (k+h^\vee) L_n$ and hence $\mW{1}n =
\ve_1^2 (k+h^\vee) L_n = - \ve_1 \ve_2 L_n$. This is compatible (up to
sign) with modified Virasoro generators in \subsecref{sec:Vir}, as
$\widetilde L^{(i)}_n = \ve_1\ve_2 L^i_n$.

\begin{NB}
Be warned that this filtration is different from the standard
filtration (see \cite[\S3.5]{Arakawa2007} and [Li ?]) or one in
\cite[\S4.11]{Arakawa2007}, as we assign $\cohdeg \ve_2 = 1$.
\end{NB}%

\subsection{Specialization at \texorpdfstring{$\ve_1 = 0$}{epsilon1=0}}\label{subsec:oper}

In this subsection, we study the specialization at $\ve_1 = 0$. This
is the classical limit of the $\scW$-algebra, but it also contains
$\ve_2$ as a parameter. The relevant computation can be found in
\cite[\S15.4.1$\sim$6]{F-BZ}.

\begin{NB}
  Slightly edited on May 17.
\end{NB}
Let us set $\ve_1 = 0$ in \eqref{eq:diff2}. Since $\widetilde
J^\bara(z)$ and $\widetilde J^\barb(z)$ commute at $\ve_1 = 0$ (see
\eqref{eq:65}), the complex is identified with polynomials in the
commuting variables $\widetilde J^\bara_n$ ($n < 0$) and
anti-commuting variables $\psi^*_{\alpha,m}$ ($m\le 0$). Therefore
\begin{equation}
  C^\bullet_{\bA}(\g)_0\otimes_{\bA} \bB_2 %(\Q[\ve_1]/\ve_1\Q[\ve_1])
%  C^\bullet_{\ve_1=0}(\g)_0
  \cong
  \operatorname{Sym} \mathfrak b_-((t))/b_-[[t]]\otimes_\CC
  \Wedge^\bullet \mathfrak n_+[[t]]^* \otimes \bB_2,
\end{equation}
where $\bB_2 = \CC[\ve_2] = \bA/\ve_1\bA$.
The differential is specialized as
{\allowdisplaybreaks
\begin{equation}\label{eq:ve1=0}
  \begin{split}
    & [\widetilde\chi,\widetilde J^{\bara}(z)] = \sum_{i\in I}
    \sum_{\beta\in\Delta_+} c_{\alpha_i}^{\bara\beta} \psi_\beta^*(z),
    \\
    & [\widetilde\chi,\psi_\alpha^*(z)]_+ = 0,
    \\
    & [\dst,\widetilde J^\bara(z)] =
    \begin{aligned}[t]
    & \sum_{\barb,\alpha} c^{\alpha
      \bara}_{\barb} \widetilde J^\barb(z)\psi_\alpha^*(z)
    - \ve_2
    \sum_\alpha (J^\bara, J^\alpha) \partial_z \psi_\alpha^*(z),
    \end{aligned}
    \\
    & [\dst,\psi^*_\alpha(z)]_+ = -\frac12 \sum_{\beta,\gamma}
    c^{\beta\gamma}_\alpha \psi^*_\beta(z) \psi^*_\gamma(z),
  \end{split}
\end{equation}
where} power series in $z$ contain only terms with non-negative degrees in $z$.
This is exactly the same complex as in \cite[\S15.4.2]{F-BZ}, if we
set $\ve_2 = -1$. It is the complex at the classical limit
$k\to\infty$.

By \cite[Cor.~15.4.6]{F-BZ}, the cohomology group $H^i_{\ve_1=0}(\g)$
of this complex (at $\ve_2 = -1$) vanishes for $i\neq 0$, and
$H^0_{\ve_1 = 0}(\g)$ is isomorphic to the ring of functions on
$\mathfrak a_+[[t]]$, where $\mathfrak a_+$ is the kernel of
$\operatorname{ad}p_+$. Here $p_+$ is as in the previous subsection.

In fact, $\mathfrak a_+[[t]]$ is obtained as the quotient of the space
of connections of the form
\begin{equation}\label{eq:oper}
  \nabla = \partial_t + p_- + A(t), \qquad
  A(t)\in\mathfrak b_+[[t]],
\end{equation}
modulo the action of the gauge transformations $N_+[[t]]$. This is the
space $\operatorname{Op}_G(D)$ of $G$-opers on the formal disk $D =
\operatorname{Spec}\CC[[t]]$.
There exists a unique gauge transformation in $N_+[[t]]$ so that
$\nabla$ is transformed into the same form with $A(t)\in\mathfrak
a_+[[t]]$.

It is easy to put $\ve_2$ in this picture. The term with $\ve_2$
corresponds to the differential of the gauge transformation. Therefore
the cohomology of our complex is the ring of functions on the
quotient space of $(-\ve_2)$-connections
\begin{equation}
  \nabla = -\ve_2 \partial_t + p_- + A(t)
\end{equation}
modulo $N_+[[t]]$. It is the space of $(-\ve_2)$-opers on $D$. This
notion appears for example in \cite[\S5.2]{Opers}. See also \subsecref{sec:oppos-spectr-sequ} below.

We have a structure of a vertex Poisson algebra on $H^0_{\ve_1=0}(\g)$
by \cite[16.2.4]{F-BZ}. It is defined by renormalizing the polar part
of vertex operators
\begin{equation}
  Y_-(A,z) = \left.\frac1{\ve_1} Y_-(\widetilde A,z)\right|_{\ve_1=0}.
\end{equation}

We can further make $\ve_2 = 0$. Then we get $(p_- + \mathfrak
b_+[[t]])/N_+[[t]]$. This space is also equal to $\mathfrak a_+[[t]]$.
The proof in \cite[15.4.5]{F-BZ} works also at $\ve_2 = 0$.
In fact, the result is a consequence of a classical result of Kostant:
$(p_- + \mathfrak b_+)/N_+ \cong \mathfrak a_+$. See
\cite[\S5.4]{Opers} for further detail.
Therefore the cohomology group $H^i_{\ve_1,\ve_2=0}(\g)$ of the
complex at $\ve_1 = \ve_2 = 0$ vanishes for $i\neq 0$, and
$H^0_{\ve_1,\ve_2=0}(\g) \cong V(\mathfrak a_-)$.

\begin{NB}
Version on Apr. 29:

  We can further make $\ve_2 = 0$. Then we get $(p_- + \mathfrak
b_+[[t]])/N_+[[t]]$. I do not know the reference, but this space
should be also equal to $\mathfrak a_+[[t]]$. In fact, the same proof
as in \cite[15.4.4]{F-BZ} should work.
\end{NB}

\begin{NB}
  I do not study the Poisson structure at $\ve_2 = 0$. I guess that it
  should be something like the Kostant-Kirillov Poisson
  structure. More precisely, we first identify
  $\operatorname{Fun}(\mathfrak a_+[[t]]) =
  \operatorname{Sym}(\mathfrak a_-\otimes t^{-1}\Q[t^{-1}])$ with the
  ring $\operatorname{Fun}(\g[[t]])^{\g[[t]]}$ of $\g[[t]]$-invariant
  functions on $\g[[t]]$. And we should construct the structure of a
  vertex Poisson algebra on it. I am still have a confusion between a
  vertex Poisson algebra and the usual Poisson algebra, so I cannot
  make this more precise now.
\end{NB}

The argument for \eqref{eq:67} works also here, i.e., the specialization commutes with cohomology group. We have
\begin{equation}
    \begin{split}
        & \scW_\bA(\g)\otimes_{\bA}\bB_2 \cong H^0_{\ve_1=0}(\g),
\\
      & \scW_\bA(\g)\otimes_{\bA}\CC \cong H^0_{\ve_1,\ve_2=0}(\g)
      \cong V(\mathfrak a_-),
    \end{split}
\end{equation}
where $\bB_2 = \bA/\ve_1\bA$, $\CC = \bA/(\ve_1,\ve_2)$.

\begin{NB}
Here is an older version. (edited on Nov. 25, 2013)

Recall that the complex at $\ve_1 = 0$ is given by the tensor product
$C^\bullet_{\bA}(\g)_0\otimes_{\bA}\bB_2$. We apply the K\"unneth
formula to compute $\scW_\bA(\g)\otimes_{\bA}\bB_2$. Note that
$C^\bullet_\bA(\g)_0$ is flat over $\bB_1$ and $H^i_\bA(\g)$
vanishes for $i\neq 0$. Therefore
\begin{equation}
  \scW_\bA(\g)\otimes_{\bA}\bB_2 \cong H^0_{\ve_1=0}(\g).
\end{equation}
The same applies to the further specialization at $\ve_1 = \ve_2 =
0$. We have
\begin{equation}
  \scW_\bA(\g)\otimes_{\bA}\CC \cong H^0_{\ve_1,\ve_2=0}(\g) \cong V(\mathfrak a_-),
\end{equation}
where $\CC = \bA/(\ve_1,\ve_2)$.
\end{NB}

\subsection{The opposite spectral sequence}\label{sec:oppos-spectr-sequ}

The embedding of the $\scW$-algebra into the Heisenberg algebra is given
by considering the `opposite' spectral sequence associated with the
double complex $C^\bullet_k(\g)_0$, where the $E_1$-term is the
cohomology with respect to $\dst$. The detail is explained in
\cite[\S15.4.10]{F-BZ}, and we give a brief review in order to see that
the embedding is compatible with integral forms.

Let $\widetilde H^i_k(\g)$ be the $i^{\mathrm{th}}$ cohomology of the
complex $C^\bullet_k(\g)_0$ with respect to $\dst$. This notation is
taken from \cite{F-BZ} and has nothing to do with our notation for
elements in the integral form.
\index{Htildekg@$\widetilde H^\bullet_k(\g)$}
Let $\widehat h^i(z)$ denote $\widehat J^\bara(z)$ for $\bara = i\in
I$. Then we have
\begin{equation}
  [\dst, \widehat h^i(z)] = 0, \qquad
  [\dst, \psi^*_{\alpha_i}(z)]_+ = 0
\end{equation}
by \eqref{eq:diff}.
Therefore we have linear maps $\CC[\widehat h^i_n]_{i\in I, n <
  0}|0\rangle \to \widetilde H^0_k(\g)$, $\bigoplus_i \CC[\widehat
h^j_n]_{j\in I, n<0} \psi^*_{\alpha_i,0}|0\rangle \to \widetilde
H^1_k(\g)$ respectively. In fact, they live in the uppermost row as
$\operatorname{bideg} \widehat h^i(z) = (0,0)$, $\operatorname{bideg}
\psi^*_{\alpha_i}(z) = (1,0)$.
\begin{NB}
  The bidegree $(0,0)$ part is exactly $\CC[\widehat h^i_n]_{i\in I, n
    < 0}|0\rangle$. Therefore $\CC[\widehat h^i_n]_{i\in I, n <
    0}|0\rangle \to \widetilde H^0_k(\g)$ is injective. There are
  bidegree $(p,-p)$ parts also, but they do not contribute to
  $\widetilde H^0_k(\g)$ for generic $k$.

  We have
  \begin{equation}
    \begin{split}
    & C^{1,-1}_k(\g)_0 = \bigoplus_{i,m<0} \CC[\widehat h^j_n]_{i\in I, n\in\Z}
     \widehat f_{i,m} |0\rangle,
\\
    & C^{1,0}_k(\g)_0 = \bigoplus_{i,m\le 0} \CC[\widehat h^j_n]_{i\in I, n\in\Z}
     \psi^*_{\alpha_i,m} |0\rangle,
    \end{split}
  \end{equation}
  where $\widehat f_{i,m}$ is the Fourier mode of $\widehat
  J^\bara(z)$ corresponding to the basis element $f_i =
  f^{\alpha_i}$. We have
  \begin{equation}
    [\dst, \widehat f_i(z)]
    = \frac2{(\alpha_i,\alpha_i)}\left(
    \normal{\widehat h^i(z) \psi^*_{\alpha_i}(z)}
    + (k+h^\vee) \partial_z \psi^*_{\alpha_i}(z)
    \right).
  \end{equation}
  (See \cite[\S15.4.11]{F-BZ}.)

  The bidegree $(1,0)$ part is $\bigoplus_{i,m\le 0} \CC[\widehat
  h^j_n]_{j\in I, n<0} \psi^*_{\alpha_i,m}|0\rangle$. For generic $k$,
  $\dst$ kills the $m < 0$ part.
\end{NB}%
Then by considering the limit $k\to\infty$, one can see that both
cohomology groups are exactly the same as the above spaces
respectively if $k$ is generic.
Moreover one can identify $\widetilde H^0_k(\g)$ with the Heisenberg
vertex algebra associated with the Cartan subalgebra $\h$ of
$\mathfrak g$. This is because $\widehat h^i_n$ satisfies the
commutation relation \eqref{eq:hhatcomm}.
Modified generators $\overline{h}^i_n = \widehat{h}^i_n /
\sqrt{k+h^\vee}$ satisfy the usual commutation rule
\begin{equation}
  [\overline{h}^i_m, \overline{h}^j_n]
  = m \delta_{m,-n} (\alpha_i,\alpha_j).
\end{equation}
And $\widetilde H^1_k(\g)$ is its module. It is a direct sum of
$(\# I)$ Fock modules. The highest weights are given
by the formula
\begin{equation}
  \overline{h}^i_0 \psi^*_{\alpha_j,0}|0\rangle
  = - \frac{(\alpha_i,\alpha_j)}{\sqrt{k+h^\vee}} \psi^*_{\alpha_j,0}|0\rangle.
\end{equation}
Another differential $\chi$ induces a homomorphism $\widetilde
H^0_k(\g) \to \widetilde H^1_k(\g)$. Since $\widetilde H^1_k(\g)$
lives only at bidegree $(1,0)$, we have $\scW_k(\g) = H^0_k(\g) \cong
\Ker\chi$ for generic $k$.

Moreover $\chi$ is the sum of the residue of the field
$\psi^*_{\alpha_i}(z)$, which is given by the vertex operator in terms
of the Heisenberg algebra:
\begin{equation}
  \psi^*_{\alpha_i}(z) = V_{-\alpha_i/\sqrt{k+h^\vee}}(z)
\end{equation}
where
\begin{equation}\label{eq:vo}
  V_\lambda(z) = S_\lambda z^{\lambda b_0}
  \exp\left( -\lambda \sum_{n<0} \frac{b_n}n z^{-n} \right)
  \exp\left( -\lambda \sum_{n>0} \frac{b_n}n z^{-n} \right).
\end{equation}
This formula is given in \cite[(5.2.8)]{F-BZ}. The operator
$S_\lambda$ sends the highest weight vector $|0\rangle$ to the highest
weight vector $|\lambda\rangle$ and commutes with all $b_n$, $n\neq 0$.
And $\lambda b_n$ is
replaced by
\begin{equation}\label{eq:lb}
  \lambda b_n
  \begin{NB}
    = -\frac{\sqrt{(\alpha_i,\alpha_i)}}{\sqrt{k+h^\vee}}
      \cdot \frac{\overline{h}^i_n}{\sqrt{(\alpha_i,\alpha_i)}}
  \end{NB}%
  = - \frac{\overline{h}^i_n}{\sqrt{k+h^\vee}}
  = -\frac{\widehat{h}^i_n}{k+h^\vee},
\end{equation}
and $S_\lambda$ sends $|0\rangle$ to $\psi^*_{\alpha_i,0}|0\rangle$
here.
\begin{NB}
  The formula follows from \eqref{eq:dp} below,
  $\psi^*_{\alpha_i,n}|0\rangle = \delta_{n,0}
  \psi^*_{\alpha_i,0}|0\rangle$ ($n\ge 0$), and the commutation
  relation
  \begin{equation}
    [\overline{h}^i_n, \psi^*_{\alpha_j}(z)]
    = - z^n\frac{(\alpha_i,\alpha_j)}{\sqrt{k+h^\vee}}
  \psi^*_{\alpha_j}(z).
  \end{equation}
\end{NB}

Now we consider the cohomology group $\widetilde
H^i_{\bA}(\g)$ over $\bA$.
\index{HtildeAg@$\widetilde H^\bullet_\bA(\g)$}
The $0^{\mathrm{th}}$ cohomology $\widetilde H^0_{\bA}(\g) = \Ker\dst$
is a direct sum of $\bA[\widetilde P^i_n]_{i\in I, n< 0}$ with
bidegree $(0,0)$ and the other parts with bidegree $(p,-p)$ with $p >
0$. Here we put
\(
  \widetilde P^i_n = \ve_1 \widehat h^i_n
\)
so that they satisfy the commutation relation \eqref{eq:mPrel}.
\index{P@$\widetilde P_n^i$|textit}
\begin{NB}
  Since $k+h^\vee = - \ve_2/\ve_1$, we have $\ve_1^2(k+h^\vee) =
  -\ve_1\ve_2$. Therefore the commutation relation is compatible.
\end{NB}%
Since $\dst$ on $(p,-p)$ part is injective for generic $(\ve_1,\ve_2)$
by the above computation, it is injective as an
$\bA$-homomorphism. Therefore we have
\begin{Lemma}\label{lem:H0}
\begin{equation}\label{eq:tH0}
  \widetilde H^0_{\bA}(\g)
  = \bA[\widetilde P^i_n]_{i\in I, n< 0}|0\rangle.
\end{equation}
\end{Lemma}

This is an $\bA$-form of the Heisenberg vertex algebra, denoted by
$\Heis_\bA(\h)$ in \subsecref{sec:heis-algebra-assoc}.

We have an induced homomorphism $\scW_\bA(\g) = H^0_\bA(\g)\to \widetilde
H^0_\bA(\g)$, taking the bidegree $(0,0)$ component. It is injective
as $\Ker \dst = 0$ on $(p,-p)$ with $p > 0$. Therefore we can consider
$\scW_\bA(\g)$ as an $\bA$-submodule of $\widetilde H^0_\bA(\g)$.
We have an induced homomorphism $\widetilde\chi\colon \widetilde
H^0_{\bA}(\g) \to \widetilde H^1_{\bA}(\g)$ and the double complex
tells us that $\scW_\bA(\g)$ is contained in $\Ker\widetilde\chi$.

When we compare the embedding with the usual one $\scW_k(\g)\to
\widetilde H^0_k(\g)$ in the literature via the identification of
$\scW_k(\g)$ and $\scW_\bA(\g)$ in \subsecref{subsec:grfil}, we use the
relations
\(
  \widetilde P^i_n = \ve_1 \widehat h^i_n
\)
as before.

For example, consider $\mW{1}n$ for $\g = \algsl_2$. It is given by
\eqref{eq:mVir} up to sign,
\begin{NB}
  I do not consider the constant term ($n=0$) yet.

  Now zero mode is corrected.
\end{NB}%
and is contained in $\widetilde H^0_\bA(\g)$. The formula follows from
the computation in the literature, say \cite[\S15.4.14]{F-BZ}, with
the rule for the change of generators above.

\begin{NB}
Oct. 19 : The whole paragraph is comment out.

Let us consider $\widetilde H^1_{\bA}(\g)$ next. It is not isomorphic
to $\bigoplus_i \bA[\widetilde P^j_n]_{j\in I, n<0}
\psi^*_{\alpha_i,0}|0\rangle$, as we will see soon.
We have an induced homomorphism $\widetilde\chi\colon \widetilde
H^0_{\bA}(\g) \to \widetilde H^1_{\bA}(\g)$ and the double complex
tells us that $\scW_\bA(\g)$ is contained in $\Ker\widetilde\chi$.
If the spectral sequence degenerates at $E_2$-term, we have
$\scW_\bA(\g)\cong \Ker\widetilde\chi$. But it is probably not true as
the embedding of $\scW_k(\g)$ into the Heisenberg algebra behaves nicely
only for generic $k$.
\begin{NB2}
  If the spectral sequence may not degenerate at $E_2$-term over
  $\bA$, we need to consider higher differentials
  \begin{equation}
    \begin{split}
    & d_2\colon \Ker\widetilde \chi \to E_2^{2,-1}, \quad
\\
    & d_3\colon \Ker d_2 \to E_3^{3,-2}, \cdots,
    \end{split}
  \end{equation}
  so $\scW_{\bA}(\g)$ may be smaller than $\Ker\widetilde\chi$.
\end{NB2}%
This suggests that $\Ker\widetilde\chi$ could be larger than
$\scW_\bA(\g)$.
\begin{NB2}
Corrected. May 22:
This suggests us that $\Ker\widetilde\chi$ is not a
right object to study.
\end{NB2}

\begin{NB2}
  Let us consider the double complex at $\ve_1 = 0$. Like as in the
  identification of the cohomology with the space of opers in
  \subsecref{subsec:oper}, the cohomology with respect to $\dst$ can
  be identified with the Lie algebra cohomology for the action of
  $N_+[[t]]$ on the space of $(-\ve_2)$-connections
  \begin{equation}
    \nabla = -\ve_2 \partial_t - A(t), \qquad A(t)\in\mathfrak b_+[[t]].
  \end{equation}

  Suppose $\ve_2\neq 0$.
  By \cite[Lem.~15.4.9]{F-BZ} the action becomes free if we replace
  $N_+[[t]]$ by its normal subgroup $N_+[[t]]_0$, consisting of
  elements which are identity at $t=0$. And the quotient space is the
  space of connections
  \begin{equation}
    \nabla = -\ve_2 \partial_t - B(t), \qquad A(t)\in\h[[t]].
  \end{equation}

  From this description, we get
  \begin{equation}
    \widetilde H^\bullet_{\ve_1=0}(\g)\otimes_{\Q[\ve_2]}\Q(\ve_2)
    \cong (\operatorname{Fun}\h[[t]])\otimes_\Q
    H^\bullet(\mathfrak n_+,\Q)\otimes_\Q \Q(\ve_2)
  \end{equation}
  by the Serre-Hochschild spectral sequence for $t \mathfrak
  n_+[[t]]\subset \mathfrak n_+[[t]]$. This gives us the embedding of
  the $\scW$-algebra into the Heisenberg algebra at $\ve_1 = 0$.

  Let us explain why the action is free. Consider the equation
  \begin{equation}
     \exp(\operatorname{ad} U(t))\cdot (-\ve_2\partial_t - B(t))
     = -\ve_2\partial_t - A(t),
  \end{equation}
  where $U(t)\in t \mathfrak n_+[[t]]$, $B(t)\in \h[[t]]$. We
  write
  \begin{equation}
    U(t) = U_1 t + U_2 t^2 + \cdots, \quad
    B(t) = B_0 + B_1 t + \cdots, \quad
    A(t) = A_0 + A_1 t + \cdots.
  \end{equation}
  Then $B_i + \ve_2 U_{i+1}$ is expressed by $A_i$, $B_0$, \dots,
  $B_{i-1}$, $U_1$, \dots, $U_i$. For example, $B_0 + \ve_2 U_1 =
  A_0$. Therefore $B_i$ and $U_{i+1}$ are uniquely determined
  recursively from $A_0$, $A_1$, \dots, $A_i$.

  It is clear that this argument does not work at $\ve_2 =
  0$. Therefore the cohomology $\widetilde H^\bullet(\g)$ is
  complicated at $\ve_1 = \ve_2 = 0$.
\end{NB2}
\end{NB}

Let us look at $\widetilde H^1_\bA(\g)$ more closely.
From the definition, we have
\begin{equation}\label{eq:C1}
  \begin{split}
    & C^{1,-1}_\bA(\g)_0 = \bigoplus_{i,m<0} \bA[\widetilde P^j_n]_{j\in I,
      n< 0} \widetilde f_{i,m} |0\rangle,
    \\
    & C^{1,0}_\bA(\g)_0 = \bigoplus_{i,m\le 0} \bA[\widetilde P^j_n]_{j\in
      I, n < 0} \psi^*_{\alpha_i,m} |0\rangle,
  \end{split}
\end{equation}
where $\widetilde f_{i,m}$ is the Fourier mode of $\widetilde
J^\bara(z)$ corresponding to the basis element $f_i =
f^{\alpha_i}$. The differential $\dst\colon C^{1,-1}_\bA(\g)_0\to
C^{1,0}_\bA(\g)_0$ can be calculated from \eqref{eq:diff2}, in
particular we have
\begin{gather}
  [\dst, \widetilde P^i(z)] = 0, % \qquad
%  [\dst, \psi^*_{\alpha_i}(z)]_+ = 0
\\
  [\dst, \widetilde f_i(z)] =
  \frac2{(\alpha_i,\alpha_i)}\left(
    \normal{\widetilde{P}^i(z) \psi^*_{\alpha_i}(z)}
    - \ve_2 \partial_z \psi^*_{\alpha_i}(z)
    \right).
\end{gather}
See the formula in the middle of \cite[p.261]{F-BZ}.
From the second formula we have
\begin{equation}\label{eq:dp}
  -\ve_2 \partial_z \psi^*_{\alpha_i}(z)
  = \normal{\widetilde{P}^i(z) \psi_{\alpha_i}^*(z)}
\end{equation}
modulo $\dst$-exact term.
\begin{NB}
Alternatively we can start from \cite[(15.4.8)]{F-BZ} and use
  \begin{equation}
    \frac{\overline{h}^i(z)}\nu = \frac{\widehat{h}^i(z)}{k+h^\vee}
    = - \frac{\widetilde{P}^i(z)}{\ve_2}
  \end{equation}
to get the same result.
\end{NB}%
If $\ve_2$ would be invertible, we could replace $\psi^*_{\alpha_i,m}
|0\rangle$ with $m\neq 0$ in \eqref{eq:C1} by an element in
$\bA[\widetilde P^i_n] \psi^*_{\alpha_i,0}|0\rangle$ so that $\widetilde
H^1_{\bA}(\g)$ is isomorphic to $\bigoplus_i \bA[\widetilde P^j_n]
\psi^*_{\alpha_i,0}|0\rangle$.
As $\ve_2$ is not invertible in $\bA$, this cannot be true.

From this consideration, we set $\ve_2 = -1$ in the double complex
\eqref{eq:diff2}, and consider it over $\bB_1 = \Q[\ve_1]$, as in
\subsecref{subsec:grfil}. We denote it by $C^\bullet_{\bB_1}(\g)_0$.
This is not any loss of the information for our purpose, as
$\scW_\bA(\g)$ can be recovered from $\scW_{\bB_1}(\g)$ together with its
natural filtration, as explained in \subsecref{subsec:grfil}.

However, the higher cohomology groups $\widetilde H^{> 0}_\bA(\g)$ may
not vanish nor be free. Hence the cohomology group $\widetilde
H^\bullet_{\bB_1}(\g)$ of $C^\bullet_{\bB_1}(\g)_0$ with respect to
$\dst$ may be different from
\(
    \widetilde H^\bullet_\bA(\g)\otimes_{\bA}\bB_1.
\)
We will see that $\widetilde H^\bullet_{\bB_1}(\g)$ behaves better than
$\widetilde H^\bullet_\bA(\g)$ at $\ve_2 = 0$ below.

\begin{NB}
Here is an older version. (edited on Nov. 25, 2013)

The total cohomology group $\widetilde H^\bullet_\bA(\g)$ may not be torsion
free over $\bA$, hence they cannot be recovered from the filtration.
Here the filtration is given by $\cohdeg \widetilde P^i_n = 1$,
$\cohdeg \ve_1 = 1$. The $\bA$-module defined from $\widetilde
H^\bullet_\bB(\g)$ with the filtration by the Rees algebra
construction is possibly different from $\widetilde
H^\bullet_\bA(\g)$.
But it is fine for our purpose, as we only need $\widetilde
H^0_\bA(\g)$ and $\Ker\widetilde\chi$, which are torsion free.
\begin{NB2}
  Corrected. May 22
\end{NB2}%
The cohomology group $\widetilde H^\bullet_{\bB_1}(\g)$ of
$C^\bullet_{\bB_1}(\g)_0$ with respect to $\dst$ behaves better than
$\widetilde H^\bullet_\bA(\g)$ at $\ve_2 = 0$, as we will see later.
\end{NB}

Let us study first two terms of $\widetilde H^\bullet_{\bB_1}(\g)$.
We have $\widetilde H^0_{\bB_1}(\g)\cong \bB_1[\widetilde P^i_n]_{i\in I, n<
  0}|0\rangle$ by the same argument as in \eqref{eq:tH0}.
Let $\widetilde H^{1,0}_{\bB_1}(\g)$ be the $(1,0)$ part of the
cohomology. We do not know $\widetilde H^1_{\bB_1}(\g)\cong \widetilde
H^{1,0}_{\bB_1}(\g)$, but $\widetilde\chi$ maps $\widetilde H^0_{\bB_1}(\g)$
to $\widetilde H^{1,0}_{\bB_1}(\g)$ anyway.
From the above argument we have a surjective homomorphism $\bigoplus
\bB_1[\widetilde P^j_n]\psi^*_{\alpha_i,0}|0\rangle \to \widetilde
H^{1,0}_{\bB_1}(\g)$. It is an isomorphism for generic $\ve_1$, in other
words over $\CC(\ve_1)$. Therefore it must be injective also over
$\bB_1$. We thus get

\begin{Lemma}\label{lem:H1free}
\begin{equation}
    \begin{split}
        & \widetilde H^0_{\bB_1}(\g)\cong \bB_1[\widetilde
        P^i_n]_{i\in I, n< 0}|0\rangle,
        \\
        & \widetilde H^{1,0}_{\bB_1}(\g)\cong \bigoplus_i
        \bB_1[\widetilde P^j_n]_{j\in I, n< 0}\psi^*_{\alpha_i,0}|0\rangle.
    \end{split}
\end{equation}
\end{Lemma}

The substitution $\ve_2 = -1$ makes the vertex operator \eqref{eq:vo}
well-defined: We replace $\lambda b_n$ by
\eqref{eq:lb}, hence
\begin{equation}
  \lambda b_n
  \begin{NB}
    = - \frac{\widetilde{P}^i_n}{\ve_1(k+h^\vee)}
    = \frac{\widetilde{P}^i_n}{\ve_2}
  \end{NB}
  = - \widetilde P^i_n.
\end{equation}
\begin{NB}
  \begin{equation}
    V_\lambda(z) =
  S_{-\sqrt{\ve_1}{\alpha_i}} % z^{\lambda b_0}
  \exp\left( \sum_{n<0} \frac{\widetilde P^i_n}{n} z^{-n} \right)
  \exp\left( \sum_{n>0} \frac{\widetilde P^i_n}{n} z^{-n} \right).
  \end{equation}
\end{NB}%
The vertex operator is a homomorphism between $\bB_1$-modules.

Now we let $\ve_1 = 0$. We have the K\"unneth theorem
\begin{equation}\label{eq:Kunneth1}
  0 \to \widetilde H^n_{\bB_1}(\g)\otimes_{\bB_1}\CC
  \to H^n(C^\bullet_{\bB_1}(\g)_0\otimes_{\bB_1}\CC)
  \to \operatorname{Tor}_1^{\bB_1}(\widetilde H^{n+1}_{\bB_1}(\g), \CC)\to 0,
\end{equation}
where $\CC = \bB_1/\ve_1\bB_1$. The middle term is the cohomology at the
classical limit, and is known (see \cite[\S15.4.8]{F-BZ}). In
particular, we get
\begin{equation}\label{eq:69}
  \begin{gathered}[m]
  \CC[\widetilde P^i_n]%_{i\in I, n< 0}
  |0\rangle =
  \widetilde H^0_{\bB_1}(\g)\otimes_{\bB_1}\CC \cong
  H^0(C^\bullet_{\bB_1}(\g)_0\otimes_{\bB_1}\CC),
\\
   \bigoplus%_i
   \CC[\widetilde P^j_n]\psi^*_{\alpha_i,0}|0\rangle =
  \widetilde H^{1,0}_{\bB_1}(\g)\otimes_{\bB_1}\CC \cong
  H^1(C^{1,\bullet}_{\bB_1}(\g)_0\otimes_{\bB_1}\CC),
\\
  \widetilde H^{p+1,-p}_{\bB_1}(\g)\otimes_{\bB_1}\CC =
  H^1(C^{p+1,\bullet}_{\bB_1}(\g)_0\otimes_{\bB_1}\CC) = 0 \quad\text{for $p > 0$}.
  \end{gathered}
\end{equation}
\begin{NB}
I comment out this sentence, as it is not clear it is enough to `what'.

Therefore it is enough to study the cohomology group of
$C^\bullet_{\bB_1}(\g)_0\otimes_{\bB_1}\CC$.
\end{NB}

Next we study $\widetilde\chi$ at $\ve_1 = 0$.
Recall that $\widetilde\chi = \chi/\ve_1$, so
we need to divide $\int V_\lambda(z)$ in \eqref{eq:vo} by
$\ve_1$. We see that the induced operator
\begin{multline}
  \left.\widetilde\chi\right|_{\substack{\ve_1 = 0 \\\ve_2=-1}} \colon
  \widetilde H^0_{\bB_1}(\g)\otimes_{\bB_1}\CC = \CC[\widetilde
  P^i_n]%_{i\in I, n<0}
  | 0\rangle
\\
  \to
  \widetilde H^1_{\bB_1}(\g)\otimes_{\bB_1}\CC =
  \bigoplus%_i
  \CC[\widetilde P^j_n]%_{j\in I, n < 0}
  \psi^*_{\alpha_i,0} |0\rangle
\end{multline}
is given by the formula
\begin{equation}
  \sum_i \sum_{j=1}^\ell (\alpha_i,\alpha_j) \sum_{m\le 0}
  \mathbf V_i[m]
  \frac{\partial}{\partial \widetilde P^j_{m-1}},
\end{equation}
with
\begin{equation}
  \sum_{n\le 0} \mathbf V_i[n] z^{-n}
  = S_i \exp\left(\sum_{n<0} \frac{\widetilde P^i_n}{n} z^{-n} \right).
\end{equation}
Here the operator $S_i$ sends the highest weight vector $|0\rangle$ to
$\psi^*_{\alpha_i,0}|0\rangle$.
The point here is the commutation relation $[\widetilde P^i_m,
\widetilde P^j_n] = m \ve_1 (\alpha_i,\alpha_j) \delta_{m,-n}$ at
$\ve_2 = -1$. This vanishes at $\ve_1 = 0$, and hence only linear
terms in the expansion of the second exponential in \eqref{eq:vo}
survive.

\begin{NB}
  Let us check that the modified Virasoro elements \eqref{eq:mVir}
  at $\ve_1 = 0$ are contained in
  $\left.\Ker\widetilde\chi\right|_{\ve_1 = 0}$. This can, of course,
  be checked at generic $\ve_1$, but the computation becomes much
  easier. Also to see that our specialization $\ve_2 = -1$ is natural,
  we put it back. Let us consider the $\algsl_2$-case, so we drop the
  super/subscript $i$. We have
  \begin{equation}\label{eq:mmL}
    \begin{split}
    {\mL}_n &= - \frac12
    \sum_l \normal{\widetilde{P}_l \widetilde{P}_{n-l}}
    - (n+1) \ve_2 \widetilde{P}_n
\\
     &= - \frac12
    \sum_{n < l < 0} \widetilde{P}_l \widetilde{P}_{n-l}
    - (n+1) \ve_2 \widetilde{P}_n
    \end{split}
  \end{equation}
  as we have $\widetilde P_n = 0$ for $n\ge 0$ at $\ve_1 =
  0$. Remember that we perform the computation on the Fock space.

  The screening operator is
  \begin{equation}
    \sum_{m\le 0} \mathbf V[m] \frac{\partial}{\partial \widetilde P_{m-1}}
  \end{equation}
with
\begin{equation}\label{eq:V}
    \sum_{n\le 0} \mathbf V[n] z^{-n}
  = S \exp\left(-\sum_{n<0} \frac{\widetilde P_n}{n\ve_2} z^{-n} \right).
\end{equation}
  We have
  \begin{equation}\label{eq:tovanish}
    \begin{split}
      & \sum_{m\le 0} \mathbf V[m] \frac{\partial}{\partial \widetilde
        P_{m-1}} \mL_n
      \\
      =\; & -\frac12 \sum_{n < l < 0} \left(\mathbf V[l+1] \widetilde
        P_{n-l} + \widetilde P_l \mathbf V[n-l+1]\right) - (n+1)\ve_2
      \mathbf V[n+1]
      \\
      =\; & - \sum_{n < l < 0} \mathbf V[l+1] \widetilde P_{n-l} -
      (n+1)\ve_2 \mathbf V[n+1].
    \end{split}
  \end{equation}
We differentiate \eqref{eq:V} with respect to $z$ to get
\begin{equation}
  \begin{split}
  & - \sum_{m < 0} m \mathbf V[m] z^{-m-1}
  = S \sum_{l < 0} \frac{\widetilde P_l}{\ve_2} z^{-n-1}
  \exp\left(-\sum_{n<0} \frac{\widetilde P_n}{n\ve_2} z^{-n} \right)
\\
  =\; &
  \sum_{l < 0, k\le 0} \frac{\widetilde P_l}{\ve_2} \mathbf V[k] z^{-k-l-1}.
  \end{split}
\end{equation}
Therefore
\begin{equation}
  - m\mathbf V[m] = \sum_{\substack{l < 0, k\le 0\\ l+k=m}}
   \frac{\widetilde P_l}{\ve_2} \mathbf V[k].
\end{equation}
Hence \eqref{eq:tovanish} vanishes.

Note that \eqref{eq:mmL} can be specialized at $\ve_2 = 0$. However, I
do not see how to specialize the screening operator as we have
$\widetilde P_n/\ve_2$.
Since $\ve_2$ appears in the denominator, we loose the filtration when
we put $\ve_2 = -1$.
\end{NB}%

This computation appears in the study of the classical limit of the
$\scW$-algebra \cite[Chap.~8]{Fr}. In particular, the followings were
shown there:
\begin{itemize}
\item
\(
  \widetilde H^0_{\bB_1}(\g)\otimes_{\bB_1}\CC
\)
is isomorphic to the ring of functions on the space
$\operatorname{MOp}_G(D)_{\mathrm{gen}}$ of generic Miura
opers on the formal disk $D$.

\item Each generic Miura oper can be uniquely transformed into the
  following form
  \begin{equation}\label{eq:Miura}
    \nabla = \partial_t + p_- + \mathbf u(t), \qquad
    \mathbf u(t)\in\h[[t]].
  \end{equation}

\item The kernel of $\left.\widetilde\chi\right|_{\substack{\ve_1 = 0
      \\\ve_2=-1}}$ is isomorphic to the ring of functions on the
  space $\operatorname{Op}_G(D)$ of opers. The inclusion $\Ker
  (\left.\widetilde\chi\right|_{\substack{\ve_1 = 0 \\\ve_2=-1}})\to
  \widetilde H^0_{\bB_1}(\g)\otimes_{\bB_1}\CC$ is given by the
  forgetting morphism $\operatorname{MOp}_G(D)_{\mathrm{gen}}\to
  \operatorname{Op}_G(D)$.
\end{itemize}

We do not recall the definition of generic Miura opers here, as it is
enough to consider the space of connections of the form
\eqref{eq:Miura}. The morphism
$\operatorname{MOp}_G(D)_{\mathrm{gen}}\to \operatorname{Op}_G(D)$ is
given just by considering a connection in \eqref{eq:Miura} as a
$G$-oper.
As we have already known that $\scW_\bA(\g)$ at $\ve_1 = 0$, $\ve_2 = -1$
is the ring of functions on $\operatorname{Op}_G(D)$ in
\subsecref{subsec:oper}, we get
\begin{equation}\label{eq:iso}
  \scW_{\bB_1}(\g)\otimes_{\bB_1}\CC
  = \Ker (\left.\widetilde\chi\right|_{\substack{\ve_1 = 0 \\\ve_2=-1}}).
\end{equation}

Finally we study the filtration in the both sides of
\eqref{eq:iso}.
\begin{NB}
    Here is an older version.  Since the filtration of the right hand
    side is not obtained in that way, we must be more careful.
    (edited on Nov. 25, 2013)

Filtrations are the same on the left and right hand sides, as both
come from the specialization of the grading at $\ve_2 = -1$.

\end{NB}%
The left hand side has a filtration as it comes from the
specialization of the grading on $\scW_\bA(\g)$ at $\ve_1=0$, $\ve_2 =
-1$.
\begin{NB}
    Added on Mar.~14
\end{NB}%
On the other hand, we have filtration on $\widetilde H^0_{\bB_1}(\g)$
and $\widetilde H^0_{\bB_1}(\g)\otimes_{\bB_1}\CC$ given by $\cohdeg
\widetilde P^i_n = 1$, as they are polynomial rings (see
\lemref{lem:H1free} and \eqref{eq:69}.) Since $\widetilde
H^0_{\bA}(\g)$ is also free by Lemma~\ref{lem:H0}, the filtrations
come from the specialization. We give an induced filtration on $\Ker
(\left.\widetilde\chi\right|_{\substack{\ve_1 = 0 \\\ve_2=-1}})$ as a
subspace of $\widetilde H^0_{\bB_1}(\g)\otimes_{\bB_1}\CC$. Then
\eqref{eq:iso} respects the filtration as the inclusion
$\scW_{\bB_1}(\g)\to \widetilde H^0_{\bB_1}(\g)$ does.

On the ring of functions on $\operatorname{Op}_G(D)$,
the filtration can be understood by considering $(-\ve_2)$-opers
\cite[\S3.1.14]{BeilinsonDrinfeld} as follows.
A filtration on an algebra can be identified with a graded flat
$\CC[\ve_2]$-algebra with $\deg\ve_2 = 1$. The latter is considered as
the ring of functions on a flat affine scheme $X$ over $\AA^1 =
\operatorname{Spec}\CC[\ve_2]$ with a $\mathbb G_m$-action compatible
with the action by homotheties on $\AA^1$. The space of
$(-\ve_2)$-opers provides such a scheme, where the $\mathbb
G_m$-action is given by $\nabla\mapsto\lambda\nabla$ for
$\lambda\in\mathbb G_m$. More precisely, we need to compose it with a
gauge transformation so that the form \eqref{eq:oper} is preserved.
Since $(-\ve_2)$-opers appear at the specialization at $\ve_1=0$ in \subsecref{subsec:oper}, our filtration is given in this way.

The action is induced from the action
$\lambda\operatorname{Ad}(\lambda)$ on $\mathfrak a_+$ under
$\operatorname{Op}_G(D)\cong \mathfrak a_+[[t]]$, where
$\operatorname{Ad}(\lambda)$ is given by the $SL_2$ embedding
associated with the nilpotent element $p_-$.
It is known that the degrees of the $\mathbb G_m$-action on $\mathfrak
a_+$ are given by $d_\kappa+1$ ($\kappa=1,\dots,\ell$), hence are the
same as our `$\cohdeg$' by \subsecref{subsec:gen}.
This is another reason why we define the degree in that way.

We can define the $\mathbb G_m$-action on
$\operatorname{MOp}_G(D)_{\mathrm{gen}}$ in the same way so that the
morphism $\operatorname{MOp}_G(D)_{\mathrm{gen}}\to
\operatorname{Op}_G(D)$ is $\mathbb G_m$-equivariant. Under
$\operatorname{MOp}_G(D)_{\mathrm{gen}}\cong \h[[t]]$, it is just
homotheties on $\h$. The corresponding filtration is the same as ours.

The homomorphism between the associated graded of
$\Ker (\left.\widetilde\chi\right|_{\substack{\ve_1 = 0 \\\ve_2=-1}})$ and
$\widetilde H^0_{\bB_1}(\g)\otimes_{\bB_1}\CC$
\begin{NB}
  I think that the following original version is not clear as we first
  specialize to $\varepsilon_1 = 0$, and then we take the associated
  graded.

  The specialization at $\ve_1 = \ve_2 = 0$ is given by considering
  the associated graded of the filtration. Therefore the embedding
\begin{equation}\label{eq:emb00}
  \left.\scW_{\bA}(\g)\right|_{\ve_1,\ve_2 = 0}
  \to \left.\widetilde H^0_\bA(\g)\right|_{\ve_1,\ve_2 = 0}
\end{equation}
\end{NB}%
is induced by the morphism
\begin{multline}\label{eq:0oper}
  \left\{ \nabla = p_- + \mathbf u(t) \mid
    \mathbf u(t)\in\h[[t]]\right\}
\\
  \to
  \left\{ \nabla = p_- + A(t) \mid
  A(t)\in\mathfrak b_+[[t]] \right\}/N_+[[t]]
\end{multline}
of $0$-opers.

Let us write down the embedding of the $\scW$-algebra into the Heisenberg
algebra at $\ve_1 = \ve_2 = 0$ induced from the morphism
\eqref{eq:0oper} of $0$-opers explicitly. It is given in
\cite[\S3.3.4]{Fr}. Let $F^{(\kappa)}\in S(\h)^W$
($\kappa=1,\dots,\ell$) be generators of degree $d_\kappa+1$,
corresponding to $p^{(\kappa)}_-$ in \subsecref{subsec:gen}. We regard
it as a polynomial in $h^i$, i.e., $F^{(\kappa)}(h^i) =
F^{(\kappa)}(h^1,\dots,h^\ell)$.
Then $\mW{\kappa}n$ (at $\ve_1$, $\ve_2=0$) is given by the formula
\begin{equation}\label{eq:93}
  F^{(\kappa)}\left(\sum_{n<0} \widetilde P^{(i)}_n z^{-n-1}\right)
  = \sum_{n < 0} \mW{\kappa}n z^{-n-d_\kappa-1}.
\end{equation}
\begin{NB}
  Since the left hand side is in $V[[z]]$, we should have
  $\mW{\kappa}n = 0$ for $-d_\kappa-1 < n < 0$.
\end{NB}%

For example, we have
\begin{equation}
      {\mL}_n
     = - \frac14
    \sum_{n < l < 0} \widetilde{P}_l \widetilde{P}_{n-l}
\end{equation}
for $\algsl_2$.

\subsection{Kernel of the screening operator}

Recall that we have a natural inclusion
$\scW_{\bB_1}(\g)\subset\Ker(\left.\widetilde\chi\right|_{\ve_2=-1})$ from
the construction. They coincide for generic $\ve_1$. We prove a
stronger result.

\begin{Theorem}\label{thm:screening}
We have isomorphisms
\begin{gather}
    \label{eq:70}
        \scW_{\bB_1}(\g) \cong \Ker(\left.\widetilde\chi\right|_{\ve_2=-1}),\\
    \label{eq:66}
      \scW_\bA(\g) \cong \bigcap_i \left.\Vir_{i,\bA}
      \right|_{\ve_1\to \ve_1'}
        \otimes_\bA
      \Heis_{\bA}(\alpha_i^\perp),
\end{gather}
where $\left.\Vir_{i,\bA} \right|_{\ve_1\to \ve_1'}$ is the $\bA$-form
of the Virasoro algebra with $\ve_1$ replaced by
$\ve_1' = \frac{\ve_1(\alpha_i,\alpha_i)}2$.
Moreover \eqref{eq:70} preserves filtrations.
\end{Theorem}

\begin{NB}
    As we will see in the proof, I do not know whether \eqref{eq:66}
    remains true if we replace the right hand side by
    $\Ker\widetilde\chi$ (over $\bA$).
\end{NB}

\begin{proof}
    Let us first consider \eqref{eq:70} and denote $\widetilde\chi$ at
    $\ve_2=-1$ also by $\widetilde\chi$ for brevity:
  \begin{equation}
    \widetilde\chi\colon
    \widetilde H^0_{\bB_1}(\g) \to \widetilde H^{1,0}_{\bB_1}(\g).
  \end{equation}
  We know that both $\widetilde H^0_{\bB_1}(\g)$ and $\widetilde
  H^{1,0}_{\bB_1}(\g)$ are free over $\bB_1$ (see
  \lemref{lem:H1free}). We also know that their specialization is the
  cohomology group at $\ve_1=0$, $\ve_2=-1$ (see \eqref{eq:69}).
  Therefore we have an exact sequence
  \begin{equation}
    0 \to \Ker \widetilde\chi \otimes_{\bB_1}\CC
      \to
      \Ker (\left.\widetilde\chi\right|_{\substack{\ve_1 = 0 \\\ve_2=-1}})
      \to
      \operatorname{Tor}_1^{\bB_1}(\operatorname{Cok}
      \widetilde\chi, \CC)\to 0.
  \end{equation}
  \begin{NB2}
    This is a special case of K\"unneth formula, used in
    \eqref{eq:Kunneth1}. See note on 2013-11-20. First consider
    \begin{equation}
        \operatorname{Tor}^{\bB_1}_1(\Ima\widetilde\chi, \CC) \to
        \Ker\widetilde\chi\otimes_{\bB_1}\CC
        \to
        \widetilde H^{1,0}_{\bB_1}(\g)\otimes \CC
        \to \Ima\widetilde\chi\otimes_{\bB_1}\CC \to 0.
    \end{equation}
    Since $\widetilde H^{1,0}_{\bB_1}(\g)$ is torsion free over
    $\bB_1$, so is $\Ima\widetilde\chi$. Therefore
    $\operatorname{Tor}^{\bB_1}_1(\Ima\widetilde\chi, \CC) = 0$, and
    \begin{equation}
        \Ker\widetilde\chi\otimes_{\bB_1}\CC
        \cong\Ker\left(
          \widetilde H^{0}_{\bB_1}(\g)\otimes \CC
        \to \Ima\widetilde\chi\otimes_{\bB_1}\CC
          \right).
    \end{equation}
    Then we consider
    \begin{equation}
        0 \to
        \Ker\left(
          \widetilde H^{0}_{\bB_1}(\g)\otimes \CC
        \to \Ima\widetilde\chi\otimes_{\bB_1}\CC
          \right)
          \to
        \Ker\left(
          \widetilde H^{0}_{\bB_1}(\g)\otimes \CC
        \to \widetilde H^{1,0}_{\bB_1}(\g)\otimes \CC
          \right).
    \end{equation}
    The quotient is
    \begin{equation}
        \Ker\left(\Ima\widetilde\chi_{\bB_1}\otimes \CC
          \to \widetilde H^{1,0}_{\bB_1}(\g)\otimes \CC
          \right).
    \end{equation}
    This is nothing but $\operatorname{Tor}_1^{\bB_1}(\operatorname{Cok}
      \widetilde\chi, \CC)$ as can be seen from the exact sequence
      \begin{equation}
          0\to \Ima\widetilde\chi\to \widetilde H^{1,0}_{\bB_1}(\g)
          \to \Coker\widetilde\chi\to 0.
      \end{equation}
  \end{NB2}

  We have a homomorphism from $\scW_{\bB_1}(\g)\otimes_{\bB_1}\CC$ to the
  first term $\Ker \widetilde\chi \otimes_{\bB_1}\CC$, and its
  composition to the middle term is an isomorphism by \eqref{eq:iso}.
  Therefore we have
  \begin{equation}
    \scW_{\bB_1}(\g)\otimes_{\bB_1}\CC \cong
    \Ker \widetilde\chi \otimes_{\bB_1}\CC
    \cong
    \Ker (\left.\widetilde\chi\right|_{\substack{\ve_1 = 0 \\\ve_2=-1}}).
  \end{equation}

  Since \eqref{eq:iso} preserves the filtration, we have an induced
  isomorphism between the associated graded
  \begin{equation}\label{eq:79}
    \gr \left(\scW_{\bB_1}(\g)\otimes_{\bB_1}\CC\right)
    \cong \gr
    \left(
      \Ker \widetilde\chi \otimes_{\bB_1}\CC
    \right).
  \end{equation}

  Let $0 = F_{-1} \subset F_0 \subset F_1 \subset\cdots$ be the
  filtration on $\scW_{\bB_1}(\g)$ as before. Then the filtration on
  $\scW_{\bB_1}(\g)\otimes_{\bB_1}\CC$ is given by
  \begin{equation}
      0\subset \nicefrac{F_0}{\ve_1 \scW_{\bB_1}(\g)\cap F_0}\subset
      \nicefrac{F_1}{\ve_1 \scW_{\bB_1}(\g)\cap F_1}\subset\cdots,
  \end{equation}
  as $\scW_{\bB_1}(\g)\otimes_{\bB_1}\CC\cong \scW_{\bB_1}(\g)/\ve_1
  \scW_{\bB_1}(\g)$. From the definition of $F_p$, we have
  $\ve_1 \scW_{\bB_1}(\g)\cap F_p = \ve_1 F_{p-1}$. Therefore
  \begin{equation}
      \gr\left(\scW_{\bB_1}(\g)\otimes_{\bB_1}\CC\right)
      = \bigoplus_{p>0} \nicefrac{F_p}{\ve_1 F_{p-1} + F_{p-1}}
     \cong \nicefrac{\gr \scW_{\bB_1}(\g)}{\ve_1 \gr \scW_{\bB_1}(\g)}.
  \end{equation}
  (Here we have used $\gr W/ \ve_1 \gr W =
 \bigoplus (F_p/F_{p-1})/\ve_1 (F_{p-1}/F_{p-2})$ as $\ve_1$ shift the grading by $1$).
  The same is true for $\gr
    \left(
      \Ker \widetilde\chi \otimes_{\bB_1}\CC
    \right)
    $.

  \begin{NB}
  Recall that $\scW_{\bA}(\g)$ is the Rees algebra of $\scW_{\bB_1}(\g)$
  with respect to the filtration. Therefore the associated graded is
  given by $\otimes_{\bB_2}\bB_2/\ve_2\bB_2$. Therefore we have
  \begin{equation}\label{eq:80}
    \gr \left(\scW_{\bB_1}(\g)\otimes_{\bB_1}\CC\right)
    \cong \left(\gr \scW_{\bB_1}(\g)\right) \otimes_{\bB_1}\CC.
  \end{equation}

  We claim that the same is true for $\Ker\widetilde\chi$:
  \begin{equation}\label{eq:78}
    \gr \left(
      \Ker \widetilde\chi \otimes_{\bB_1}\CC
    \right)
    \cong
        \left(\gr \Ker \widetilde\chi\right) \otimes_{\bB_1}\CC.
  \end{equation}

  \begin{NB2}\linelabel{NB2}
    The following argument contains a serious gap. The multiplication
    by $\ve_1$ sends $F_p$ to $F_{p+1}$, and hence \eqref{eq:84} is
    not a homomorphism of $\bB_1$-modules.
  \end{NB2}

  Let us first consider the filtration on
  $\Ker\widetilde\chi\otimes_{\bB_1}\CC$. It is given as
  $(\Ker\widetilde\chi\otimes_{\bB_1}\CC)\cap F_\bullet(\widetilde
  H^0_{\bB_1}(\g)\otimes\CC)$, where filtration $F_\bullet(\widetilde
  H^0_{\bB_1}(\g)\otimes\CC)$ is the filtration on $\widetilde
  H^0_{\bB_1}(\g)\otimes\CC$, as we explained in
  \subsecref{sec:oppos-spectr-sequ}. We also have an induced
  filtration on $\Ker\widetilde\chi$ from the filtration
  $F_\bullet(\widetilde H^0_{\bB_1}(\g))$ on $\widetilde
  H^0_{\bB_1}(\g)$. Consider a homomorphism
  \begin{equation}\label{eq:84}
      \widetilde H^0_{\bB_1}(\g)
      \to \widetilde H^{1,0}_{\bB_1}(\g) \oplus
      \left(\nicefrac{\widetilde H^0_{\bB_1}(\g)}
        {F_p(\widetilde H^0_{\bB_1}(\g))}\right)
  \end{equation}
  given by the sum of $\widetilde\chi$ and the projection. The kernel
  of this homomorphism is $\Ker\widetilde\chi\cap F_p(\widetilde
  H^0_{\bB_1}(\g))$.
  Recall that $\widetilde H^{1,0}_{\bB_1}(\g)$ is torsion free over
  $\bB_1$. The second summand is also flat over $\bB_1$, as
  $F_p(\widetilde H^0_{\bB_1}(\g))$ consist of polynomials in
  $\widetilde P^i_n$ of at most degree $p$ in the description in
  \lemref{lem:H1free}. Therefore the above induces an injective
  homomorphism
  \begin{multline}
      \label{eq:77}
      \left\{
      \Ker\widetilde\chi \cap F_p(\widetilde H^0_{\bB_1}(\g))\right\}
      \otimes_{\bB_1}\CC
\\
      \to
      (\Ker\widetilde\chi\otimes_{\bB_1}\CC)\cap F_p(\widetilde
      H^0_{\bB_1}(\g)\otimes\CC).
  \end{multline}

  Let us denote the filtration of $\scW_{\bB_1}(\g)$ by
  $F_p(\scW_{\bB_1}(\g))$. Then the filtration of
  $\scW_{\bB_1}(\g)\otimes_{\bB_1}\CC$ is given by
  $F_p(\scW_{\bB_1}(\g))\otimes_{\bB_1}\CC$.
  \begin{NB2}
    I think that this is obvious, but I do not find a `word'
    explaining why it is obvious.
  \end{NB2}%
  We have an isomorphism from $F_p(\scW_{\bB_1}(\g))\otimes_{\bB_1}\CC$
  to the second term of \eqref{eq:77}, which factors through the first
  term. Therefore \eqref{eq:77} is surjective, and hence is an isomorphism.

  Now \eqref{eq:78} is a consequence of
  \begin{equation}\label{eq:68}
    \begin{NB2}
    \bigoplus_p \operatorname{Tor}^{\bB_1}_1(
    \nicefrac{\Ker\widetilde\chi\cap F_{p+1}(\widetilde
      H^0_{\bB_1}(\g))}{\Ker\widetilde\chi\cap F_p(\widetilde
      H^0_{\bB_1}(\g))}, \CC) \cong
    \end{NB2}%
    \operatorname{Tor}^{\bB_1}_1(\gr \Ker \widetilde\chi, \CC) = 0.
  \end{equation}
  Since $\gr\Ker\widetilde\chi\subset\gr \widetilde
  H^0_{\bB_1}(\g)$
  \begin{NB2}
      $(\Ker\widetilde\chi\cap F_{p+1})/(\Ker\widetilde\chi\cap F_{p})
      \to F_{p+1}/F_p$ is clearly injective.
  \end{NB2}%
  and $\gr \widetilde H^0_{\bB_1}(\g) \cong H^0_{\bB_1}(\g)$ is
  torsion free, $\gr\Ker\widetilde\chi$ is also torsion free. Therefore
  \eqref{eq:68} is true.

  Combining (\ref{eq:80}, \ref{eq:78}) with \eqref{eq:79}, we get
  \begin{equation}
    (\gr \scW_{\bB_1}(\g))\otimes_{\bB_1}\CC \cong
    (\gr \Ker\widetilde\chi)\otimes_{\bB_1}\CC.
  \end{equation}
  \end{NB}%
  By graded Nakayama's lemma, we conclude $\gr
  \scW_{\bB_1}(\g)\cong\gr\Ker\widetilde\chi$. Using it again, we get
  \eqref{eq:70}.

  Next consider \eqref{eq:66}. Since both sides are Rees algebras of
  the corresponding vertex algebras at $\ve_2 = -1$ with the induced
  filtration, it is enough to show that we have a filtration
  preserving isomorphism at $\ve_2 = -1$:
  \begin{equation}
      \label{eq:82}
      \scW_{\bB_1}(\g) \cong \bigcap_i \left.\Vir_{i,\bB_1}
        \right|_{\ve_1 \to \ve_1'}
        \otimes_{\bB_1}\Heis_{\bB_1}(\alpha_i^\perp),
  \end{equation}
  where $\Vir_{i,\bB_1}$, $\Heis_{\bB_1}(\alpha_i^\perp)$ are defined
  in an obvious manner.

  We use \eqref{eq:70} $\scW_{\bB_1}(\algsl_2) =
  \Vir_{\bB_1}\cong\Ker(\left.\widetilde\chi\right|_{\ve_2=-1})$ for
  $\g = \algsl_2$ and the observation that $\widetilde\chi$ is the sum
  of operators over $i\in I$, we see that the right hand side is
  $\Ker(\left.\widetilde\chi\right|_{\ve_2=-1})$. The substitution
  $\ve_1 \to \ve_1'=\frac{(\alpha_i,\alpha_i)\ve_1}2$ is necessary, as the
  Heisenberg commutation \eqref{eq:mPrel} involves
  $(\alpha_i,\alpha_j)$.
  Now we use \eqref{eq:70} for the original $\g$ and deduce
  \eqref{eq:82}.
\end{proof}

From this result, we extend the duality for the $\scW$-algebra in
\cite[Prop.~15.4.16]{F-BZ} from generic to arbitrary level.
\begin{Corollary}\label{cor:duality}
  Let ${}^L\g$ be the Langlands dual of $\g$. Then we have
  \begin{equation}
    \scW_\bA(\g) \cong \left.\scW_\bA({}^L\g)\right|_{\substack{\ve_1\to
        r^\vee \ve_2\\ \ve_2\to \ve_1}},
  \end{equation}
  where $r^\vee$ is the maximal number of edges connecting two
  vertices of the Dynkin diagram of $\g$ \textup(the lacing
  number\textup).
\end{Corollary}

This is because $\Vir_{i,\bA}$ is invariant under
$\ve_1\leftrightarrow \ve_2$ and $(\ve_1,\ve_2)\to (c\ve_1,c\ve_2)$
($c\in\CC^*$).

\begin{NB}
The lacing number appears from our normalization of the inner product:
\begin{equation}
  ({}^L\alpha_i,{}^L\alpha_j) = \frac1{r^\vee}
  (\alpha_i^\vee,\alpha_j^\vee) = \frac1{r^\vee}
  \frac{4(\alpha_i,\alpha_j)}{(\alpha_i,\alpha_i)(\alpha_j,\alpha_j)}.
\end{equation}
Therefore
\begin{equation}
  (\frac{\ve_1(\alpha_i,\alpha_i)}2, \ve_2)
  \to
  (\ve_2, \frac{\ve_1(\alpha_i,\alpha_i)}2)
  \to
  (\frac{2\ve_2}{(\alpha_i,\alpha_i)}, \ve_1)
  = (\frac{r^\vee({}^L\alpha_i,{}^L\alpha_i)\ve_2}{2}, \ve_1).
\end{equation}
Hence we should set
\begin{equation}
    ({}^L\ve_1, {}^L\ve_2) = (r^\vee \ve_2, \ve_1).
\end{equation}
Let us check the compatibility with \cite[Prop.~15.4.16]{F-BZ}:
\begin{equation}
    k+h^\vee = - \frac{\ve_2}{\ve_1} = -\frac{{}^L\ve_1}{r^\vee {}^L\ve_2}
    = \frac1{r^\vee}({}^Lk+{}^L h^\vee)^{-1}.
\end{equation}
\end{NB}

\begin{NB}
  It remains to study the $\bA$-form of the Verma module. I must
  switch to Arakawa's paper \cite{Arakawa2007}.
\end{NB}

\subsection{The embedding
\texorpdfstring{$\scW_{\bA}(\g)\to \scW_{\bA}(\mathfrak l)$}
{W(g) to W(l)}
}

The result in this subsection will not be used elsewhere, but shows
that the hyperbolic restriction functor $\Phi_{L,G}$ for general
$L$ corresponds to in the $\scW$-algebra side.

Let $L$ be a standard Levi subgroup of $G$ with Lie algebra $\mathfrak l$. We can write $\fl$ as
$[\fl,\fl]\oplus \z(\fl)$, where $\z(\fl)$ denotes the center of $\fl$. The above discussion can be applied
to the Lie algebra $\fl$ instead of $\g$ and we get a well-defined vertex operator algebra
$\scW_{\bA}(\fl)$ over $\bA$ and we have an embedding
$\scW_{\bA}(\fl)\hookrightarrow \Heis(\h)$. It is also clear that $\scW_{\bA}(\fl)$
is isomorphic to $\scW_{\bA}([\fl,\fl])\underset{\bA}\otimes \Heis_{\bA}(\z(\fl))$.

\begin{Theorem}\label{W-Levi}
There exists an embedding $\scW_{\bA}(\g)\to \scW_{\bA}(\fl)$ compatible with the embedding of both
algebras into $\Heis(\h)$.
\end{Theorem}
\begin{proof}
Clearly, it is enough to construct any map $\scW_{\bA}(\g)\to \scW_{\bA}(\fl)$ whose composition with the embedding
$\scW_{\bA}(\fl)\hookrightarrow \Heis(\h)$ gives the map $\scW_{\bA}(\g)\hookrightarrow \Heis(\h)$ constructed before.
To this end, we are going to construct another double complex structure on $C_{\bA}^{\bullet}(\g)_0$ (with the same
total complex).

Let $\p$ be the parabolic subalgebra containing $\fl$ and $\n_+$ and let $\n(\p)$ be its nilpotent radical. We can write $\n_+=\n_+(\fl)\oplus \n(\p)$. Accordingly, we can decompose $\chi=\chi_1+\chi_2$ where $\chi_1\in \n_+(\fl)^*$ and $\chi_2\in \n(\p)^*$.
Let $h_{\fl}\in \z(\fl)$ denote the (unique) element such that for every simple root $\alpha_i$ we have
$\text{ad}_{h_{\fl}}(e_i)=e_i$ if $e_i$ is not in $\fl$ and $\text{ad}_{h_{\fl}}(e_i)=0$ otherwise.
Now define a new grading on $C_{\bA}^{\bullet}(\g)_0$ in a way similar to~(\ref{bidegree}) but where instead of the principal gradation and the root height we use the eigenvalue with respect to $ad_{h_{\fl}}$. Then the action of $\chi_2$ has bidegree $(1,0)$
and the action of $d_{st}+\chi_1$ has bidegree $(0,1)$. In this way we get a new bicomplex structure
on $C_{\bA}^{\bullet}(\g)_0$ with the same total differential and total degree.

It is easy to see that we have $C_{\bA}^{p,q}(\g)_0=0$ unless $p\geq 0$ and $p+q\geq 0$. Note that it is no longer true that for $p=0$
the complex $C^{0,q}_{\mathbf A}({\mathfrak g})_0$
vanishes unless $q=0$;
moreover, the complex $C_{\bA}^{0,\bullet}(\g)_0$ (with respect to the differential $d_{st}+\chi_1$) is just $C_{\bA}^{\bullet}(\fl)_0$.
Thus we get a morphism $H^0(C_{\bA}^{\bullet}(\g)_0)\to H^0(C_{\bA}^{\bullet}(\fl)_0)$ by mapping every cocycle to its degree $(0,0)$-component with respect to the above grading.
\end{proof}

%%% Local Variables:
%%% mode: latex
%%% TeX-master: "BFN"
%%% End:

\bibliographystyle{myamsplain}
\bibliography{nakajima,mybib,tensor2,MO}

\def\indexname{List of notations}
\printindex
\end{document}